\documentclass[11pt]{article}
\usepackage[margin=0.9in]{geometry}
\usepackage{amsmath,amsthm,amssymb,xcolor}
\usepackage{fontspec}
\usepackage{fancyvrb}
\usepackage[colorlinks=true,linkcolor=blue!50!black,citecolor=blue!50!black,urlcolor=blue!50!black]{hyperref}
\usepackage{alltt}
\newcounter{datablock}
\newcommand{\datablock}[2]{\par\medskip\refstepcounter{datablock}\noindent\textbf{Data \thedatablock\ (#1).}\label{#2}\par\nopagebreak\smallskip}
\newcommand{\leanI}[1]{\path{#1}}% Lean names; used only in the sections on computation and formal verification
\newcommand{\NA}{N_A}

\newcommand{\cD}{\mathcal{D}}

\DefineVerbatimEnvironment{formalquote}{Verbatim}%
  {fontsize=\footnotesize,xleftmargin=0.8em,frame=leftline,framerule=0.4pt,framesep=0.5em}
\newcommand{\muM}{{\mu_M}}% the mean free count
\newcommand{\alphaV}{{\alpha_{\mathrm V}}}% the multipliers of the variance bound

\newcommand{\gammaV}{{\gamma_{\mathrm V}}}
\newcommand{\ppar}{{p_{\mathrm{par}}}}% the parent's message in a comparison of the variance bound
\newcommand{\phipar}{{\varphi_{\mathrm{par}}}}% the parent's log-mass in a comparison of the variance bound
\newcommand{\Ypar}{{Y_{\mathrm{par}}}}% the parent's children's log-mass in a comparison of the variance bound
\newcommand{\EKL}{{E_{\mathrm{KL}}}}% the first coefficient of the comparisons
\newcommand{\NF}{{N_F}}% neighbourhoods in F
\newcommand{\zresp}{{\tilde z}}% the response of the variance-ratio bound
\newcommand{\pmsg}{{p'}}% a message value in the unit inequality
\newcommand{\qval}{{\tilde q}}% a value of q in the box test
\newcommand{\binp}{{\mathrm{bin}}}% the binomial probability bin_M(j;q)
\newcommand{\tg}{{\tilde g}}% a test function of (M,K_C)
\newcommand{\Iv}{{\mathcal X}}% the statement in an Iverson bracket
\newcommand{\Ihat}{{\hat I}}% the 55 subintervals of [1/3,7/3]
\newcommand{\lamhat}{{\hat\lambda}}% their endpoints
\newcommand{\ebar}{{\bar e}}% the bound of a generic Laplace constraint
\newcommand{\xpq}{{\bar x}}% the coordinate p/q of the variance-ratio data

\newcommand{\llo}{{\lambda_{\mathrm{lo}}}}
\newcommand{\lhi}{{\lambda_{\mathrm{hi}}}}

\newcommand{\sigb}{\bar\sigma}% sqrt(v mu)
\newcommand{\aG}{a_{\mathrm G}}% e^{-1/2}/sqrt(2 pi)
\newcommand{\artanh}{\operatorname{artanh}}
\newcommand{\Def}{\operatorname{Def}}% the deficit of a spike fiber
\newcommand{\Var}{\operatorname{Var}}
\newcommand{\E}{\mathbb E}
\newcommand{\wL}{w_{\mathrm L}}
\newcommand{\wR}{w_{\mathrm R}}
\newcommand{\fL}{f_{\mathrm L}}
\newcommand{\fR}{f_{\mathrm R}}
\newcommand{\oT}{\omega^{\mathrm T}}% the log-ratios of the activity-one bound
\newcommand{\Ghat}{\hat G}
\newcommand{\Ff}{\mathsf F}% the fiber function
\newcommand{\Sl}{\mathsf S}% the Gaussian slack

\newtheorem{theorem}{Theorem}[section]
\newtheorem{conjecture}{Conjecture}[section]
\newtheorem{lemma}[theorem]{Lemma}
\newtheorem{proposition}[theorem]{Proposition}
\newtheorem{corollary}[theorem]{Corollary}
\theoremstyle{definition}
\newtheorem{definition}[theorem]{Definition}
\newtheorem{remark}[theorem]{Remark}
\newtheorem{example}[theorem]{Example}

\title{Unimodality of Forest Independence Polynomials}
\hypersetup{pdftitle={Unimodality of Forest Independence Polynomials}}
\author{Wei Li$^{1}$\qquad Kevin Vallier$^{2}$\qquad Tong Zhang$^{1}$\\[2pt]
\parbox[t]{0.55\textwidth}{\centering\normalsize
$^{1}$School of Mathematics and Statistics\\
Northwestern Polytechnical University\\
Xi'an, Shaanxi, 710129, P.~R.~China}
\hspace{0.02\textwidth}
\parbox[t]{0.36\textwidth}{\centering\normalsize
$^{2}$University of Toledo\\
Toledo, Ohio, USA}}
\date{Draft, version 2.2, 6 October 2026}

\begin{document}
\maketitle

\begin{abstract}
For a finite forest $F$ let $i_k(F)$ be the number of independent sets of $F$ with $k$ vertices. Zhang and Li proved that the
sequence $i_0(F),i_1(F),\dots,i_{\alpha(F)}(F)$ is unimodal for every finite forest $F$, which answers Erd\H{o}s
Problem~993. We give a second proof. It starts from their decomposition relative to a fixed independent set and from the
bounds of Zhang and Li and of Fang, Lu, Nevo, Yao and Zheng that confine a valley of the sequence to an explicit window
of ranks. For a forest with at least $25$ vertices, one moment argument excludes a valley at every rank of the window:
at the activity where the hard-core mean equals the rank, the size of a random independent set is a mixture of binomial
laws over an independent set of maximum weight, a valley is a moment inequality for this mixture, and it is excluded by
duality given three bounds that hold for every forest, on the variance of the number of free vertices and on its
Laplace transforms, and on the variance ratio. The variance bound is proved by hand up to finitely many interval checks and
the other two bounds are verified by computer on finite interval-arithmetic coverings; on the resulting parameter domain a
valley is excluded by exact tests on finitely many rational boxes while the mean number of free vertices is below an
explicit starting mean between $19$ and $50$, and above it by one inequality, with explicit constants, for the fibers of a
weighted valley kernel, proved by hand up to a finite list of explicit checks and averaged over the mixture. Forests with at most $24$ vertices are treated by exact
counting, by hand except for exact rational evaluations of two explicit formulas at $43$ parameter triples. No forest is
enumerated. A formal proof of the theorem in Lean~4 accompanies the paper.
\end{abstract}

\begin{quote}\small
\textbf{Provisional manuscript.} We are rewriting this manuscript before journal submission and share the current
result, its supporting computations, and its formalization for examination by the community. AI systems produced
the new mathematical arguments and formal proofs reported here and wrote most of the text, under the direction
described in Section~\ref{sec:ack}. The paper has not undergone external peer review; that section records the
human and AI contributions and the reported scope of review.
\end{quote}

\section{Introduction}\label{sec:intro}

For a graph \(G\), let \(V(G)\) and \(E(G)\) denote its vertex set and edge set, respectively. Throughout this paper all graphs considered are finite and simple. The degree of a vertex \(v\) in a graph \(G\), denoted by \(d_G(v)\), is the number of edges of \(G\) incident with \(v\). An acyclic graph is one that contains no cycles. The graph with no vertices is the null graph. A nonnull connected acyclic graph is called a tree. Acyclic graphs are called forests. For \(X\subseteq V(G)\), the subgraph of \(G\) induced by \(X\), denoted by \(G[X]\), has vertex set \(X\) and edge set consisting of all edges of \(G\) with both end vertices in \(X\). A claw is the graph with four vertices \(u,v_1,v_2,v_3\) and three edges \(uv_1,uv_2,uv_3\). A graph \(G\) is claw-free if no induced subgraph of it is a claw.

A stable set in a graph is a set of vertices no two of which are adjacent. Stable sets are also commonly known as independent sets. A stable set is maximum if the graph contains no larger stable set. The stability number of \(G\), denoted by \(\alpha (G)\), is the size of a maximum stable set in \(G\). A stable set is maximal if the set cannot be extended to a larger stable set. A tree having at most one vertex of degree \(\ge 3\) is called a spider. A graph is well-covered if all maximal independent sets are of the same size.

Let \(i_k(G)\) be the number of independent sets of size \(k\) in \(G\), where \(0\le k\le \alpha (G)\), and let \(i_0(G)=1\). The sequence \((i_0(G),i_1(G),\ldots ,i_{\alpha (G)}(G))\) is the independence sequence, or vertex independence sequence, of \(G\). The independence polynomial is \(I(G;x)=\sum_{k=0}^{\alpha (G)}i_k(G)x^k\). A sequence \((a_0,\ldots,a_m)\) is weakly increasing if \(a_0\le\cdots\le a_m\) and weakly decreasing if \(a_0\ge\cdots\ge a_m\). A sequence \((a_0,a_1,\ldots ,a_m)\) is unimodal if there is an index \(r\) such that \(a_0\le \cdots \le a_r\ge \cdots \ge a_m\). We say that a polynomial is unimodal if its sequence of coefficients is unimodal. A sequence \((a_0,\ldots ,a_m)\) of positive numbers is log-concave if \(a_k^2\ge a_{k-1}a_{k+1}\) for \(1\le k\le m-1\). A polynomial with positive coefficients is called log-concave if the sequence of its coefficients is log-concave. Any log-concave sequence of positive numbers is also unimodal. The independence polynomial was first defined by Gutman and Harary in 1983~\cite{GH}. Unimodality problems arise naturally in many branches of mathematics and have been extensively investigated.

A matching in a graph \(G\) is a set \(M\subseteq E(G)\) such that no vertex of \(G\) is incident with more than one edge of \(M\). The matching sequence of a graph \(G\) is the sequence whose \(k\)th term is the number of matchings of size \(k\) in \(G\). In 1981, Schwenk proved that the matching sequence of a graph is unimodal~\cite{Schwenk}. H. Wilf asked whether the vertex independence sequence was likewise unimodal. In 1987, Alavi, Malde, Schwenk and Erd\H{o}s showed that the independent set sequence of a graph is not in general unimodal~\cite{AMSE}. This led to the question of which classes of graphs have unimodal independent set sequences. In the same paper, Alavi, Malde, Schwenk and Erd\H{o}s posed the following question~\cite{AMSE}.

\begin{conjecture}\label{conj:993}
For trees (or perhaps forests), is the vertex independence sequence unimodal? 
\end{conjecture}

There have been numerous partial results, mostly exhibiting families of trees with unimodal independent set sequences. In 1990, Hamidoune proved that the independence polynomial of a claw-free graph is log-concave~\cite{Hamidoune}. Chudnovsky and Seymour proved that the independence polynomial of any claw-free graph has only real zeros and therefore is unimodal~\cite{CS}. Levit and Mandrescu proved that the independence polynomials of well-covered spiders and some other well-covered trees are unimodal~\cite{LMSpiders}. Li, Li, Yang and Zhang proved that all spiders with a vertex of degree at least 3 have log-concave independence polynomials~\cite{LLYZ}.

There are also results that establish monotonicity for the two ends of the independence sequence. For every forest \(F\), Levit and Mandrescu proved \(i_{\lceil (2\alpha (F)-1)/3\rceil }(F)\ge \cdots \ge i_{\alpha (F)}(F)\)~\cite{LMEnds}. For a graph \(G\) on \(n\ge1\) vertices, Basit and Galvin proved that \(i_k(G)\ge i_{k+1}(G)\) whenever \(\lceil \alpha (G)(n-1)/(\alpha (G)+n)\rceil \le k<\alpha (G)\)~\cite{BG}. For a tree \(T\) chosen uniformly from all trees with vertex set \(\{1,\ldots ,n\}\), Basit and Galvin proved that \((i_0(T),\ldots ,i_{\lfloor 0.280n\rfloor}(T))\) is weakly increasing and \((i_{\lceil 0.347n\rceil}(T),\ldots ,i_{\alpha (T)}(T))\) is weakly decreasing, with probability tending to 1 as \(n\) tends to infinity~\cite{BG}.

In 2004, Levit and Mandrescu conjectured that the independence polynomials of all forests are log-concave~\cite{LM2004}. Computations run by Radcliffe showed that all trees on 25 or fewer vertices have log-concave independent set sequences~\cite{BGHW}. Kadrawi, Levit, Yosef and Mizrachi found two trees on 26 vertices whose independence sequences are not log-concave~\cite{KLYM}. Kadrawi and Levit subsequently constructed infinite families of such trees~\cite{KL}. Galvin showed that log-concavity can fail arbitrarily far from the end of the independent set sequence~\cite{Galvin}. The failure of log-concavity does not imply a failure of unimodality. In 2026, Fang, Lu, Nevo, Yao and Zheng proved that there is a positive integer \(N_0\) such that the independence sequence of every forest with at least \(N_0\) vertices is unimodal~\cite{FLNYZ}. However, this result does not settle Conjecture~\ref{conj:993} for trees and forests of all sizes. In this paper, we prove Conjecture~\ref{conj:993}. Our main result is the following theorem.

\begin{theorem}\label{thm:main}
For every finite forest \(F\) the independence sequence \(i_0(F),i_1(F),\ldots ,i_{\alpha (F)}(F)\) is unimodal.
\end{theorem}

% BEGIN INTEGRATED CONTEXT
Zhang and Li previously gave a proof of this result~\cite{Z,ZL}. The present argument retains their decomposition of
independent sets relative to a fixed independent set, the conditional binomial mixture, and the bounds that confine a
possible valley to a window of ranks; the initial increasing segment is credited to Fang, Lu, Nevo, Yao and
Zheng~\cite{FLNYZ}. It also uses several of their counting inequalities and the forms of their analytic bounds. The
new route treats forests with at most $24$ vertices by exact counting and forests with at least $25$ vertices by a
single moment argument. The finite-order certificates of the earlier proof are not used in the counting argument,
while the analytic part still depends on explicitly recorded finite computations. The sources of the individual
inputs are credited where they are used.

The computer-assisted steps have different roles. For forests with at most $24$ vertices, exact rational evaluations
supplement the counting argument. For larger forests, finite integer and interval checks establish explicit analytic
bounds, rational box certificates cover the parameter region below the starting means, and finite checks support the
fiber inequality above them. Section~\ref{sec:computer} records the scope and arithmetic of these computations. The
accompanying Lean~4 development~\cite{Lean4}, using Mathlib~\cite{Mathlib}, is described in
Section~\ref{sec:verification}, including the distinction between the written proof and the formal proof for small
forests and the use of compiled evaluations.
% END INTEGRATED CONTEXT

Section~\ref{p1:sec:inventory} fixes the notation, recalls the decomposition of independent sets, and reduces unimodality
to the exclusion of valleys in an explicit window of ranks. Section~\ref{p2:sec:overview} develops the moment argument
for forests with at least $25$ vertices and derives Theorem~\ref{p2:thm:main} from its analytic inputs.
Section~\ref{p2:sec:inputs} establishes the variance, variance-ratio, and Laplace-transform bounds required by this
argument. Section~\ref{p2:sec:atlas} proves the no-valley inequality throughout the parameter domain, using rational
box certificates below explicit starting means and a fiber inequality above them. Section~\ref{p1:sec:overview}
treats forests with at most $24$ vertices by exact counting, completing the proof of Theorem~\ref{thm:main}.
Section~\ref{sec:computer} documents the computer-assisted steps and the supplementary data.
Section~\ref{sec:verification} describes the formal verification and its scope, and Section~\ref{sec:ack} records the
author contributions and the use of AI. Appendices~\ref{apx:p1tables} and~\ref{apx:fiber} give, respectively, the exact
evaluations for small forests and the finite checks and starting means used in the fiber inequality.

\section{Preliminaries}\label{p1:sec:inventory}

This section states the decomposition of Zhang and Li (Lemma~\ref{p1:lem:fiber}) and reduces unimodality to the
absence of a valley in the window (Corollary~\ref{p1:cor:window}); both parts of the proof start from them. Every symbol
has one meaning, except some letters to which Section~\ref{fib:sec} gives local meanings, listed at its start; a symbol
used in one section only, or inside one proof, is defined there. Subscripts, superscripts,
accents and typefaces are part of a symbol's name ($m_{25}$, $z^{\mathrm T}$, $\tilde Y$, $\mathbb E$, $\mathcal J$ differ
from the bare letters), a letter followed by an argument ($m(J)$, $I(F;x)$, $\pi(A)$, $e(F[X])$, $V(F)$) names a function,
distinct from the bare letter, the letter $k$ always denotes a rank, and the bare letters $e$ and $\pi$ denote Euler's
number and the circle constant.

\subsection{Notation}\label{sec:prelim:notation}

We use the terminology of Bondy and Murty~\cite{BM}. Graphs are finite and simple; the empty forest is allowed. For a
forest $F$, $V(F)$ and $E(F)$ are its vertex and edge sets, $n=|V(F)|$, $\NF(v)$ is the neighbourhood of a vertex $v$
and $\NF[v]=\NF(v)\cup\{v\}$ its closed neighbourhood; for $X\subseteq V(F)$, $F[X]$ is the induced subgraph, $e(F[X])$
its number of edges, and $F-X$ the forest obtained by deleting $X$. We write $\alpha(F)$ for the independence number,
$\nu(F)$ for the matching number, $\operatorname{cc}(F)$ for the number of components, $i_k=i_k(F)$ for the number of independent
sets of size $k$, and $I(F;x)=\sum_{k\ge0}i_k(F)x^k$ for the independence polynomial. We put $\binom{r}{s}=0$ unless
$r\ge0$ and $0\le s\le r$, and $[\Iv]=1$ if $\Iv$ holds and $[\Iv]=0$ otherwise.

\begin{definition}[valley]\label{p1:def:valley}
A sequence $(a_r)_{r\ge0}$ of reals has a \emph{valley} at $k\ge1$ if $a_{k-1}\ge a_k\le a_{k+1}$, that is, if
$a_k\le a_{k-1}$ and $a_k\le a_{k+1}$; the valley is \emph{strict} if both inequalities are strict. The sequence is
\emph{unimodal} if $a_0\le\dots\le a_j\ge a_{j+1}\ge\cdots$ for some $j$.
\end{definition}

Excluding a valley at $k$ means that the sequence rises strictly into $k$ or falls strictly out of $k$.

\paragraph{The hard-core measure.}
For $\lambda>0$ the hard-core measure $P_\lambda$ on the independent sets of $F$ (the repulsive lattice gas
of~\cite{SS} on $F$) gives an independent set $S$ the probability $\lambda^{|S|}/I(F;\lambda)$; $\mathbb{E}_\lambda$ and
$\operatorname{Var}_\lambda$ are the expectation and the variance under $P_\lambda$, and $S$ denotes the random
independent set. Thus $P_\lambda(|S|=k)=\lambda^ki_k/I(F;\lambda)$. We write $\mu_F(\lambda)=\mathbb{E}_\lambda|S|$
for the hard-core mean and $q=q(\lambda)=\lambda/(1+\lambda)$.

\subsection{The decomposition of Zhang and Li}\label{p1:subsec:fiber}

Let $A$ be an independent set of $F$ and $C=V(F)\setminus A$. For $J\subseteq C$ let $\NA(J)$ be the set of vertices
of $A$ with a neighbour in $J$ (the \emph{blockers} of $J$), and $m(J)=|A\setminus\NA(J)|$, the number of vertices of
$A$ with no neighbour in $J$. For integers $y,m\ge0$ the \emph{fiber numbers} of $A$ are
\begin{equation}\label{p1:eq:inventory}
N(y,m)=\#\{J\subseteq C:\ J\text{ independent},\ |J|=y,\ m(J)=m\}.
\end{equation}
Under $P_\lambda$, the \emph{free count} is $M=m(S\cap C)$: the number of vertices of $A$ with no neighbour in the
random set $S$ (a vertex of $A$ has all its neighbours in $C$).

\begin{lemma}[the decomposition; Zhang and Li {\cite[Section~2.1, (2) to (3)]{ZL}}; also {\cite[Section~2.1, (2) to (3)]{Z}}]\label{p1:lem:fiber}
For every independent set $A$ of $F$, with $J$ ranging over the independent subsets of $C$,
\begin{equation}\label{p1:eq:fiber}
\begin{gathered}
I(F;x)=\sum_{J}x^{|J|}(1+x)^{m(J)}=\sum_{y,m}N(y,m)\,x^y(1+x)^m,\\
i_k=\sum_{y,m}N(y,m)\binom{m}{k-y}\qquad(k\ge0).
\end{gathered}
\end{equation}
\end{lemma}

\begin{proof}
If $S$ is independent, $J=S\cap C$ is independent and $A'=S\cap A$ avoids $\NA(J)$. Conversely, for independent
$J\subseteq C$ and $A'\subseteq A\setminus\NA(J)$, the set $J\cup A'$ is independent. So the independent $S$ with
$S\cap C=J$ are the sets $J\cup A'$, $A'\subseteq A\setminus\NA(J)$, and they contribute $x^{|J|}(1+x)^{m(J)}$.
\end{proof}

The array $N(y,m)$ is denoted $t_{j,m}$ in \cite[Section~2.1]{Z} and \cite[Section~2.1]{ZL} ($j=|J|$ playing the role of
$y$), where the decomposition is taken relative to a maximum independent set, as in Section~\ref{p1:sec:overview}; the
moment argument takes it relative to an independent set of maximum weight (Proposition~\ref{p2:prop:mixture}).

\subsection{The window, the prefix and the tail}\label{p1:subsec:window}

Let $\nu=\nu(F)$ and
\begin{equation}\label{p1:eq:brow}
B_F(x)=(1+2x)^{\nu}(1+x)^{n-2\nu}=\sum_{s\ge0}b_sx^s,\qquad
b_s=\sum_{j}\binom{\nu}{j}2^j\binom{n-2\nu}{s-j}.
\end{equation}
Let $h_B=h_B(F)$ be the first mode of $(b_s)$, the least index of a maximal entry. The \emph{window} of $F$ is the
set of integers
\[
W(F)=\{k:\ \lceil n/4\rceil<k<h_B(F)\}.
\]
The index $h_B(F)$ is the first mode of $B_F$ and does not depend on the independent set $A$.

\begin{proposition}[prefix, tail and mode bound]\label{p1:prop:inputs}
For every forest $F$ with $n$ vertices:

\textup{(a)} $(n-3k)\,i_k\le(k+1)\,i_{k+1}$ for all $k\ge0$; hence $i_k\le i_{k+1}$ whenever $4k+1\le n$, and
$i_0\le\dots\le i_{\lceil n/4\rceil}$.
\textup{(b)} $i_s\,b_{s-1}\le i_{s-1}\,b_s$ for all $s\ge1$; hence $i_{k+1}\le i_k$ for every $k\ge h_B$.
\textup{(c)} $h_B\le\lceil(3n-2\nu)/6\rceil$.

\end{proposition}

\emph{Credit.} Parts (a) and (b) are, in this form, Lemma~3.1 of Zhang and Li~\cite[Section~3]{ZL}. There the initial
segment $i_0\le\dots\le i_{\lceil n/4\rceil}$ is proved from the same bound $(k+1)i_{k+1}\ge(n-3k)i_k$ and is credited
to Fang, Lu, Nevo, Yao and Zheng~\cite[Prop.~8.2]{FLNYZ}; the terminal segment is proved by ratio domination against
the polynomial $(1+2x)^{\nu}(1+x)^{\alpha(F)-\nu}$ of a maximum matching, which for a forest ($\alpha(F)+\nu=n$) is the
generating function of our $(b_s)$, with the mode bound $\lfloor(n+2\alpha(F))/6\rfloor+1$ in place of our (c). The formal proofs of
(a) to (c) (Section~\ref{sec:verification}) were written for this paper.

\begin{proof}
(a) The polynomial $\Delta_F(x)=(1+3x)I'(F;x)-n\,I(F;x)$ has coefficient $(k+1)i_{k+1}+3ki_k-ni_k$ at $x^k$, and
nonnegative coefficients by induction: $\Delta_\emptyset=0$; $\Delta_F=(1+x)\Delta_{F-v}+2x\,I(F-v;x)$ for an
isolated vertex $v$; $\Delta_F=\Delta_{F-u}+x\,\Delta_{F-u-v}+x\,\bigl(I(F-u-v;x)-I(F-\NF[v];x)\bigr)$ for a leaf $u$
with neighbour $v$, the bracket being nonnegative since $F-\NF[v]\subseteq F-u-v$.
(b) The independent sets form a down-closed family in the product of the blocks of a maximum matching (its edges and
unmatched vertices), whose rank numbers are the $b_s$; the normalized matching property of such products (Hsieh and
Kleitman~\cite{HK}, Harper~\cite{Harper}) gives $i_s/b_s$ nonincreasing (the formal proof is an induction on the
blocks). $B_F$ is log-concave with positive coefficients exactly on $[0,n-\nu]$, so $b$ is nonincreasing from $h_B$,
and $i_s=0$ for $s>n-\nu$.
(c) Part (c) is one half of Darroch's theorem \cite{Darroch}, the mean of $(b_s)$ being $n/2-\nu/3$; it is proved from
an Euler-operator identity in Lemma~\ref{apC:lem:B7}.
\end{proof}

\begin{lemma}[rank logic]\label{p1:lem:ranklogic}
Let $(a_r)_{r\ge0}$ be real and $k_{\mathrm{lo}},k_{\mathrm{hi}}$ integers with $a_r\le a_{r+1}$ for $r<k_{\mathrm{lo}}$,
$a_{r+1}\le a_r$ for $r\ge k_{\mathrm{hi}}$, and no valley at any $k$ with $k_{\mathrm{lo}}<k<k_{\mathrm{hi}}$. Then $(a_r)$
is unimodal.
\end{lemma}

\begin{proof}
If no strict fall $a_{i+1}<a_i$ is followed by a strict rise $a_{j+1}>a_j$, $j>i$, the sequence is unimodal (peak at
the first strict fall, or at $\max(k_{\mathrm{hi}},0)$ if there is none). Otherwise take such $i<j$ and let $k$ be the least index
greater than $i$ with $a_{k+1}>a_k$. If $k=i+1$ then $a_k<a_{k-1}$; otherwise minimality gives $a_k\le a_{k-1}$. So
there is a valley at $k$. The fall at $i$ and the prefix give $i\ge k_{\mathrm{lo}}$, so $k>k_{\mathrm{lo}}$; the rise at $k$
and the tail give $k<k_{\mathrm{hi}}$. Contradiction.
\end{proof}

\begin{corollary}[the window reduction]\label{p1:cor:window}
If a forest $F$ has no valley at any $k\in W(F)$, its independence sequence is unimodal.
\end{corollary}

\begin{proof}
Lemma~\ref{p1:lem:ranklogic} with $k_{\mathrm{lo}}=\lceil n/4\rceil$ and $k_{\mathrm{hi}}=h_B(F)$: the hypotheses on the prefix and the tail are
Proposition~\ref{p1:prop:inputs}(a) and~(b).
\end{proof}

% ---------------------------------------------------------------------------------------------------------------
\section{The moment argument}\label{p2:sec:overview}

This section proves Theorem~\ref{p2:thm:main}, the case $n\ge25$ of Theorem~\ref{thm:main}, from the inputs of
Section~\ref{p2:sec:inputs} and the no-valley inequality on the parameter domain of Section~\ref{p2:sec:atlas}.

\begin{theorem}\label{p2:thm:main}
Every forest with at least $25$ vertices has a unimodal independence sequence.
\end{theorem}

% ---------------------------------------------------------------------------------------------------------------
\subsection{The proof from its inputs}\label{p2:sec:assembly}

This subsection proves Theorem~\ref{p2:thm:main} from the results of Sections~\ref{p2:sec:mixture} to~\ref{p2:sec:atlas}.
The activities $[1/3,7/3]$ are divided into $55$ subintervals $\Ihat_\iota$, and $\iota(\lambda)$ is the index of the one
containing $\lambda$ (Section~\ref{p2:sec:inputs}); $q$, $M$, $\delta$ and the forest mixture are defined in
Section~\ref{p2:sec:mixture}.

\begin{proof}[Proof of Theorem~\ref{p2:thm:main}]
Let $n\ge25$ and $k\in W(F)$.
First, choose $\lambda\in(1/3,7/3)$ with $\mu_F(\lambda)=k$ (Proposition~\ref{p2:prop:activity}) and an independent set
  $A$ of maximum weight at $\lambda$, and consider the forest mixture of $A$ at $\lambda$, with $\muM=\mathbb{E}M$.

The forest mixture is a probability mixture with $\mathbb{E}\delta=0$ (Proposition~\ref{p2:prop:mixture}), and it
  satisfies the hypotheses of Theorem~\ref{p2:thm:O4}: the variance-ratio bound (Theorem~\ref{p2:thm:O2}), the variance
  bound (Theorem~\ref{p2:thm:O1}), the Laplace-transform bounds (Theorem~\ref{p2:thm:O3}) and the lower bound
  $\muM\ge m_{25}(\iota(\lambda))$ for the mean free count (Theorem~\ref{p2:thm:O5}).

Finally, by Theorem~\ref{p2:thm:O4} the forest mixture has no valley at $k$. By Proposition~\ref{p2:prop:mixture} its law is
  $P_r=i_r\lambda^r/I(F;\lambda)$, so $F$ has no valley at $k$, and Corollary~\ref{p1:cor:window} concludes.

The hypothesis $n\ge25$ is used only through Theorem~\ref{p2:thm:O5}.
\end{proof}

\subsection{The mixture over an independent set of maximum weight}\label{p2:sec:mixture}

This subsection writes the law of $|S|$ as a mixture of shifted binomial laws (Proposition~\ref{p2:prop:mixture}), the
input of the criterion of Section~\ref{p2:sec:criterion}; the mixture over an independent set of maximum weight is due to
Zhang and Li~\cite{ZL}. With the hard-core measure $P_\lambda$ of Section~\ref{sec:prelim:notation} write $K=|S|$,
$\pi_v=\pi_v(\lambda)=P_\lambda(v\in S)$ and $q=q(\lambda)=\lambda/(1+\lambda)$, so $q/(1-q)=\lambda$; for $\lambda\in[1/3,7/3]$,
$q\in[1/4,7/10]$.

\paragraph{Independent sets of maximum weight.}
The \emph{weight} of an independent set $A$ at $\lambda$ is $\pi(A)=\sum_{v\in A}\pi_v$. In
Sections~\ref{p2:sec:overview} to~\ref{p2:sec:atlas}, $A$ is an independent set of maximum weight at $\lambda$ (one exists;
ties are allowed, and every statement holds for every maximizer); it is not the maximum independent set of
Section~\ref{p1:sec:overview}. Two facts are used. (a) $\mu_F(\lambda)\le 2\pi(A)$: the colour classes of a proper
$2$-colouring are independent and their weights add up to $\mu_F$ \cite[Section 3]{Z}. (b) (L1) If $u$ is a leaf whose
neighbour $v$ has another neighbour, then $u\in A$ and $v\notin A$ (since $\NF[u]\subsetneq\NF[v]$ forces
$\pi_u>\pi_v$); (L2) each single-edge component has exactly one endpoint in $A$; (L3) isolated vertices lie in $A$
(\cite[proof of Lemma 4.1]{Z}). The variance bound uses (L1), the Laplace-transform bounds use (L1) and (L2).

\paragraph{The fibers.}
Put $C=V(F)\setminus A$. Every independent set is uniquely $S=J\cup(S\cap A)$ with $J=S\setminus A$ independent in $C$
and $S\cap A$ any subset of the $m(J)$ vertices of $A$ with no neighbour in $J$ (Lemma~\ref{p1:lem:fiber}). Hence
$I(F;\lambda)=\sum_J\lambda^{|J|}(1+\lambda)^{m(J)}$, and given $J$ the set $S\cap A$ is a $q$-random subset of these
$m(J)$ free vertices. For $J=S\setminus A$ put $K_C=|J|$, $M=m(J)$ and $\muM=\mathbb{E}_\lambda M$. Then
$K=K_C+\operatorname{Bin}(M,q)$.

\paragraph{Mixtures.}
A \emph{mixture} is a finite family of states $J$ with weights $w(J)\in\mathbb{R}$, free counts
$M(J)\in\mathbb{N}$ and sizes $K_C(J)\in\mathbb{N}$. Its law is $P_k=\sum_J w(J)\,\binp_{M(J)}(k-K_C(J);q)$ ($k\in\mathbb{Z}$), where
$\binp_M(j;q)=\binom Mj q^j(1-q)^{M-j}$ (zero outside $[0,M]$); we write $\mathbb{E}\tg(M,K_C)=\sum_Jw(J)\,\tg(M(J),K_C(J))$,
$\muM=\mathbb{E}M$, $\operatorname{Var}M=\mathbb{E}(M-\muM)^2$ and $P(M\le M')=\sum_{J:\,M(J)\le M'}w(J)$. It is a
\emph{probability mixture} if $w\ge0$ and $\sum_Jw(J)=1$. The \emph{forest mixture} of $A$ at $\lambda$ has the
independent sets $J\subseteq C$ as states, with $w(J)=\lambda^{|J|}(1+\lambda)^{m(J)}/I(F;\lambda)$, $M(J)=m(J)$ and
$K_C(J)=|J|$. A \emph{valley} of a mixture at the integer $k$ is the pair of inequalities
$(1-q)P_k\le qP_{k-1}$ and $qP_k\le(1-q)P_{k+1}$; for $P_k=i_k\lambda^k/I(F;\lambda)$ it says exactly $i_k\le i_{k-1}$ and
$i_k\le i_{k+1}$, because $q/(1-q)=\lambda$.

\begin{proposition}[Zhang and Li {\cite[Section 3, (41) to (42)]{ZL}}; also {\cite[(45) to (46)]{Z}}]\label{p2:prop:mixture}
For $\lambda>0$ and every independent $A$: the weights of the forest mixture form a probability; its law is
$P_r=i_r\lambda^r/I(F;\lambda)$ for all $r$; if $\mu_F(\lambda)=k$ then $\mathbb{E}\delta=0$ and
$\mathbb{E}\delta^2=\operatorname{Var}_\lambda K-q(1-q)\muM$, where $\delta=k-K_C-qM$; and $q\muM=\pi(A)$.
Consequently $\operatorname{Var}_\lambda|S\cap A|=(1-q)\,\pi(A)+q^2\operatorname{Var}M$, and if $A$ has maximum weight,
\begin{equation}\label{p2:eq:consequences}
\muM\ \ge\ \mu_F(\lambda)/(2q).
\end{equation}
\end{proposition}
\begin{proof}
The master formula
$\mathbb{E}_\lambda[\tg(|S\setminus A|,|S\cap A|)]=\mathbb{E}_{J}\big[\sum_{j\le M}\tg(K_C,j)\,\binp_M(j;q)\big]$ follows
from the decomposition and $(1+\lambda)^M\binp_M(j;q)=\binom Mj\lambda^j$; apply it to the functions $1$, $[K_C+j=r]$ and
$j$ of $(K_C,j)$. The law of total variance, given $J$, gives $\operatorname{Var}_\lambda K=\mathbb{E}(K_C+qM-\mu_F(\lambda))^2
+q(1-q)\muM$ and $\operatorname{Var}_\lambda|S\cap A|=q^2\operatorname{Var}M+q(1-q)\muM$. The last claim is (a) with
$q\muM=\pi(A)$.
\end{proof}

% ---------------------------------------------------------------------------------------------------------------
\subsection{The criterion}\label{p2:sec:criterion}

This subsection proves the MGF criterion (Theorem~\ref{p2:thm:mgf2}), which excludes a valley of a probability mixture
by duality from moment information; the no-valley inequality on the parameter domain (Theorem~\ref{p2:thm:O4}) applies it,
with certificates read off from rational boxes for small means and built in closed form from a fiber inequality for the
others (Section~\ref{p2:sec:atlas}).
With $f_{\mathrm L}(M,j)=(1-q)\binp_M(j)-q\,\binp_M(j-1)$ and $f_{\mathrm R}(M,j)=q\,\binp_M(j)-(1-q)\binp_M(j+1)$, the
valley kernel with weights $w_{\mathrm L},w_{\mathrm R}\ge0$ is
\[
\kappa_2(M,j)=w_{\mathrm L}f_{\mathrm L}(M,j)/q+w_{\mathrm R}f_{\mathrm R}(M,j)/(1-q)\qquad(M\in\mathbb{N},\ j\in\mathbb{Z}),
\]
and $\kappa=f_{\mathrm L}/q+f_{\mathrm R}/(1-q)$ is the case $w_{\mathrm L}=w_{\mathrm R}=1$. In this subsection mixtures are
arbitrary mixtures in the sense of Section~\ref{p2:sec:mixture}, $0<q<1$, and $\delta=k-K_C-qM$ at the rank $k$
considered.

\begin{lemma}[the valley as a moment inequality]\label{apC:lem:valley}
For every mixture, every real $q$ and every integer $k$,
\[
\mathbb{E}f_{\mathrm L}(M,k-K_C)=(1-q)P_k-qP_{k-1},\qquad \mathbb{E}f_{\mathrm R}(M,k-K_C)=qP_k-(1-q)P_{k+1}.
\]
Consequently, if $0<q<1$, $w_{\mathrm L},w_{\mathrm R}\ge0$ and the mixture has a valley at $k$, then
$\mathbb{E}\kappa_2(M,k-K_C)\le0$.
\end{lemma}
\begin{proof}
For a state $J$, $f_{\mathrm L}(M(J),k-K_C(J))=(1-q)\binp_{M(J)}(k-K_C(J))-q\,\binp_{M(J)}((k-1)-K_C(J))$; multiply by $w(J)$ and
sum over the states: the two sums are $P_k$ and $P_{k-1}$ by definition. The second identity is the same computation
with $(k-K_C)+1=(k+1)-K_C$. By linearity
$\mathbb{E}\kappa_2=\frac{w_{\mathrm L}}q\big((1-q)P_k-qP_{k-1}\big)+\frac{w_{\mathrm R}}{1-q}\big(qP_k-(1-q)P_{k+1}\big)$,
and a valley makes both brackets nonpositive.
\end{proof}

A \emph{Laplace constraint} $(s,\ebar)$ stands for the moment constraint $\mathbb{E}s^M\le\ebar$; constraints are indexed
by $r$, and for tilts $\tau_r$ with rates $\ell_r$ the Laplace constraints at $(q,\muM)$ are
$(1-q+q\tau_r,\,e^{-\ell_r\muM})$. Prices $z_r\ge0$ are attached to the constraints by position.

\begin{definition}[MGF certificate; valley-free parameters]\label{p2:def:mgf2}
Let $q,\muM,\theta,D$ be reals and $\mathcal E=((s_r,e_r))_r$ a finite list of Laplace constraints. An \emph{MGF
certificate} for $(q,\muM,\theta,D,\mathcal E)$ consists of a real $\eta>0$, reals $\xi$, $\alpha$ and $\beta$, nonnegative
reals $\gamma$, $w_{\mathrm L}$ and $w_{\mathrm R}$, and prices $z_r\ge0$ such that, with $\delta=j-qM$,
\begin{align*}
\text{(P)}\quad&\alpha+\beta(M-\muM)-\gamma(M-\muM)^2-\textstyle\sum_rz_rs_r^M\ \le\ \kappa_2(M,j)+\xi\,\delta+\eta\,\delta^2
\qquad(M\in\mathbb{N},\ j\in\mathbb{Z}),\\
\text{(Mg)}\quad&\eta\,\theta\,q(1-q)\,\muM\ <\ \alpha-\gamma D\muM-\textstyle\sum_rz_re_r .
\end{align*}
The parameters $(q,\muM,\theta,D,\mathcal E)$ are \emph{valley-free} if no probability mixture with $\mathbb{E}M=\muM$,
$\mathbb{E}\delta=0$ and $\mathbb{E}\delta^2\le\theta\,q(1-q)\,\muM$ for $\delta=k-K_C-qM$ at an integer $k$,
$\operatorname{Var}M\le D\muM$, and $\mathbb{E}s_r^M\le e_r$ for every constraint of $\mathcal E$ has a valley at $k$.
\end{definition}

For $0<q<1$ and $\muM\ge0$, an MGF certificate for $(q,\muM,\theta,D,\mathcal E)$ is one for every smaller $\theta$, $D$ and $e_r$, and valley-freeness
is antitone in them: smaller values admit fewer mixtures.

\begin{theorem}[the MGF criterion]\label{p2:thm:mgf2}
Let $0<q<1$. If an MGF certificate exists for $(q,\muM,\theta,D,\mathcal E)$, then $(q,\muM,\theta,D,\mathcal E)$ is
valley-free.
\end{theorem}
\begin{proof}
Suppose a probability mixture has $\mathbb{E}M=\muM$, $\mathbb{E}\delta=0$, $\mathbb{E}\delta^2\le\theta q(1-q)\muM$,
$\operatorname{Var}M\le D\muM$, $\mathbb{E}s_r^M\le e_r$ for every constraint, and a valley at $k$. Put
$\psi(M)=\alpha+\beta(M-\muM)-\gamma(M-\muM)^2$, multiply (P) at $(M(J),k-K_C(J))$ by $w(J)\ge0$ and sum:
\[
\mathbb{E}\psi(M)-\textstyle\sum_rz_r\mathbb{E}s_r^M\ \le\ \mathbb{E}\kappa_2+\xi\,\mathbb{E}\delta+\eta\,\mathbb{E}\delta^2\ \le\ \eta\theta q(1-q)\muM ,
\]
by Lemma~\ref{apC:lem:valley}, $\mathbb E\delta=0$ and $\eta>0$. Since the weights sum to $1$ and $\mathbb{E}M=\muM$,
$\mathbb{E}\psi(M)=\alpha-\gamma\operatorname{Var}M\ge\alpha-\gamma D\muM$ ($\gamma\ge0$), and
$\sum_rz_r\mathbb{E}s_r^M\le\sum_rz_re_r$ ($z_r\ge0$). This contradicts (Mg).
\end{proof}

MGF certificates have the structure of the rational mixture certificates of Zhang and Li~\cite[Section~5, (45) to
(46)]{ZL}: a pointwise inequality between a weighted valley kernel and terms quadratic in $\delta$ and $M$ with Laplace
terms, whose expectation under the source bounds gives an explicit budget. They have no curvature term
$f_{\mathrm C}(M,j)=2\binp_M(j)-\binp_M(j-1)-\binp_M(j+1)$ and none of the count terms $(1-q)^M$, and they centre the
quadratic in $M$ at $\muM$.

% ---------------------------------------------------------------------------------------------------------------
\subsection{The activity of a window rank}\label{p2:sec:window}

This subsection proves that every rank of the window is the hard-core mean at an activity in $(1/3,7/3)$
(Proposition~\ref{p2:prop:activity}), used in step~1 of the proof of Theorem~\ref{p2:thm:main}. Here
$\mu_F(\lambda)=\mathbb E_\lambda|S|=\sum_kk\,i_k\lambda^k/I(F;\lambda)$, and $B_F$, its coefficients $b_s$ and its first
mode $h_B$ are as in Section~\ref{p1:subsec:window}.

\begin{lemma}[the mean at $1/3$]\label{apC:lem:B2}
For every forest and every $k\ge0$, $(k+1)\,i_{k+1}\le(n-k)\,i_k$; for every $\lambda\ge0$,
$\mu_F(\lambda)(1+\lambda)\le n\lambda$; in particular $\mu_F(1/3)\le n/4$.
\end{lemma}
\begin{proof}
Every $k$-subset of an independent $(k+1)$-set is independent, so the shadow $\partial\mathcal F$ of the family
$\mathcal F$ of independent $(k+1)$-sets has at most $i_k$ members, and the local LYM inequality
$(k+1)|\mathcal F|\le(n-k)|\partial\mathcal F|$ gives the first claim (for $k\ge n$ both sides vanish). Multiply by
$\lambda^{k+1}\ge0$ and sum over $0\le k<n$: as $\lambda I'(F;\lambda)=\sum_kk\,i_k\lambda^k$ and the term $k=n$ on the
right is zero, $\lambda I'(F;\lambda)=\sum_{k<n}(k+1)i_{k+1}\lambda^{k+1}\le\sum_{k<n}(n-k)i_k\lambda^{k+1}
=\lambda(nI(F;\lambda)-\lambda I'(F;\lambda))$. Dividing by $I(F;\lambda)\ge i_0=1$ gives the second claim, and
$\lambda=1/3$ the third.
\end{proof}

\begin{lemma}[the mode of $B_F$]\label{apC:lem:B7}
Let $b_s=b_s(\nu',\nu'')$ be the coefficient of $x^s$ in $(1+2x)^{\nu'}(1+x)^{\nu''}$ ($\nu',\nu''\ge0$). If
$6s\ge4\nu'+3\nu''$ then $b_{s+1}\le b_s$. Consequently $2\nu\le n$ and $h_B(F)\le\lfloor(3n-2\nu+5)/6\rfloor=\lceil(3n-2\nu)/6\rceil$.
\end{lemma}
\begin{proof}
Write $b'_s=b_s(\nu'-1,\nu'')$ and $b''_s=b_s(\nu',\nu''-1)$ (a term with a factor $\nu'=0$ or
$\nu''=0$ below is zero). With $P=(1+2x)^{\nu'}(1+x)^{\nu''}$, Euler's identity
$(\nu'+\nu'')P=xP'+\nu'P/(1+2x)+\nu''P/(1+x)$ (as $2x=(1+2x)-1$ and $x=(1+x)-1$) gives at $x^s$
$(\nu'+\nu'')\,b_s=s\,b_s+\nu'\,b'_s+\nu''\,b''_s$, and multiplying out one factor gives $b_{s+1}=b'_{s+1}+2b'_s$ (if $\nu'\ge1$) and $b_{s+1}=b''_{s+1}+b''_s$ (if
$\nu''\ge1$). Induct on $\nu'+\nu''$. If $\nu'=\nu''=0$ then $b_{s+1}=0$. Otherwise the hypothesis $6s\ge4\nu'+3\nu''$
implies $6s\ge4(\nu'-1)+3\nu''$ and $6s\ge4\nu'+3(\nu''-1)$, so by induction $b'_{s+1}\le b'_s$ and $b''_{s+1}\le b''_s$.
Hence $\nu' b_{s+1}\le3\nu' b'_s$ and $\nu'' b_{s+1}\le2\nu''\,b''_s$, and
$(2\nu'+3\nu'')\,b_{s+1}\le6(\nu' b'_s+\nu'' b''_s)=6(\nu'+\nu''-s)\,b_s\le(2\nu'+3\nu'')\,b_s$, the last step because $6(\nu'+\nu'')\le 2\nu'+3\nu''+6s$. Divide by $2\nu'+3\nu''>0$.

For the forest take $\nu'=\nu$, the matching number, and $\nu''=n-2\nu$ ($2\nu\le n$, as the edges of a matching are
disjoint). Then $4\nu'+3\nu''=3n-2\nu$, so the coefficients $b_s$ of $B_F$ are nonincreasing from
$s_0=\lfloor(3n-2\nu+5)/6\rfloor$ on, as $6s_0\ge3n-2\nu$; hence the maximum is attained at an index at most $s_0$, and
the first mode is at most $s_0$. This is the half of Darroch's theorem~\cite{Darroch} that is needed, proved directly.
\end{proof}

\begin{lemma}[the integer cases]\label{apC:lem:cases}
Let $\Upsilon(j)=1+(7/3)(3/10)^j$ ($j\ge0$) and, for integers $\bar k$, $\bar D$ and a real $\hat P$,
$\mathrm{num}_3(\hat P)=(726\bar k-740)\hat P+1736-2499\bar k+714\bar D$ and
$\mathrm{num}_5(\hat P)=(1567\bar k-1749)\hat P+5082-5880\bar k+1680\bar D$. For integers $\bar k\ge2$ and $\bar D\ge0$
write $\bar D=\bar k\hat a+\hat s$ with $0\le\hat s<\bar k$ and
$\hat P_*(\bar k,\bar D)=\Upsilon(\hat a)^{\bar k-\hat s}\Upsilon(\hat a+1)^{\hat s}$. Then
$\mathrm{num}_3(\hat P_*(\bar k,\bar D))>0$ for each of the $181$ pairs $(\bar k,\bar D)$ with $2\le\bar k\le12$, $\bar D\ge0$
and $714\bar D<1773\bar k-996$, and $\mathrm{num}_5(\hat P_*(\bar k,\bar D))>0$ for each of the $247$ pairs with
$2\le\bar k\le14$, $\bar D\ge0$ and $1680\bar D<4313\bar k-3333$. This is checked by computer in exact integer
arithmetic, by kernel evaluation (family~\ref{chk:act} of Section~\ref{sec:computer:families}).
\end{lemma}

\begin{theorem}[occupation at $7/3$]\label{apC:thm:B5}
For every forest, $\mu_F(7/3)\ge n/2-\nu/3+\operatorname{cc}(F)/7$, where $\operatorname{cc}(F)$ is the number of
components of $F$, isolated vertices included.
\end{theorem}
\begin{proof}
Fix $\lambda=7/3$, so $1+\lambda=10/3$. For a vertex set $X$ put $I_{X}=I(F[X];\lambda)$ and
$\dot I_{X}=\lambda I'(F[X];\lambda)$, let $\nu(X)$ and $\operatorname{cc}(X)$ be the matching
number and the number of components of $F[X]$, $\mu^\ast(X)=|X|/2-\nu(X)/3+\operatorname{cc}(X)/7$ and
$\Xi(X)=\dot I_{X}-\mu^\ast(X)I_{X}$. We prove $\Xi(X)\ge0$ by strong induction on $X$; at $X=V(F)$, where $I_X>0$, this is the
theorem.

\emph{Tools.} (i) $I_{X}=I_{X-v}+\lambda I_{X-\NF[v]}$ and
$\dot I_{X}=\dot I_{X-v}+\lambda(I+\dot I)_{X-\NF[v]}$; $I_X$ is multiplicative and $\dot I_X$
obeys Leibniz's rule over unions with no edges between the parts; an edgeless $X_0$ has
$I_{X_0}=(1+\lambda)^{|X_0|}$, $\dot I_{X_0}=(7/10)|X_0|(1+\lambda)^{|X_0|}$; $I_X$ is monotone in $X$ and
$I_{X}\le(1+\lambda)^{|X_0|}I_{X\setminus X_0}$. (ii) $\nu$ is monotone, $\nu(X)\le\nu(X-v)+1$,
$\nu(X)=\nu(X-v_1-v_2)+1$ if $v_1$ is a leaf of $X$ at $v_2$, and $\nu$ adds over edge-free unions. (iii)
$\operatorname{cc}(X-x)+1=\operatorname{cc}(X)+\deg_X(x)$.
\emph{Isolated vertex.} If $x$ is isolated in $X$, then $\Xi(X)=(1+\lambda)\Xi(X-x)+(4/21)I_{X-x}$: $\mu^\ast$
grows by $1/2+1/7$. Also $\Xi(\emptyset)=0$. \emph{A leaf-stem vertex.} If $X\ne\emptyset$ has no isolated vertex, some $v\in X$ has a set $V_{\mathrm{leaf}}$ of
$\bar t\ge1$ leaf neighbours and at most one other neighbour $u$ (a \emph{stem}): take an adjacent pair $(x,y)$
minimizing the size, at least $2$, of the branch at $x$ that contains $y$, and $v=y$, or an isolated edge if no such pair
exists. Put $X'=(X-v)\setminus V_{\mathrm{leaf}}$ and let $V_{\mathrm{stem}}$ be the set of stems; then $\nu(X)=\nu(X')+1$,
$\operatorname{cc}(X)+|V_{\mathrm{stem}}|=\operatorname{cc}(X')+1$, and (i) gives
$I_{X}=(1+\lambda)^{\bar t}I_{X'}+\lambda I_{X'\setminus V_{\mathrm{stem}}}$ and
$\dot I_{X}=(1+\lambda)^{\bar t}\dot I_{X'}+(7/10)\bar t(1+\lambda)^{\bar t}I_{X'}
+\lambda(I+\dot I)_{X'\setminus V_{\mathrm{stem}}}$. \emph{Star} ($V_{\mathrm{stem}}=\emptyset$): $\Xi(X)=((1+\lambda)^{\bar t}+\lambda)\Xi(X')+I_{X'}
\big[(1+\lambda)^{\bar t}(\bar t/5-13/42)+\lambda(29/42-\bar t/2)\big]$, and the bracket is
$5/63$, $107/378$ for $\bar t=1,2$ and positive for $\bar t\ge3$ since $(1+\lambda)^{\bar t}\ge1000/27$.
\emph{One stem $u$.} Let $X''=X'-u$ and $\hat\nu=\nu(X')-\nu(X'')\in\{0,1\}$. Here $\deg_{X'}(u)\ge1$: $u$ is not a leaf
of $X$ (all leaf neighbours of $v$ lie in $V_{\mathrm{leaf}}$), so it has a neighbour other than $v$, which is not a
leaf neighbour of $v$ and hence lies in $X'$. Then
$\Xi(X)=(1+\lambda)^{\bar t}\Xi(X')+\lambda\Xi(X'')+(1+\lambda)^{\bar t}I_{X'}(\bar t/5-1/6)
+\lambda I_{X''}\big(\deg_{X'}(u)/7+\hat\nu/3-\bar t/2+4/21\big)$.
For $\bar t=1$, $\hat\nu=1$ all terms are nonnegative; for $\bar t=1$, $\hat\nu=0$, $\deg_{X'}(u)\ge2$, the last factor
is at least $-1/42$ and $I_{X''}\le I_{X'}$, so $\Xi(X)\ge I_{X''}/18$; for
$\bar t\ge2$, $(1+\lambda)^{\bar t}\ge100/9$ and the last factor is at least $1/3-\bar t/2$, so
$\Xi(X)\ge I_{X'}(19\bar t/18-29/27)>0$ when that factor is negative. The remaining case
$\bar t=1$, $\hat\nu=0$, $\deg_{X'}(u)=1$ is a pendant path of length $3$ ending at the unique neighbour $x$ of $u$ in
$X'$, with $x$ matching-essential: $\nu(X'')=\nu(X''-x)+1$.

\emph{Pendant paths.} Let $X$ consist of a set $X_1\ni x$ and a path of length $3$ or $5$ hanging at $x$, with
$\nu(X_1)=\nu(X_1-x)+1$. Put $X_2=X_1-x$, $X_3=X_1-\NF[x]$, $\bar k=\deg_{X_1}(x)\ge1$,
$\bar D=\sum_{v'}\deg_{X_2}(v')$ over the neighbours $v'$ of $x$, and $\tilde\nu=\nu(X_2)-\nu(X_3)\ge0$. Let
$\bar m\in\{3,5\}$ be the number of vertices of the path (its length), $\mathrm{Path}_j$ the path with $j$ vertices and $\Xi_i=\Xi(X_i)$ ($i=2,3$). Conditioning on $x$
gives $I_{X}=I(\mathrm{Path}_{\bar m};\lambda)I_{X_2}+\lambda I(\mathrm{Path}_{\bar m-1};\lambda)I_{X_3}$; moreover
$|X|=|X_2|+\bar m+1$, $\nu(X)=\nu(X_2)+(\bar m+1)/2$, $\operatorname{cc}(X)=\operatorname{cc}(X_2)-\bar k+1$, $|X_3|=|X_2|-\bar k$
and $\operatorname{cc}(X_3)=\operatorname{cc}(X_2)+\bar D-\bar k$. With the path recurrences these give the exact
identities, for length $3$ and $5$ respectively,
\begin{align*}
378\,\Xi(X)&=5082\,\Xi_2+4998\,\Xi_3+(726\bar k-740)I_{X_2}+(1736-2499\bar k+714\bar D+1666\tilde\nu)I_{X_3},\\
189\,\Xi(X)&=10969\,\Xi_2+11760\,\Xi_3+(1567\bar k-1749)I_{X_2}+(5082-5880\bar k+1680\bar D+3920\tilde\nu)I_{X_3}.
\end{align*}
Deleting the neighbours of $x$ one by one gives $\hat P\,I_{X_3}\le I_{X_2}\le(1+\lambda)^{\bar k}I_{X_3}$
with $\hat P=\prod_{v'}\Upsilon(\deg_{X_2}(v'))$. If $\bar k=1$ and $\bar D=0$, then
$\Xi_2=(1+\lambda)\Xi_3+(4/21)I_{X_3}$ and $378\Xi(X)=21938\Xi_3+(475/3)I_{X_3}$, resp.\
$189\Xi(X)=(144970/3)\Xi_3+(2054/3)I_{X_3}$. If $\bar k=1$, the length is $3$ and $\bar D=1$, the path
extends through $x$ and its neighbour to a pendant path of length $5$ of the same $X$, whose attachment vertex is again
matching-essential, and the case of length $5$ applies. If $\bar k=1$ and $\bar D$ is larger ($\bar D\ge2$ for length
$3$, $\bar D\ge1$ for length $5$), $I_{X_2}\le(10/3)I_{X_3}$ bounds the negative coefficient of
$I_{X_2}$. If $\bar k\ge2$, $\hat P\,I_{X_3}\le I_{X_2}$ reduces the claim to
$\mathrm{num}_3(\hat P)\ge0$, resp.\ $\mathrm{num}_5(\hat P)\ge0$. These hold in three regimes: when the constant terms
dominate ($1773\bar k\le714\bar D+996$, resp.\ $4313\bar k\le1680\bar D+3333$) from $\hat P\ge1$; for $\bar k\ge13$, resp.\
$\bar k\ge15$, from the affine minorant $(3/10)^j\ge(837-189j)/10000$ and
$\hat P\ge1+\sum_{v'}(7/3)(3/10)^{\deg_{X_2}(v')}$, which give $\hat P\ge\hat P_{\mathrm a}:=1+7(837\bar k-189\bar D)/30000$
(as $\mathrm{num}_3$ increases in $\hat P$, it suffices that $\mathrm{num}_3(\hat P_{\mathrm a})>0$, which is affine in $\bar D$:
at the end $\bar D=(1773\bar k-996)/714$ of the first regime it equals $7(726\bar k-740)(837\bar k-189\bar D)/30000>0$, and at
$\bar D=0$ it is a quadratic in $\bar k$ that increases for $\bar k\ge7$ and equals $30.35\ldots$ at $\bar k=13$; for
$\mathrm{num}_5$ the end is $(4313\bar k-3333)/1680$, the quadratic increases for $\bar k\ge8$, and its value at $\bar k=15$
is $2372.2\ldots$); and in the remaining $181$, resp.\ $247$, pairs
$(\bar k,\bar D)$ with $\bar k\le12$, resp.\ $\bar k\le14$, from the balanced product $\hat P_*(\bar k,\bar D)\le\hat P$
($\log\Upsilon$ is convex) and Lemma~\ref{apC:lem:cases}. Since $X_2$ and $X_3$ are proper subsets of $X$, the
induction hypothesis gives $\Xi_2,\Xi_3\ge0$, and $\Xi(X)\ge0$ in every case.
\end{proof}

\begin{proposition}[exact activity selection]\label{p2:prop:activity}
For every forest $F$ and every $k\in W(F)$ there is an activity $\lambda$ with $1/3<\lambda<7/3$ and
$\mu_F(\lambda)=k$.
\end{proposition}

\begin{proof}
Let $k\in W(F)$, that is, $\lceil n/4\rceil<k<h_B(F)$. Then $\mu_F(1/3)\le n/4<k$ by Lemma~\ref{apC:lem:B2}. By
Lemma~\ref{apC:lem:B7}, $k+1\le\lfloor(3n-2\nu+5)/6\rfloor$, so $6k+1+2\nu\le3n$, that is, $k<n/2-\nu/3\le\mu_F(7/3)$ by
Theorem~\ref{apC:thm:B5} (as $\operatorname{cc}(F)\ge0$). As $\mu_F$ is continuous on $[0,\infty)$ (a ratio of polynomials with $I(F;\lambda)\ge1$), the
intermediate value theorem gives $\lambda\in(1/3,7/3)$ with $\mu_F(\lambda)=k$.
\end{proof}

% ---------------------------------------------------------------------------------------------------------------
\section{The inputs}\label{p2:sec:inputs}

This section proves the four inputs of the moment argument, used in the input-verification part of the proof of Theorem~\ref{p2:thm:main}:
the variance bound (Theorem~\ref{p2:thm:O1}), the variance-ratio bound (Theorem~\ref{p2:thm:O2}), the Laplace-transform
bounds (Theorem~\ref{p2:thm:O3}) and the lower bound for the mean free count (Theorem~\ref{p2:thm:O5}). The first three
hold for every finite forest, every $\lambda\in[1/3,7/3]$ and every independent set $A$ of maximum weight at $\lambda$; they
have the forms of the source inequalities of Zhang and Li~\cite[Section~5, (44)]{ZL}, $\operatorname{Var}M\le D_0\muM$,
$\mathbb{E}\delta^2\le\alpha_0\muM+\beta_0$ and $\mathbb{E}(1-q+q\tau_r)^M\le e^{-\ell_r\muM}$. Only the fourth uses
$n\ge25$.

\paragraph{Subintervals and caps.}
The constants are indexed by $55$ subintervals $\Ihat_\iota=[\lamhat_\iota,\lamhat_{\iota+1}]$, $\iota=0,\dots,54$,
$\lamhat_0=1/3<\dots<\lamhat_{55}=7/3$, refining a partition of $[1/3,7/3]$ into $30$ intervals, the \emph{bands}
$[\check\lambda_i,\check\lambda_{i+1}]$; the end points are listed in the supplementary data (Section~\ref{sec:computer}).
For $\lambda\in[1/3,7/3]$, $\iota(\lambda)$ is the index of the subinterval containing $\lambda$, the left one at a shared
end point, and a constant of index $\iota$ is verified on the closed interval $\Ihat_\iota$. The caps $\theta(\iota)$ of the
variance-ratio bound are given in Data~\ref{p2:tab:theta}.
\datablock{Caps for the variance-ratio bound}{p2:tab:theta}
For each group of subinterval indices, the ordered pair records its activity interval and the cap
$\theta(\iota)$ of the variance-ratio bound:
For indices 0 to 5, the pair is $([1/3,3/5], 1)$. For indices 6 to 9, the pair is $([3/5,39/50], 7/10)$. For index 10, the pair is $([39/50,4/5], 13/20)$. For indices 11 to 37, the pair is $([4/5,13/10], 3/5)$.

For indices 38 to 43, the pair is $([13/10,8/5], 7/10)$. For index 44, the pair is $([8/5,33/20], 9/10)$. For indices 45, 46, the pair is $([33/20,7/4], 19/20)$. For indices 47 to 54, the pair is $([7/4,7/3], 1)$.
The \emph{coarse cap} of a band is the cap $\theta=1$, $7/10$, $3/5$, $7/10$ or $1$ of the one of the closed intervals
$[1/3,3/5]$, $[3/5,4/5]$, $[4/5,13/10]$, $[13/10,8/5]$, $[8/5,7/3]$ that contains its closure (Lemma~\ref{cs:lem:tables}(a)).
The cap of the variance bound is the function
\begin{equation}\label{p2:eq:Dhat}
\hat D(\lambda)=\tfrac65\ \text{on}\ [\tfrac13,\tfrac{263}{200}],\quad \tfrac75\ \text{on}\ (\tfrac{263}{200},\tfrac32],\quad
\tfrac85\ \text{on}\ (\tfrac32,\tfrac74],\quad \tfrac95\ \text{on}\ (\tfrac74,2],\quad \tfrac{21}{10}\ \text{on}\ (2,\tfrac73],
\end{equation}
and $D(\iota)=\max\{\hat D(\lambda):\lambda\in\Ihat_\iota\}$: $6/5$ for $\iota\le37$, $7/5$ for $38\le\iota\le42$, $8/5$ for
$43\le\iota\le46$, $9/5$ for $47\le\iota\le51$ and $21/10$ for $52\le\iota\le54$.

\begin{lemma}[the subintervals against the bands]\label{cs:lem:tables}
For every $\iota<55$: (a) $\theta(\iota)$ is the coarse cap of the band containing $\Ihat_\iota$, or $13/20$ with
$\Ihat_\iota\subseteq[39/50,4/5]$, or $9/10$ with $\Ihat_\iota\subseteq[8/5,33/20]$, or $19/20$ with
$\Ihat_\iota\subseteq[33/20,7/4]$; so $\theta(\iota)$ is at most the coarse cap of its band; and the closure of each band
lies in one of the five closed intervals $[1/3,3/5]$, $[3/5,4/5]$, $[4/5,13/10]$, $[13/10,8/5]$, $[8/5,7/3]$, whose cap
($1$, $7/10$, $3/5$, $7/10$, $1$) is the coarse cap of the band. (b) The rates of the Laplace-transform bounds on $\Ihat_\iota$ are
those of the band containing it, except the rate of $\tau_0$ on the $26$ subintervals that carry an additional interval,
where it is that interval's rate, and that interval contains $\Ihat_\iota$.
\end{lemma}
This is checked by kernel evaluation (Section~\ref{sec:computer}).

% ---------------------------------------------------------------------------------------------------------------
\subsection{Rooted forests and messages}\label{apB:sec:conv}

This subsection fixes the notation for rooted forests and messages and the two properties (M1) and (M2) of the hard-core
law from which the recursions of the variance bound and of the variance-ratio bound follow. The letter $p$, with or
without a subscript, is a message (a cavity probability).

For a finite set $X$ of vertices of $F$, $P_X$ is the hard-core law of $F[X]$: $P_X(S')=\lambda^{|S'|}/I(F[X];\lambda)$ for
independent $S'\subseteq X$ ($I(F[\emptyset];\lambda)=1$), and $S_X$ is a random set with law $P_X$. For a set $Q$ of
vertices (the counting set), $\mu_Q(X)$ and $V_Q(X)$ are the mean and variance of $|S_X\cap Q|$. The variance bound uses
$Q=A$, the variance-ratio bound uses $Q=V(F)$.
The independence property (M1) states that if $X\cap X'=\emptyset$ and no edge joins $X$ to $X'$, then $P_{X\cup X'}=P_X\otimes P_{X'}$, so $\mu_Q$ and
$V_Q$ are additive.

For the conditioning property (M2), let $v\in X$, $p=P_X(v\in S_X)$, and let $1_Q$ be the indicator function of $Q$. Given $v\notin S_X$, $S_X$ has
law $P_{X-v}$; given $v\in S_X$, $S_X-v$ has law $P_{X-\NF[v]}$. By total expectation and total variance,
\begin{equation}\label{apB:eq:M2}
\begin{aligned}
\mu_Q(X)&=(1-p)\,\mu_Q(X-v)+p\,\bigl(1_Q(v)+\mu_Q(X-\NF[v])\bigr),\\
V_Q(X)&=(1-p)\,V_Q(X-v)+p\,V_Q(X-\NF[v])\\
&\qquad+p(1-p)\bigl(1_Q(v)+\mu_Q(X-\NF[v])-\mu_Q(X-v)\bigr)^2 .
\end{aligned}
\end{equation}
On the whole vertex set, $\mu_A(V(F))=\pi(A)$, $V_A(V(F))=\operatorname{Var}_\lambda|S\cap A|$ and
$V_{V(F)}(V(F))=\operatorname{Var}_\lambda K$.

A rooted subtree is a pair $(X,v)$ with $X$ connected and $v\in X$. The children of $v$ are its neighbours in $X$; for a
child $v'$ the branch $X_{v'}$ is the component of $X-v$ containing $v'$. As $F$ is a forest, $X-v$ is the disjoint union
of the branches, $X-\NF[v]$ that of the sets $X_{v'}-v'$, and no edge joins two branches, so (M1) applies to both. The
\emph{message} of $(X,v)$ is $p_v=P_X(v\in S_X)$, its \emph{log-mass} is $\varphi_v=-\log(1-p_v)$, and
$Y_v=\sum_{v'}\varphi_{v'}$ is the children's log-mass, where $p_{v'}$ and $\varphi_{v'}$ are the message and the log-mass of
the branch $(X_{v'},v')$. Put
\[
\mathrm{msg}_\lambda(s)=\frac{\lambda e^{-s}}{1+\lambda e^{-s}},\qquad
\mathrm{lm}_\lambda(s)=\log\bigl(1+\lambda e^{-s}\bigr)=-\log\bigl(1-\mathrm{msg}_\lambda(s)\bigr),
\]
so $\mathrm{msg}_\lambda(0)=q$ and $\mathrm{lm}_\lambda(0)=\log(1+\lambda)=:\bar\varphi$, the log-mass of a vertex without
children.

\begin{lemma}[messages]\label{apB:lem:msg}
$p_v/(1-p_v)=\lambda\prod_{v'}(1-p_{v'})$; hence $p_v=\mathrm{msg}_\lambda(Y_v)$ and $\varphi_v=\mathrm{lm}_\lambda(Y_v)$. In
particular $0<p_v\le q$, with equality if and only if $v$ has no child, and $0<\varphi_v\le\log(1+\lambda)$.
\end{lemma}
\begin{proof}
$I(F[X];\lambda)=I(F[X-v];\lambda)+\lambda I(F[X-\NF[v]];\lambda)$, and by (M1) $I(F[X-v];\lambda)=\prod_{v'}I(F[X_{v'}];\lambda)$
and $I(F[X-\NF[v]];\lambda)=\prod_{v'}I(F[X_{v'}-v'];\lambda)$, with $I(F[X_{v'}-v'];\lambda)/I(F[X_{v'}];\lambda)=1-p_{v'}$. So
$p_v/(1-p_v)=\lambda\prod_{v'}(1-p_{v'})=\lambda e^{-Y_v}$; each $p_{v'}>0$, so $Y_v>0$ when $v$ has a child.
\end{proof}

% ---------------------------------------------------------------------------------------------------------------
\subsection{The variance bound}\label{apB:sec:O1}

This subsection proves the variance bound by hand, up to $4{,}321$ interval checks; the no-valley inequality on the
parameter domain uses it through the term $\gamma D\muM$ of the criterion.

\begin{theorem}[the variance bound]\label{p2:thm:O1}
For every finite forest, every $\lambda\in[1/3,7/3]$ and every independent set $A$ of maximum weight at $\lambda$,
$\operatorname{Var}_\lambda|S\cap A|\le(1+(\hat D(\lambda)-1)q)\,\pi(A)$ and $\operatorname{Var}M\le\hat D(\lambda)\,\muM$,
with $\hat D$ as in \eqref{p2:eq:Dhat}.
\end{theorem}

\emph{Idea.} Root each component at a vertex outside $A$. By induction from the leaves, a reserve of each rooted subtree
dominates the function $\alphaV\varphi+\gammaV\chi^2/\varphi$ of its log-mass $\varphi$ and its port $\chi$; one step of
the induction reduces to three comparisons at one parent and one child, which are proved along a curve of multipliers.
The maximality of $A$ is used only through the condition
(LP): every leaf of a component of order at least three lies in $A$,
which follows from (L1) of Section~\ref{p2:sec:mixture}; the proof gives the conclusion for every independent $A$ with (LP).

\begin{lemma}[conversion]\label{apB:lem:conv}
Let $A$ be independent, $D\in\mathbb{R}$ and $A_D=1+(D-1)q$. If $\operatorname{Var}_\lambda|S\cap A|\le A_D\,\pi(A)$,
then $\operatorname{Var}M\le D\,\muM$.
\end{lemma}
\begin{proof}
By Proposition~\ref{p2:prop:mixture},
$q^2\operatorname{Var}M=\operatorname{Var}_\lambda|S\cap A|-(1-q)\pi(A)\le(A_D-1+q)\pi(A)=Dq\,\pi(A)=Dq^2\muM$, and $q>0$.
\end{proof}

\paragraph{The reserves and the recursion.}
Fix $A$ and $A_D$. For a rooted subtree $(X,v)$ put $\mathbf 1_v=1_A(v)$, $\mu^0=\mu_A(X)$, $V^0=V_A(X)$,
$\mu^1=\mu_A(X-v)$, $V^1=V_A(X-v)$: the mean and variance of the number of vertices of $A$ in the random set of the
subtree when the parent of $v$ is vacant (the subtree then has law $P_X$) and when it is occupied ($v$ is then vacant).
Put
\[
R^0(X)=A_D\mu^0-V^0,\qquad R^1(X,v)=R^0(X-v)=A_D\mu^1-V^1,\qquad \chi(X,v)=\mu^0-\mu^1;
\]
the \emph{reserves} $R^0$, $R^1$ and the \emph{port} $\chi$, the response of the mean count of $A$ to the parent's
occupancy. For a child $v'$ write $R^0_{v'}=R^0(X_{v'})$, $R^1_{v'}=R^1(X_{v'},v')$, $\chi_{v'}=\chi(X_{v'},v')$.

\begin{lemma}[the recursion]\label{apB:lem:rec}
With $p=p_v$ and sums over the children $v'$ of $v$:
\begin{gather*}
R^1(X,v)=\sum_{v'}R^0_{v'},\qquad \chi(X,v)=p\Bigl(\mathbf 1_v-\sum_{v'}\chi_{v'}\Bigr),\\
R^0(X)=A_Dp\mathbf 1_v+(1-p)\sum_{v'}R^0_{v'}+p\sum_{v'}R^1_{v'}-p(1-p)\Bigl(\mathbf 1_v-\sum_{v'}\chi_{v'}\Bigr)^2 .
\end{gather*}
\end{lemma}
\begin{proof}
By (M1) over the branches, $\mu_A(X-v)=\sum_{v'}\mu_A(X_{v'})$, $\mu_A(X-\NF[v])=\sum_{v'}\mu_A(X_{v'}-v')$, and likewise for
$V_A$; hence $R^1(X,v)=\sum_{v'}R^0_{v'}$, $R^0(X-\NF[v])=\sum_{v'}R^1_{v'}$ and
$\mathbf 1_v+\mu_A(X-\NF[v])-\mu_A(X-v)=\mathbf 1_v-\sum_{v'}\chi_{v'}$. Insert these into \eqref{apB:eq:M2} with $Q=A$.
\end{proof}

\paragraph{The reserve step.}\label{apB:sec:O1step}
The invariant attached to a rooted subtree is the \emph{reserve inequality}
\begin{equation}\label{apB:eq:R}
R^0(X)\ \ge\ \mathcal Q\bigl(\varphi_v,\chi(X,v)\bigr),\qquad \mathcal Q(\varphi,\chi)=\alphaV\varphi+\gammaV\chi^2/\varphi\qquad(\varphi>0),
\end{equation}
with multipliers $\alphaV\ge0$, $\gammaV>0$ depending on $\lambda$ only; so $\mathcal Q\ge0$. The step reduces to
comparisons at one parent-child pair, parametrized by a \emph{feasible pair} $(\tilde Y,\Ypar)$: $\tilde Y\ge0$, $\Ypar>0$
and $\Ypar\ge\mathrm{lm}_\lambda(\tilde Y)$, where $\Ypar=Y_v$ is the parent's children's log-mass and $\tilde Y=Y_{v'}$ that of
the chosen child $v'$. Put $\ppar=\mathrm{msg}_\lambda(\Ypar)$, $\phipar=\mathrm{lm}_\lambda(\Ypar)$ (parent),
$\tilde p=\mathrm{msg}_\lambda(\tilde Y)$, $\varphi=\mathrm{lm}_\lambda(\tilde Y)$ (child), $\mathcal J=\tilde p^2\tilde Y/\varphi$,
\begin{gather*}
\EKL=1-\ppar+\frac{\ppar\tilde Y}{\varphi}-\frac{\phipar}{\Ypar},\qquad \mathcal H=\ppar(1-\ppar)+\frac{\gammaV\ppar^2}{\phipar},\\
L=\gammaV(1-\ppar)-\mathcal H\Ypar,\qquad Z=\ppar\gammaV+L\mathcal J,\\
\mathrm{CC}=\alphaV\EKL,\qquad \mathrm{CA}=\alphaV\EKL+\frac{\ppar\gammaV L(\tilde p/\varphi)^2}{Z},\qquad
\mathrm{AC}=\alphaV\EKL+\frac{A_D\ppar-\mathcal H}{\Ypar}-\frac{\mathcal H^2\mathcal J}{Z}.
\end{gather*}
The first letter of a comparison is the type of the parent and the second that of the child, C for a vertex of $C$ and A
for a vertex of $A$; a parent in $A$ with a child in $A$ does not occur.

\begin{definition}[step certificate]\label{apB:def:cert}
At the activity $\lambda$, $(\alphaV,\gammaV,A_D)$ is a \emph{step certificate} if $\alphaV\ge0$, $\gammaV>0$, the
\emph{base inequality} at a childless vertex of $A$,
\begin{equation}\label{mc:eq:base}
\mathrm{sl}(\lambda):=A_Dq-q(1-q)-\alphaV\bar\varphi-\gammaV q^2/\bar\varphi\ \ge\ 0,\qquad \bar\varphi=\log(1+\lambda),
\end{equation}
holds, and at every feasible pair $Z>0$, $\mathrm{CC}\ge0$, $\mathrm{CA}\ge0$ and $\mathrm{AC}\ge0$. (With
$A_D=1+(D-1)q$, $A_Dq-q(1-q)=Dq^2$.)
\end{definition}

\begin{lemma}[the child's share]\label{apB:lem:child}
Fix a feasible pair with $\Ypar>0$, indicators $\mathbf 1_v,\mathbf 1_{v'}\in\{0,1\}$ (parent $v$, child $v'$) not both
$1$, a real $\chi$ and reals $R^0_{v'},R^1_{v'}$ with $R^0_{v'}\ge\mathcal Q(\varphi,\chi)$ and either
(i) $\tilde Y>0$ and $R^1_{v'}\ge\alphaV\tilde Y+\gammaV(\mathbf 1_{v'}-\chi/\tilde p)^2/\tilde Y$, or
(ii) $\tilde Y=0$, $\mathbf 1_{v'}=1$, $R^1_{v'}=0$ and $\chi=\tilde p$ (the child is a childless vertex of $A$).
If $Z>0$ and $\mathrm{CC},\mathrm{CA},\mathrm{AC}\ge0$ at the pair, the child's share
\[
\mathcal S_{v'}=\frac\varphi\Ypar\bigl(A_D\ppar\mathbf 1_v-\mathcal H\mathbf 1_v^2-\alphaV\phipar\bigr)+2\mathcal H\mathbf 1_v\chi
-\frac{\mathcal H\Ypar\chi^2}{\varphi}+(1-\ppar)R^0_{v'}+\ppar R^1_{v'}
\]
is nonnegative.
\end{lemma}
\begin{proof}
In case (i) put $\hat\chi=\chi/\tilde p$. As an identity of rational functions in independent variables,
\[
\mathcal S_{v'}=\mathcal U+(1-\ppar)\bigl(R^0_{v'}-\mathcal Q(\varphi,\chi)\bigr)+\ppar\Bigl(R^1_{v'}-\alphaV\tilde Y-\frac{\gammaV(\mathbf 1_{v'}-\hat\chi)^2}{\tilde Y}\Bigr),
\]
\[
\mathcal U=\alphaV\varphi\EKL+\frac\varphi\Ypar\bigl(A_D\ppar\mathbf 1_v-\mathcal H\mathbf 1_v^2\bigr)+2\mathcal H\mathbf 1_v\tilde p\hat\chi
+\frac{(Z-\ppar\gammaV)\hat\chi^2+\ppar\gammaV(\mathbf 1_{v'}-\hat\chi)^2}{\tilde Y}.
\]
The last two terms of the first line are nonnegative ($0<\ppar<1$). Completing the square in $\hat\chi$ in the three
possible cases (the second uses $Z-\ppar\gammaV=L\mathcal J=L\tilde p^2\tilde Y/\varphi$):
\begin{align*}
(\mathbf 1_v,\mathbf 1_{v'})=(0,0):\quad \mathcal U&=\varphi\,\mathrm{CC}+\frac Z{\tilde Y}\hat\chi^2,\qquad
(\mathbf 1_v,\mathbf 1_{v'})=(0,1):\quad \mathcal U=\varphi\,\mathrm{CA}+\frac Z{\tilde Y}\Bigl(\hat\chi-\frac{\ppar\gammaV}{Z}\Bigr)^2,\\
(\mathbf 1_v,\mathbf 1_{v'})=(1,0):\quad \mathcal U&=\varphi\,\mathrm{AC}+\frac Z{\tilde Y}\Bigl(\hat\chi+\frac{\mathcal H\tilde p\tilde Y}{Z}\Bigr)^2,
\end{align*}
each nonnegative as $\varphi,\tilde Y,Z>0$. In case (ii), $\mathbf 1_v=0$, $\mathcal J=0$, $Z=\ppar\gammaV$, and
$\mathcal S_{v'}=(\alphaV\varphi\EKL+L\tilde p^2/\varphi)+(1-\ppar)(R^0_{v'}-\mathcal Q(\varphi,\chi))$ with
$\alphaV\varphi\EKL+L\tilde p^2/\varphi=\varphi\,\mathrm{CA}\ge0$ at $\tilde Y=0$.
\end{proof}

\begin{proposition}[the vertex step]\label{apB:prop:step}
Let $(\alphaV,\gammaV,A_D)$ be a step certificate at $\lambda$. Consider a vertex $v$ with indicator $\mathbf 1_v$ and
children $i$ (possibly none, and then $\mathbf 1_v=1$) with $\tilde Y_i\ge0$, $\mathbf 1_i\in\{0,1\}$ and
$\mathbf 1_v=1\Rightarrow\mathbf 1_i=0$; suppose each child satisfies $R^0_i\ge\mathcal Q(\mathrm{lm}_\lambda(\tilde Y_i),\chi_i)$,
the bound $R^1_i\ge\alphaV\tilde Y_i+\gammaV(\mathbf 1_i-\chi_i/p_i)^2/\tilde Y_i$ with $p_i=\mathrm{msg}_\lambda(\tilde Y_i)$ if
$\tilde Y_i>0$, and $\mathbf 1_i=1$,
$R^1_i=0$, $\chi_i=q$ if $\tilde Y_i=0$. With $\Ypar=\sum_i\mathrm{lm}_\lambda(\tilde Y_i)$, $\ppar=\mathrm{msg}_\lambda(\Ypar)$,
$\phipar=\mathrm{lm}_\lambda(\Ypar)$:
\[
A_D\ppar\mathbf 1_v+(1-\ppar)\sum_iR^0_i+\ppar\sum_iR^1_i-\ppar(1-\ppar)\Bigl(\mathbf 1_v-\sum_i\chi_i\Bigr)^2\ \ge\
\mathcal Q\Bigl(\phipar,\ \ppar\bigl(\mathbf 1_v-\textstyle\sum_i\chi_i\bigr)\Bigr).
\]
\end{proposition}
\begin{proof}
With no child, $\mathbf 1_v=1$, $\Ypar=0$, $\ppar=q$, $\phipar=\bar\varphi$, and the claim is the base inequality. Otherwise
$\Ypar>0$, each $(\tilde Y_i,\Ypar)$ is feasible, and Lemma~\ref{apB:lem:child} gives $\mathcal S_i\ge0$ for every child.
The weights $\varphi_i/\Ypar$ sum to $1$; with $\bar\chi=\sum_i\chi_i$,
$0\le\sum_i\mathcal S_i=A_D\ppar\mathbf 1_v-\mathcal H\mathbf 1_v^2-\alphaV\phipar+2\mathcal H\mathbf 1_v\bar\chi
-\mathcal H\Ypar\sum_i\chi_i^2/\varphi_i+(1-\ppar)\sum_iR^0_i+\ppar\sum_iR^1_i$. By Cauchy-Schwarz
$\bar\chi^2\le\Ypar\sum_i\chi_i^2/\varphi_i$, and $\mathcal H\ge0$, so
$(1-\ppar)\sum R^0_i+\ppar\sum R^1_i\ge\alphaV\phipar-A_D\ppar\mathbf 1_v+\mathcal H(\mathbf 1_v-\bar\chi)^2$. Add
$A_D\ppar\mathbf 1_v-\ppar(1-\ppar)(\mathbf 1_v-\bar\chi)^2$ and use $\mathcal H-\ppar(1-\ppar)=\gammaV\ppar^2/\phipar$.
\end{proof}

A rooted subtree $(X,v)$ is \emph{good} if every vertex of $X\setminus A$ has a child: $v\notin A$ has a neighbour in $X$,
and every $x\in X\setminus A$, $x\ne v$, has at least two neighbours in $X$. If $(X,v)$ is good, every branch
$(X_{v'},v')$ is good (a vertex $x\in X_{v'}$, $x\ne v'$, has all its neighbours in $X_{v'}$, and $v'\notin A$ has a
neighbour other than $v$), and a childless root of a good subtree lies in $A$.

\begin{theorem}[the variance bound from a step certificate]\label{apB:thm:band}
If $(\alphaV,\gammaV,A_D)$ is a step certificate at $\lambda$, with $A_D=1+(D-1)q$, then the reserve inequality
\eqref{apB:eq:R} holds at every good rooted subtree, and for every finite forest and every independent $A$ with (LP),
$\operatorname{Var}_\lambda|S\cap A|\le A_D\,\pi(A)$ and $\operatorname{Var}M\le D\,\muM$.
\end{theorem}
\begin{proof}
\emph{The induction.} Strong induction on $|X|$. The branches $(X_{v'},v')$ are proper, connected and good, so
$R^0_{v'}\ge\mathcal Q(\varphi_{v'},\chi_{v'})$ with $\varphi_{v'}=\mathrm{lm}_\lambda(Y_{v'})$ (Lemma~\ref{apB:lem:msg}). If
$Y_{v'}>0$, the branches of $(X_{v'},v')$ carry \eqref{apB:eq:R}, and Cauchy-Schwarz with the weights $\varphi_{v''}$ of the
children $v''$ of $v'$, with $\mathbf 1_{v'}-\chi_{v'}/p_{v'}=\sum_{v''}\chi_{v''}$ (Lemma~\ref{apB:lem:rec}), gives
$R^1_{v'}=\sum_{v''}R^0_{v''}\ge\alphaV Y_{v'}+\gammaV(\mathbf 1_{v'}-\chi_{v'}/p_{v'})^2/Y_{v'}$. If $Y_{v'}=0$, then $v'$ has no
child, $X_{v'}=\{v'\}$, $v'\in A$, $R^1_{v'}=0$ and $\chi_{v'}=p_{v'}=q$. If $\mathbf 1_v=1$, no child lies in $A$.
Proposition~\ref{apB:prop:step} with Lemma~\ref{apB:lem:rec} gives $R^0(X)\ge\mathcal Q(\varphi_v,\chi(X,v))$.

\emph{The rooting.} By (M1) it suffices to show $V_A(X)\le A_D\mu_A(X)$ for each component $X$. If $X$ has no vertex of
$A$ both sides vanish; if $X$ is a single vertex of $A$, it is its own root. Otherwise $X$ has a vertex $v\notin A$, and
$(X,v)$ is good: if $x\in X\setminus A$, $x\ne v$, had only one neighbour $x'$, then $x$ would be a leaf, by (LP) the
component would be $\{x,x'\}$ with $v=x'$, and it would contain no vertex of $A$. So
$A_D\mu_A(X)-V_A(X)=R^0(X)\ge\mathcal Q(\varphi_v,\chi)\ge0$. Summing over the components and Lemma~\ref{apB:lem:conv}
give the claim.
\end{proof}

\begin{lemma}[the entropy identity]\label{apB:lem:entropy}
At every feasible pair,
$\EKL\,\varphi=\mathrm{KL}\bigl(\mathrm{Ber}(\ppar)\,\|\,\mathrm{Ber}(\tilde p)\bigr)+(\Ypar-\varphi)\bigl(\ppar+\phipar/\Ypar\bigr)$,
where $\mathrm{KL}$ is the Kullback-Leibler divergence and $\mathrm{Ber}(p)$ the Bernoulli law with mean $p$. Hence
$\EKL\ge0$ and $\mathrm{CC}\ge0$.
\end{lemma}
\begin{proof}
$\log(1-\ppar)=-\phipar$ and $\log\ppar=\log\lambda-\Ypar-\phipar$, likewise for $\tilde p$ with $(\tilde Y,\varphi)$; so
$\mathrm{KL}=\ppar(\tilde Y-\Ypar)+\varphi-\phipar$, and $\EKL\varphi=(1-\ppar)\varphi+\ppar\tilde Y-\phipar\varphi/\Ypar$ equals
$\mathrm{KL}+(\Ypar-\varphi)(\ppar+\phipar/\Ypar)$ by expansion. Both terms are nonnegative, as $\Ypar\ge\varphi>0$.
\end{proof}

So a step certificate needs, besides the base inequality, only $Z>0$, $\mathrm{AC}\ge0$, and $\mathrm{CA}\ge0$ where $L<0$:
where $L\ge0$ and $Z>0$, $\mathrm{CA}\ge\alphaV\EKL\ge0$.

\paragraph{The curve of multipliers.}\label{hand:var:set}
The multipliers are given along a curve in the activity, piecewise on the six segments of Data~\ref{mc:tab:curve}: on a
segment with parameters $(\bar a,\bar g,D)$,
\[
\alphaV=\bar aq,\qquad \gammaV=\lambda(3+\lambda)/\bar g,\qquad A_D=1+(D-1)q,
\]
and $\hat D(\lambda)$ of \eqref{p2:eq:Dhat} is the least cap $D$ of a segment containing $\lambda$. Theorem~\ref{apB:thm:band}
uses one activity at a time, so a step certificate at each $\lambda$ is all that is needed; at a common end point of two
segments both give one. With $\mu=\log\lambda$ (so $\mathrm{msg}_\lambda(\zeta)=\mathrm{msg}_1(\zeta-\mu)$ and
$\mathrm{lm}_\lambda(\zeta)=\mathrm{lm}_1(\zeta-\mu)$) the multipliers of a segment increase in $\lambda$, with
$\mu$-derivatives
\[
\dot\alpha_{\mathrm V}=\bar a\,q(1-q),\qquad\dot\gamma_{\mathrm V}=\lambda(3+2\lambda)/\bar g,\qquad\dot A_D=(D-1)\,q(1-q),
\qquad q(1-q)=\lambda/(1+\lambda)^2 .
\]
\datablock{The curve of multipliers and its checks}{mc:tab:curve}
For each segment, the first tuple is $(\Lambda,\bar a,\bar g,D)$. The second tuple
$(n_{\rm sub},n_{\rm poly},n_{\rm cmp},n_{\rm pc},n_{\rm base})$ gives the numbers of activity sub-intervals,
child polygons, comparison checks, polygon checks and base-inequality checks, respectively
(Theorem~\ref{tgt:thm}). Five comparison checks on each sub-interval are those of
Lemma~\ref{tgt:lem:hand}. For each segment we specify whether second derivatives and transport bounds are taken over the whole parent
interval or its two halves, and report the least lower bound of a comparison after transport, rounded down.
On the \textnormal{lo} segment, the parameter tuple is $([1/3,11/20], 11/25, 21/2, 6/5)$ and the check record is $(6, 6, 312, 216, 2)$.
Second derivatives and transport bounds use the whole parent interval; the least lower bound is $0.00011$. On the \textnormal{mid} segment, the parameter tuple is $([11/20,263/200], 23/50, 13, 6/5)$ and the check record is $(21, 19, 1{,}283, 678, 6)$.
Second derivatives and transport bounds use the two halves of each parent interval; the least lower bound is $0.0006$.

On the \textnormal{u1} segment, the parameter tuple is $([263/200,3/2], 9/20, 23/2, 7/5)$ and the check record is $(4, 4, 232, 144, 1)$.
Second derivatives and transport bounds use the whole parent interval; the least lower bound is $0.00018$. On the \textnormal{u2} segment, the parameter tuple is $([3/2,7/4], 47/100, 43/4, 8/5)$ and the check record is $(5, 5, 273, 193, 1)$.
Second derivatives and transport bounds use the whole parent interval; the least lower bound is $0.00004$.

On the \textnormal{u3} segment, the parameter tuple is $([7/4,2], 1/2, 41/4, 9/5)$ and the check record is $(5, 5, 258, 186, 1)$.
Second derivatives and transport bounds use the whole parent interval; the least lower bound is $0.00005$. On the \textnormal{u4} segment, the parameter tuple is $([2,7/3], 57/100, 39/4, 21/10)$ and the check record is $(6, 6, 306, 227, 2)$.
Second derivatives and transport bounds use the whole parent interval; the least lower bound is $0.00002$.
The totals $(n_{\rm sub},n_{\rm poly},n_{\rm cmp},n_{\rm pc},n_{\rm base})$ are
$(47, 45, 2{,}664, 1{,}644, 13)$.

For the parent of a feasible pair put $\mathsf e_\lambda=e^{\Ypar}/\lambda$, $\mathsf k=\ppar/\phipar$ and
$\mathsf w=\mathsf k+\ppar-1$, so that $(1-\ppar)/\ppar=\mathsf e_\lambda$, and for the child $\mathsf r=\tilde Y/\varphi$, so that
$\mathcal J=\tilde p^2\mathsf r$. Put
\[\hat h=\mathcal H/\ppar=1-\ppar+\gammaV\mathsf k,\qquad \bar L=L/\ppar=\gammaV\mathsf e_\lambda-\hat h\Ypar,\]
\[\bar Z=Z/\ppar=\gammaV+\mathcal J\bar L,\qquad \mathsf N=A_D-\hat h-\alphaV/\mathsf k .\]
Since $\EKL/\ppar=\mathsf e_\lambda+\mathsf r-1/(\mathsf k\Ypar)$, dividing by $\ppar>0$ gives
\[\mathrm{AC}/\ppar=\mathsf A(\Ypar;\mathsf r,\mathcal J):=\alphaV\mathsf e_\lambda+\frac{\mathsf N}\Ypar+\alphaV\mathsf r-\frac{\mathcal J\hat h^2}{\bar Z},\qquad
\mathrm{CA}/\ppar=\alphaV\frac\EKL\ppar+\frac{\gammaV\bar L(\tilde p/\varphi)^2}{\bar Z},\]
and we put $\mathsf A^{\gammaV}(\Ypar;\mathsf r,\mathcal J):=\alphaV\mathsf e_\lambda+\mathsf N/\Ypar+\alphaV\mathsf r-\mathcal J\hat h^2/\gammaV$.
Since $1/\bar Z-1/\gammaV=-\mathcal J\bar L/(\gammaV\bar Z)$, where $\bar Z>0$,
\begin{equation}\label{arc:eq:maj}
\mathsf A-\mathsf A^{\gammaV}=\frac{\hat h^2\bar L\mathcal J^2}{\gammaV\bar Z},\quad\text{so for }\mathcal J>0\text{, }\mathsf A^{\gammaV}\le\mathsf A\text{ exactly where }\bar L\ge0 .
\end{equation}
Below, $\bar L$ is a quotient, not a derivative; a prime denotes $d/d\Ypar$, except for functions of the child, where it
denotes $d/d\tilde Y$, and a dot denotes $d/d\mu$. An \emph{activity sub-interval} is an interval $\Lambda=[\llo,\lhi]$ inside
one segment, and a \emph{parent interval} $[Y_0,Y_1]$ is an interval of values of $\Ypar$. The lemmas below hold at every
activity of $[1/3,7/3]$ and for the multipliers of every segment, except where a range is stated.

\begin{lemma}[two consequences of feasibility]\label{tgt:lem:feas}
Let $(\tilde Y,\Ypar)$ be a feasible pair at the activity $\lambda$.
The first consequence, (a), is that $\mathsf r\ge\tilde Y/\Ypar$ and $\tilde Y\ge\mu-\log(e^{\Ypar}-1)$; hence
$\EKL/\ppar\ge\mathsf E_-(\Ypar):=\mathsf e_\lambda+\bigl(\mu-\log(e^{\Ypar}-1)-1/\mathsf k\bigr)/\Ypar$.

For the second consequence, (b), if $L<0$, put $\hat L=-\bar L>0$. If $\mathcal J\le\bar{\mathcal J}$ and $\gammaV-\bar{\mathcal J}\hat L>0$, then
$\mathrm{CA}/\ppar\ge\mathsf C(\Ypar):=\alphaV\mathsf E_-(\Ypar)-\gammaV\hat L/(\gammaV-\bar{\mathcal J}\hat L)$.
\end{lemma}
\begin{proof}
(a) Feasibility gives $\varphi=\mathrm{lm}_\lambda(\tilde Y)\le\Ypar$, so $\mathsf r=\tilde Y/\varphi\ge\tilde Y/\Ypar$; and
$\Ypar\ge\log(1+\lambda e^{-\tilde Y})$ is equivalent to $\tilde Y\ge\mu-\log(e^{\Ypar}-1)$. Since the coefficient $1/\Ypar$ of $\tilde Y$ is
positive, $\EKL/\ppar=\mathsf e_\lambda+\mathsf r-1/(\mathsf k\Ypar)\ge\mathsf e_\lambda+(\tilde Y-1/\mathsf k)/\Ypar\ge\mathsf E_-(\Ypar)$.
(b) $(\tilde p/\varphi)^2<1$ because $-\log(1-\tilde p)>\tilde p$, and $\mathcal J\mapsto1/(\gammaV-\mathcal J\hat L)$ increases on $[0,\bar{\mathcal J}]$ (as $\hat L>0$ and the denominator
stays positive there), so
$\gammaV\bar L(\tilde p/\varphi)^2/\bar Z=-\gammaV\hat L(\tilde p/\varphi)^2/(\gammaV-\mathcal J\hat L)\ge-\gammaV\hat L/(\gammaV-\bar{\mathcal J}\hat L)$.
\end{proof}

\begin{lemma}[derivatives in $\Ypar$]\label{tgt:lem:dY}\label{arc:lem:der}
At a fixed activity, with $'=d/d\Ypar$: $\ppar'=-\ppar(1-\ppar)$, $\phipar'=-\ppar$, $\mathsf e_\lambda'=\mathsf e_\lambda$, $\mathsf k'=\mathsf k\mathsf w$,
$\mathsf w'=\mathsf k\mathsf w-\ppar(1-\ppar)$, and $\mathsf w>0$. Further
\begin{align*}
\hat h'&=\ppar(1-\ppar)+\gammaV\mathsf k\mathsf w,\\ \hat h''&=-\ppar(1-\ppar)(1-2\ppar)+\gammaV\mathsf k\bigl(\mathsf w^2+\mathsf k\mathsf w-\ppar(1-\ppar)\bigr),\\
\mathsf N'&=-\hat h'+\frac{\alphaV\mathsf w}{\mathsf k},\qquad \mathsf N''=-\hat h''+\alphaV\Bigl(\mathsf w-\frac{\ppar(1-\ppar)}{\mathsf k}-\frac{\mathsf w^2}{\mathsf k}\Bigr),\\
\Bigl(\frac{\mathsf N}\Ypar\Bigr)''&=\frac{\mathsf N''\Ypar^2-2\mathsf N'\Ypar+2\mathsf N}{\Ypar^3},\\
\bar Z'&=\mathcal J\bar L',\quad \bar Z''=\mathcal J\bar L'',\qquad \bar L'=\gammaV\mathsf e_\lambda-\hat h-\hat h'\Ypar,\quad \bar L''=\gammaV\mathsf e_\lambda-2\hat h'-\hat h''\Ypar,\\
\Bigl(\frac{\mathcal J\hat h^2}{\bar Z}\Bigr)''&=\mathcal J\Bigl[\frac{2\hat h'^2+2\hat h\hat h''}{\bar Z}-\frac{4\hat h\hat h'\bar Z'}{\bar Z^2}-\frac{\hat h^2\bar Z''}{\bar Z^2}+\frac{2\hat h^2\bar Z'^2}{\bar Z^3}\Bigr],
\end{align*}
so that, with $\mathsf r$ and $\mathcal J$ fixed, $(\mathcal J\hat h^2/\bar Z)'=\mathcal J(2\hat h\hat h'\bar Z-\hat h^2\bar Z')/\bar Z^2$,
$\mathsf A'=\alphaV\mathsf e_\lambda+(\mathsf N'\Ypar-\mathsf N)/\Ypar^2-(\mathcal J\hat h^2/\bar Z)'$,
$(\mathsf A^{\gammaV})'=\alphaV\mathsf e_\lambda+(\mathsf N'\Ypar-\mathsf N)/\Ypar^2-2\mathcal J\hat h\hat h'/\gammaV$,
$\mathsf A''=\alphaV\mathsf e_\lambda+(\mathsf N/\Ypar)''-(\mathcal J\hat h^2/\bar Z)''$ and
$(\mathsf A^{\gammaV})''=\alphaV\mathsf e_\lambda+(\mathsf N/\Ypar)''-\mathcal J(2\hat h'^2+2\hat h\hat h'')/\gammaV$. For $\mathsf C$, with
$\mathsf n=\mu-\log(e^{\Ypar}-1)-1/\mathsf k$ and $\Theta(\hat L)=\gammaV\hat L/(\gammaV-\bar{\mathcal J}\hat L)$:
\begin{gather*}\mathsf n'=-\frac{e^{\Ypar}}{e^{\Ypar}-1}+\frac{\mathsf w}{\mathsf k},\qquad \mathsf n''=\frac{e^{\Ypar}}{(e^{\Ypar}-1)^2}+\mathsf w-\frac{\ppar(1-\ppar)}{\mathsf k}-\frac{\mathsf w^2}{\mathsf k},\\
\mathsf E_-''=\mathsf e_\lambda+\frac{\mathsf n''\Ypar^2-2\mathsf n'\Ypar+2\mathsf n}{\Ypar^3},\end{gather*}
$\hat L'=\hat h+\hat h'\Ypar-\gammaV\mathsf e_\lambda$, $\hat L''=2\hat h'+\hat h''\Ypar-\gammaV\mathsf e_\lambda$, and, for the derivatives of $\Theta$ in its
argument, $\Theta'=\gammaV^2/(\gammaV-\bar{\mathcal J}\hat L)^2$, $\Theta''=2\bar{\mathcal J}\gammaV^2/(\gammaV-\bar{\mathcal J}\hat L)^3$, and
$\mathsf C''=\alphaV\mathsf E_-''-\Theta''\hat L'^2-\Theta'\hat L''$.
\end{lemma}
\begin{proof}
$\phipar'=-e^{-\Ypar}\lambda/(1+\lambda e^{-\Ypar})=-\ppar$ and $\ppar'=-\ppar(1-\ppar)$ are direct, and
$\mathsf k'=(\ppar'\phipar-\ppar\phipar')/\phipar^2=-(1-\ppar)\mathsf k+\mathsf k^2=\mathsf k\mathsf w$. From $-\log(1-\ppar)<\ppar/(1-\ppar)$ we get
$\mathsf k>1-\ppar$, that is $\mathsf w>0$. The rest follows from the product and quotient rules; for the last line of the
display write $\mathcal J\hat h^2/\bar Z=\hat h^2\mathsf s$ with $\mathsf s=\mathcal J/\bar Z$, $\mathsf s'=-\bar L'\mathsf s^2$.
\end{proof}

\begin{lemma}[the child]\label{tgt:lem:child}
(a) At fixed $\tilde Y$, $\mathsf r$ decreases and $\mathcal J$ increases in $\lambda\in(0,\frac52]$.
(b) Fix $\lambda$ and a chain $t_0<t_1<\dots<t_r<t_{r+1}=\infty$. If $1/t_-+1/(1+\tilde p(t_+)/2)-2+2\tilde p(t_-)<0$ on
each segment $[t_-,t_+]=[t_i,t_{i+1}]$ of the chain, where $\tilde p(t)=\mathrm{msg}_\lambda(t)$ and $\tilde p(\infty)=0$, then
$\mathcal J$ decreases in $\tilde Y$ on $[t_0,\infty)$.
\end{lemma}
\begin{proof}
(a) $\varphi=\mathrm{lm}_\lambda(\tilde Y)$ increases in $\lambda$, so $\mathsf r$ decreases. With $\mathsf z=\lambda e^{-\tilde Y}\le\lambda$,
$\tilde p^2/\varphi=\mathsf z^2/((1+\mathsf z)^2\log(1+\mathsf z))$ has logarithmic derivative
$(2\log(1+\mathsf z)-\mathsf z)/(\mathsf z(1+\mathsf z)\log(1+\mathsf z))>0$ for $0<\mathsf z\le\frac52$, since
$2\log(1+\mathsf z)-\mathsf z$ is concave, vanishes at $0$ and equals $2\log\frac72-\frac52>0$ at $\frac52$.
(b) Since $\tilde p'=-\tilde p(1-\tilde p)$ and $\varphi'=-\tilde p$ in $\tilde Y$, $(\log\mathcal J)'=1/\tilde Y+\tilde p/\varphi-2(1-\tilde p)$. From
$-\log(1-\tilde p)\ge\tilde p+\tilde p^2/2$, $\tilde p/\varphi\le1/(1+\tilde p/2)$; on $[t_-,t_+]$, $1/\tilde Y\le1/t_-$, $\tilde p\le\tilde p(t_-)$ and
$\tilde p\ge\tilde p(t_+)$.
\end{proof}

\paragraph{Transport in the activity.}
\begin{lemma}[translation]\label{tgt:lem:tr}\label{mc:lem:ext}
$\mathsf e_\lambda,\ppar,\phipar,\mathsf k,\mathsf w$ are functions of $\mathsf x=\Ypar-\mu$. Hence at fixed $\Ypar$: $\partial_\mu\mathsf e_\lambda=-\mathsf e_\lambda$,
$\partial_\mu\ppar=\ppar(1-\ppar)$, $\partial_\mu\mathsf k=-\mathsf k\mathsf w$. On a box $[Y_0,Y_1]\times[\llo,\lhi]$ with $Y_0\ge0$,
$\mathsf e_\lambda,\ppar,\phipar,\mathsf k$ are enclosed by their values at the corners $(Y_0,\lhi)$ and $(Y_1,\llo)$; so is
$\mathsf w$ if $\lhi\le\frac{263}{200}$, and in general $\mathsf w$ lies in the interval $\mathsf k+\ppar-1$ computed from these
enclosures.
\end{lemma}
\begin{proof}
$\mathsf e_\lambda=e^{\Ypar-\mu}$, $\ppar=\mathrm{msg}_1(\Ypar-\mu)$, $\phipar=\mathrm{lm}_1(\Ypar-\mu)$; the derivatives are those of
Lemma~\ref{tgt:lem:dY} with the sign of the argument reversed. In $\mathsf x=\Ypar-\mu$, $\mathsf e_\lambda$ and $\mathsf k$ increase
($\mathsf k'=\mathsf k\mathsf w>0$) and $\ppar$, $\phipar$ decrease. For $\mathsf w$, $d\mathsf w/d\mathsf x=\mathsf k\mathsf w-\ppar(1-\ppar)$, and with
$t=\phipar=-\log(1-\ppar)$ one has $\mathsf k\mathsf w<\ppar(1-\ppar)$ if and only if $e^t-1<t+t^2$; for $0<t\le1$,
$e^t-1-t-t^2\le-\frac{t^2}2+\frac e6t^3<0$. So $\mathsf w$ decreases in $\mathsf x$ wherever $\ppar\le1-e^{-1}$, which holds for
$\Ypar\ge0$ and $\lambda\le\frac{263}{200}$ since then $\ppar\le q<\frac47$. The corners $(Y_0,\lhi)$ and $(Y_1,\llo)$ give the
least and the largest $\mathsf x$ on the box. Only the enclosure of $\mathsf w$ by the corners uses a bound on $\lambda$ (it needs
$\ppar\le1-e^{-1}$, which fails near $\Ypar=0$ once $\lambda>e-1$).
\end{proof}

\begin{lemma}[the $\mu$-derivatives]\label{tgt:lem:dmu}
For the multipliers of a segment, with their $\mu$-derivatives as above,
\[\dot{\hat h}=-\ppar(1-\ppar)-\gammaV\mathsf k\mathsf w+\dot\gamma_{\mathrm V}\mathsf k,\qquad\dot{\mathsf N}=\dot A_D-\dot{\hat h}-\frac{\dot\alpha_{\mathrm V}}{\mathsf k}-\frac{\alphaV\mathsf w}{\mathsf k},\qquad
\dot{\bar Z}=\dot\gamma_{\mathrm V}+\mathcal J(\dot\gamma_{\mathrm V}\mathsf e_\lambda-\gammaV\mathsf e_\lambda-\dot{\hat h}\Ypar),\]
at fixed $(\Ypar,\mathsf r,\mathcal J)$:
\begin{align*}
\partial_\mu\mathsf A&=(\dot\alpha_{\mathrm V}-\alphaV)\mathsf e_\lambda+\frac{\dot{\mathsf N}}\Ypar+\dot\alpha_{\mathrm V}\mathsf r-\mathcal J\,\frac{2\hat h\dot{\hat h}\bar Z-\hat h^2\dot{\bar Z}}{\bar Z^2},\\
\partial_\mu\mathsf A^{\gammaV}&=(\dot\alpha_{\mathrm V}-\alphaV)\mathsf e_\lambda+\frac{\dot{\mathsf N}}\Ypar+\dot\alpha_{\mathrm V}\mathsf r-\mathcal J\Bigl(\frac{2\hat h\dot{\hat h}}{\gammaV}
-\frac{\hat h^2\dot\gamma_{\mathrm V}}{\gammaV^2}\Bigr),\\
\partial_\mu\mathsf C&=\dot\alpha_{\mathrm V}\mathsf E_-+\alphaV\Bigl(-\mathsf e_\lambda+\frac{1-\mathsf w/\mathsf k}{\Ypar}\Bigr)+\frac{\dot\gamma_{\mathrm V}\bar{\mathcal J}\hat L^2
-\gammaV^2(\dot{\hat h}\Ypar-\dot\gamma_{\mathrm V}\mathsf e_\lambda+\gammaV\mathsf e_\lambda)}{(\gammaV-\bar{\mathcal J}\hat L)^2}\qquad(\bar{\mathcal J}\text{ fixed}).
\end{align*}
\end{lemma}
\begin{proof}
Product and quotient rules with Lemma~\ref{tgt:lem:tr} and $d\alphaV/d\mu=\bar a\lambda\,dq/d\lambda$,
$d\gammaV/d\mu=\lambda\,d\gammaV/d\lambda$, $dA_D/d\mu=(D-1)\lambda\,dq/d\lambda$: $\partial_\mu\hat h=-\partial_\mu\ppar+\dot\gamma_{\mathrm V}\mathsf k+\gammaV\partial_\mu\mathsf k$,
$\partial_\mu(\alphaV/\mathsf k)=\dot\alpha_{\mathrm V}/\mathsf k+\alphaV\mathsf w/\mathsf k$, $\partial_\mu(\mu-\log(e^{\Ypar}-1)-1/\mathsf k)=1-\mathsf w/\mathsf k$,
$\partial_\mu\hat L=\dot{\hat h}\Ypar-\dot\gamma_{\mathrm V}\mathsf e_\lambda+\gammaV\mathsf e_\lambda$, and
$\partial_{\gammaV}\Theta=-\bar{\mathcal J}\hat L^2/(\gammaV-\bar{\mathcal J}\hat L)^2$, $\partial_{\hat L}\Theta=\gammaV^2/(\gammaV-\bar{\mathcal J}\hat L)^2$
for $\Theta=\gammaV\hat L/(\gammaV-\bar{\mathcal J}\hat L)$.
\end{proof}

\begin{lemma}[transport inequality]\label{tgt:lem:mvt}
Let $\Lambda=[\llo,\lhi]\ni\lambda_c$, $\varepsilon_-\ge\log(\lambda_c/\llo)$ and $\varepsilon_+\ge\log(\lhi/\lambda_c)$.
Let $\mathsf F$ be defined on $[Y_0,Y_1]\times\Lambda$ and continuously differentiable in $\mu$. If
$\mathsf F(\cdot,\lambda_c)\ge\underline{\mathsf F}$ on $[Y_0,Y_1]$ and $d_-\le\partial_\mu\mathsf F\le d_+$ on
$[Y_0,Y_1]\times\Lambda$, then on $[Y_0,Y_1]\times\Lambda$
\[\mathsf F\ \ge\ \underline{\mathsf F}-\max\bigl(\varepsilon_+\max(0,-d_-),\ \varepsilon_-\max(0,d_+)\bigr).\]
\end{lemma}
\begin{proof}The mean value theorem in $\mu$ at fixed $\Ypar$, separately for $\mu\ge\mu_c$ and $\mu\le\mu_c$.\end{proof}

\begin{lemma}[the child polygon moves monotonically]\label{tgt:lem:poly}
Let $\mathcal J\le \mathsf a_i+\mathsf b_i\mathsf r$ with $\mathsf b_i\ge0$ ($i=0,1,2$) and $\mathcal J\le\mathcal J_{\rm cap}$ hold on the child curve
$\{(\mathsf r(\tilde Y),\mathcal J(\tilde Y)):\tilde Y\ge0\}$ at an activity $\lambda_p\le\frac52$. Then they hold on the child curve at every
activity $\lambda\le\lambda_p$.
\end{lemma}
\begin{proof}
At fixed $\tilde Y$, $\mathcal J_\lambda-\mathsf b_i\mathsf r_\lambda\le\mathcal J_{\lambda_p}-\mathsf b_i\mathsf r_{\lambda_p}\le \mathsf a_i$ and $\mathcal J_\lambda\le\mathcal J_{\lambda_p}\le
\mathcal J_{\rm cap}$ by Lemma~\ref{tgt:lem:child}(a), where the subscript is the activity.
\end{proof}
The \emph{child polygon} at an activity $\lambda_p$ consists of three lines with $\mathsf b_0>\mathsf b_1>\mathsf b_2>0$ and a
cap; $45$ polygons are used, at activities $\lambda_p$, and each activity sub-interval $[\llo,\lhi]$ of the checks uses one
at an activity $\lambda_p\ge\lhi$ (their constants and this assignment are in the supplementary data,
Section~\ref{sec:computer}).
The upper boundary
$\mathsf U(\mathsf r)=\min(\mathsf a_0+\mathsf b_0\mathsf r,\mathsf a_1+\mathsf b_1\mathsf r,\mathsf a_2+\mathsf b_2\mathsf r,\mathcal J_{\rm cap})$ is concave, with vertices
$(\mathsf r_0,\mathcal J_0)=(0,\mathsf a_0)$ and the three consecutive intersections $(\mathsf r_i,\mathcal J_i)$, $i=1,2,3$. Each of the inequalities
$\mathsf f:=\mathcal J-\mathsf b_i\mathsf r\le\mathsf a_i$ ($i=0,1,2$) and $\mathsf f:=\mathcal J\le\mathcal J_{\rm cap}$ is proved at $\lambda_p$ as follows, with $t=\tilde Y$
and a prime for $d/dt$. On $[0,t_\ast]$ the interval is cut into subintervals with centre $c_t$ and half-width $d_t$, and
$\mathsf f(t)\le\mathsf f(c_t)+\mathsf f'(c_t)(t-c_t)+\frac12\max(0,\sup\mathsf f'')(t-c_t)^2$, whose maximum is at an end;
$\mathsf f(c_t)$, $\mathsf f'(c_t)$ are enclosed at the centre, and $\sup\mathsf f''$ is the upper end of the interval value of $\mathsf f''$ on the
subinterval, computed from $\tilde p'=-\tilde p(1-\tilde p)$, $\tilde p''=\tilde p(1-\tilde p)(1-2\tilde p)$, $\varphi'=-\tilde p$,
$\varphi''=\tilde p(1-\tilde p)$ with $\tilde p$ and $\varphi$ enclosed by their values at the two ends (both decrease in $\tilde Y$). On
$[t_\ast,\infty)$, $\mathcal J$ decreases by Lemma~\ref{tgt:lem:child}(b) (a chain of at most eight segments) and $\mathsf r$ increases, so
$\mathsf f$ decreases ($t_\ast$ is part of the polygon's data). We put
$\bar{\mathcal J}:=\mathcal J_{\rm cap}(\lambda_p)$, which bounds $\mathcal J$ for every child at every $\lambda\le\lambda_p$.

\begin{lemma}[the feasibility cut]\label{tgt:lem:cut}
If $(\tilde Y,\Ypar)$ is feasible at an activity $\lambda\ge\llo$ and $\Ypar\le Y_1$, then
$\mathsf r\ge\mathsf r_{\rm f}(Y_1):=\max\bigl(0,\log\llo-\log(e^{Y_1}-1)\bigr)/Y_1$.
\end{lemma}
\begin{proof}By Lemma~\ref{tgt:lem:feas}(a), $\mathsf r\ge\tilde Y/\Ypar\ge\tilde Y/Y_1$ and
$\tilde Y\ge\max(0,\mu-\log(e^{\Ypar}-1))\ge\max(0,\log\llo-\log(e^{Y_1}-1))$.\end{proof}
\noindent In the computation $\mathsf r_{\rm f}(Y_1)$ is replaced by a smaller six-decimal rational $\underline{\mathsf r}$, which only
enlarges the region used below.

\begin{lemma}[bounds on an activity sub-interval]\label{tgt:lem:hand}
Let $\Lambda=[\llo,\lhi]$ lie in one segment of Data~\ref{mc:tab:curve}, with $\lhi\le\lambda_p\le\frac73$, write
$\alphaV(\lambda),\gammaV(\lambda),A_D(\lambda)$ for the multipliers of the segment at the activity $\lambda$, and $\ppar(\Ypar,\lambda)$, $\mathsf k(\Ypar,\lambda)$ for $\ppar$, $\mathsf k$ at $(\Ypar,\lambda)$. The following bounds hold for every
$\lambda\in\Lambda$.
The uniform bound (i) states that for every $\Ypar>0$, $-\bar L\le\mathsf c_1\max\bigl(0,\log(\mathsf c_1/\mathsf c_2)-1\bigr)$, where
$\mathsf c_1=1+\gammaV(\lhi)$ and $\mathsf c_2=\gammaV(\llo)/\lhi$; hence, for every child,
$\bar Z\ge\gammaV(\llo)-\bar{\mathcal J}\mathsf c_1\max\bigl(0,\log(\mathsf c_1/\mathsf c_2)-1\bigr)$, and
$\gammaV-\bar{\mathcal J}\max(0,-\bar L)$ has the same lower bound.

The small-parent bound (ii) states that for a rational $Y_L>0$ (fixed for each sub-interval of the checks), for $0<\Ypar\le Y_L$ and every child of a
feasible pair,
$\mathrm{AC}/\ppar\ge\alphaV(\llo)/\lhi+(\mathsf N_{\rm lo}+\alphaV(\llo)t_L)/Y_L-\bar{\mathcal J}\hat h_{\rm hi}^2/\bar Z_{\rm lo}$, where
$t_L=\max(0,\log\llo-\log(e^{Y_L}-1))$, $\hat h_{\rm hi}=1-\ppar(Y_L,\llo)+\gammaV(\lhi)\mathsf k(Y_L,\llo)$,
$\mathsf N_{\rm lo}=A_D(\llo)-\hat h_{\rm hi}-\alphaV(\lhi)/\mathsf k(0,\lhi)$ and
$\bar Z_{\rm lo}=\gammaV(\llo)+\bar{\mathcal J}\min(0,\gammaV(\llo)/\lhi-\hat h_{\rm hi}Y_L)$, provided $\mathsf N_{\rm lo}+\alphaV(\llo)t_L\ge0$ and
$\bar Z_{\rm lo}>0$.

The large-parent bound (iii) states that for $\Ypar\ge3$ and every child,
$\mathrm{AC}/\ppar\ge\alphaV(\llo)e^3/\lhi+\mathsf N^{(3)}_{\rm lo}/3-(1+\gammaV(\lhi))^2\bar{\mathcal J}/\bar Z^{(3)}_{\rm lo}$, where
$\mathsf N^{(3)}_{\rm lo}=A_D(\llo)-1-\gammaV(\lhi)-\alphaV(\lhi)/\mathsf k(3,\lhi)$ and
$\bar Z^{(3)}_{\rm lo}=\min\bigl(\gammaV(\llo),\gammaV(\llo)-\bar{\mathcal J}\max(0,3(1+\gammaV(\lhi))-\gammaV(\llo)e^3/\lhi)\bigr)$, provided
$\mathsf N^{(3)}_{\rm lo}<0$, $\bar Z^{(3)}_{\rm lo}>0$ and $\log((1+\gammaV(\lhi))\lhi/\gammaV(\llo))\le3$.

Finally, the two bounds in (iv) are $\bar L\ge\gammaV(\llo)/\lhi-(1+\gammaV(\lhi))\Ypar$ for $0<\Ypar\le\frac1{500}$, and
$\bar L\ge\gammaV(\llo)e^{\Ypar}/\lhi-(1+\gammaV(\lhi))\Ypar$, which increases for $\Ypar\ge3$ under the last proviso of (iii).
On each of the $47$ activity sub-intervals of the checks (Theorem~\ref{tgt:thm}) the right-hand sides of
(i), (ii), (iii) and the two bounds of (iv) at $\Ypar=\frac1{500}$ and $\Ypar=3$ are positive and the provisos hold: five
interval-arithmetic checks on each sub-interval.
\end{lemma}
\begin{proof}
(i) $\hat h\le1+\gammaV\le\mathsf c_1$ since $\mathsf k\le1$, and $\gammaV\mathsf e_\lambda\ge\mathsf c_2e^{\Ypar}$, so
$-\bar L\le\mathsf c_1\Ypar-\mathsf c_2e^{\Ypar}$; for $\Ypar>0$ this is at most $\mathsf c_1(\log(\mathsf c_1/\mathsf c_2)-1)$,
its maximum at $\Ypar=\log(\mathsf c_1/\mathsf c_2)$, if $\mathsf c_1>\mathsf c_2$, and negative otherwise. With
$\gammaV\ge\gammaV(\llo)$ and $0\le\mathcal J\le\bar{\mathcal J}$ this gives both bounds. (ii)
$\alphaV\mathsf e_\lambda\ge\alphaV(\llo)/\lhi$. By Lemma~\ref{tgt:lem:tr}, $1-\ppar\le1-\ppar(Y_L,\llo)$, $\mathsf k\le\mathsf k(Y_L,\llo)$ and
$\mathsf k\ge\mathsf k(0,\lhi)$, so $\hat h\le\hat h_{\rm hi}$ and $\mathsf N\ge\mathsf N_{\rm lo}$. By Lemma~\ref{tgt:lem:cut},
$\alphaV\mathsf r\ge\alphaV(\llo)t_L/\Ypar$, so
$\mathsf N/\Ypar+\alphaV\mathsf r\ge(\mathsf N_{\rm lo}+\alphaV(\llo)t_L)/\Ypar\ge(\mathsf N_{\rm lo}+\alphaV(\llo)t_L)/Y_L$. Finally
$\bar L\ge\gammaV(\llo)/\lhi-\hat h_{\rm hi}Y_L$, and $\mathcal J/(\gammaV+\mathcal J\bar L)$ increases in $\mathcal J$, so
$\mathcal J\hat h^2/\bar Z\le\bar{\mathcal J}\hat h_{\rm hi}^2/\bar Z_{\rm lo}$. (iii) $\hat h\le1+\gammaV(\lhi)$ and $\mathsf k\ge\mathsf k(3,\lhi)$ give
$\mathsf N\ge\mathsf N^{(3)}_{\rm lo}$, so $\mathsf N/\Ypar\ge\mathsf N^{(3)}_{\rm lo}/3$ as $\mathsf N^{(3)}_{\rm lo}<0$; $\alphaV\mathsf e_\lambda\ge\alphaV(\llo)e^3/\lhi$
and $\mathsf r\ge0$; $-\bar L\le(1+\gammaV(\lhi))\Ypar-\gammaV(\llo)e^{\Ypar}/\lhi$, which decreases for
$\Ypar\ge\log((1+\gammaV(\lhi))\lhi/\gammaV(\llo))$, so $\bar Z\ge\bar Z^{(3)}_{\rm lo}$ and
$\mathcal J\hat h^2/\bar Z\le(1+\gammaV(\lhi))^2\bar{\mathcal J}/\bar Z^{(3)}_{\rm lo}$. (iv) $\hat h\le1+\gammaV(\lhi)$ and
$\gammaV\mathsf e_\lambda\ge\gammaV(\llo)e^{\Ypar}/\lhi$. The proof uses the multipliers only through their increase in
$\lambda$ on $\Lambda$ ($\bar a,\bar g>0$, $D\ge1$), and the activity only through the corner enclosures of $\ppar$ and $\mathsf k$
(Lemma~\ref{tgt:lem:tr}), Lemma~\ref{tgt:lem:cut} and Lemma~\ref{tgt:lem:poly} (valid for $\lambda_p\le\frac52$).
\end{proof}

\begin{proposition}[reduction to the vertices]\label{tgt:prop:red}
Let $\lambda\in[\llo,\lambda_p]$, $\Ypar\in[Y_0,Y_1]$, and let $0\le\underline{\mathsf r}\le\mathsf r_{\rm f}(Y_1)$, with $\bar Z>0$ for
$\mathcal J\le\bar{\mathcal J}$ as in Lemma~\ref{tgt:lem:hand}(i). The minimum of $\mathrm{AC}/\ppar$ over the children of feasible pairs
with parent log-mass $\Ypar$ is at least $\min_i\min(\mathsf A(\Ypar;\mathsf r_i,\mathcal J_i),\mathsf A^{\gammaV}(\Ypar;\mathsf r_i,\mathcal J_i))$, the
minimum over the vertices $(\mathsf r_i,\mathcal J_i)$ of the polygon at $\lambda_p$ with $\mathsf r_i\ge\underline{\mathsf r}$ together with the
point $(\underline{\mathsf r},\mathsf U(\underline{\mathsf r}))$. For $\underline{\mathsf r}=0$ these are the four vertices.
\end{proposition}
\begin{proof}
Put $\mathsf M(\mathcal J)=\hat h^2\mathcal J/(\gammaV+\min(\bar L,0)\mathcal J)\ge\mathcal J\hat h^2/\bar Z$. As the denominator is positive for
$0\le\mathcal J\le\bar{\mathcal J}$, $\mathsf M$ increases in $\mathcal J$ and is linear ($\bar L\ge0$) or convex ($\bar L<0$). By
Lemmas~\ref{tgt:lem:poly} and~\ref{tgt:lem:cut} every such child lies in
$\mathsf G=\{\mathsf r\ge\underline{\mathsf r},\ 0\le\mathcal J\le\mathsf U(\mathsf r)\}$. On $\mathsf G$, $\alphaV\mathsf r-\mathsf M(\mathcal J)$ increases in $\mathsf r$,
decreases in $\mathcal J$ and is concave along each edge of the graph of $\mathsf U$, and it increases in $\mathsf r$ along the cap; so its
minimum over $\mathsf G$ is at $(\underline{\mathsf r},\mathsf U(\underline{\mathsf r}))$ or at a vertex to its right. Finally
$\alphaV\mathsf e_\lambda+\mathsf N/\Ypar+\alphaV\mathsf r-\mathsf M(\mathcal J)=\min(\mathsf A,\mathsf A^{\gammaV})$, by \eqref{arc:eq:maj}.
\end{proof}

\paragraph{The checks.}\label{hand:var:curve}
Every finite check below is an evaluation of explicit formulas in interval arithmetic with exact rational end points:
$\exp$ and $\log$ are enclosed by series with explicit remainders, and every operation is rounded outward to a multiple of
$2^{-48}$. On an activity sub-interval $\Lambda$ with centre $\lambda_c=(\llo+\lhi)/2$, $\varepsilon_\pm$ are the upper ends of
enclosures of $\log(\lhi/\lambda_c)$ and $\log(\lambda_c/\llo)$. A lower bound of a function $\mathsf F$ of $\Ypar$ on a parent
interval at $\lambda_c$ is obtained as follows: $\mathsf F_0\le\mathsf F(Y_0,\lambda_c)$, $\mathsf F_1\le\mathsf F(Y_1,\lambda_c)$,
$\mathsf F_0'\le\partial_{\Ypar}\mathsf F(Y_0,\lambda_c)$ and $\mathsf F_1'\ge\partial_{\Ypar}\mathsf F(Y_1,\lambda_c)$ are ends of
enclosures, $\underline{\mathsf F}_2\le\min(0,\inf\partial_{\Ypar}^2\mathsf F(\cdot,\lambda_c))$ comes from an interval evaluation of the
formula of Lemma~\ref{tgt:lem:dY} (with $\mathsf w$ and $\ppar(1-\ppar)$ enclosed as below), and by Taylor's theorem from
each end $\mathsf F(\Ypar,\lambda_c)\ge\max(\mathsf P_0(\Ypar),\mathsf P_1(\Ypar))$ with
$\mathsf P_i(\Ypar)=\mathsf F_i+\mathsf F_i'(\Ypar-Y_i)+\frac12\underline{\mathsf F}_2(\Ypar-Y_i)^2$; the minimum
$\underline{\mathsf F}$ of $\max(\mathsf P_0,\mathsf P_1)$ on $[Y_0,Y_1]$ is attained at $Y_0$, at $Y_1$ or at the crossing point,
since $\mathsf P_0-\mathsf P_1$ is affine. The bound is carried to the whole of $\Lambda$ by Lemma~\ref{tgt:lem:mvt}, with $d_\pm$
the ends of an interval evaluation of the formula of Lemma~\ref{tgt:lem:dmu} on $[Y_0,Y_1]\times\Lambda$, in which
$\mathsf e_\lambda,\ppar,\phipar,\mathsf k,\mathsf w$ are enclosed by Lemma~\ref{tgt:lem:tr}, $\ppar(1-\ppar)$ by its values at
the ends of the range of $\ppar$ (and $\frac14$ when that range contains $\frac12$), $\alphaV,\gammaV,A_D$ by their values at
$\llo,\lhi$, and $\dot\alpha_{\mathrm V},\dot\gamma_{\mathrm V},\dot A_D$ by interval arithmetic in $\lambda\in\Lambda$, or, on the
segment mid, whose sub-intervals do not contain $1$ in their interior, by their values at $\llo,\lhi$ ($q(1-q)$ is monotone on
each). Where the halves indicator in Data~\ref{mc:tab:curve} is yes, $\underline{\mathsf F}_2$ and $d_\pm$ are taken over the
two halves of each parent interval; elsewhere one evaluation on the whole parent interval is used.

\begin{theorem}[the step certificate along the curve; $4{,}321$ interval checks]\label{tgt:thm}
For every $\lambda\in[\frac13,\frac73]$ the multipliers of every segment of Data~\ref{mc:tab:curve} containing $\lambda$,
each with its cap $D$, form a step certificate at $\lambda$.
\end{theorem}
\begin{proof}
Cover $[\frac13,\frac73]$ by the $47$ activity sub-intervals of the checks (listed in the supplementary data, Part~4, each
with the activity $\lambda_p\ge\lhi$ of its polygon, its $Y_L$ and whether the feasibility cut is used; counts per segment in
Data~\ref{mc:tab:curve}; those of mid have log-width at most $0.069$). On each, every comparison is bounded below
at the centre by Taylor bounds in $\Ypar$, and the bound is carried to the whole sub-interval by the transport inequality,
as follows.

$Z>0$ by Lemma~\ref{tgt:lem:hand}(i).

AC: Lemma~\ref{tgt:lem:hand}(ii) on $(0,Y_L]$ and (iii) on $[3,\infty)$. On $[Y_L,3]$, Proposition~\ref{tgt:prop:red} reduces
AC to the functions $\mathsf F\in\{\mathsf A(\cdot;\mathsf r_i,\mathcal J_i),\mathsf A^{\gammaV}(\cdot;\mathsf r_i,\mathcal J_i)\}$ at the
vertices that it requires, with $\underline{\mathsf r}=0$ on the sub-intervals without a feasibility cut, and otherwise, on
each parent interval $[Y_0,Y_1]$, with $\underline{\mathsf r}$ a rational below the lower end of an enclosure of
$\mathsf r_{\rm f}(Y_1)$ (Lemma~\ref{tgt:lem:cut}); each such function is bounded below on each parent interval as described
above.

CA: where $L\ge0$, $\mathrm{CA}\ge\alphaV\EKL\ge0$. Where $L<0$, $\mathrm{CA}/\ppar\ge\mathsf C$ by Lemma~\ref{tgt:lem:feas}(b), whose
hypothesis $\gammaV-\bar{\mathcal J}\hat L>0$ is the last bound of Lemma~\ref{tgt:lem:hand}(i), as $\hat L=\max(0,-\bar L)$
there. On $(0,\frac1{500}]$ and $[3,\infty)$, $\bar L>0$ by Lemma~\ref{tgt:lem:hand}(iv). On each parent interval of
$[\frac1{500},3]$ either $\bar L>0$ on $[Y_0,Y_1]\times\Lambda$, by the enclosures of Lemma~\ref{tgt:lem:tr}, or $\mathsf C>0$ there
by Lemma~\ref{tgt:lem:mvt} with $\underline{\mathsf F}=\mathsf C(Y_m)-|\mathsf C'(Y_m)|d_Y+\frac12\min(0,\inf\mathsf C'')d_Y^2$ at
$\lambda_c$ ($Y_m$ the midpoint, $d_Y$ the half-width), a lower bound of $\mathsf C(\cdot,\lambda_c)$ by Taylor's theorem at
$Y_m$.

CC: $\mathrm{CC}=\alphaV\EKL\ge0$ by Lemma~\ref{apB:lem:entropy}.

The base inequality \eqref{mc:eq:base}: its slack $\mathrm{sl}=Dq^2-\alphaV\bar\varphi-\gammaV q^2/\bar\varphi$ has
\[\partial_\mu\mathrm{sl}=2Dq^2(1-q)-\dot\alpha_{\mathrm V}\bar\varphi-\alphaV q-\dot\gamma_{\mathrm V}q^2/\bar\varphi
-2\gammaV q^2(1-q)/\bar\varphi+\gammaV q^3/\bar\varphi^2\]
($\partial_\mu q=q(1-q)$, $\partial_\mu\bar\varphi=q$), and Lemma~\ref{tgt:lem:mvt} applies on activity intervals with centre the
midpoint, $\mathrm{sl}$ enclosed at the centre and $\partial_\mu\mathrm{sl}$ by interval arithmetic in $\lambda$: the segments are cut
at $\frac{1491}{2000},\frac{387}{400},\frac{461}{400},\frac{633}{500},\frac{1309}{1000}$ (mid, with the two halves of each
piece), $\frac{2687}{6000}$ (lo) and $\frac{58}{25}$ (u4).

These cover every $\Ypar>0$ and every $\tilde Y\ge0$, so no bound for the tails is needed. Every check is positive; the
least lower bound of a comparison after transport is in Data~\ref{mc:tab:curve}, that of the base inequality is
$4\cdot10^{-5}$, and on the segments other than mid the least slack of a polygon line is $2.0\cdot10^{-7}$
($\lambda_p=11/20$).
\end{proof}

\begin{proof}[Proof of Theorem~\ref{p2:thm:O1}]
At every $\lambda\in[\frac13,\frac73]$ the multipliers of a segment containing $\lambda$ with cap $\hat D(\lambda)$ form a step
certificate (Theorem~\ref{tgt:thm}), and an independent set of maximum weight has (LP), by (L1). Theorem~\ref{apB:thm:band}
gives the claim.
\end{proof}

% ---------------------------------------------------------------------------------------------------------------
\subsection{The variance-ratio bound}\label{apB:sec:O2}

This subsection proves the variance-ratio bound, which bounds $\mathbb E\delta^2$ in the input-verification part of the proof of
Theorem~\ref{p2:thm:main}: the variance of $|S|$ is written as a sum of one term per vertex, and the terms are paid by the
weight of an independent set. The innovation identity is that of Zhang and Li~\cite[Section~7.1, (83) to (86)]{ZL} for the
source identically equal to $1$, which they derive from the tree recursion of Scott and Sokal~\cite{SS}, and the dynamic
programme that turns the field $s'$ (defined below) into an independent set is a version of the realization in their
competitor construction~\cite[Section~7.2, (95)]{ZL}; the local payment is verified on cells by computer (Lemma~\ref{apB:lem:O2box}).

\begin{theorem}[the variance-ratio bound]\label{p2:thm:O2}
For every forest, every $\lambda\in[1/3,7/3]$ with $\iota=\iota(\lambda)$ and every independent set $A$ of maximum
weight at $\lambda$, with the cap $\theta(\iota)$ of Data~\ref{p2:tab:theta},
\[
\operatorname{Var}_\lambda K\le(1+\theta(\iota))(1-q)\,\pi(A);
\]
at $\mu_F(\lambda)=k$ this is $\mathbb{E}\delta^2\le\theta(\iota)\,q(1-q)\,\muM$ (Proposition~\ref{p2:prop:mixture}).
\end{theorem}

\begin{theorem}[the variance-ratio bound on a parameter interval]\label{apB:thm:O2row}
For each of the $17$ parameter intervals $[\lambda_{\mathrm{lo}},\lambda_{\mathrm{hi}}]$ of Data~\ref{apB:tab:O2rows}, with its
price $\gamma_\theta=1+\theta$ and its rational data sets (paragraph ``The endpoint data''), every finite forest, every
$\lambda\in[\lambda_{\mathrm{lo}},\lambda_{\mathrm{hi}}]$ and every independent set $A$ of maximum weight at $\lambda$:
$\operatorname{Var}_\lambda K\le\bar\gamma\,\pi(A)$ with the price $\bar\gamma=\gamma_\theta/(1+\lambda)=\gamma_\theta(1-q)$.
\end{theorem}

\datablock{Parameter intervals for the variance-ratio bound}{apB:tab:O2rows}
The following ordered triples give the activity interval, the price $\gamma_\theta$ and the number of
cells in the verified covering; $\theta=\gamma_\theta-1$. The first $14$ intervals cover $[1/3,7/3]$,
and the three additional intervals are marked $^\ast$.
The covering begins with the records $([1/3,\allowbreak 3/5],\allowbreak  2,\allowbreak  4{,}221)$, $([3/5,\allowbreak 13/20],\allowbreak  17/10,\allowbreak  1{,}177)$, $([13/20,\allowbreak 3/4],\allowbreak  17/10,\allowbreak  1{,}883)$, $([3/4,\allowbreak 4/5],\allowbreak  17/10,\allowbreak  717)$. It continues with $([4/5,\allowbreak 9/10],\allowbreak  8/5,\allowbreak  5{,}311)$, $([9/10,\allowbreak 1],\allowbreak  8/5,\allowbreak  3{,}128)$, $([1,\allowbreak 23/20],\allowbreak  8/5,\allowbreak  3{,}232)$, $([23/20,\allowbreak 13/10],\allowbreak  8/5,\allowbreak  3{,}715)$.

The remaining covering records are $([13/10,\allowbreak 8/5],\allowbreak  17/10,\allowbreak  7{,}000)$, $([8/5,\allowbreak 17/10],\allowbreak  2,\allowbreak  743)$, $([17/10,\allowbreak 19/10],\allowbreak  2,\allowbreak  1{,}351)$, $([19/10,\allowbreak 21/10],\allowbreak  2,\allowbreak  1{,}581)$, $([21/10,\allowbreak 9/4],\allowbreak  2,\allowbreak  1{,}729)$, and $([9/4,\allowbreak 7/3],\allowbreak  2,\allowbreak  1{,}611)$. The additional records are $([39/50,\allowbreak 4/5]^\ast,\allowbreak  33/20,\allowbreak  317)$, $([8/5,\allowbreak 33/20]^\ast,\allowbreak  19/10,\allowbreak  460)$, and $([33/20,\allowbreak 7/4]^\ast,\allowbreak  39/20,\allowbreak  777)$.
The covering starts at $p/q=2^{-16}$. It has $37{,}399$ cells for the $14$ covering intervals and
$1{,}554$ for the three additional ones, $38{,}953$ in all. Every interval has $12$ root boxes,
with one compiled evaluation for each, giving $17\times12=204$ evaluations
(Lemma~\ref{apB:lem:O2box}).

\begin{proof}[Proof of Theorem~\ref{p2:thm:O2} from Theorem~\ref{apB:thm:O2row}]
The $14$ unstarred intervals of Data~\ref{apB:tab:O2rows} tile the five regions $[1/3,3/5]$, $[3/5,4/5]$, $[4/5,13/10]$,
$[13/10,8/5]$, $[8/5,7/3]$, with one price $\gamma_\theta=1+\theta$ per region ($\theta$ is the cap of the region). By
Lemma~\ref{cs:lem:tables}(a) (a kernel evaluation over the $55$ subintervals), either $\theta(\iota)$ is the coarse cap of the band
containing $\Ihat_\iota$ and the closure of the band lies in a region with this cap, or $\Ihat_\iota$ lies
in the starred interval $[39/50,4/5]$, $[8/5,33/20]$ or $[33/20,7/4]$ and $\theta(\iota)$ is its $\gamma_\theta-1$ ($13/20$, $9/10$
or $19/20$). Either way $\lambda\in\Ihat_\iota$ lies in a parameter interval with $\gamma_\theta=1+\theta(\iota)$; apply
Theorem~\ref{apB:thm:O2row}.
\end{proof}

\paragraph{Local data and the innovation identity.}\label{apB:sec:O2rec}
Root each component $X$ at an arbitrary vertex $v_0$, and let $X_v$ be the subtree of a vertex $v$ ($X_{v_0}=X$). The
\emph{local data} of $v$ are $(p_v,\allowbreak \varrho_v,\allowbreak p_{\mathrm{par}}(v),\allowbreak \zresp_v)$: the message $p_v$ of $(X_v,\allowbreak v)$; the probability
$\varrho_v$ that the parent of $v$ is vacant ($\varrho_{v_0}=1$); the parent's message $p_{\mathrm{par}}(v)$ (for $v\ne v_0$); and
the signed response of the total count
\[
\zresp_v=\mathbb{E}\bigl[\,|S_{X_v}|\;\big|\;v\in S_{X_v}\bigr]-\mathbb{E}\bigl[\,|S_{X_v}|\;\big|\;v\notin S_{X_v}\bigr]=1+\mu_{V(F)}(X_v-\NF[v])-\mu_{V(F)}(X_v-v)
\]
(unrelated to the prices $z_r$ of Section~\ref{p2:sec:criterion}). As for the variance bound, $\varphi_v=-\log(1-p_v)$ and
$Y_v=\sum_{v'}\varphi_{v'}$ over the children $v'$ of $v$.

\begin{lemma}[local data]\label{apB:lem:records}
(a) $\pi_v=\varrho_vp_v$, and $\varrho_{v'}=1-\varrho_vp_v$ for every child $v'$ of $v$.
(b) $\zresp_v=1-\sum_{v'}p_{v'}\zresp_{v'}$; $\zresp_v=1$ if $v$ has no child.
(c) $Y_v=\log\bigl(\lambda(1-p_v)/p_v\bigr)$.
\end{lemma}
\begin{proof}
(a) Given everything outside $X_v$, the law of $S\cap X_v$ is $P_{X_v}$ if the parent of $v$ is vacant and
$P_{X_v-v}$ if it is occupied (the parent is the only neighbour of $X_v$ outside it). So $\pi_v=\varrho_vp_v$, and a child
$v'$ has $\varrho_{v'}=1-\pi_v$. (b) By (M1) and \eqref{apB:eq:M2} with $Q=V(F)$, where $1_Q(v)=1$, $\mu_{V(F)}(X_{v'})-\mu_{V(F)}(X_{v'}-v')=p_{v'}\zresp_{v'}$
and $\mu_{V(F)}(X_v-\NF[v])-\mu_{V(F)}(X_v-v)=-\sum_{v'}(\mu_{V(F)}(X_{v'})-\mu_{V(F)}(X_{v'}-v'))$. (c) Lemma~\ref{apB:lem:msg}.
\end{proof}

\begin{lemma}[the innovation identity]\label{apB:lem:innov}
On a component $X$ rooted at $v_0$, $V_{V(F)}(X)=\sum_{v\in X}\varrho_vp_v(1-p_v)\zresp_v^2=\sum_{v\in X}\pi_v(1-p_v)\zresp_v^2$.
\end{lemma}
\begin{proof}
For a subtree $(X_v,\allowbreak v)$ whose root receives an arbitrary $\varrho'$ in place of $\varrho_v$ (its children then receive
$1-\varrho' p_v$, and so on down), $\sum_{x\in X_v}\varrho_xp_x(1-p_x)\zresp_x^2=\varrho' V_{V(F)}(X_v)+(1-\varrho')V_{V(F)}(X_v-v)$, by
induction on $|X_v|$: \eqref{apB:eq:M2} gives $V_{V(F)}(X_v)=(1-p_v)V_{V(F)}(X_v-v)+p_vV_{V(F)}(X_v-\NF[v])+p_v(1-p_v)\zresp_v^2$, and (M1) gives
$V_{V(F)}(X_v-v)=\sum_{v'}V_{V(F)}(X_{v'})$, $V_{V(F)}(X_v-\NF[v])=\sum_{v'}V_{V(F)}(X_{v'}-v')$, so
$\varrho' V_{V(F)}(X_v)+(1-\varrho')V_{V(F)}(X_v-v)=\varrho' p_v(1-p_v)\zresp_v^2+\sum_{v'}\bigl[(1-\varrho' p_v)V_{V(F)}(X_{v'})+\varrho' p_vV_{V(F)}(X_{v'}-v')\bigr]$, and each bracket is
the sum over $X_{v'}$ by induction. At the root $\varrho'=1$.
\end{proof}

\begin{definition}[admissible local data]\label{apB:def:adm}
Let $\lambda>0$ and $q=\lambda/(1+\lambda)$. A triple $(p,\allowbreak \varrho,\allowbreak p_{\mathrm{par}})$ of reals is \emph{admissible} (at $\lambda$) if
\[
0<p\le q,\qquad \varrho\le1,\qquad 1-\varrho\le p_{\mathrm{par}}\le\frac{\lambda(1-p)}{1+\lambda(1-p)},\qquad
p_{\mathrm{par}}\le\frac{(1+\lambda)(1-\varrho)}{1+\lambda(1-\varrho)} .
\]
\end{definition}

\begin{lemma}[admissible local data]\label{apB:lem:phys}
For every nonroot vertex $v$ the triple $(p_v,\allowbreak \varrho_v,\allowbreak p_{\mathrm{par}}(v))$ is admissible, and
$\varrho_v(1+\lambda(1-p_v))\ge1$.
\end{lemma}
\begin{proof}
Induction from the root, carrying $\varrho_v\le1$ and $\varrho_v(1+\lambda(1-p_v))\ge1$ (at the root $\varrho_{v_0}=1$). For a child $v'$
of $v$: $0<p_{v'}\le q$ (Lemma~\ref{apB:lem:msg}); $\varrho_{v'}=1-\varrho_vp_v\le1$ as $\varrho_v\ge0$; $1-\varrho_{v'}=\varrho_vp_v\le p_v$;
$p_v/(1-p_v)=\lambda\prod_{w}(1-p_{w})\le\lambda(1-p_{v'})$, the product over the children $w$ of $v$, gives the upper
bound; and the last bound is
$1+\lambda \varrho_vp_v\le(1+\lambda)\varrho_v$, i.e.\ $\varrho_v(1+\lambda(1-p_v))\ge1$. The inequality $\varrho_{v'}(1+\lambda(1-p_{v'}))\ge1$
follows from $1-\varrho_{v'}\le p_{\mathrm{par}}(v')\le\lambda(1-p_{v'})/(1+\lambda(1-p_{v'}))$.
\end{proof}

\paragraph{The source and the rounding.}\label{apB:sec:O2src}
Put $f_q(x)=\max\big((1+q)x/(2q),\,(1+q)x/(2q^2)+(q-1)/(2q)\big)$ (the function $f_b$ of Zhang and
Li~\cite[Section~7.2]{ZL} with $q_b=q$), $H_c(p)=\min\big(q,\,1-p/(\lambda(1-p))\big)$,
$M_{\mathrm{up}}(p,\varrho)=\max\big(1-\varrho,\,(1-p\varrho)H_c(p)\big)$, the \emph{source}
$g(p,\varrho)=\max\big(f_q(p\varrho),\,(1+f_q(p\varrho)-f_q(M_{\mathrm{up}}(p,\varrho)))/2\big)$ and, with
$\bar s_v=\max(0,\max_{w\sim v}f_q(\pi_w))$ (the maximum over the neighbours $w$ of $v$), the \emph{field}
$s'_v=(1+f_q(\pi_v)-\bar s_v)/2$. On $[0,q]$, $f_q$ is nondecreasing with $0\le f_q\le1$, and
\begin{equation}\label{apB:eq:fqedge}
x+qx'\le q\ \text{ and }\ x'+qx\le q\quad\Longrightarrow\quad f_q(x)+f_q(x')\le1:
\end{equation}
of the four sums of the affine branches of $f_q$, $(1+q)(x+x')/(2q)$ and $(1+q)(x+x')/(2q^2)+(q-1)/q$ are $\le1$ by the sum of the
two hypotheses, and the two mixed sums are $\le1$ if and only if $qx+x'\le q$, respectively $x+qx'\le q$.

\begin{lemma}[the field is a fractional independent set]\label{apB:lem:field}
For every edge $xx'$ of $F$, $\pi_x+q\pi_{x'}\le q$; and $\pi_x\le q$ for every vertex. Hence
$0\le f_q(\pi_v)\le s'_v\le1$ at every vertex, and $s'_x+s'_{x'}\le1$ on every edge.
\end{lemma}
\begin{proof}
$I(F;\lambda)=I(F-x-x';\lambda)+\lambda I(F-\NF[x];\lambda)+\lambda I(F-\NF[x'];\lambda)$, $\pi_x=\lambda I(F-\NF[x];\lambda)/I(F;\lambda)$,
$\pi_{x'}=\lambda I(F-\NF[x'];\lambda)/I(F;\lambda)$, and $I(F-\NF[x];\lambda)\le I(F-x-x';\lambda)$. So
$qI(F;\lambda)-\lambda I(F-\NF[x];\lambda)-q\lambda I(F-\NF[x'];\lambda)=q\,(I(F-x-x';\lambda)-I(F-\NF[x];\lambda))\ge0$, using
$(1-q)\lambda=q$. And $\pi_x/(1-\pi_x)=\lambda I(F-\NF[x];\lambda)/I(F-x;\lambda)\le\lambda$.
By these bounds and \eqref{apB:eq:fqedge}, $0\le f_q(\pi_v)\le1$ and $f_q(\pi_x)+f_q(\pi_{x'})\le1$ on
edges, so $f_q(\pi_v)+\bar s_v\le1$, which is $s'_v\ge f_q(\pi_v)\ge0$; $s'_v\le1$ as $f_q(\pi_v)\le1\le1+\bar s_v$. On an
edge, $\bar s_x\ge f_q(\pi_{x'})$ and $\bar s_{x'}\ge f_q(\pi_x)$, so
$s'_x+s'_{x'}=1+\tfrac12(f_q(\pi_x)+f_q(\pi_{x'})-\bar s_x-\bar s_{x'})\le1$.
\end{proof}

\begin{lemma}[the source is paid by the field]\label{apB:lem:source}
At every vertex, $\varrho_vp_v\,g(p_v,\varrho_v)\le\pi_v\,s'_v$.
\end{lemma}
\begin{proof}
$\pi_v=\varrho_vp_v$ (Lemma~\ref{apB:lem:records}). Every neighbour of $v$ has marginal at most $M_{\mathrm{up}}(p_v,\varrho_v)$:
the parent has $1-\varrho_v$, and a child $v'$ has $(1-\varrho_vp_v)p_{v'}$ with $p_{v'}\le H_c(p_v)$, because $p_{v'}\le q$ and
$p_v/(\lambda(1-p_v))=\prod_{w}(1-p_{w})\le1-p_{v'}$, the product over the children $w$ of $v$. By monotonicity $\bar s_v\le f_q(M_{\mathrm{up}})$, so
$\tfrac12(1+f_q(\pi_v)-f_q(M_{\mathrm{up}}))\le s'_v$; and $f_q(\pi_v)\le s'_v$ (Lemma~\ref{apB:lem:field}). So
$g\le s'_v$; multiply by $\varrho_vp_v\ge0$.
\end{proof}

\begin{lemma}[the rounding]\label{apB:lem:round}
Let $\hat s$ be a field on the vertices of a forest with $0\le\hat s_v\le1$ at every vertex and
$\hat s_x+\hat s_{x'}\le1$ on every edge, and let $\hat\pi$ be any real weights. Then some independent set $A'$ satisfies
$\sum_v\hat\pi_v\hat s_v\le\sum_{v\in A'}\hat\pi_v$.
\end{lemma}
\begin{proof}
For each rooted subtree $(X,\allowbreak v)$ we find, by induction on $|X|$, independent sets $A'_1\subseteq X$ and
$A'_0\subseteq X-v$ with $\sum_{x\in X}\hat\pi_x\hat s_x\le\hat s_v\hat\pi(A'_1)+(1-\hat s_v)\hat\pi(A'_0)$, where $\hat\pi(\cdot)$ is
the total weight of a set. For each child $v'$ take $(A'^{v'}_1,\allowbreak A'^{v'}_0)$ from $(X_{v'},\allowbreak v')$ and let $A'^{v'}_{\max}$ be the
heavier of the two. Put $A'_1=\{v\}\cup\bigcup_{v'}A'^{v'}_0$ and $A'_0=\bigcup_{v'}A'^{v'}_{\max}$; both are independent
(distinct branches are nonadjacent and $A'^{v'}_0$ avoids $v'$). Using $\hat s_{v'}\ge0$,
$\hat\pi(A'^{v'}_{\max})\ge\hat\pi(A'^{v'}_1),\hat\pi(A'^{v'}_0)$ and $\hat s_{v'}\le1-\hat s_v$ (the edge $vv'$),
\[
\hat s_{v'}\hat\pi(A'^{v'}_1)+(1-\hat s_{v'})\hat\pi(A'^{v'}_0)\le\hat\pi(A'^{v'}_0)+\hat s_{v'}\bigl(\hat\pi(A'^{v'}_{\max})-\hat\pi(A'^{v'}_0)\bigr)\le\hat s_v\hat\pi(A'^{v'}_0)+(1-\hat s_v)\hat\pi(A'^{v'}_{\max}),
\]
and summing over $v'$, with $\hat\pi_v\hat s_v$, gives the claim. At the root of a component the heavier of $A'_1,A'_0$
works since $0\le\hat s_v\le1$; components are joined by union.
\end{proof}

\paragraph{The endpoint data.}\label{apB:sec:O2pay} A parameter interval carries two rational data sets, at $\lambda=\lambda_{\mathrm{lo}}$ and
$\lambda=\lambda_{\mathrm{hi}}$, of four fields at the nine knots $\xpq=j/8$ ($j=0,\dots,8$) and two constants. At the
actual $\lambda$ each knot value and constant is interpolated linearly in $\lambda$ between the data sets, and a field at
a cavity probability $p$ is the piecewise-linear interpolation of its knot values at $\xpq=p/q$. We write
$\check a,\check b,\check c,\check d$ for the four fields (the multipliers of the negative-response,
positive-response, child-variance and signed telescoping transfers) and $\check m,\check l$ for the two constants
(the packed child-mass and root-log multipliers); $\check a,\check b,\check c,\check m,\check l\ge0$ on the interval
(Lemma~\ref{cs:lem:O2side}(a)). For $0<p\le q$ let $Y(p)=\log(\lambda(1-p)/p)\ge0$ (so $Y(p_v)=Y_v$),
$\hat k=\lfloor Y(p)/\log(1+\lambda)\rfloor$, $\mathcal O(p)=\hat k\lambda+e^{Y(p)-\hat k\log(1+\lambda)}-1$,
$\mathcal M(p)=\hat kq+1-e^{-(Y(p)-\hat k\log(1+\lambda))}$, $\bar h(p)=p/(-\log(1-p))\in(0,1]$ and
$\mathcal G(p)=\log(p/(\lambda(1-p)^2))=-\log(1-p)-Y(p)$. At $p=q$, $Y=\mathcal O=\mathcal M=0$.

\begin{lemma}[packing]\label{apB:lem:pack}
For every vertex $v$: $\sum_{v'}p_{v'}/(1-p_{v'})\le\mathcal O(p_v)$ and $\sum_{v'}p_{v'}\ge\mathcal M(p_v)$, over the children $v'$ of $v$.
\end{lemma}
\begin{proof}
The masses $\varphi_{v'}$ lie in $(0,\bar \varphi]$, $\bar \varphi=\log(1+\lambda)$, and sum to $Y_v$. For real $t$ and masses in
$[0,\bar \varphi]$ with total $Y$, $\sum(e^{t \varphi_{v'}}-1)\le \hat k(e^{t\bar \varphi}-1)+(e^{t(Y-\hat k\bar \varphi)}-1)$ with
$\hat k=\lfloor Y/\bar \varphi\rfloor$: add one mass $\varphi$ at a time to the packed arrangement, whose residual $\hat \varphi=Y-\hat k\bar \varphi$ lies
in $[0,\bar \varphi)$; if $\hat \varphi+\varphi<\bar \varphi$ this is $e^{t \varphi}+e^{t\hat \varphi}\le1+e^{t(\hat \varphi+\varphi)}$, otherwise
$e^{t \varphi}+e^{t\hat \varphi}\le e^{t\bar \varphi}+e^{t(\hat \varphi+\varphi-\bar \varphi)}$, both instances of the two-point exchange
for the convex function $e^{t\,\cdot}$. Take $t=1$ ($e^{\varphi_{v'}}-1=p_{v'}/(1-p_{v'})$, $e^{\bar \varphi}-1=\lambda$) and
$t=-1$ ($e^{-\varphi_{v'}}-1=-p_{v'}$, $e^{-\bar \varphi}-1=-q$).
\end{proof}

\paragraph{The payment.} With $t_+=\max(t,0)$ and $t_-=\max(-t,0)$ for real $t$, and the convention $x/0=0$,
the local payment of nonroot local data $(p,\allowbreak \varrho,\allowbreak p_{\mathrm{par}},\allowbreak \zresp)$ is
\begin{gather*}
\Phi=\bar\gamma\,\varrho\,g(p,\varrho)-\varrho(1-p)\zresp^2+\mathcal R_{\check a}+\mathcal R_{\check b}+\mathcal R_{\check c}+\mathcal R_{\check d}+\check m\,\mathcal R_{\check m}-\check l\,\mathcal G(p)/p,\\
\mathcal R_{\check a}=\frac{\check a(p)(\zresp-1)_+^2}{p\,Y(p)}-\check a(p_{\mathrm{par}})\bar h(p)\zresp_-^2,\qquad
\mathcal R_{\check b}=\frac{\check b(p)(1-\zresp)_+^2}{p\,Y(p)}-\check b(p_{\mathrm{par}})\bar h(p)\zresp_+^2,\\
\mathcal R_{\check c}=\frac{\varrho\,\check c(p)(1-\zresp)^2}{\mathcal O(p)}-(1-\varrho)\check c(p_{\mathrm{par}})(1-p)\zresp^2,\\
\mathcal R_{\check d}=\varrho\,\check d(p)(\zresp-1)+(1-\varrho)\check d(p_{\mathrm{par}})\zresp,\qquad \mathcal R_{\check m}=\zresp-1+\varrho\,\mathcal M(p),
\end{gather*}
and the root payment $\Phi_{\mathrm{root}}(p,\zresp)$ is $\Phi$ at $\varrho=1$ without the terms at $p_{\mathrm{par}}$; $\Phi_v$ is $\Phi$ at
the local data of a nonroot vertex $v$. As $\check a,\check b,\bar h\ge0$,
$\Phi_{\mathrm{root}}(p,\zresp)=\Phi(p,1,0,\zresp)+(\check a(0)\zresp_-^2+\check b(0)\zresp_+^2)\bar h(p)\ge\Phi(p,1,0,\zresp)$.
Expanding, with $\varphi=-\log(1-p)$, gives the \emph{own} and \emph{incoming} parts:
$p\,\Phi(p,\varrho,p_{\mathrm{par}},\zresp)=\mathrm{own}(p,\zresp,\varrho)-\mathrm{inc}(p,\zresp,\varrho,p_{\mathrm{par}})$ and
$p\,\Phi_{\mathrm{root}}(p,\zresp)=\mathrm{own}(p,\zresp,1)-\check l\,\varphi$, where
\begin{align*}
\mathrm{own}&=\bar\gamma \varrho pg-\varrho p(1-p)\zresp^2+\frac{\check a(p)(\zresp-1)_+^2+\check b(p)(1-\zresp)_+^2}{Y(p)}+\frac{\varrho p\,\check c(p)(1-\zresp)^2}{\mathcal O(p)}\\
&\qquad+\varrho p\,\check d(p)(\zresp-1)+\check m\,\varrho p\,(\zresp-1+\mathcal M(p))+\check l\,Y(p),\\
\mathrm{inc}&=p\,\bar h(p)\bigl(\check a(p_{\mathrm{par}})\zresp_-^2+\check b(p_{\mathrm{par}})\zresp_+^2\bigr)+(1-\varrho)p\,\check c(p_{\mathrm{par}})(1-p)\zresp^2\\
&\qquad-(1-\varrho)p\,\check d(p_{\mathrm{par}})\zresp+\check m(1-\varrho)p(1-\zresp)+\check l\,\varphi  .
\end{align*}

\begin{lemma}[the per-vertex identity]\label{apB:lem:vertex}
At a vertex $v$ whose children $v'$ carry $(p_{v'},\allowbreak \varrho_{v'}=1-\varrho_vp_v,\allowbreak p_{\mathrm{par}}(v')=p_v,\allowbreak \zresp_{v'})$, with $g_v=g(p_v,\varrho_v)$,
\[
\mathrm{own}_v-\sum_{v'}\mathrm{inc}_{v'}=\bar\gamma \varrho_vp_vg_v-\varrho_vp_v(1-p_v)\zresp_v^2-\check a(p_v)\epsilon_{\check a}(v)-\check b(p_v)\epsilon_{\check b}(v)
-\varrho_vp_v\check c(p_v)\epsilon_{\check c}(v)-\check m\,\varrho_vp_v\epsilon_{\check m}(v),
\]
where the four \emph{residuals} are
\begin{gather*}
\epsilon_{\check a}(v)=\sum_{v'}p_{v'}\bar h(p_{v'})(\zresp_{v'})_-^2-\frac{(\zresp_v-1)_+^2}{Y_v},\qquad
\epsilon_{\check b}(v)=\sum_{v'}p_{v'}\bar h(p_{v'})(\zresp_{v'})_+^2-\frac{(1-\zresp_v)_+^2}{Y_v},\\
\epsilon_{\check c}(v)=\sum_{v'}p_{v'}(1-p_{v'})\zresp_{v'}^2-\frac{(1-\zresp_v)^2}{\mathcal O(p_v)},\qquad
\epsilon_{\check m}(v)=\sum_{v'}p_{v'}-\mathcal M(p_v).
\end{gather*}
\end{lemma}
\begin{proof}
Substitute $\varrho_{v'}=1-\varrho_vp_v$ and $p_{\mathrm{par}}(v')=p_v$ in $\sum_{v'}\mathrm{inc}_{v'}$. The $\check d$ terms cancel exactly,
$\varrho_vp_v\check d(p_v)(\zresp_v-1)+\varrho_vp_v\check d(p_v)\sum_{v'}p_{v'}\zresp_{v'}=0$, by $\zresp_v=1-\sum_{v'}p_{v'}\zresp_{v'}$; the logarithmic terms cancel,
$\check lY_v-\check l\sum_{v'}\varphi_{v'}=0$; the $\check m$ terms give
$\check m\,\varrho_vp_v(\zresp_v-1+\mathcal M(p_v)-\sum_{v'}p_{v'}(1-\zresp_{v'}))=-\check m\,\varrho_vp_v\epsilon_{\check m}(v)$; and the $\check a$,
$\check b$, $\check c$ terms group into
$-\check a(p_v)\epsilon_{\check a}(v)-\check b(p_v)\epsilon_{\check b}(v)-\varrho_vp_v\check c(p_v)\epsilon_{\check c}(v)$. What remains is
$\bar\gamma \varrho_vp_vg_v-\varrho_vp_v(1-p_v)\zresp_v^2$.
\end{proof}

\begin{lemma}[the residuals are nonnegative]\label{apB:lem:resid}
$\epsilon_{\check a}(v),\epsilon_{\check b}(v),\epsilon_{\check c}(v),\epsilon_{\check m}(v)\ge0$ at every vertex.
\end{lemma}
\begin{proof}
From $\zresp_v-1=-\sum_{v'}p_{v'}\zresp_{v'}$: $(\zresp_v-1)_+\le\sum_{v'}p_{v'}(\zresp_{v'})_-$ and $(1-\zresp_v)_+\le\sum_{v'}p_{v'}(\zresp_{v'})_+$. Cauchy-Schwarz with the
weights $\varphi_{v'}>0$, $\sum_{v'}\varphi_{v'}=Y_v$, gives
$(\sum_{v'}p_{v'}(\zresp_{v'})_\mp)^2/Y_v\le\sum_{v'}p_{v'}^2(\zresp_{v'})_\mp^2/\varphi_{v'}=\sum_{v'}p_{v'}\bar h(p_{v'})(\zresp_{v'})_\mp^2$. For $\epsilon_{\check c}$: $(1-\zresp_v)^2=(\sum_{v'}p_{v'}\zresp_{v'})^2\le(\sum_{v'}p_{v'}/(1-p_{v'}))(\sum_{v'}p_{v'}(1-p_{v'})\zresp_{v'}^2)$
and Lemma~\ref{apB:lem:pack}. For $\epsilon_{\check m}$: Lemma~\ref{apB:lem:pack}. (At a childless vertex $\zresp_v=1$ and
all four vanish.)
\end{proof}

\begin{theorem}[the component identity]\label{apB:thm:ledger}
On a component $X$ rooted at $v_0$,
\begin{multline*}
\bar\gamma\sum_{v\in X}\varrho_vp_vg_v-V_{V(F)}(X)=p_{v_0}\Phi_{\mathrm{root}}(p_{v_0},\zresp_{v_0})+\sum_{v\ne v_0}p_v\Phi_v\\
+\sum_{v\in X}\bigl[\check a(p_v)\epsilon_{\check a}(v)+\check b(p_v)\epsilon_{\check b}(v)+\varrho_vp_v\check c(p_v)\epsilon_{\check c}(v)
+\check m\,\varrho_vp_v\epsilon_{\check m}(v)\bigr]+\check l\,\varphi_{v_0}.
\end{multline*}
If $\Phi_v\ge0$ at every nonroot vertex and $\Phi_{\mathrm{root}}(p_{v_0},\zresp_{v_0})\ge0$, then
$V_{V(F)}(X)\le\bar\gamma\sum_{v\in X}\pi_vg(p_v,\varrho_v)$.
\end{theorem}
\begin{proof}
Sum $p\Phi=\mathrm{own}-\mathrm{inc}$ over the nonroot vertices and add
$p_{v_0}\Phi_{\mathrm{root}}=\mathrm{own}_{v_0}-\check l\varphi_{v_0}$; regroup each incoming part at the parent (each
nonroot vertex has one parent) and apply Lemma~\ref{apB:lem:vertex} at every vertex and Lemma~\ref{apB:lem:innov}. The
residuals are nonnegative (Lemma~\ref{apB:lem:resid}), $\check l\ge0$, $\varphi_{v_0}>0$, and $\varrho_vp_v=\pi_v$.
\end{proof}

\begin{proof}[Proof of Theorem~\ref{apB:thm:O2row}]
By the nonnegativity of the payment (proved below), $\Phi_v\ge0$ at every nonroot vertex and $\Phi_{\mathrm{root}}\ge0$ at every root,
for every $\lambda$ of the interval. Theorem~\ref{apB:thm:ledger} on each component, (M1) and
Lemma~\ref{apB:lem:source} give $\operatorname{Var}_\lambda K\le\bar\gamma\sum_v\pi_vg(p_v,\varrho_v)\le\bar\gamma\sum_v\pi_vs'_v$
($\bar\gamma\ge0$). By Lemmas~\ref{apB:lem:field} and~\ref{apB:lem:round}, with the marginals as weights, some independent set
$A'$ has $\sum_v\pi_vs'_v\le\pi(A')\le\pi(A)$, the last step because $A$ has maximum weight; this is the only use of the
maximality of $A$.
\end{proof}

\paragraph{The nonnegativity of the payment.}\label{apB:sec:O2nonneg}
Let $\xpq_{\min}=2^{-16}$ for all $17$ intervals. On a parameter interval, for every $\lambda$ of the interval and every
admissible triple $(p,\allowbreak \varrho,\allowbreak p_{\mathrm{par}})$, $\Phi\ge0$ holds for every real $\zresp$ when $\xpq_{\min}q\le p<q$ (the
\emph{local payment condition}, by computer, Lemma~\ref{apB:lem:O2box}(a)) and when $0<p<\xpq_{\min}q$ (the \emph{small-message
condition}, analytic, Lemma~\ref{apB:lem:small}), and for $\zresp=1$ when $p=q$ (the \emph{leaf endpoint condition}, by computer
at $\xpq=1$, Lemma~\ref{apB:lem:O2box}(b)). This covers every vertex: the triple $(p_v,\allowbreak \varrho_v,\allowbreak p_{\mathrm{par}}(v))$ of a
nonroot vertex is admissible (Lemma~\ref{apB:lem:phys}); $p_v<q$ if $v$ has a child, and $p_v=q$, $\zresp_v=1$ if not
(Lemmas~\ref{apB:lem:msg} and~\ref{apB:lem:records}); and a root has the admissible triple $(p_{v_0},\allowbreak 1,\allowbreak 0)$, and
$\Phi_{\mathrm{root}}\ge\Phi(p_{v_0},1,0,\zresp_{v_0})$.

\begin{lemma}[the parent knots]\label{apB:lem:knots}
Fix $\lambda,p,\varrho,\zresp$ and $0\le t_{\mathrm{lo}}\le t_{\mathrm{hi}}\le q$. If $\Phi\ge0$ at $p_{\mathrm{par}}=t_{\mathrm{lo}}$,
at $p_{\mathrm{par}}=t_{\mathrm{hi}}$ and at every knot $qj/8$ between them, then $\Phi\ge0$ on
$[t_{\mathrm{lo}},t_{\mathrm{hi}}]$.
\end{lemma}
\begin{proof}
$\Phi$ is affine in $(\check a,\allowbreak \check b,\allowbreak \check c,\allowbreak \check d)(p_{\mathrm{par}})$, and these are affine in
$p_{\mathrm{par}}$ between consecutive knots; an affine function nonnegative at both ends of an interval is
nonnegative on it.
\end{proof}
The verifications use $t_{\mathrm{lo}}=1-\varrho$ and $t_{\mathrm{hi}}=\lambda(1-p)/(1+\lambda(1-p))$, so they cover the whole
interval of Definition~\ref{apB:def:adm} for $p_{\mathrm{par}}$; its last inequality (the retained-parent bound) is not used.

\begin{lemma}[small messages]\label{apB:lem:small}
Suppose a parameter interval satisfies the rational guards (checked by kernel evaluation for all $17$ intervals,
Lemma~\ref{cs:lem:O2side}(b)):
$1/3\le\lambda_{\mathrm{lo}}$, $\lambda_{\mathrm{hi}}\le7/3$, $\xpq_{\min}\le2^{-16}$, and in both data sets
$0\le\check a,\check b\le1$, $0\le\check c\le4$, $|\check d|\le4$, $0\le\check m\le2$, $\check l\ge1/1000$, and
$\check a,\check b\ge1/100$ at the knots $0$ and $1/8$. Then $\Phi(p,\varrho,p_{\mathrm{par}},\zresp)\ge0$ at every admissible triple
$(p,\allowbreak \varrho,\allowbreak p_{\mathrm{par}})$ with $p<\xpq_{\min}q$ and every real $\zresp$.
\end{lemma}
\begin{proof}
Here $p<2^{-16}$ and $\xpq=p/q\le1/8$, so $\check a(p),\check b(p)\ge1/100$. Next $p\,Y(p)\le1/1300$:
$Y(p)\le\log(4/p)$ and $p\log(4/p)=p\log2^{18}+p\log(2^{-16}/p)\le\frac{25}2\cdot2^{-16}+(2^{-16}-p)$,
as $\log2^{18}\le25/2$ (Section~\ref{sec:computer:numeric}). So the own $\check a,\check b$ terms are at least $13(\zresp-1)_+^2+13(1-\zresp)_+^2=13(\zresp-1)^2$. The
incoming $\check a,\check b$ charges are at most $\zresp_-^2+\zresp_+^2=\zresp^2$ ($\bar h\le1$); the own $\check c$ term is nonnegative
and the incoming one at most $4\zresp^2$; the innovation term $\varrho(1-p)\zresp^2$ is at most $\zresp^2$ ($0<\varrho\le1$); the $\check d$ terms
are at least $-4|\zresp-1|-4$ (with $|\zresp|\le|\zresp-1|+1$); the $\check m$ term is at least $-2|\zresp-1|$ ($\mathcal M\ge0$); the source
term is nonnegative; and $-\check l\mathcal G(p)/p\ge580$, as $\mathcal G(p)\le-9$ (from $e^9\le8104$) and
$1/p>65536$. Hence $\Phi\ge13(\zresp-1)^2-6\zresp^2-6|\zresp-1|+576$, which equals $7(\zresp-1)^2-12(\zresp-1)-6|\zresp-1|+570>0$.
\end{proof}

\begin{lemma}[the local payment on the parameter intervals]\label{apB:lem:O2box}
For each of the $17$ parameter intervals of Data~\ref{apB:tab:O2rows} and every $\lambda$ in the interval:
(a)~$\Phi(p,\varrho,p_{\mathrm{par}},\zresp)\ge0$ for every admissible $(p,\allowbreak \varrho,\allowbreak p_{\mathrm{par}})$ with $\xpq_{\min}q\le p<q$ and every
real $\zresp$;
(b)~the leaf endpoint condition: $\Phi(q,\varrho,p_{\mathrm{par}},1)\ge0$, that is,
$\bar\gamma \varrho\,g(q,\varrho)-\varrho(1-q)-\check b(p_{\mathrm{par}})\frac q{\log(1+\lambda)}-(1-\varrho)(1-q)\check c(p_{\mathrm{par}})+(1-\varrho)\check d(p_{\mathrm{par}})-\check l\frac{\log(1+\lambda)}q\ge0$,
whenever $(q,\allowbreak \varrho,\allowbreak p_{\mathrm{par}})$ is admissible.
Verified by computer: compiled evaluations of the checker below (Section~\ref{sec:computer:families}).
\end{lemma}

\begin{lemma}[side conditions of the variance-ratio bound]\label{cs:lem:O2side}
For each of the $17$ parameter intervals of Data~\ref{apB:tab:O2rows}: (a) the fields $\check a,\check b,\check c$ and the
constants $\check m,\check l$ of its data sets are nonnegative on the interval; (b) the rational conditions of
Lemma~\ref{apB:lem:small} hold; (c) $\check a=\check b=\check c=0$ at the knot $p/q=1$; (d) the data sets used by the
checker of Lemma~\ref{apB:lem:O2box} are the data sets of the interval. (Kernel evaluation, Section~\ref{sec:computer:arith};
rechecked by a program written from the text; least margin $3.83\cdot10^{-4}$.)
\end{lemma}

\paragraph{The checker.}\label{apB:sec:O2cert}
In the coordinates $\xpq=p/q\in[2^{-16},1]$ and $\bar\pi=(1-\varrho)/q$ (the parent's marginal divided by $q$), admissible
triples have $0\le\bar\pi\le\bar P(\lambda,\xpq)=(1-q\xpq)/(1-q^2\xpq)$, where
$q\bar P=\lambda(1-p)/(1+\lambda(1-p))$. The checker verifies at every point $(\lambda,\allowbreak \xpq)$ of the interval
$\times[2^{-16},1]$ that for every $\bar\pi\in[0,\bar P]$ and every parent case
$p_{\mathrm{par}}/q\in\{\bar\pi,\bar P\}\cup\{j/8:1\le j\le7,\ \bar\pi\le j/8\le\bar P\}$,
$\Phi(q\xpq,1-q\bar\pi,p_{\mathrm{par}},\zresp)\ge0$ for every real $\zresp$ if $\xpq<1$, and for $\zresp=1$ (the leaf endpoint
condition) if $\xpq=1$. By Lemma~\ref{apB:lem:knots} this gives parts (a) and (b) of Lemma~\ref{apB:lem:O2box}.

A cell is a product of two closed intervals in the coordinates
$(\lambda,\allowbreak \xpq)$. Each interval has twelve root boxes $[\lambda_{\mathrm{lo}},\lambda_{\mathrm{hi}}]\times[\xpq_i,\xpq_{i+1}]$
($0\le i\le11$), where $(\xpq_0,\allowbreak \dots,\allowbreak \xpq_{12})=(2^{-16},\allowbreak 2^{-12},\allowbreak 2^{-10},\allowbreak 2^{-8},\allowbreak 2^{-6},\allowbreak \tfrac18,\allowbreak \tfrac28,\allowbreak \dots,\allowbreak \tfrac78,\allowbreak 1)$, each covered by a binary subdivision
(each piece split in $\lambda$, split in $\xpq$, or a leaf; a leaf may be bisected at most four more times). For each parent case,
$[0,\bar P]$ is cut where the source $g$ or a parent field changes its affine formula; on each resulting segment the three response
quadratics of $\Phi$ (in $\hat z=-\zresp$ on $\zresp\le0$, $\hat z=1-\zresp$ on $[0,1]$, $\hat z=\zresp-1$ on $\zresp\ge1$) have coefficients of
degree at most two in the segment coordinate, and each of their degree-two Bernstein controls
$\hat q_2\hat z^2+\hat q_1\hat z+\hat q_0$ is shown to be nonnegative on the response range: on a half-line,
$\hat q_0,\hat q_2\ge0$ and ($\hat q_1\ge0$, or $4\hat q_2\hat q_0\ge\hat q_1^2$, or $\hat q_2>0$ and
$\hat q_0-\hat q_1^2/(4\hat q_2)\ge0$); on $[0,1]$, both end values $\ge0$ and ($\hat q_2\le0$, or
$\hat q_0+\hat q_1/2\ge0$, or the derivative has one sign at both ends, or the minimum or discriminant test). A segment
may be subdivided by de Casteljau up to three levels. The coefficients are bounded on the cell by the centred
(mean-value) form in $(\lambda,\allowbreak \xpq)$, with slopes propagated through every operation, in fixed-point
interval arithmetic with $64$ fractional bits and outward rounding, $\exp$ being enclosed by argument halving, the
Taylor polynomial of degree $17$ with its remainder bound and repeated squaring, and $\log t$ by
$2\operatorname{atanh}((t-1)/(t+1))$ with twenty terms and a geometric tail (Section~\ref{sec:computer:arith}). For $\xpq\ge7/8$ the own Cauchy coefficients are taken in
removable forms (they vanish at $\xpq=1$), which needs $\check a=\check b=\check c=0$ at the knot
$\xpq=1$ (Lemma~\ref{cs:lem:O2side}(c)). The checker is proved sound for arbitrary data (Section~\ref{sec:computer:arith}),
so twelve passing root boxes give parts (a) and (b) of Lemma~\ref{apB:lem:O2box} for the interval, whose data sets are those
used by the checker (Lemma~\ref{cs:lem:O2side}(d)).

\begin{remark}[scope of the variance-ratio bound]\label{apB:rem:O2not}
The bound is claimed for every independent set of maximum weight (all ties), whose maximality enters only through
$\pi(A')\le\pi(A)$, and not for the maximum independent set of Section~\ref{p1:sec:overview}.
\end{remark}

% ---------------------------------------------------------------------------------------------------------------
\subsection{The Laplace-transform bounds}\label{apC:sec:O3}

This subsection proves the Laplace-transform bounds (Theorem~\ref{p2:thm:O3}); they give the Laplace constraints of the
MGF criterion in the no-valley inequality on the parameter domain (Theorem~\ref{p2:thm:O4}).
Theorem~\ref{apC:thm:T32} proves the bound for every forest from a numerical hypothesis, the unit inequality, which
Lemma~\ref{apC:lem:O3box} verifies on cells.

\begin{theorem}[the Laplace-transform bounds]\label{p2:thm:O3}
Let $(\tau_0,\tau_1,\tau_2,\tau_3,\tau_4)=(0,1/20,1/10,1/5,3/10)$, in this order, and let $\ell_r(\iota)$ be the
verified rate of index $r$ of the subinterval $\iota$. For every forest, every $\lambda\in[1/3,7/3]$ with
$\iota=\iota(\lambda)$ and every independent set $A$ of maximum weight at $\lambda$,
\[
\mathbb{E}_\lambda(1-q+q\tau_r)^M\le e^{-\ell_r(\iota)\,\muM}\qquad(r<5).
\]
\end{theorem}
Here $\mathbb{E}_\lambda(1-q+q\tau)^M=\mathbb{E}_\lambda\tau^{|S\cap A|}$, the generating function of the count
$|S\cap A|$. The rate $\ell_r(\iota)$ is that of the band containing $\Ihat_\iota$, except $\ell_0$ on the subintervals that
carry one of $26$ additional intervals at $\tau=0$ (Lemma~\ref{cs:lem:tables}(b)); the rates and the constants $a_{\mathrm U}$
below are in the supplementary data (Section~\ref{sec:computer}).

For real $t$ and $0<\pmsg<1$ put $\mathcal L(\pmsg)=-\log(1-\pmsg)$, $Y(\pmsg)=\log(\lambda(1-\pmsg)/\pmsg)$,
$\hat Y(t)=\log\lambda-t-\log(1-e^{-t})$ and $\hat Y_+(t)=\max\big(\hat Y(t),\,0\big)$ for $t>0$,
$\omega=1-2/(1+\sqrt{1+4\lambda})$ (the root in $(0,1)$ of $\omega=\lambda(1-\omega)^2$),
$\rho=\log\big((1+\lambda)/(1+\lambda \tau)\big)/\log(1+\lambda)$, $\mathrm{msg}_\lambda(t)=\lambda e^{-t}/(1+\lambda e^{-t})$,
$\pi_{\min}(p)=p/(1+\lambda(1-p))$ and $\ell^\ast(t)=(\ell/q)\big(1-\pi_{\min}(\mathrm{msg}_\lambda(t))\big)$. The \emph{unit
inequality} $U(\lambda,\tau,a_{\mathrm U},\ell)$ states that
\begin{multline*}
-\log\big(\bar p+(1-\bar p)e^{-\rho Y_A}\big)+a_{\mathrm U}(1-\omega)Y_C+[Y_C>0]\,a_{\mathrm U}\omega\,\hat Y_+(Y_C)-a_{\mathrm U}\mathcal L(\bar p)\\
+Y_A\,\frac{a_{\mathrm U}Y(\pmsg)+[\pmsg<q]\,a_{\mathrm U}\omega\,\hat Y_+(Y(\pmsg))-a_{\mathrm U}\omega\mathcal L(\pmsg)-\ell^\ast(Y_A+Y_C)\,\pmsg}{\mathcal L(\pmsg)}\ \ge\ 0,
\end{multline*}
with $\bar p=\mathrm{msg}_\lambda(Y_A+Y_C)$, for all reals $Y_A>0$, $Y_C\ge0$ and $0<\pmsg\le\min(q,1-e^{-Y_A})$. We write
$\mathcal L_\tau(p)=-\log(1-(1-\tau)p)$.

\begin{theorem}[Laplace bound from the unit inequality]\label{apC:thm:T32}
Let $A$ be an independent set of a forest such that every vertex $v\notin A$ with a neighbour $w$ has a second
neighbour, or else $w\in A$ and $v$ is the only neighbour of $w$ (the \emph{leaf condition}). Let $0<\lambda\le6$,
$0\le\tau\le1$, $a_{\mathrm U}\ge0$, $\ell\le\log\frac{1+\lambda}{1+\lambda\tau}$ (the \emph{isolated-vertex bound}), and
suppose
$U(\lambda,\tau,a_{\mathrm U},\ell)$. Then $\mathbb{E}_\lambda(1-q+q\tau)^M\le e^{-\ell \muM}$ for the forest mixture of $A$.
\end{theorem}
Every independent set of maximum weight at $\lambda>0$ satisfies the leaf condition, which is parts (L1) and (L2) of
property (b) of Section~\ref{p2:sec:mixture}.

\emph{Idea.} The transform factors over the components; each component is bounded by a product of vertex factors
along the tree recursion of the messages, and an induction up each tree shows that the unit inequality makes the
logarithms of these factors pay for the weight of $A$.
\begin{proof}
\emph{Step 0: reduction to the hard-core model.} For a vertex set $X'$ let
$I^\tau(X')$ be the sum of $\lambda^{|\mathcal I\setminus A|}(\lambda\tau)^{|\mathcal I\cap A|}$ over the independent
$\mathcal I\subseteq X'$, and $\mathcal T(X')=I^\tau(X')/I(F[X'];\lambda)$; $\mathcal T(V(F))$ is the generating function of
$|S\cap A|$, and $\mathcal T$ is multiplicative over unions with no edges between the parts. The decomposition
$S=J\cup(S\cap A)$ of Section~\ref{p2:sec:mixture} and $\lambda^{|J|}(1+\lambda)^M(1-q+q\tau)^M=\lambda^{|J|}(1+\lambda\tau)^M$
give $\mathbb{E}_\lambda(1-q+q\tau)^M=\mathcal T(V(F))$, and $q\muM=\pi(A)=\sum_{u\in A}\pi_u$. If $\ell<0$ the claim is
trivial ($\mathcal T\le1$). Otherwise it suffices to prove, for each component $X$,
$\mathcal T(X)\le\exp\big(-\frac\ell q\sum_{u\in A\cap X}\pi_u\big)$ (marginals are local). If $X=\{v\}$ with $v\notin A$,
$\mathcal T(X)=1$ and there is no weight of $A$. If $X=\{u\}$ with $u\in A$,
$\mathcal T(X)=\frac{1+\lambda\tau}{1+\lambda}\le e^{-\ell}$ and $\pi_u=q$: this is the only use of the isolated-vertex
bound. Otherwise $X$ contains an edge, hence a vertex $v\notin A$; root $X$ at $v$. By the leaf condition every vertex of
$C$ in $X$ has a child, and no vertex of $C$ other than the root is a leaf.

\emph{Step 1: messages.} For a vertex $v$ with subtree $X_v$ let $p_v=\lambda I(F[X_v-\NF[v]];\lambda)/I(F[X_v];\lambda)$, the
probability that $v$ is occupied in the hard-core model of $X_v$ alone. Then $p_v/(1-p_v)=\lambda\prod_{v'}(1-p_{v'})$
over the children $v'$, so $\sum_{v'}\mathcal L(p_{v'})=Y(p_v)$; $0<p_v\le q$, with equality exactly when $v$ has no
child; and $p=\mathrm{msg}_\lambda(s)$ whenever $Y(p)=s$.

\emph{Step 2: the product bound.} For $v\in C$ put $\mathcal V^A_v=\prod_u(1-(1-\tau)p_u)$ over its children $u\in A$,
and $\mathcal V_v=p_v+(1-p_v)\mathcal V^A_v$. Then $\mathcal T(X)\le\prod_{v\in C\cap X}\mathcal V_v$. Indeed
$\mathcal T(X')=p_{v_0}\tau^{[v_0\in A]}\mathcal T(X'-\NF[v_0])+(1-p_{v_0})\mathcal T(X'-v_0)$ for a vertex set $X'$ rooted
at $v_0$, with $\mathcal T(X'-v_0)=\prod_{v'}\mathcal T(X_{v'})$ and $\mathcal T(X'-\NF[v_0])=\prod_{v'}\mathcal T(X_{v'}-v')$
over the children $v'$ of $v_0$; induction proves $\mathcal T(X'-v_0)\le\mathcal V^A_{v_0}\mathcal V_{<v_0}$ and
$\mathcal T(X')\le(p_{v_0}\tau^{[v_0\in A]}+(1-p_{v_0})\mathcal V^A_{v_0})\mathcal V_{<v_0}$, where $\mathcal V_{<v_0}$ is
the product of $\mathcal V_v$ over the vertices $v$ of $C$ strictly below $v_0$, using $\mathcal V^A_v\le\mathcal V_v\le1$
and $\mathcal V^A_v=1$ at vertices of $A$ (whose children are in $C$).

\emph{Step 3: the gain.} $\mathcal V^A_v=\exp(-\sum_u\mathcal L_\tau(p_u))$, and $\mathcal L_\tau(p)\ge\rho\mathcal L(p)$
for $0\le p\le q$: $x\mapsto x^\rho$ is concave for $0\le\rho\le1$ and $(1-q)^\rho=\frac{1+\lambda\tau}{1+\lambda}$, so
$(1-p)^\rho\ge1-(1-\tau)p$. Hence $\mathcal V_v\le p_v+(1-p_v)e^{-\rho Y_A}$ with $Y_A=\sum_u\mathcal L(p_u)$.

\emph{Step 4: the weight of $A$.} With $p^A_v=\sum_up_u$ over the children $u\in A$ of $v$:
$\sum_{u\in A\cap X}\pi_u=\sum_{v\in C\cap X}(1-\pi_v)p^A_v\le\sum_{v\in C\cap X}(1-\pi_{\min}(p_v))p^A_v$. The identity
conditions each $u$ on its parent; the inequality is $\pi_v\ge\pi_{\min}(p_v)$, because the parent of $v$ is unoccupied
with probability at least $1/(1+\lambda(1-p_v))$.

\emph{Step 5: children's messages.} For a nonempty family $x_i\in(0,q]$ with $Y_0=\sum_i\mathcal L(x_i)$:
$\hat Y_+(Y_0)\le\sum_iY(x_i)$. Each $Y(x_i)\ge0$; with $j'$ terms and $\prod(1-x_i)=e^{-Y_0}$, the claim
$\hat Y(Y_0)\le\sum Y(x_i)$ reads $\lambda\prod x_i\le\lambda^{j'}(1-e^{-Y_0})$, which follows by induction on $j'$ from
$x_i\le q\le\lambda$.

\emph{Step 6: the boundary $Y_A=0$.} For $\varphi>0$: $(a_{\mathrm U}-a_{\mathrm U}\omega)\varphi+a_{\mathrm U}\omega\hat Y_+(\varphi)
-a_{\mathrm U}\mathcal L(\mathrm{msg}_\lambda(\varphi))\ge0$. Take $a_{\mathrm U}=1$; let $\lambda_\ast=(1+\sqrt{1+4\lambda})/2$,
so $\lambda=\lambda_\ast(\lambda_\ast-1)$, $\omega=1-1/\lambda_\ast$, and $\lambda\le6$ gives $\lambda_\ast\le3$. Put
$\hat x=e^\varphi$. For $\hat x\le1+\lambda$ the expression is
$\mathcal W(\hat x)=(2-\omega)\log\hat x-\omega\log(\hat x-1)+\omega\log\lambda-\log(\hat x+\lambda)$; its derivative in
$\log\hat x$ has the sign of $(\hat x-\lambda_\ast)((2-\lambda_\ast)\hat x+\lambda_\ast^2-1)$, and the second factor is
positive on $[1,1+\lambda]$ for $1<\lambda_\ast\le3$ (for $\lambda_\ast>2$ it is at least
$1+\lambda_\ast(\lambda_\ast-1)(3-\lambda_\ast)$); since $\mathcal W(\lambda_\ast)=0$, $\mathcal W\ge0$. For
$\hat x>1+\lambda$, $\hat Y_+=0$ and
$(1-\omega)\varphi\ge(1-\omega)\log(1+\lambda)\ge\log\frac{1+2\lambda}{1+\lambda}\ge\mathcal L(\mathrm{msg}_\lambda(\varphi))$, the middle
inequality being $\mathcal W(1+\lambda)\ge0$. This is the only use of $\lambda\le6$.

\emph{Step 7: the unit step.} Let $v\in C$ have children in $A$ with messages $p_i\in(0,q]$ and children in $C$ with
messages $p_j\in(0,q]$, not both empty (sums over $i$ and $j$ run over these children), $Y_A=\sum_i\mathcal L(p_i)$,
$Y_C=\sum_j\mathcal L(p_j)$, so $p_v=\mathrm{msg}_\lambda(Y_A+Y_C)$, and let
$0<\mathcal V\le p_v+(1-p_v)e^{-\rho Y_A}$. Then with $\Psi(p)=a_{\mathrm U}\mathcal L(p)+a_{\mathrm U}\omega Y(p)$ and
$\varsigma=-\log\mathcal V-\ell^\ast(Y_A+Y_C)\sum_ip_i$:
\[
\Psi(p_v)\ \le\ \varsigma+\sum_i\big[a_{\mathrm U}Y(p_i)+[p_i<q]\,a_{\mathrm U}\omega\hat Y_+(Y(p_i))\big]+\sum_j\Psi(p_j).
\]
Indeed $Y(p_v)=Y_A+Y_C$; $\sum_j\Psi(p_j)\ge a_{\mathrm U}Y_C+[Y_C>0]a_{\mathrm U}\omega\hat Y_+(Y_C)$ by Step 5; and the
$A$-terms equal $\sum_i\Psi^A(p_i)+a_{\mathrm U}\omega Y_A+\ell^\ast(Y_A+Y_C)\sum_ip_i$, where
$\Psi^A(\pmsg)=a_{\mathrm U}Y(\pmsg)+[\pmsg<q]a_{\mathrm U}\omega\hat Y_+(Y(\pmsg))-a_{\mathrm U}\omega\mathcal L(\pmsg)-\ell^\ast(Y_A+Y_C)\,\pmsg$.
It remains that $\Psi_0+\sum_i\Psi^A(p_i)\ge0$, where $\Psi_0$ is the part of the left side of the unit inequality
without the term in $\pmsg$. If $v$ has no child in $A$ this is Step 6. Otherwise $Y_A>0$, and each $\pmsg=p_i$ satisfies
$0<\pmsg\le q$ and $\pmsg\le1-e^{-Y_A}$ (as $\mathcal L(p_i)\le Y_A$), so the unit inequality gives
$\Psi^A(p_i)\ge-\frac{\Psi_0}{Y_A}\mathcal L(p_i)$; summing over $i$ gives $\sum_i\Psi^A(p_i)\ge-\Psi_0$ whatever the sign
of $\Psi_0$.

\emph{Step 8: the induction.} Put $\varsigma_v=-\log\mathcal V_v-\frac\ell q(1-\pi_{\min}(p_v))p^A_v$ for $v\in C$ (so
$\frac\ell q(1-\pi_{\min}(p_v))=\ell^\ast(Y(p_v))$), $\varsigma_u=0$ for $u\in A$, and let $\bar\varsigma_v$ be the sum
of the unit surpluses $\varsigma$ over the subtree of $v$. Then $\bar\varsigma_v\ge\Psi(p_v)$ for $v\in C$ and
$\bar\varsigma_u\ge a_{\mathrm U}Y(p_u)+[p_u<q]a_{\mathrm U}\omega\hat Y_+(Y(p_u))$ for $u\in A$: the first by Steps 3
and 7 with the induction hypothesis at the children; the second from
$\sum_{v'}\Psi(p_{v'})=a_{\mathrm U}Y(p_u)+a_{\mathrm U}\omega\sum_{v'}Y(p_{v'})$ over the children $v'$ of $u$ and Step 5
if $u$ has children, and from $p_u=q$, $Y(q)=0$ if not.

\emph{Step 9: conclusion.} At the root $v$, $0\le\Psi(p_v)\le\bar\varsigma_v$, that is
$\sum_{v\in C\cap X}\log\mathcal V_v\le-\frac\ell q\sum_v(1-\pi_{\min}(p_v))p^A_v\le-\frac\ell q\sum_{u\in A\cap X}\pi_u$
by Step 4; with Step 2 this is the per-component bound.
\end{proof}

\begin{lemma}[the unit inequality on the bands and the additional intervals]\label{apC:lem:O3box}
(a) For each of the $30$ bands $[\check\lambda_i,\check\lambda_{i+1}]$ and each tilt index $r<5$, with the pair $(a_{\mathrm U},\ell)$ of the
band for that tilt ($\ell$ the band's rate of index $r$), $U(\lambda,\tau_r,a_{\mathrm U},\ell)$ holds for every
$\lambda$ in the closed band. (b) For each of the $26$ additional intervals at $\tau=0$, with its pair
$(a_{\mathrm U},\ell)$, $U(\lambda,0,a_{\mathrm U},\ell)$ holds for every $\lambda$ in the interval. (c) For each band and each
$r<5$, with its pair $(a_{\mathrm U},\ell)$: $a_{\mathrm U}\ge0$, $0\le\tau_r\le1$ and
$\ell\le\log\bigl((1+\lambda)/(1+\lambda\tau_r)\bigr)$ for every $\lambda$ of the band; for each additional interval, carried
by the subinterval $\iota$: $0\le\ell\le1$, $a_{\mathrm U}\ge0$ and $e^{\ell}\le1+\lamhat_\iota$.
Parts (a) and (b) are verified by the cell checker described below, in interval arithmetic with outward rounding: $30$
compiled evaluations for (a), each covering the five tilts of its band ($150$ certificates), and $26$ for (b); part (c) by
kernel evaluation (Section~\ref{sec:computer}), for the bands by the Taylor bound
$e^\ell\le\sum_{i\le4}\ell^i/i!+\ell^5/100$ at the lower end of the band and the monotonicity of
$\frac{1+\lambda}{1+\lambda\tau}$.
\end{lemma}

\paragraph{The checker.} If the checker accepts $(\lambda_{\mathrm{lo}},\lambda_{\mathrm{hi}},a_{\mathrm U},\ell,\tau)$,
the ends of an activity interval and its constants, then $U(\lambda,\tau,a_{\mathrm U},\ell)$ holds for every
$\lambda\in[\lambda_{\mathrm{lo}},\lambda_{\mathrm{hi}}]$ (soundness, proved for arbitrary data). It requires
$0<\lambda_{\mathrm{lo}}\le\lambda_{\mathrm{hi}}\le6$, $0\le\tau\le1$, $a_{\mathrm U},\ell\ge0$, and encloses
$\omega\in[\omega(\lambda_{\mathrm{lo}}),\omega(\lambda_{\mathrm{hi}})]$ ($\omega$ increases), $\rho\ge\rho(\lambda_{\mathrm{hi}},\tau)$
($\rho$ decreases in $\lambda$) and $\ell^\ast\le\ell/q(\lambda_{\mathrm{lo}})$. The dependence on $\pmsg$ is handled by a
table on the grid $2^{i-46}$ ($i\le32$), then steps of $2^{-13}$ up to $q(\lambda_{\mathrm{hi}})$, of lower bounds of
$(a_{\mathrm U}Y+[\pmsg<q]a_{\mathrm U}\omega\hat Y_+(Y))/\mathcal L-a_{\mathrm U}\omega$ and upper bounds of $\pmsg/\mathcal L(\pmsg)$
per cell (the $\hat Y$ credit only on cells below $q(\lambda_{\mathrm{lo}})$, which handles the discontinuity at
$\pmsg=q$), with prefix minima against $65$ values of $\ell^\ast$, interpolated linearly; below $2^{-46}$ the term is
nonnegative. The $(Y_A,Y_C)$-plane is covered by adaptive bisection (depth at most $60$) of $(0,30]\times\{0\}$,
$(0,30]\times(0,30]$ and $(0,30]\times[30,\infty)$, each cell bounded below by monotone bounds of its three parts, and the
region $Y_A\ge30$ is closed by one linear inequality in $Y_A$; the arithmetic is $64$-bit fixed point with outward rounding.

\begin{proof}[Proof of Theorem~\ref{p2:thm:O3}]
Let $\iota=\iota(\lambda)$, $i$ the band containing $\Ihat_\iota$, and $r<5$. Then $0<\lambda\le7/3\le6$, and the leaf
condition holds. The band's rate of index $r$ satisfies the hypotheses of Theorem~\ref{apC:thm:T32} with $\tau=\tau_r$, by
Lemma~\ref{apC:lem:O3box}(a) and (c). If $r=0$ and the subinterval $\iota$ carries an additional interval, $\ell_0(\iota)$ is
that interval's rate, the interval contains $\Ihat_\iota$ (Lemma~\ref{cs:lem:tables}(b)), and Theorem~\ref{apC:thm:T32} applies
with $\tau=0$, by Lemma~\ref{apC:lem:O3box}(b) and (c) ($e^{\ell}\le1+\lamhat_\iota\le1+\lambda$, as
$\lambda\ge\lamhat_\iota$); otherwise $\ell_r(\iota)$ is the band's rate. The tilt and the rate
are paired by position, and no monotonicity of the rates in $\lambda$ is used.
\end{proof}
% ---------------------------------------------------------------------------------------------------------------
\subsection{The lower bound for the mean free count}\label{apC:sec:O5}

This subsection proves the lower edge $m_{25}(\iota)$ of the parameter domain, used in the input-verification part of the proof of
Theorem~\ref{p2:thm:main}.

\begin{theorem}[the lower bound for the mean free count]\label{p2:thm:O5}
For every forest with $n\ge25$, every $\lambda\in[1/3,7/3]$ with $\mu_F(\lambda)\in W(F)$, and every independent set
$A$ of maximum weight at $\lambda$:
\[
\muM\ge m_{25}(\iota)=\frac{8}{2\,q(\lamhat_{\iota+1})},\qquad \iota=\iota(\lambda).
\]
\end{theorem}
\begin{proof}
Let $k=\mu_F(\lambda)\in W(F)$. Since $n\ge25$, $k>\lceil n/4\rceil\ge7$, so $k\ge8$. By \eqref{p2:eq:consequences} and
since $q$ is increasing, $8\le k\le2q(\lambda)\muM\le2q(\lamhat_{\iota+1})\muM$.
\end{proof}

The bounds $m_{25}(\iota)$ range from $40/7\approx5.71$ ($\iota=54$) to $136/9\approx15.11$ ($\iota=0$). This is the only
place where the order $n$ enters the moment argument.

% ---------------------------------------------------------------------------------------------------------------
\section{The parameter domain}\label{p2:sec:atlas}

This section proves the no-valley inequality on the parameter domain (Theorem~\ref{p2:thm:O4}), used in the final part of the
proof of Theorem~\ref{p2:thm:main}. The domain is a set of pairs $(q,\muM)$, and $n$ enters only through its lower edge
$m_{25}(\iota)$. Below an explicit starting mean $\mu_0(\iota)\in[19,50]$ it is covered by rational boxes, each carrying a
dual certificate that a box test in exact rational arithmetic verifies at every point of the box
(Sections~\ref{apC:sec:mgfbox} and~\ref{apC:sec:families}); from $\mu_0(\iota)$ on, a dual certificate is given in closed
form by a fiber inequality for a weighted valley kernel (Section~\ref{fib:sec}).

% ---------------------------------------------------------------------------------------------------------------
\subsection{The box test}\label{apC:sec:mgfbox}

This subsection defines the box test and proves that a passing box gives an MGF certificate at every point of the box
(Theorem~\ref{apC:thm:mgfbox}) and that a passing covering is valley-free on its rectangle (Theorem~\ref{apC:thm:mgfpiece});
the proof of Theorem~\ref{p2:thm:O4} rests on them.

A \emph{box} is a rectangle $q\in[q_l,q_h]$, $\muM\in[m_l,m_h]$ with a rational dual $(\eta,\xi,\alpha,\beta,\gamma,
w_{\mathrm L},w_{\mathrm R},z_0,\dots,z_4)$ centred at $m_0$, i.e.\ with the quadratic
$\psi_0(M)=\alpha+\beta(M-m_0)-\gamma(M-m_0)^2$, and a cutoff $\bar M$; the
parameters of a covering are $\theta$, $D$, the tilts $\tau_r$ with their rates $\ell_r$ and the activity endpoints. Put
$\varpi(M)=\max(w_{\mathrm L},w_{\mathrm R})$ for a fiber $M$. By Lemma~\ref{apC:lem:atoms}(a) below,
$\kappa_2^q(M',j)\ge-\varpi(M)$ for every $q\in[q_l,q_h]$, every $j$ and every $M'$. The box test is a Boolean
function in exact rational arithmetic:

\textup{(S)} $0<q_l<q_h<1$, $\eta>0$, $\gamma\ge0$, $w_{\mathrm L},w_{\mathrm R}\ge0$, $z\ge0$, $0<m_l\le m_h$;
\textup{(T)} fibers $M\ge\bar M$: $\psi_0(\bar M)\le-\varpi(\bar M)-\xi^2/(4\eta)$ and $\beta-2\gamma(\bar M-m_0)\le0$, using
  $\kappa_2\ge-\varpi(\bar M)$ and $\eta\delta^2+\xi\delta\ge-\xi^2/(4\eta)$;
\textup{(F)} fibers $M<\bar M$, all $q\in[q_l,q_h]$: with $s_r=1-q_h+q_h\tau_r$, $\mathrm{rhs}(M)=\psi_0(M)-\sum_rz_rs_r^M$,
  $\delta_*=-\xi/(2\eta)$ (the minimiser of $\eta\delta^2+\xi\delta$) and
  $\mathrm{lev}=\mathrm{rhs}(M)+\varpi(M)+\xi^2/(4\eta)$: (i)~$\mathrm{rhs}(M)$ is at most the
  minimum of $\eta\delta^2+\xi\delta$ on each of the half-lines $\delta\le-2-q_lM$ and $\delta\ge M+2-q_hM$, which
  contain every $\delta=j-qM$ with $j\le-2$, resp.\ $j\ge M+2$ (where $\kappa_2=0$); (ii)~either $\mathrm{lev}\le0$, or
  every atom $j$ of an explicit range $[j_{\mathrm{lo}},j_{\mathrm{hi}}]\subseteq[-1,M+1]$ passes, where the ends are
  certified: if $j_{\mathrm{lo}}>-1$ then $\Delta_-=q_lM+\delta_*-(j_{\mathrm{lo}}-1)\ge0$ and
  $\eta\Delta_-^2\ge\mathrm{lev}$, and if $j_{\mathrm{hi}}<M+1$ then $\Delta_+=(j_{\mathrm{hi}}+1)-q_hM-\delta_*\ge0$ and
  $\eta\Delta_+^2\ge\mathrm{lev}$ (the checker proposes $\lfloor q_lM+\delta_*\rfloor-r_{\mathrm{box}}$ and
  $\lceil q_hM+\delta_*\rceil+r_{\mathrm{box}}$, $r_{\mathrm{box}}=\lceil\sqrt{\mathrm{lev}/\eta}\rceil+1$, clipped to
  $[-1,M+1]$, and uses $-1$, resp.\ $M+1$, when the certificate fails). An enumerated atom passes if it skips
  ($\eta\,\mathrm{dist}(\delta_*,[j-q_hM,j-q_lM])^2\ge\mathrm{lev}$) or passes a first- or second-order bound uniform in
  $q$ around the midpoint (defined below). If $\mathrm{lev}\le0$, or $j$ lies outside the range, then
  $\eta(\delta-\delta_*)^2\ge\mathrm{lev}$ for every $q\in[q_l,q_h]$, and the atom needs no binomial value;
\textup{(Mg)} $\alpha+\min\beta(\muM-m_0)-\gamma\big(Dm_h+\max(\muM-m_0)^2\big)-\eta\theta\,v_{\max}m_h-\sum_rz_r\hat e_r>0$, the min
  and max over $\muM\in\{m_l,m_h\}$, $v_{\max}=\max_{[q_l,q_h]}q(1-q)$, $\hat e_r=1/\sum_{i\le120}(\ell_rm_l)^i/i!\ge e^{-\ell_rm_l}$.

A box \emph{passes} if it passes the four tests (S), (T), (F) and (Mg). For $0\le j\le M$ the atom bounds use
$\kappa_2=\mathcal B(\qval)\big[(w_{\mathrm L}\kappa^{\mathrm L}_j)(1-\qval)^2+(w_{\mathrm R}\kappa^{\mathrm R}_j)\qval^2\big]$ with
$\mathcal B=\binp_M(j;\qval)/(\qval(1-\qval))$, $\kappa^{\mathrm L}_j=(M+1-2j)/(M-j+1)$, $\kappa^{\mathrm R}_j=(2j+1-M)/(j+1)$; the binomial values at $q_l$, $(q_l+q_h)/2$, $q_h$ are carried exactly
along $M$ as integer vectors. The two atom bounds are Taylor bounds around the midpoint $q_0=(q_l+q_h)/2$, with
half-width $\Delta q=(q_h-q_l)/2$. Let $Q_\delta$ be the exact minimum of $\eta\delta^2+\xi\delta$ over
$\delta\in[j-q_hM,\,j-q_lM]$, $\tilde\kappa(q)=\kappa_2+\xi\delta+\eta\delta^2$ with $\delta=j-qM$, $\delta_0=j-q_0M$, and let
$K_1$, $K_2$ be the rational bounds of $|\kappa_2'|$ and $|\kappa_2''|$ on $[q_l,q_h]$ given below. The atom passes
the first-order bound if $\mathrm{rhs}\le\kappa_2(q_0)-\Delta q\,K_1+Q_\delta$, and the second-order bound if
$\mathrm{rhs}\le\tilde\kappa(q_0)-\Delta q\,|\tilde\kappa'(q_0)|-\tfrac12\Delta q^2\,(K_2+2\eta M^2)$, where
$\tilde\kappa'(q_0)=\kappa_2'(q_0)-M(\xi+2\eta\delta_0)$ is computed exactly. For $0\le j\le M$ write
$\kappa_2=\mathcal B\mathcal K$ with
$\mathcal K(\qval)=(w_{\mathrm L}\kappa^{\mathrm L}_j)(1-\qval)^2+(w_{\mathrm R}\kappa^{\mathrm R}_j)\qval^2$ and
$L_{\mathcal B}=\mathcal B'/\mathcal B=(j-1)/\qval-(M-j-1)/(1-\qval)$; then
$\kappa_2'(q_0)=\mathcal B(q_0)\big(\mathcal K'(q_0)+L_{\mathcal B}(q_0)\mathcal K(q_0)\big)$,
$K_1=\mathcal B_{\max}\big(|\mathcal K'|_{\max}+|L_{\mathcal B}|_{\max}|\mathcal K|_{\max}\big)$ and
$K_2=\mathcal B_{\max}\big(|\mathcal K''|+2|L_{\mathcal B}|_{\max}|\mathcal K'|_{\max}+(|L_{\mathcal B}|_{\max}^2+|L_{\mathcal B}'|_{\max})|\mathcal K|_{\max}\big)$,
where the maxima are rational upper bounds over $[q_l,q_h]$ computed by the checker. For $j=-1$,
$\kappa_2=-w_{\mathrm R}(1-q)^M$, $K_1=w_{\mathrm R}M(1-q_l)^{M-1}$ and $K_2=w_{\mathrm R}M(M-1)(1-q_l)^{M-2}$; for $j=M+1$,
$\kappa_2=-w_{\mathrm L}q^M$, $K_1=w_{\mathrm L}Mq_h^{M-1}$ and $K_2=w_{\mathrm L}M(M-1)q_h^{M-2}$. A \emph{covering} of an activity interval $[\lambda_{\mathrm{lo}},\lambda_{\mathrm{hi}}]$
and a range $[m_{\mathrm{lo}},m_{\mathrm{hi}}]$ \emph{passes} if its parameters are sane ($\theta,D\ge0$, $0\le\tau_r\le1$,
$\ell_r\ge0$), every box passes, and the boxes cover
$[q(\lambda_{\mathrm{lo}}),q(\lambda_{\mathrm{hi}})]\times[m_{\mathrm{lo}},m_{\mathrm{hi}}]$ by consecutive $q$-slabs, each with a
chain of boxes in $\muM$. The
binomial values $\binp_M(j;\qval)$ at $\qval\in\{q_l,q_0,q_h\}$, written $\qval=\hat p/\hat q$ in lowest terms, are carried along $M$ as
the integer vectors $\hat q^M\binp_M(j;\qval)$, by
$\hat q^{M+1}\binp_{M+1}(j;\qval)=(\hat q-\hat p)\hat q^M\binp_M(j;\qval)+\hat p\,\hat q^M\binp_M(j-1;\qval)$, together with Pascal's row and the
powers $s_r^M$. No floating point is used.

\begin{lemma}[the atom bounds]\label{apC:lem:atoms}
Let $0<q_l\le q\le q_h<1$ and $\kappa_2^{q}(M,j)=w_{\mathrm L}f_{\mathrm L}/q+w_{\mathrm R}f_{\mathrm R}/(1-q)$ at
parameter $q$, $w_{\mathrm L},w_{\mathrm R}\ge0$.

\textup{(a)} $\kappa_2^q(M,j)\ge-\max(w_{\mathrm L},w_{\mathrm R})$; $\kappa_2^q(M,j)=0$ for $j\le-2$ and $j\ge M+2$;
$\kappa_2^q(M,-1)=-w_{\mathrm R}(1-q)^M$ and $\kappa_2^q(M,M+1)=-w_{\mathrm L}q^M$.
\textup{(b)} For $0\le j\le M$, $\kappa_2^q=\mathcal B(q)\mathcal K(q)$ and $\mathcal B'=\mathcal BL_{\mathcal B}$.
\textup{(c)} $\kappa_2^q\ge\kappa_2^{q_0}-\Delta q\,K_1$ for every $q\in[q_l,q_h]$, for $-1\le j\le M+1$.
\textup{(d)} $\tilde\kappa(q)\ge\tilde\kappa(q_0)-\Delta q\,|\tilde\kappa'(q_0)|-\tfrac12\Delta q^2(K_2+2\eta M^2)$ for every $q\in[q_l,q_h]$.

Here the maxima in $K_1$, $K_2$ are the checker's: $\mathcal B_{\max}$ is the largest of $\mathcal B$ at $q_l$, $q_h$
and, when $M\ne2$, at $(j-1)/(M-2)$ if this point lies strictly inside $(q_l,q_h)$ (when $M=2$ there is no such point:
$\mathcal B'$ has the sign of $j-1$, so $\mathcal B$ is monotone and only $q_l$ and $q_h$ are used); $|\mathcal K|_{\max}$ uses the
endpoints and, when $w_{\mathrm L}\kappa^{\mathrm L}_j+w_{\mathrm R}\kappa^{\mathrm R}_j\ne0$, the vertex
$w_{\mathrm L}\kappa^{\mathrm L}_j/(w_{\mathrm L}\kappa^{\mathrm L}_j+w_{\mathrm R}\kappa^{\mathrm R}_j)$ of the quadratic $\mathcal K$ if it lies
strictly inside (when this sum is $0$, $\mathcal K=w_{\mathrm L}\kappa^{\mathrm L}_j(1-2\qval)$ is affine, has no vertex, and only
the endpoints are used); $|\mathcal K'|_{\max}$ the endpoints of the affine $\mathcal K'$; $|L_{\mathcal B}|_{\max}$ the four combinations of the
monotone terms of $L_{\mathcal B}$; $|L_{\mathcal B}'|_{\max}=|j-1|/q_l^2+|M-j-1|/(1-q_h)^2$.
\end{lemma}
\begin{proof}
(a) $f_{\mathrm L}/q=\frac{1-q}q \binp_M(j)-\binp_M(j-1)\ge-\binp_M(j-1)$ and $f_{\mathrm R}/(1-q)\ge-\binp_M(j+1)$, and
$\binp_M(j-1)+\binp_M(j+1)\le1$ (two distinct masses of one law); the rest is the support of $\binp_M$. (b) From
$\qval(M+1-j)\binp_M(j-1)=j(1-\qval)\binp_M(j)$ and $\qval(M-j)\binp_M(j)=(j+1)(1-\qval)\binp_M(j+1)$ one gets $f_{\mathrm L}=(1-\qval)\binp_M(j)\kappa^{\mathrm L}_j$ and
$f_{\mathrm R}=\qval \binp_M(j)\kappa^{\mathrm R}_j$, and $L_{\mathcal B}$ is the logarithmic derivative of $\binom Mj \qval^{j-1}(1-\qval)^{M-j-1}$. (c)
$\kappa_2^q-\kappa_2^{q_0}=(\mathcal B(q)-\mathcal B(q_0))\mathcal K(q)+\mathcal B(q_0)(\mathcal K(q)-\mathcal K(q_0))$; the mean value theorem
with $|\mathcal B'|=\mathcal B|L_{\mathcal B}|\le\mathcal B_{\max}|L_{\mathcal B}|_{\max}$ bounds the first term by
$\mathcal B_{\max}|L_{\mathcal B}|_{\max}|\mathcal K|_{\max}|q-q_0|$, the second by $\mathcal B_{\max}|\mathcal K'|_{\max}|q-q_0|$, and
$|q-q_0|\le\Delta q$. For $j=-1$, $|\partial_\qval(-w_{\mathrm R}(1-\qval)^M)|\le w_{\mathrm R}M(1-q_l)^{M-1}$; similarly at $M+1$. (d) Taylor's theorem
with the Lagrange remainder about the midpoint: $|\tilde\kappa''|\le|\partial^2_\qval\kappa_2|+2\eta M^2$, and
$|\partial_\qval^2(\mathcal B\mathcal K)|=|\mathcal B(\mathcal K''+2L_{\mathcal B}\mathcal K'+(L_{\mathcal B}^2+L_{\mathcal B}')\mathcal K)|$ with
$\mathcal K''=2(w_{\mathrm L}\kappa^{\mathrm L}_j+w_{\mathrm R}\kappa^{\mathrm R}_j)$. The maxima are valid because $\mathcal B'$ has the sign of the affine function $(j-1)-(M-2)\qval$, so
that the maximum of $\mathcal B$ on $[q_l,q_h]$ is at an endpoint or, for $M\ne2$, at $(j-1)/(M-2)$ (for $M=2$ this function is
the constant $j-1$), $\mathcal K$ is a quadratic (affine when $w_{\mathrm L}\kappa^{\mathrm L}_j+w_{\mathrm R}\kappa^{\mathrm R}_j=0$), $\mathcal K'$ is
affine, and each term of $L_{\mathcal B}$ and $L_{\mathcal B}'$ is monotone in $\qval$.
\end{proof}

\begin{theorem}[what a passing box asserts]\label{apC:thm:mgfbox}
Let the parameters of a covering be sane ($\theta,D\ge0$, $0\le\tau_r\le1$, $\ell_r\ge0$) and let a box pass the
box test. Then for every real $q\in[q_l,q_h]$ and $\muM\in[m_l,m_h]$ an MGF certificate exists for $(q,\muM,\theta,D)$ with
the Laplace constraints $(1-q+q\tau_r,\ e^{-\ell_r\muM})_r$.
\end{theorem}
\emph{Idea.} The certificate is read off from the box; fibers at or above the cutoff are handled by the tail test (T),
smaller fibers by the atom bounds, and the margin by the test (Mg) at the ends $m_l,m_h$ of the box.
\begin{proof}
Take $\eta,\xi,\gamma,w_{\mathrm L},w_{\mathrm R},(z_r)$ from the box, $\alpha'=\alpha+\beta(\muM-m_0)-\gamma(\muM-m_0)^2$ and
$\beta'=\beta-2\gamma(\muM-m_0)$; then $\alpha'+\beta'(M-\muM)-\gamma(M-\muM)^2=\psi_0(M)$ identically, with
$\psi_0(M)=\alpha+\beta(M-m_0)-\gamma(M-m_0)^2$. Nonnegativity of $\eta^{-1}$, $\gamma$, $w_{\mathrm L}$, $w_{\mathrm R}$,
$z_r$ is (S).

(P). Since $0\le\tau_r\le1$ and $q\le q_h$, $1-q+q\tau_r\ge s_r\ge0$, so $\sum_rz_r(1-q+q\tau_r)^M\ge\sum_rz_rs_r^M$; it
therefore suffices that $\mathrm{rhs}(M)\le\kappa_2^q(M,j)+\xi\delta+\eta\delta^2$, $\delta=j-qM$, for all $M$, $j$ and
$q\in[q_l,q_h]$. \emph{Fibers $M\ge\bar M$}: $\psi_0(M)-\psi_0(\bar M)=(M-\bar M)\big(\beta-\gamma(M+\bar M-2m_0)\big)
\le(M-\bar M)(\beta-2\gamma(\bar M-m_0))\le0$ by (T) and $\gamma\ge0$; the priced constraints are nonnegative; so
$\mathrm{rhs}(M)\le\psi_0(\bar M)\le-\varpi(\bar M)-\xi^2/(4\eta)\le\kappa_2+\xi\delta+\eta\delta^2$ (by the bound stated with
the definition of $\varpi$). \emph{Fibers $M<\bar M$} (with the loop invariants of the integer vectors, Pascal's row and
the powers): if $j\le-2$ then $\kappa_2=0$ and $\delta\le-2-q_lM$, so the first half-line test applies; $j\ge M+2$ is
symmetric. For $-1\le j\le M+1$: if $\mathrm{lev}\le0$, then
$\mathrm{rhs}\le-\varpi(M)-\xi^2/(4\eta)$ and we conclude as for the tail. Otherwise, if
$j<j_{\mathrm{lo}}$ then $\delta-\delta_*\le-\Delta_-$ with $\Delta_-\ge0$,
$\eta\Delta_-^2\ge\mathrm{lev}$, so
$\kappa_2+\xi\delta+\eta\delta^2\ge-\varpi(M)+\eta\Delta_-^2-\xi^2/(4\eta)\ge\mathrm{rhs}$;
$j>j_{\mathrm{hi}}$ is symmetric; and an enumerated atom passes one of the three tests, which give the bound by the same
computation (skip), resp.\ by Lemma~\ref{apC:lem:atoms}(c) with $\xi\delta+\eta\delta^2\ge Q_\delta$, resp.\
Lemma~\ref{apC:lem:atoms}(d).

(Mg). For $\muM\in[m_l,m_h]$: $\beta(\muM-m_0)\ge\min\{\beta(m_l-m_0),\beta(m_h-m_0)\}$ (linear in $\muM$);
$(\muM-m_0)^2\le\max\{(m_l-m_0)^2,(m_h-m_0)^2\}$ (convex); $\gamma D\muM\le\gamma Dm_h$; $\eta\theta q(1-q)\muM\le\eta\theta v_{\max}m_h$;
and $e^{-\ell_r\muM}\le e^{-\ell_rm_l}\le\hat e_r$ because $\ell_r\ge0$, $\muM\ge m_l>0$ and $e^{x}\ge\sum_{i\le120}x^i/i!$ for
$x\ge0$. With $z_r\ge0$ the checked margin gives $\eta\theta q(1-q)\muM<\alpha'-\gamma D\muM-\sum_rz_re^{-\ell_r\muM}$.
\end{proof}

\begin{theorem}[a passing covering]\label{apC:thm:mgfpiece}
If a covering of $[\lambda_{\mathrm{lo}},\lambda_{\mathrm{hi}}]$ and $[m_{\mathrm{lo}},m_{\mathrm{hi}}]$ passes, then for every
$q\in[q(\lambda_{\mathrm{lo}}),q(\lambda_{\mathrm{hi}})]$ and $\muM\in[m_{\mathrm{lo}},m_{\mathrm{hi}}]$ an MGF certificate exists
for $(q,\muM,\theta,D)$ with the Laplace constraints $(1-q+q\tau_r,e^{-\ell_r\muM})_r$ of the covering, and these parameters
are valley-free.
\end{theorem}
\begin{proof}
The test of the covering checks consecutive $q$-slabs, the first starting at or below $q(\lambda_{\mathrm{lo}})$,
each starting at or below the end of the previous one, the last ending at or above
$q(\lambda_{\mathrm{hi}})$; for each slab a chain of boxes each containing the slab in $q$, the first with
$m_l\le m_{\mathrm{lo}}$, each next one with $m_l$ at most the previous $m_h$, the last with $m_h\ge m_{\mathrm{hi}}$. By
induction along the slabs and the chains every $(q,\muM)$ of the rectangle lies in a listed box, which passes;
Theorem~\ref{apC:thm:mgfbox} and (S) give the claims, and Theorem~\ref{p2:thm:mgf2} gives valley-freeness.
\end{proof}

% ---------------------------------------------------------------------------------------------------------------
\subsection{The coverings below the starting means and the no-valley inequality}\label{apC:sec:families}

This subsection states the verified coverings of the parameter domain below $\muM=50$, the starting means $\mu_0(\iota)$ and
the no-valley inequality above them (Theorem~\ref{fib:thm:intro-novalley}, proved in Section~\ref{fib:sec}), and proves from
them the no-valley inequality on the parameter domain (Theorem~\ref{p2:thm:O4}).

\begin{lemma}[the coverings below $50$]\label{p2:lem:strips}
For every $\iota<55$ there is a covering of $\Ihat_\iota$ and $[m_{25}(\iota),m_\ast(\iota)]$ and a covering of
$\Ihat_\iota$ and $[m_\ast(\iota),50]$, both with the caps $\theta(\iota)$, $D(\iota)$ and the rates $\ell_r(\iota)$, that pass
the box test of Section~\ref{apC:sec:mgfbox}; here $m_\ast(\iota)\in[35/3,68/3]$. The $110$
coverings have $779$ boxes: $550$ on $109$ $q$-slabs below $m_\ast(\iota)$ and $229$ on $60$ $q$-slabs above it.
\end{lemma}
This is verified by computer in exact rational arithmetic (family~\ref{chk:strips} of Section~\ref{sec:computer}). All
$110$ coverings are verified in the formal proof (Section~\ref{sec:verification}); for $36$ subintervals the coverings below
$m_\ast(\iota)$ are verified there on a rectangle containing
$[q(\lamhat_\iota),q(\lamhat_{\iota+1})]\times[m_{25}(\iota),m_\ast(\iota)]$ and at a cap at least $D(\iota)$, which suffices since (Mg)
is monotone in $D$ ($\gamma\ge0$). The proof of Theorem~\ref{p2:thm:O4} uses $693$ of the $779$ boxes: the $550$ below
$m_\ast(\iota)$ and the $143$ above it whose range of means meets $[m_\ast(\iota),\mu_0(\iota))$ (Appendix~\ref{apx:fiber}).

The \emph{starting mean} $\mu_0(\iota)$ of the subinterval $\iota$ is the entry of Data~\ref{fib:tab:mu0} in
Appendix~\ref{apx:fiber}; $19\le\mu_0(\iota)\le50$ for every $\iota$. From the starting mean on, no box is needed:

\begin{theorem}[no valley above the starting means]\label{fib:thm:intro-novalley}
Let $\iota\in\{0,\dots,54\}$. For every $q$ with $q/(1-q)\in\Ihat_\iota$ and every real $\muM\ge\mu_0(\iota)$, no probability
mixture with $\mathbb{E}M=\muM$, $\mathbb{E}\delta=0$ and $\mathbb{E}\delta^2\le\theta(\iota)q(1-q)\muM$ at an integer $k$,
$\operatorname{Var}M\le D(\iota)\muM$ and $\mathbb{E}(1-q)^M\le e^{-\ell_0(\iota)\muM}$ has a valley at $k$.
\end{theorem}
The proof (Section~\ref{fib:sec}) builds an MGF certificate in closed form from a lower bound for the fibers of a weighted
valley kernel, the fiber inequality, and uses only the Laplace constraint of the tilt $\tau_0=0$.

\begin{theorem}[the no-valley inequality on the parameter domain]\label{p2:thm:O4}
Let $\lambda\in[1/3,7/3]$, $\iota=\iota(\lambda)$, $q=q(\lambda)$, and let a probability mixture satisfy
$\mathbb{E}M=\muM\ge m_{25}(\iota)$, $\mathbb{E}\delta=0$ and $\mathbb{E}\delta^2\le\theta(\iota)q(1-q)\muM$ at an integer $k$,
$\operatorname{Var}M\le\hat D(\lambda)\muM$, and $\mathbb{E}(1-q+q\tau_r)^M\le e^{-\ell_r(\iota)\muM}$ ($r<5$). Then the mixture
has no valley at $k$.
\end{theorem}
\begin{proof}
Since $m_{25}(\iota)\le136/9<19\le\mu_0(\iota)$ for every $\iota$ (Data~\ref{fib:tab:mu0} and Part~1 of the supplementary data
file, Section~\ref{sec:computer:data}), the two cases below cover every $\muM\ge m_{25}(\iota)$.

\emph{(a) $m_{25}(\iota)\le\muM<\mu_0(\iota)$.} Here $\muM<50$. By Lemma~\ref{p2:lem:strips} the point $(q,\muM)$ lies in a box
of one of the two coverings of $\Ihat_\iota$ (of $[m_{25}(\iota),m_\ast(\iota)]$ and of $[m_\ast(\iota),50]$), and that box is not
one whose range of means lies in $[\mu_0(\iota),50]$, because it contains $\muM<\mu_0(\iota)$. The box passes the box test and
the parameters of the covering are sane (both by Lemma~\ref{p2:lem:strips}: a covering passes only if every box passes and its
parameters are sane, Section~\ref{apC:sec:mgfbox}), so by Theorem~\ref{apC:thm:mgfbox} an MGF certificate exists for
$(q,\muM,\theta(\iota),D(\iota))$ with the Laplace constraints $(1-q+q\tau_r,e^{-\ell_r(\iota)\muM})_r$, and
Theorem~\ref{p2:thm:mgf2} excludes a valley at $k$, as $\hat D(\lambda)\le D(\iota)$.

\emph{(b) $\muM\ge\mu_0(\iota)$.} Theorem~\ref{fib:thm:intro-novalley} applies, with $\operatorname{Var}M\le\hat D(\lambda)\muM\le
D(\iota)\muM$ and the constraint $\mathbb{E}(1-q)^M\le e^{-\ell_0(\iota)\muM}$, which is the case $r=0$, $\tau_0=0$ of the
hypothesis.
\end{proof}

\begin{remark}\label{fib:rem:P58}
This proof uses only the boxes of Lemma~\ref{p2:lem:strips} whose range of means meets $[m_{25}(\iota),\mu_0(\iota))$; the
other $86$ boxes are not needed (Appendix~\ref{apx:fiber}). The earlier computational route documented in the
supplementary data (Section~\ref{sec:computer:data}) establishes Theorem~\ref{p2:thm:O4}, with an additional
hypothesis on the rates of the bands (Section~\ref{sec:verification}), in
three cases instead: all $779$ boxes up to mean $50$, coverings of the bands of activities from
mean $50$ to $400$, and a large-mean theorem above $400$. The formal proof described in
Section~\ref{sec:verification} follows the proof given here; below the starting means it checks all $779$ boxes.
\end{remark}

% ---------------------------------------------------------------------------------------------------------------
\subsection{The no-valley inequality above the starting means}\label{fib:sec}

This subsection proves Theorem~\ref{fib:thm:intro-novalley}, used in case~(b) of the proof of Theorem~\ref{p2:thm:O4}. The
proof has two steps. The first is a lower bound, with explicit constants, for the fibers of a weighted valley kernel: the fiber
inequality on the central range of activities (Theorem~\ref{fib:thm:env}), with its forms for the wing and for the outer
subintervals (Theorems~\ref{fib:thm:wing} and~\ref{fib:thm:window}). The second is an averaging argument in which the
fibers below the range of the first step, the spike fibers, are priced by the Laplace constraint of the tilt $\tau_0=0$
(Section~\ref{fib:sec:novalley}). The finite inequalities that the proofs use are stated in Appendix~\ref{apx:fiber}. We use the
mixtures, the valley kernel $\kappa_2$ and the MGF certificates of Section~\ref{p2:sec:criterion} with a single Laplace
constraint, $(s_0,e_0)=(1-q,e^{-\ell_0\mu})$, and throughout this subsection we write $\mu$ for the mean free count $\muM$.
Some letters have local meanings in this subsection, fixed in Section~\ref{fib:sec:setting}: $r$ is not the index of a tilt,
$v$ not a vertex, $x$ not the variable of a polynomial, $J$ not a state of a mixture, and $\gamma$ not a coefficient of a box
certificate; in Proposition~\ref{fib:prop:avg}, $\alpha$ and $\beta$ are the coefficients of a minorant, those of the
certificate being $\alpha_c$, $\beta_c$ and $\gamma_c$. $I(x)$ in Section~\ref{fib:sec:fiber} is not the independence
polynomial, and inside single proofs $A$, $c$, $K$, $n$ and $\xi$ are auxiliary constants or variables defined there.

% ---------------------------------------------------------------------------------------------------------------
\subsubsection{The weighted kernel, the fiber function and the averaging step}\label{fib:sec:setting}

Throughout this subsection, $q\in[\frac14,\frac34]$ and
\[
\wL=q^2(3-4q),\qquad \wR=(1-q)^2(4q-1),
\]
which are nonnegative exactly on $[\frac14,\frac34]$; $\kappa_w$ denotes $\kappa_2$ with these weights (the \emph{weighted kernel}).
We write
\[
r=2q-1,\quad v=q(1-q)=\tfrac{1-r^2}4,\quad \sigb=\sqrt{v\mu},\quad \aG=\frac{e^{-1/2}}{\sqrt{2\pi}}
\]
(in this subsection $r$ is this number, never the index of a tilt), and for integers $M\ge0$ and $j$
\[
N=M+2,\quad J=j+1,\quad t=\frac M\mu,\quad B=\frac N\mu=t+\frac2\mu,\quad \rho=\frac N{N+1},\quad
B_N(J)=\binom NJq^J(1-q)^{N-J}
\]
($B_N(J)=\binp_N(J;q)$, which is $0$ outside $0\le J\le N$; here $J$ is this shifted index, not a state of a mixture). The
\emph{offset} and its standardized square are
\[
u=\delta-\tfrac r2=J-q(N-1)-\tfrac12,\qquad x=\frac u{\sqrt{vN}},\qquad w=x^2,
\]
and the \emph{bonus} is
\[
\Delta=\frac{r^2(N^2-N+1)}{(1-r^2)N},\qquad \gamma=1+\Delta=\frac{N+r^2(N-1)^2}{(1-r^2)N}.
\]
The \emph{fiber function} and the target are
\[
\Ff(M,j)=\frac{\sigb^3\kappa_w(M,j)}{4v^2\aG}+\frac{u^2}{\sigb^2},\qquad \psi(t)=t-\tfrac92(t-1)^2 .
\]
At $q=\frac12$ one has $\wL=\wR=\frac14$, $u=\delta$, and $\Ff{}=\sigb^3\kappa/\aG+\delta^2/\sigb^2$ with the unweighted kernel
$\kappa=\fL/q+\fR/(1-q)$. From Section~\ref{fib:sec:fiber} on, statements about $\min_j\Ff(M,j)$ are proved for $q\ge\frac12$, by the
symmetry of Lemma~\ref{fib:lem:sym}; there $r=|2q-1|\in[0,\frac12]$, and every range $r\le r_m$ or $r\in[r_a,r_b]$ is a range of
$|2q-1|$. The spike weight $(1-q)^{-M}$ and the inputs $\theta,D,\ell_0$ below are always those of the original $q$.

For fixed $\mu$, the set of pairs $(M,j)$ with $j\in\mathbb Z$ is the \emph{fiber} $M$, and $t=M/\mu$ is its coordinate.

\paragraph{The subintervals.}\label{fib:sec:inputs}
For a subinterval $\iota$ we write $q_a=q(\lamhat_\iota)$, $q_b=q(\lamhat_{\iota+1})$, $L_b=-\log(1-q_b)$,
$r_{\rm hi}=\max(|2q_a-1|,|2q_b-1|)$ and $v_{\min}=\min(q_a(1-q_a),q_b(1-q_b))$; $\theta(\iota)$, $D(\iota)$ and the zero-tilt rate
$\ell_0(\iota)$ are those of Section~\ref{p2:sec:inputs}, listed with $m_{25}(\iota)$ and $m_\ast(\iota)$ in Part~1 of the
supplementary data file (Section~\ref{sec:computer:data}). For every $\lambda\in\Ihat_\iota$, $\hat D(\lambda)\le D(\iota)$ by
\eqref{p2:eq:Dhat}.

\paragraph{The averaging step.}
\begin{proposition}[averaging with the shifted penalty]\label{fib:prop:avg}
Let $q\in[\frac14,\frac34]$, $\mu>0$, and suppose that for reals $\alpha,\beta$ and $\Gamma\ge0$, $Z'\ge0$,
\begin{equation}\label{fib:eq:FL}
\Ff(M,j)\ \ge\ \alpha+\beta(t-1)-\Gamma(t-1)^2-Z'(1-q)^M\qquad\text{for all }M\ge0,\ j\in\mathbb Z .
\end{equation}
Then the weights $\wL,\wR$, $\eta_c=4v^2\aG/\sigb^5$, $\xi_c=-r\eta_c$, $\alpha_c=\frac{4v^2\aG}{\sigb^3}(\alpha-\frac{r^2}{4\sigb^2})$,
$\beta_c=\frac{4v^2\aG\beta}{\sigb^3\mu}$, $\gamma_c=\frac{4v^2\aG\Gamma}{\sigb^3\mu^2}$ and the price $z_0=\frac{4v^2\aG Z'}{\sigb^3}$ on
$(1-q,e^{-\ell_0\mu})$ satisfy (P), and (Mg) holds if and only if
\begin{equation}\label{fib:eq:margin}
\alpha-\theta-\frac{\Gamma D}\mu-\frac{r^2}{(1-r^2)\mu}-Z'e^{-\ell_0\mu}>0 .
\end{equation}
In that case no probability mixture with $\E M=\mu$, $\E\delta=0$, $\E\delta^2\le\theta v\mu$, $\Var M\le D\mu$ and
$\E(1-q)^M\le e^{-\ell_0\mu}$ has a valley.
\end{proposition}

\begin{proof}
Multiply \eqref{fib:eq:FL} by $4v^2\aG/\sigb^3>0$. The left side becomes $\kappa_w+\eta_cu^2$, and
$\eta_cu^2=\eta_c\delta^2+\xi_c\delta+\eta_c\frac{r^2}4$ since $u^2=\delta^2-r\delta+\frac{r^2}4$. With $t-1=(M-\mu)/\mu$ the right side
becomes $\frac{4v^2\aG}{\sigb^3}\alpha+\beta_c(M-\mu)-\gamma_c(M-\mu)^2-z_0(1-q)^M$, and
$\frac{4v^2\aG}{\sigb^3}\alpha-\eta_c\frac{r^2}4=\alpha_c$. This is (P) with $\kappa_2=\kappa_w$, which is allowed because
$\wL,\wR\ge0$ on $[\frac14,\frac34]$. Condition (Mg) reads $\eta_c\theta v\mu<\alpha_c-\gamma_cD\mu-z_0e^{-\ell_0\mu}$; multiply it by
$\sigb^3/(4v^2\aG)$ and use $\sigb^2=v\mu$ and $\frac{r^2}{4\sigb^2}=\frac{r^2}{(1-r^2)\mu}$ to get \eqref{fib:eq:margin}. The MGF criterion
(Theorem~\ref{p2:thm:mgf2}) concludes.
\end{proof}

The coefficient $\beta$ does not enter \eqref{fib:eq:margin}: in the proof of Theorem~\ref{p2:thm:mgf2} the term $\beta(M-\mu)$ has
expectation $0$, because $\E M=\mu$ is the same $\mu$ as in $t=M/\mu$. So any real $\beta$ may be used in \eqref{fib:eq:FL}. The shift
$\frac r2$ of the penalty costs $\frac{r^2}{(1-r^2)\mu}$ in the margin.

% ---------------------------------------------------------------------------------------------------------------
\subsubsection{The weighted kernel identity and the activity-one bound}\label{fib:sec:template}

\paragraph{The weighted kernel identity.}

\begin{lemma}[the weighted kernel identity]\label{fib:lem:T1}
For $q\in(0,1)$, every $M\ge0$ and every $j\in\mathbb Z$,
\[
\kappa_w(M,j)=\frac{B_N(J)\,\big(N+r^2(N-1)^2-4u^2\big)}{N(N-1)}=\frac{4v\,B_N(J)\,(\gamma-w)}{N-1}.
\]
Consequently
\begin{equation}\label{fib:eq:Fform}
\Ff(M,j)=\frac{\mu^{3/2}\sqrt{1-r^2}\,B_N(J)}{2\aG(N-1)}\,(\gamma-w)+Bw .
\end{equation}
\end{lemma}

\begin{proof}
From the factorials,
\begin{gather*}
\binp_M(j)=B_N(J)\frac{J(N-J)}{N(N-1)v},\qquad \binp_M(j-1)=B_N(J)\frac{J(J-1)}{N(N-1)q^2},\\
\binp_M(j+1)=B_N(J)\frac{(N-J)(N-J-1)}{N(N-1)(1-q)^2},
\end{gather*}
and all sides vanish together outside the supports. Since $\wL\frac{1-q}q+\wR\frac q{1-q}=v(3-4q)+v(4q-1)=2v$,
$\wL/q^2=3-4q$ and $\wR/(1-q)^2=4q-1$,
\[
\frac{N(N-1)\kappa_w}{B_N(J)}=2J(N-J)-(3-4q)J(J-1)-(4q-1)(N-J)(N-J-1).
\]
Put $J=c+u$ with $c=q(N-1)+\frac12$. The right side is a quadratic in $u$ with leading coefficient $-2-(3-4q)-(4q-1)=-4$. Its
coefficient of $u$ is $2(N-2c)-(3-4q)(2c-1)+(4q-1)(2(N-c)-1)=2(N-1)\big[(1-2q)-q(3-4q)+(4q-1)(1-q)\big]=0$, since
$2c-1=2q(N-1)$, $N-2c=(N-1)(1-2q)$ and $2(N-c)-1=2(N-1)(1-q)$. Its value at $u=0$ is $N+r^2(N-1)^2$: with
$c(c-1)=q^2(N-1)^2-\frac14$, $(N-c)(N-c-1)=(1-q)^2(N-1)^2-\frac14$ and $c(N-c)=q(1-q)(N-1)^2+\frac{N-1}2+\frac14$, the constant
terms add up to $N$ and the coefficient of $(N-1)^2$ is $2q(1-q)-(3-4q)q^2-(4q-1)(1-q)^2=(2q-1)^2$. This is the first form.
The second uses $N+r^2(N-1)^2-4u^2=(1-r^2)N(\gamma-w)$ and $1-r^2=4v$. For \eqref{fib:eq:Fform}:
$\frac{\sigb^3}{4v^2\aG}\cdot\frac{4v}{N-1}=\frac{\mu^{3/2}\sqrt v}{\aG(N-1)}=\frac{\mu^{3/2}\sqrt{1-r^2}}{2\aG(N-1)}$ and
$\frac{u^2}{\sigb^2}=\frac{vNw}{v\mu}=Bw$.
\end{proof}

So the weighted kernel is a binomial law on $N=M+2$ trials times the quadratic factor $\gamma-x^2$, whose centre is the offset
$u=0$, that is $\delta=\frac{2q-1}2$; the penalty of $\Ff{}$ is centred at the same point.

\begin{lemma}[symmetry]\label{fib:lem:sym}
$\kappa_w$ at $(1-q,M,M-j)$ equals $\kappa_w$ at $(q,M,j)$, and $u_{1-q}(M,M-j)=-u_q(M,j)$. Hence $\Ff_{1-q}(M,M-j)=\Ff_q(M,j)$.
\end{lemma}

\begin{proof}
$\wL(1-q)=\wR(q)$, and $\binp_M(j;q)=\binp_M(M-j;1-q)$ gives $\fR^{(1-q)}(M,M-j)=\fL^{(q)}(M,j)$ and $\fL^{(1-q)}(M,M-j)=\fR^{(q)}(M,j)$;
divide by $1-q$ and $q$ respectively. Further $u=\delta-\frac r2$ with the signed $r=2q-1$, and both $\delta$ and $r$ change sign.
\end{proof}

So every statement about $\min_j\Ff(M,j)$ may be proved for $q\ge\frac12$.

\paragraph{The minimum of a minorant, and the function $G$.}

\begin{lemma}[the minimum of a minorant]\label{fib:lem:min1}
Let $P,B>0$, $\rho>0$ and $H(w)=P(1-w)e^{-\rho w/2}+Bw$. Let $\omega=\log(Be^{\rho/2}/P)$ and let $s>-1$ solve $s+\log(1+s)=\omega$.
Then $H$ decreases on $(-\infty,w_*]$ and increases on $[w_*,\infty)$, where $w_*=1-2s/\rho$, and
\[
\min_{\mathbb R}H=B\Big(1-\frac2\rho\,G(\omega)\Big),\qquad G(\omega):=\frac{s^2}{1+s}.
\]
\end{lemma}

\begin{proof}
With $\sigma(w)=\rho(1-w)/2$, $H'(w)=B-Pe^{-\rho/2}e^{\sigma}(1+\sigma)$. The map $\sigma\mapsto e^\sigma(1+\sigma)$ is negative on
$(-\infty,-1)$ and positive and increasing on $(-1,\infty)$, and $\sigma$ decreases in $w$; so $H'>0$ where $\sigma\le-1$, and where
$\sigma>-1$ the derivative $H'$ increases in $w$ from $-\infty$. Hence $H'<0$ for $w<w_*$ and $H'>0$ for $w>w_*$, where $\sigma(w_*)=s$,
that is $e^s(1+s)=Be^{\rho/2}/P$. At $w_*$, $Pe^{-\rho/2}e^s=B/(1+s)$ and $1-w_*=2s/\rho$, so
$H(w_*)=\frac{2Bs}{\rho(1+s)}+B-\frac{2Bs}\rho=B(1-\frac2\rho\frac{s^2}{1+s})$.
\end{proof}

\begin{corollary}[the closed form with a bonus]\label{fib:lem:min}
Let $P,B,\rho>0$ and $\gamma\in\mathbb R$, and $H(w)=P(\gamma-w)e^{-\rho w/2}+Bw$. With $\omega=\log(Be^{\rho\gamma/2}/P)$ and $s>-1$ the
root of $s+\log(1+s)=\omega$, $H$ decreases on $(-\infty,w_*]$ and increases on $[w_*,\infty)$, $w_*=\gamma-2s/\rho$, and
$\min_{\mathbb R}H=B(\gamma-\frac2\rho G(\omega))$.
\end{corollary}

\begin{proof}
With $w=w'+\gamma-1$, $H=Pe^{-\rho(\gamma-1)/2}(1-w')e^{-\rho w'/2}+Bw'+B(\gamma-1)$. Apply Lemma~\ref{fib:lem:min1} with
$P'=Pe^{-\rho(\gamma-1)/2}$, whose $\log(Be^{\rho/2}/P')$ is $\omega$.
\end{proof}

We write $s(\omega)$ for the root of $s+\log(1+s)=\omega$, $s>-1$ (the \emph{root map}).

\begin{lemma}[the function $G$]\label{fib:lem:G}

\textup{(a)} $s$ is increasing, $s(0)=0$, $s(\omega)$ has the sign of $\omega$, and $s'=\frac{1+s}{2+s}\in(0,1)$; hence
$s(\omega+h)\le s(\omega)+h$ for $h\ge0$, and $0\le s(\omega)\le\omega$ for $\omega\ge0$.
\textup{(b)} $G$ is $C^2$ with $G(0)=G'(0)=0$, $G'(\omega)=\frac s{1+s}<1$ and $G''(\omega)=\frac1{(1+s)(2+s)}>0$, and $G''$ is decreasing.
So $G\ge0$ is convex, decreasing on $(-\infty,0]$ and increasing on $[0,\infty)$, and $|G'|$ decreases on $(-\infty,0]$.
\textup{(c)} For $h\ge0$: $G(\omega+h)\le G(\omega)+hG'(\omega)+\frac{h^2}2G''(\omega)$.
\textup{(d)} For $\omega\ge0$: $G(\omega)\le\frac{\omega^2}4$, $G(\omega)\le\omega$ and $G(\omega)\le\omega-\log(1+\omega)$.

\end{lemma}

\begin{proof}
(a) $d\omega/ds=1+\frac1{1+s}=\frac{2+s}{1+s}>1$. If $\omega\ge0$ then $s\ge0$ and $\omega=s+\log(1+s)\ge s$. (b)
$\frac d{ds}\frac{s^2}{1+s}=\frac{s(2+s)}{(1+s)^2}$, so $G'=\frac s{1+s}$, and $G''=\frac d{ds}\frac s{1+s}\cdot\frac{1+s}{2+s}=\frac1{(1+s)(2+s)}$,
which decreases in $s$, hence in $\omega$. (c) Taylor's theorem with $G''(\xi)\le G''(\omega)$ for $\xi\ge\omega$. (d) On
$\omega\ge0$, $G''\le\frac12$ and $G'<1$, and $G(0)=G'(0)=0$; this gives the first two bounds. For the third, both sides vanish at
$0$, and by (a) $G'=\frac s{1+s}\le\frac\omega{1+\omega}$, the derivative of $\omega-\log(1+\omega)$, since $x\mapsto x/(1+x)$ increases.
\end{proof}

\begin{lemma}[the bonus split]\label{fib:lem:split}
For real $\omega_0,h$: $G(\omega_0+h)\le G(\omega_0)+h\,G'(\omega_0+h)$. Hence, for $\rho>0$, $\Delta\ge0$, real $\tilde e$ and $\gamma=1+\Delta$,
with $\sigma=G'(\omega_0+\frac\rho2\Delta+\tilde e)$,
\[
\gamma-\tfrac2\rho G\big(\omega_0+\tfrac\rho2\Delta+\tilde e\big)\ \ge\ 1-\tfrac2\rho G(\omega_0)+\Delta(1-\sigma)-\tfrac2\rho\,\tilde e\,\sigma .
\]
\end{lemma}

\begin{proof}
$G$ is convex (Lemma~\ref{fib:lem:G}(b)), so $G(\omega_0)\ge G(\omega_0+h)-hG'(\omega_0+h)$; take $h=\frac\rho2\Delta+\tilde e$.
\end{proof}

The term $\Delta(1-\sigma)$ is nonnegative ($G'<1$): inside $G$ the bonus is a shift $\frac\rho2\Delta$, and the split moves it
out of $G$ exactly.

\paragraph{The central binomial law.}

\begin{lemma}[central binomial coefficients]\label{fib:lem:wallis}
For $m\ge1$, $W_m=\binom{2m}m4^{-m}$ satisfies
\[
e^{1/(24m+1)-1/(6m)}(\pi m)^{-1/2}\ \le\ W_m\ \le\ \big(\pi(m+\tfrac14)\big)^{-1/2}.
\]
\end{lemma}

\begin{proof}
$x_m=W_m^2(m+\frac14)$ has $x_{m+1}/x_m=\frac{(2m+1)^2(4m+5)}{(2m+2)^2(4m+1)}=\frac{16m^3+36m^2+24m+5}{16m^3+36m^2+24m+4}>1$, and $x_m\to1/\pi$
by Wallis' product; so $x_m\le1/\pi$. Robbins' form of Stirling's formula~\cite{Robbins}, $n!=\sqrt{2\pi n}(n/e)^ne^{\vartheta_n}$ with
$\frac1{12n+1}<\vartheta_n<\frac1{12n}$, gives $W_m=e^{\vartheta_{2m}-2\vartheta_m}(\pi m)^{-1/2}$ and the lower bound.
\end{proof}

For $N\ge2$ put
\[
U_N=\begin{cases}\big(\pi(\frac N2+\frac14)\big)^{-1/2}&N\text{ even},\\ e^{\frac1{2(N+1)}}\big(\pi(\frac{N+1}2+\frac14)\big)^{-1/2}&N\text{ odd},\end{cases}\qquad
L_N=\begin{cases}e^{\frac1{12N+1}-\frac1{3N}-\frac1{12(N-1)}}\big(\pi\frac N2\big)^{-1/2}&N\text{ even},\\
e^{\frac1{12N+13}+\frac1{12(N+1)}}\big(\pi\frac{N+1}2\big)^{-1/2}&N\text{ odd.}\end{cases}
\]

\begin{lemma}[the symmetric upper bound]\label{fib:lem:UN}
For $N\ge2$, every integer $J$ and $d=J-\frac N2$: $\binom NJ2^{-N}\le U_Ne^{-2d^2/(N+1)}$.
\end{lemma}

\begin{proof}
By the symmetry $J\mapsto N-J$ take $d\ge0$; the bound is trivial for $J>N$. Write $b_N(J)=\binom NJ2^{-N}$. If $N=2m$ and
$J=m+d'$, $1\le d'\le m$, then $b_N(m+i)/b_N(m+i-1)=\frac{m-i+1}{m+i}=\frac{1-\alpha_i}{1+\alpha_i}$ with $\alpha_i=\frac{2i-1}{N+1}$, so
$\log\frac{b_N(J)}{W_m}=-2\sum_{i\le d'}\artanh\alpha_i\le-2\sum_{i\le d'}\alpha_i=-\frac{2d'^2}{N+1}$, and $W_m\le U_N$ by
Lemma~\ref{fib:lem:wallis}. If $N=2m+1$ and $J=m+1+d'$, then $d=d'+\frac12$, the ratios are $\frac{1-\alpha_i}{1+\alpha_i}$ with
$\alpha_i=\frac{2i}{N+1}$, $\sum_{i\le d'}\alpha_i=\frac{d^2-1/4}{N+1}$, and $b_N(m+1)=W_{m+1}$ because $\binom{2m+2}{m+1}=2\binom{2m+1}{m+1}$; so
$b_N(J)\le W_{m+1}e^{1/(2(N+1))}e^{-2d^2/(N+1)}$, and $W_{m+1}\le(\pi(\frac{N+1}2+\frac14))^{-1/2}$.
\end{proof}

\paragraph{The activity-one bound.}
For $X\in\{U,L\}$ and $N=M+2$ put
\[
a_X(N)=\frac{N^{3/2}(\pi/2)^{1/2}X_N}{N-1},\qquad P^{\mathrm T}_X=\frac{\mu^{3/2}X_N}{2\aG(N-1)},\qquad
\oT_X=\log\frac{Be^{\rho/2}}{P^{\mathrm T}_X},\qquad \omega_t=\tfrac52\log t .
\]

\begin{lemma}[the log-ratio shift]\label{fib:lem:k}
(a) $P^{\mathrm T}_X=e^{1/2}B^{-3/2}a_X(N)$ and $\oT_X=\frac52\log B-\frac1{2(N+1)}-\log a_X(N)$; for $M\ge1$, $\oT_X=\omega_t+\frac{k_X(M)}{M+3}$, where
\[
k_X(M)=(M+3)\Big[\tfrac52\log\big(1+\tfrac2M\big)-\tfrac1{2(M+3)}-\log a_X(M+2)\Big].
\]
(b) For $N\ge10$: $\frac{0.65}N\le\log a_X(N)\le\frac3{4N}+\frac{1.1}{N^2}$.\\
(c) For $M\ge8$: $3.55\le k_X(M)\le3.85+\frac{10}M$.
\end{lemma}

\begin{proof}
(a) $2\aG e^{1/2}=(2/\pi)^{1/2}$, so $P^{\mathrm T}_X=e^{1/2}\mu^{3/2}N^{-3/2}a_X(N)$; then $Be^{\rho/2}/P^{\mathrm T}_X=B^{5/2}e^{(\rho-1)/2}/a_X$ and
$1-\rho=\frac1{N+1}$; finally $\frac52\log B=\omega_t+\frac52\log\frac NM$. (b) Explicitly,
$\log a_U=-\log(1-\frac1N)-\frac12\log(1+\frac1{2N})$ ($N$ even), $\log a_U=-\log(1-\frac1N)+\frac1{2(N+1)}-\frac12\log(1+\frac3{2N})$ ($N$ odd),
$\log a_L=-\log(1-\frac1N)+\frac1{12N+1}-\frac1{3N}-\frac1{12(N-1)}$ ($N$ even) and
$\log a_L=-\log(1-\frac1N)-\frac12\log(1+\frac1N)+\frac1{12N+13}+\frac1{12(N+1)}$ ($N$ odd). Use
$x\le-\log(1-x)\le x+\frac{x^2}2+\frac{x^3}{3(1-x)}$, $y-\frac{y^2}2\le\log(1+y)\le y$, and, for the even $a_L$,
$\frac1{12N+1}-\frac1{12(N-1)}\ge-\frac{0.1}{N^2}$ ($N\ge10$); for the odd $a_L$, $\frac2{12N+13}\ge\frac2{13.3N}$ ($N\ge10$); and $a_L\le a_U$, that is $L_N\le U_N$: the prefactor of $L_N$ exceeds that of $U_N$ by the factor
$(1+\frac1{2N})^{1/2}\le e^{1/(4N)}$ ($N$ even), resp. $(1+\frac1{2(N+1)})^{1/2}\le e^{1/(4(N+1))}$ ($N$ odd), so
$\log\frac{L_N}{U_N}\le\frac1{12N+1}-\frac1{12N}-\frac1{12(N-1)}<0$, resp. $\le\frac1{12N+13}-\frac1{6(N+1)}<0$.
(c) From $x-\frac{x^2}2\le\log(1+x)\le x-\frac{x^2}2+\frac{x^3}3$ at $x=2/M$:
$5+\frac{10}M-\frac{15}{M^2}\le\frac52(M+3)\log(1+\frac2M)\le5+\frac{10}M$ (for $M\ge3$, $20(M+3)\le45M$). By (b),
$0.65\le(M+3)\log a_X(M+2)\le\frac34+\frac{1.96}{M+2}$. Hence $k_X\ge\frac{15}4+\frac{10}M-\frac{15}{M^2}-\frac{1.96}{M+2}\ge\frac{15}4\ge3.55$ and
$k_X\le5+\frac{10}M-\frac12-0.65$.
\end{proof}

\begin{definition}[the Gaussian slack]\label{fib:def:S}
For $t>0$ let $s_t=s(\omega_t)$, that is $e^{s_t}(1+s_t)=t^{5/2}$, and let $G_0=\frac{s_t^2}{1+s_t}$, $G_1=\frac{s_t}{1+s_t}$,
$G_2=\frac1{(1+s_t)(2+s_t)}$ be the values of $G,G',G''$ at $\omega_t$. The \emph{Gaussian slack} is $\Sl(t)=\frac92\frac{(t-1)^2}t-2G_0$, and for
$M>0$
\[
T(t,M)=\min_{\mathsf k\in\{3.55,\ 3.85+10/M\}}\Big[M\,\Sl(t)+2-6G_0-2\mathsf kG_1-\frac{\mathsf k^2G_2}{M+3}\Big].
\]
We also use $s$ itself as a coordinate: $t=t(s)=(e^s(1+s))^{2/5}$ is increasing in $s$, and $\omega_{t(s)}=s+\log(1+s)$.
\end{definition}

\begin{proposition}[the activity-one bound]\label{fib:prop:template}
Let $M\ge8$, $\mu>0$ and $X\in\{U,L\}$. Then
\[
\mu\Big(B\big(1-\tfrac2\rho G(\oT_X)\big)-\psi(t)\Big)=N-\mu\psi(t)-2(N+1)G(\oT_X)\ \ge\ T(t,M).
\]
\end{proposition}

\begin{proof}
The identity holds because $\mu B=N$ and $\mu B\frac2\rho=2(N+1)$. Since $B=t+\frac2\mu$, $\mu(B-\psi(t))=2+\frac92\frac{(M-\mu)^2}\mu=2+\frac92M\frac{(t-1)^2}t$.
By Lemma~\ref{fib:lem:k}, $\oT_X=\omega_t+h$ with $h=k_X/(M+3)\ge0$, and by Lemma~\ref{fib:lem:G}(c),
$2(M+3)G(\oT_X)\le2(M+3)G_0+2k_XG_1+k_X^2G_2/(M+3)$. Writing $2(M+3)G_0=2MG_0+6G_0$,
\[
\mu\Big(B\big(1-\tfrac2\rho G(\oT_X)\big)-\psi\Big)\ge M\,\Sl(t)+2-6G_0-2k_XG_1-\frac{k_X^2G_2}{M+3}.
\]
The right side is concave in $k_X$ ($G_2>0$), and $k_X\in[3.55,3.85+\frac{10}M]$ by Lemma~\ref{fib:lem:k}(c); so it is at least its smaller
value at the two ends of that interval, which is $T(t,M)$.
\end{proof}

\begin{lemma}[the Gaussian slack is nonnegative]\label{fib:lem:S}
Let $\varphi(s)=\frac{18}{25}(2-\frac s2)^2-\frac2{1+s}$, $\Lambda(a)=2a+\frac{a^2}2+\frac{a^3}3+\frac{a^4}4+\frac{a^5}5+\frac{a^6}6$ and
$Q(a)=\frac{18}{25}(\Lambda(a)/a)^2(1+\Lambda(a)^2/75)(1-a)-2$, where $\Lambda(a)/a$ denotes the polynomial
$2+\frac a2+\frac{a^2}3+\frac{a^3}4+\frac{a^4}5+\frac{a^5}6$ (the continuous extension of the quotient to $a=0$), so that $Q(0)=\frac{22}{25}$. Then
$\Sl(t(s))\ge s^2\varphi(s)\ge0$ for $s\in[0,0.67]$, and
$\Sl(t(s))\ge\frac{a^2}{1-a}Q(a)\ge0$ for $s=-a\in[-0.528,0]$. Hence $\Sl{}\ge0$ on $[t(-0.528),t(0.67)]\supset[0.5996,1.605]$.
\end{lemma}

\begin{proof}
Since $\log t=\frac25\omega_t$, $\frac{(t-1)^2}t=(t^{1/2}-t^{-1/2})^2=4\sinh^2\frac{\omega_t}5$, and $\sinh^2y\ge y^2+\frac{y^4}3$; so
$\frac92\frac{(t-1)^2}t\ge\frac{18}{25}\omega_t^2(1+\frac{\omega_t^2}{75})$. Also $G_0=\frac{s^2}{1+s}$. For $s\ge0$:
$\omega_t=s+\log(1+s)\ge2s-\frac{s^2}2\ge0$, so $\Sl{}\ge\frac{18}{25}s^2(2-\frac s2)^2-\frac{2s^2}{1+s}=s^2\varphi(s)$. The function $\varphi$ is
concave on $[0,0.67]$ ($\varphi''=0.36-4/(1+s)^3<0$), so $\varphi\ge\min(\varphi(0),\varphi(0.67))>0$ (Lemma~\ref{fib:chk:A0}). For $s=-a\le0$:
$|\omega_t|=a-\log(1-a)=2a+\sum_{i\ge2}\frac{a^i}i\ge\Lambda(a)$ and $G_0=\frac{a^2}{1-a}$, so
$\Sl{}\ge\frac{18}{25}\Lambda^2(1+\frac{\Lambda^2}{75})-\frac{2a^2}{1-a}=\frac{a^2}{1-a}Q(a)$ (as $\Lambda(a)^2=a^2(\Lambda(a)/a)^2$, also at $a=0$), and $Q\ge0$ on $[0,0.528]$ (Lemma~\ref{fib:chk:A1}). Finally
$t(-0.528)=0.59958\ldots$ and $t(0.67)=1.60500\ldots$.
\end{proof}

% ---------------------------------------------------------------------------------------------------------------
\subsubsection{The fiber inequality at every activity}\label{fib:sec:fiber}

Throughout Section~\ref{fib:sec:fiber}, $q\in[\frac12,\frac34]$, so $r=2q-1\in[0,\frac12]$; $N\ge10$ unless stated otherwise.

\paragraph{Shape bounds at a general activity.}
Let $a=\artanh r$, $I(x)=x\artanh x+\frac12\log(1-x^2)$ (so $I'=\artanh$, $I(0)=0$), and
\[
\Psi_r(z)=I(r+z)-I(r)-z\artanh r\qquad(|r+z|<1),
\]
so that $\Psi_r(0)=\Psi_r'(0)=0$ and $\Psi_r''(z)=1/(1-(r+z)^2)$. Let $J_m=q(N+1)-\frac12$, $y=J-J_m=u-r$ and
$z=2(J-J_m)/(N+1)=2(u-r)/(N+1)$. Put $c_N=W_{N/2}$ if $N$ is even and $c_N=W_{(N+1)/2}\,e^{(N+1)I(1/(N+1))}$ if $N$ is odd.

\begin{lemma}[the lower local bound]\label{fib:lem:LB}
Let $N\ge2$, $0\le r<1$, $q=\frac{1+r}2$ and $d=J-\frac N2$. Then $B_N(J)=(1-r^2)^{N/2}e^{2ad}\binom NJ2^{-N}$, and for $0\le J\le N$
\[
B_N(J)\ \ge\ \frac{c_N}{\sqrt{1-r^2}}\,e^{-(N+1)\Psi_r(z)} .
\]
\end{lemma}

\begin{proof}
$q^J(1-q)^{N-J}=2^{-N}(1+r)^{N/2+d}(1-r)^{N/2-d}=2^{-N}(1-r^2)^{N/2}(\frac{1+r}{1-r})^d$ and $\log\frac{1+r}{1-r}=2a$. This is the first
claim. For the bound write $b_N(J)=\binom NJ2^{-N}$ and $h=2/(N+1)$. If $N=2m$ and $d\ge1$, then, as in the proof of
Lemma~\ref{fib:lem:UN}, $\log\frac{b_N(m+d)}{W_m}=-2\sum_{i=1}^d\artanh\alpha_i$ with $\alpha_i=h(i-\frac12)$, the midpoint of
$[h(i-1),hi]$. Since $\artanh$ is convex on $[0,1)$, the Hermite-Hadamard inequality gives
$\artanh\alpha_i\le\frac1h\int_{h(i-1)}^{hi}\artanh$, so $\sum_{i\le d}\artanh\alpha_i\le\frac1hI(hd)$ and
$\log b_N(\frac N2+d)\ge\log W_m-(N+1)I(\frac{2d}{N+1})$. If $N=2m+1$, $J=m+1+d'$ and $d=d'+\frac12$, the nodes are $\alpha_i=hi$, the
midpoints of $[h(i-\frac12),h(i+\frac12)]$, the sum of integrals runs from $h/2$ to $hd$, and $b_N(m+1)=W_{m+1}$; this gives
$\log b_N(J)\ge\log c_N-(N+1)I(\frac{2d}{N+1})$ with the factor $e^{(N+1)I(h/2)}$ in $c_N$. Negative $d$ follow from $b_N(J)=b_N(N-J)$
and the evenness of $I$. So $\log b_N(J)\ge\log c_N-(N+1)I(x)$ with $x=\frac{2d}{N+1}$, in both parities. By the first claim,
$\log B_N=\frac N2\log(1-r^2)+\log b_N+2ad$, and $2ad-(N+1)I(x)=(N+1)(ax-I(x))=-(N+1)\Psi_r(x-r)-\frac{N+1}2\log(1-r^2)$, because
$ar-I(r)=-\frac12\log(1-r^2)$. Finally $x-r=2(J-J_m)/(N+1)=z$.
\end{proof}

\begin{lemma}[two Gaussian bounds]\label{fib:lem:gauss}
Let $0\le r<1$.

\textup{(a)} If $0<Z<1$, then $\Psi_r(z)\le\frac{z^2}{2(1-r^2)}+\frac{z^4}{12(1-Z^2)(1-r^2)}$ for $-Z\le z\le0$.
\textup{(b)} If $Z>0$, $r+Z<1$ and $D_Z=(1-(r+Z)^2)(1-r^2)$, then $\Psi_r(z)\le\frac{z^2}{2(1-r^2)}+\frac{rz^3/3+z^4/12}{D_Z}$ for $0\le z\le Z$.
\textup{(c)} With $A_0=U_N(1-r^2)^{N/2}e^{(N+1)a^2/2}$ and $y_U=(a-r)(N+1)/2\ge0$: $B_N(J)\le A_0e^{-2(y-y_U)^2/(N+1)}$ for every integer $J$.
\textup{(d)} $0\le a-r=\sum_{i\ge1}\frac{r^{2i+1}}{2i+1}\le\frac{r^3}{3(1-r^2)}$.

\end{lemma}

\begin{proof}
$\Psi_r''(s)-\frac1{1-r^2}=\frac{s(2r+s)}{(1-(r+s)^2)(1-r^2)}$. (a) For $s\in[-2r,0]$ this is $\le0$. For $s\in[-Z,-2r]$ (if $Z>2r$) we have
$s(2r+s)\le s^2$ and $|r+s|\le Z-r\le Z$, so it is at most $\frac{s^2}{(1-Z^2)(1-r^2)}$. Integrate twice
($\int_z^0(s-z)s^2\,ds=\frac{z^4}{12}$). (b) For $s\in[0,Z]$, $|r+s|\le r+Z$, so the difference is at most $\frac{s(2r+s)}{D_Z}$; integrate
twice ($\int_0^z(z-s)s\,ds=\frac{z^3}6$, $\int_0^z(z-s)s^2\,ds=\frac{z^4}{12}$). (c) By Lemma~\ref{fib:lem:UN} and Lemma~\ref{fib:lem:LB},
$B_N(J)\le(1-r^2)^{N/2}U_Ne^{2ad-2d^2/(N+1)}$, and $2ad-\frac{2d^2}{N+1}=\frac{N+1}2a^2-\frac2{N+1}(d-d_0)^2$ with $d_0=\frac{a(N+1)}2$; and
$d-d_0=y-y_U$. (d) $\sum_{i\ge1}\frac{r^{2i+1}}{2i+1}\le\frac13\sum_{i\ge1}r^{2i+1}$.
\end{proof}

Put $s_N=\sqrt{vN}$ and
\[
\varepsilon=\frac{r\sqrt N}{\sqrt v(N+1)},\quad c_0=\frac{r^2}{2v(N+1)},\quad Z_L=\frac{2(s_N+r)}{N+1},\quad Z_R=\frac{2s_N}{N+1},
\]
\[
E_4=\max\Big(\frac{(N+1)Z_L^4}{12(1-Z_L^2)(1-r^2)},\ \frac{(N+1)Z_R^4}{12(1-(r+Z_R)^2)(1-r^2)}\Big),\qquad
E_4^{(0)}=\frac{N^2}{12(N+1)(N^2+N+1)},
\]
\[
\tilde y=r+y_U,\quad \varepsilon'=\frac{4\tilde ys_N}{N+1},\quad c_0'=\frac{2\tilde y^2}{N+1},\quad c_r=\frac{a^2+\log(1-r^2)}2 .
\]
$E_4^{(0)}$ is the value of $E_4$ at $r=0$, and $c_r\ge0$ because $\frac d{dr}(a^2+\log(1-r^2))=\frac{2(a-r)}{1-r^2}\ge0$.

\begin{lemma}[exponents]\label{fib:lem:expo}
(a) $(N+1)\frac{z^2}{2(1-r^2)}=\frac\rho2w-\varepsilon x+c_0$ and $\frac{2(y-y_U)^2}{N+1}=\frac{(1-r^2)\rho}2w-\varepsilon'x+c_0'$.\\
(b) For $N\ge10$ and $w\le1$: $\frac{\sqrt{1-r^2}}{c_N}B_N(J)\ge\exp\big(-\frac\rho2w-\varepsilon|x|-c_0-E_4\big)$.\\
(c) For every $J$: $\frac{\sqrt{1-r^2}}{c_N}B_N(J)\le\frac{U_N}{c_N}\exp\big((N+1)c_r-c_0'-\frac{(1-r^2)\rho}2w+\varepsilon'|x|\big)$.
\end{lemma}

\begin{proof}
(a) Substitute $u=s_Nx$, using $vN=s_N^2$ and $2v(N+1)=\frac{(1-r^2)(N+1)}2$. (b) Here $|u|\le s_N=\sqrt{q(1-q)N}\le\frac{\sqrt N}2$, so $J\ge\frac N2-\frac{\sqrt N}2\ge0$ and
$N-J\ge(1-q)(N-1)+\frac12-\sqrt{q(1-q)N}\ge\frac14(N+1-\sqrt{3N})>0$ for $q\in[\frac12,\frac34]$ (the middle expression
decreases in $q$ there); so
$0\le J\le N$ and Lemma~\ref{fib:lem:LB} applies. If $u\le r$ ($z\le0$), then $|z|\le Z_L$ because $|u|\le s_N$; Lemma~\ref{fib:lem:gauss}(a) with $Z=Z_L$ (which is $<1$, as
$Z_L\le\frac{\sqrt N+1}{N+1}\le0.379$), (a) and $-\varepsilon x\le\varepsilon|x|$ give the bound through Lemma~\ref{fib:lem:LB}.
If $u>r$ ($0<z\le Z_R$), Lemma~\ref{fib:lem:gauss}(b) gives the extra $\frac{(N+1)(rz^3/3+z^4/12)}{D_R}$, $D_R=(1-(r+Z_R)^2)(1-r^2)$. With
$z\le\frac{2u}{N+1}$ the cubic part is at most $\frac{8ru^3}{3(N+1)^2D_R}$, which is at most $2\varepsilon x=\frac{2ru}{v(N+1)}$ as soon as
$w\le\frac{12(1-(r+Z_R)^2)}{(1-r^2)\rho}$; this holds for $w\le1$, $N\ge10$, $r\le\frac12$ ($Z_R\le\frac{\sqrt{10}}{11}<0.29$, so the right side exceeds $6$). Then
$-\varepsilon x+(\text{cubic part})\le\varepsilon x$. (c) Lemma~\ref{fib:lem:gauss}(c), (a),
$\sqrt{1-r^2}(1-r^2)^{N/2}e^{(N+1)a^2/2}=e^{(N+1)c_r}$, and $\varepsilon'x\le\varepsilon'|x|$.
\end{proof}

\begin{lemma}[amplitudes]\label{fib:lem:amp}
For $N\ge2$: $c_N\ge L_Ne^{E_4^{(0)}}$.
\end{lemma}

\begin{proof}
If $N$ is even, $c_N=W_{N/2}\ge e^{\frac1{12N+1}-\frac1{3N}}(\frac{\pi N}2)^{-1/2}=L_Ne^{\frac1{12(N-1)}}$ by Lemma~\ref{fib:lem:wallis}, and
$\frac1{12(N-1)}\ge E_4^{(0)}$ since $N^2(N-1)\le(N+1)(N^2+N+1)$. If $N$ is odd and $m=\frac{N+1}2$, then
$W_m\ge e^{\frac1{12N+13}-\frac1{3(N+1)}}(\pi m)^{-1/2}$ and $(N+1)I(\frac1{N+1})\ge\frac1{2(N+1)}$ (as $I(x)\ge x^2/2$), so
$c_N\ge e^{\frac1{12N+13}+\frac1{6(N+1)}}(\pi m)^{-1/2}=L_Ne^{\frac1{12(N+1)}}$, and $\frac1{12(N+1)}\ge E_4^{(0)}$.
\end{proof}

\paragraph{Three ranges of the offset.}
\begin{proposition}[three ranges]\label{fib:prop:pieces}
Let $N\ge10$ and $r\le\frac12$. For every $j$:

\textup{(L)} if $w\le1$: $\Ff{}\ge H_L(w):=P_L(\gamma-w)e^{-\rho_Lw/2}+Bw$, with $\rho_L=\rho+\varepsilon$ and an amplitude $P_L>0$ such that
$\log\frac{Be^{\rho_L\gamma/2}}{P_L}\le\oT_L+\frac{\rho_L}2\Delta+e_L$, where $e_L=c_0+(E_4-E_4^{(0)})_++\varepsilon$;
\textup{(M)} if $1\le w\le\gamma$: $\Ff{}\ge Bw\ge B$;
\textup{(U)} if $w\ge\gamma$: $\Ff{}\ge H_U(w):=P_U(\gamma-w)e^{-\rho_Uw/2}+Bw$, where $\rho_U=(1-r^2)\rho-\varepsilon'/\sqrt\gamma$ and
$\omega_U:=\log\frac{Be^{\rho_U\gamma/2}}{P_U}=\oT_U+\frac{\rho_U}2\Delta-e_U$, with
$e_U=\frac{r^2\rho}2+\frac{\varepsilon'}2\big(\sqrt\gamma+\frac1{\sqrt\gamma}\big)+(N+1)c_r-c_0'$.

\end{proposition}

\begin{proof}
We use \eqref{fib:eq:Fform}. (L) Here $\gamma-w\ge0$. By Lemma~\ref{fib:lem:expo}(b), the tangent $|x|\le\frac{w+1}2$ and Lemma~\ref{fib:lem:amp},
$\Ff{}\ge H_L(w)$ with $P_L:=\frac{\mu^{3/2}c_N}{2\aG(N-1)}e^{-c_0-E_4-\varepsilon/2}\ge P^{\mathrm T}_Le^{-c_0-(E_4-E_4^{(0)})-\varepsilon/2}$; then
$\log\frac{Be^{\rho_L\gamma/2}}{P_L}\le\oT_L+\frac{\rho_L\gamma-\rho}2+c_0+(E_4-E_4^{(0)})+\frac\varepsilon2$ and
$\frac{\rho_L\gamma-\rho}2=\frac{\rho_L}2\Delta+\frac\varepsilon2$. (M) $\gamma-w\ge0$. (U) Here $\gamma-w\le0$, so an upper bound of $B_N$ is needed:
Lemma~\ref{fib:lem:expo}(c) and the tangent $|x|\le\frac{w+\gamma}{2\sqrt\gamma}$ give $\Ff{}\ge H_U(w)$ with
$P_U:=P^{\mathrm T}_Ue^{(N+1)c_r-c_0'+\varepsilon'\sqrt\gamma/2}$, and the formula for $\omega_U$ is direct. The positivity of $\rho_U$
where it is used is part of the checks.
\end{proof}

\paragraph{The remainders.}

\begin{proposition}[the range (L)]\label{fib:prop:L}
Let $N\ge10$, $r\le\frac12$ and $t>0$, and put
\[
D_L=\Big(2(N+1)\,e_L\,s\big(\oT_L+e_L\big)_+-N\Delta(1-e_L)\Big)_+ .
\]
Then $\mu(\Ff(M,j)-\psi(t))\ge\min(2,\,T(t,M)-D_L)$ for every $j$ with $w\le1$, and $\mu(\Ff(M,j)-\psi(t))\ge2$ for every $j$ with
$1\le w\le\gamma$.
\end{proposition}

\begin{proof}
For $1\le w\le\gamma$: $\mu(Bw-\psi)\ge\mu(B-\psi)=2+\frac92\frac{(M-\mu)^2}\mu\ge2$. For $w\le1$ let $\omega_L=\log(Be^{\rho_L\gamma/2}/P_L)$.
If $\omega_L\le0$, then $H_L$ decreases on $[0,1]$ (Corollary~\ref{fib:lem:min}: $w_*\ge\gamma\ge1$) and $H_L(1)\ge B$. Otherwise
Corollary~\ref{fib:lem:min}, Lemma~\ref{fib:lem:split} with $\omega_0=\oT_L$ and $\tilde e=\omega_L-\oT_L-\frac{\rho_L}2\Delta\le e_L$,
$\sigma_L=G'(\omega_L)\in(0,1)$, and $\rho_L\ge\rho$ give
\[
\min_{[0,1]}H_L\ \ge\ B\big(1-\tfrac2\rho G(\oT_L)\big)+B\big(\Delta(1-\sigma_L)-\tfrac2{\rho_L}e_L\sigma_L\big).
\]
By Proposition~\ref{fib:prop:template}, $\mu$ times the first part minus $\mu\psi$ is at least $T(t,M)$. For the second part,
$s(\omega_L)\le\bar s:=s(\oT_L+e_L)+\frac{\rho_L}2\Delta$ (Lemma~\ref{fib:lem:G}(a)), so, since $s\mapsto(\Delta-cs)/(1+s)$ decreases for $c\ge0$,
\[
\mu B\big(\Delta(1-\sigma_L)-\tfrac2{\rho_L}e_L\sigma_L\big)\ \ge\ \frac{N\Delta(1-e_L)-\frac{2N}{\rho_L}e_L\,s(\oT_L+e_L)}{1+\bar s}\ \ge\ -D_L ,
\]
using $\frac{2N}{\rho_L}\le2(N+1)$, and that a negative numerator is divided by $1+\bar s\ge1$.
\end{proof}

\begin{lemma}[the range (U) keeps both copies of the bonus]\label{fib:lem:Ubonus}
Let $N\ge10$, $r\le\frac12$ and $t>0$.

\textup{(a)} If $\omega_U\ge0$, then $\mu(\min_{w\ge\gamma}H_U-\psi(t))\ge2$.
\textup{(b)} If $\omega_U<0$, then, with $\sigma_U=G'(\omega_U)<0$,
\[
\mu\big(\min_{w\ge\gamma}H_U-\psi(t)\big)\ \ge\ T(t,M)-2N\big(\tfrac1{\rho_U}-\tfrac1\rho\big)G(\oT_U)+N\Delta-|\sigma_U|\Big(\tfrac{2N}{\rho_U}e_U-N\Delta\Big),
\]
and hence $\mu(\min_{w\ge\gamma}H_U-\psi(t))\ge T(t,M)-D_U'$ with
\[
D_U'=\Big(2N\big(\tfrac1{\rho_U}-\tfrac1\rho\big)G(\oT_U)-N\Delta\Big)_++|\sigma_U|\Big(\tfrac{2N}{\rho_U}e_U-N\Delta\Big)_+ .
\]

\end{lemma}

\begin{proof}
(a) If $\omega_U\ge0$, then $H_U$ increases on $[\gamma,\infty)$ (Corollary~\ref{fib:lem:min}: $w_*\le\gamma$), and
$\mu(B\gamma-\psi)\ge\mu(B-\psi)=2+\frac{9(M-\mu)^2}{2\mu}\ge2$. (b) Corollary~\ref{fib:lem:min} gives $\min H_U\ge B(\gamma-\frac2{\rho_U}G(\omega_U))$, and
Lemma~\ref{fib:lem:split} with $\omega_0=\oT_U$, $h=\frac{\rho_U}2\Delta-e_U$ gives $G(\omega_U)\le G(\oT_U)+\sigma_U(\frac{\rho_U}2\Delta-e_U)$. Hence
$\gamma-\frac2{\rho_U}G(\omega_U)\ge1-\frac2{\rho_U}G(\oT_U)+\Delta+|\sigma_U|\Delta-\frac2{\rho_U}|\sigma_U|e_U$. Multiply by $N=\mu B$, write
$\frac{2N}{\rho_U}=2(N+1)+2N(\frac1{\rho_U}-\frac1\rho)$, and use Proposition~\ref{fib:prop:template} for $N-\mu\psi-2(N+1)G(\oT_U)$. This is the first
inequality. For the second, split the bonus $N\Delta(1+|\sigma_U|)$ into its two copies: $N\Delta$ pays the first loss and $N\Delta|\sigma_U|$ the
second, which carries the same factor $|\sigma_U|$.
\end{proof}

\begin{remark}\label{fib:rem:r0}
At $r=0$: $\Delta=\varepsilon=c_0=0$, $E_4=E_4^{(0)}$, $\varepsilon'=c_0'=c_r=0$, $\rho_U=\rho$ and $e_U=0$, so $D_L=D_U'=0$, and
Proposition~\ref{fib:prop:L} and Lemma~\ref{fib:lem:Ubonus} reduce to the activity-one bound of Proposition~\ref{fib:prop:template}. In general the skew
and the off-centre mode of the binomial law on $N$ trials cost a remainder of first order in $r$, and the bonus, which is of
second order in $r$ but grows like $N^2$, pays for it. Keeping the second copy $N\Delta|\sigma_U|$ of the bonus in
Lemma~\ref{fib:lem:Ubonus} is what makes the closed-form bounds below sharp enough.
\end{remark}

\paragraph{Bounds in the variable $\zeta=r\sqrt N$.}
\begin{lemma}[monotone bounds]\label{fib:lem:mono}
Let $N\ge10$, $\zeta=r\sqrt N$ with $0\le\zeta\le\frac{\sqrt N}2$, $Z_0=\frac{\sqrt N}{N+1}$ and $\varphi=1+\frac{2\zeta}N$. Then:

\textup{(a)} $(N+1)e_L\le\frac{2\zeta^2}{N-\zeta^2}+\frac{2\zeta}{\sqrt{1-\zeta^2/N}}+\frac1{12}(\max(\psi_L,\psi_R)-1)$, where
$\psi_L=\frac{\varphi^4(1-Z_0^2)}{(1-Z_0^2\varphi^2)(1-\zeta^2/N)}$ and $\psi_R=\frac{1+r(r+2Z_0)/(1-(r+Z_0)^2)}{1-\zeta^2/N}$;
\textup{(b)} for every $N_a$ with $\zeta^2<N_a\le N$, the numbers $\bar\Delta=\frac{\zeta^2}{1-\zeta^2/N_a}$ and $\Delta_a=\zeta^2(1-\frac1{N_a})$ satisfy
$\Delta_a\le\Delta\le\bar\Delta$, and
\begin{gather*}
2(N+1)(\rho-\rho_U)\le\frac{2(N+1)\zeta^2}N+\frac{2\big(2\zeta+(N+1)\zeta^3/(3(N-\zeta^2))\big)}{\sqrt{1+\zeta^2(1-1/N_a)}},\\
\rho_U\ge\Big(1-\frac{\zeta^2}N\Big)\rho-\frac{2\zeta/(N+1)+\zeta^3/(3(N-\zeta^2))}{\sqrt{1+\zeta^2(1-1/N_a)}},
\end{gather*}
\[
Ne_U\le\frac{\zeta^2}2+\frac12\Big(2\zeta+\frac{\zeta^3}{3(1-\zeta^2/N)}\Big)\Big(\sqrt{1+\bar\Delta}+\frac1{\sqrt{1+\Delta_a}}\Big)+\frac{(N+1)\zeta^4}{12N(1-\zeta^2/N)^2}.
\]

\end{lemma}

\begin{proof}
(a) $(N+1)c_0=\frac{2r^2}{1-r^2}=\frac{2\zeta^2}{N-\zeta^2}$ and $(N+1)\varepsilon=\frac{2\zeta}{\sqrt{1-r^2}}$; $(N+1)E_4^{(0)}\le\frac1{12}$; and
$E_4/E_4^{(0)}\le\max(\psi_L,\psi_R)$: for the first entry of $E_4$, $s_N+r\le\frac{\sqrt N}2\varphi$, so $Z_L\le Z_0\varphi$, and $1-r^2=1-\zeta^2/N$;
for the second, $Z_R\le Z_0$ and $\frac{1-Z_0^2}{1-(r+Z_0)^2}=1+\frac{r(r+2Z_0)}{1-(r+Z_0)^2}$. (b) $s_N\le\frac{\sqrt N}2$; by
Lemma~\ref{fib:lem:gauss}(d), $(a-r)\sqrt N\le\frac{\zeta^3}{3(N-\zeta^2)}$, so $\varepsilon'\le\frac{2\zeta}{N+1}+(a-r)\sqrt N$ is at most
$\frac{2\zeta}{N+1}+\frac{\zeta^3}{3(N-\zeta^2)}$; $\frac1{\sqrt\gamma}\le\frac1{\sqrt{1+\Delta_a}}$ because $\Delta\ge\zeta^2\frac{N^2-N+1}{N^2}\ge\zeta^2(1-\frac1N)\ge\Delta_a$ (as $N\ge N_a$);
$\sqrt\gamma\le\sqrt{1+\bar\Delta}$ because $\Delta\le\frac{r^2N}{1-r^2}=\frac{\zeta^2}{1-\zeta^2/N}\le\bar\Delta$ (as $0<1-\zeta^2/N_a\le1-\zeta^2/N$); $\rho-\rho_U=r^2\rho+\varepsilon'/\sqrt\gamma$; and $c_r\le\frac{r^4}{12(1-r^2)^2}$, by integrating
$\frac{2(a-r)}{1-r^2}\le\frac{2r^3}{3(1-r^2)^2}$; finally $c_0'\ge0$ is dropped.
\end{proof}

\begin{lemma}[linearised coefficients]\label{fib:lem:coef}
Let $N\ge10$, $0\le r\le r_m\le\frac12$, $c=1-\frac1N$ and $\zeta=r\sqrt N\le\zeta_m$. Then

\textup{(a)} $(N+1)e_L\le\alpha_1\zeta+\alpha_2\zeta^2$, with $\alpha_1=\frac2{\sqrt{1-r_m^2}}+\frac1{12}\max(c_{L1},c_{R1})$,
$\alpha_2=\frac2{N(1-r_m^2)}+\frac1{12}\max(c_{L2},c_{R2})$, where $c_{L1}=\frac{4K_L}{N(1-r_m^2)}$, $c_{L2}=\frac{4K_L/N^2+1/N}{1-r_m^2}$,
$c_{R1}=\frac{2K_R}{(N+1)(1-r_m^2)}$, $c_{R2}=\frac{K_R+1}{N(1-r_m^2)}$, $K_L=1+\frac{\varphi_m^2}{1-Z_0^2\varphi_m^2}$, $\varphi_m=1+\frac{2r_m}{\sqrt N}$,
$K_R=\frac1{1-(r_m+Z_0)^2}$ and $Z_0=\frac{\sqrt N}{N+1}$;
\textup{(b)} $\rho_U\ge\rho_*:=(1-r_m^2)\rho-\frac2{(N+1)\sqrt c}-\frac{r_m^2}{3(1-r_m^2)\sqrt c}$;
\textup{(c)} $2N(\frac1{\rho_U}-\frac1\rho)\le\frac{4\zeta+\kappa_A\zeta^2}{\rho_*}$, with $\kappa_A=2+\frac{2(N+1)}{3N(1-r_m^2)\sqrt c}$;
\textup{(d)} $Ne_U\le E(\zeta):=2\zeta+\frac{\zeta^2}2+\frac{5\zeta^3}{6(1-r_m^2)}+\frac{(N+1)\zeta^4}{12N(1-r_m^2)^2}+\frac{\zeta^5}{12(1-r_m^2)^2}$;
\textup{(e)} $N\Delta\ge(N-1)\zeta^2$, and $N\Delta\ge\frac{(N-1)\zeta^2}{1-r_a^2}$ if $r\ge r_a$.

The coefficients $\alpha_1$, $\alpha_2$, $\kappa_A$ and those of $E$ are nonincreasing in $N$, and $\rho_*$ is nondecreasing in $N$.
\end{lemma}

\begin{proof}
(a) Lemma~\ref{fib:lem:mono}(a) with $\frac{2\zeta^2}{N-\zeta^2}\le\frac{2\zeta^2}{N(1-r_m^2)}$, $\frac{2\zeta}{\sqrt{1-r^2}}\le\frac{2\zeta}{\sqrt{1-r_m^2}}$, and the identities
$\psi_L-1=\frac{(\varphi^2-1)(1+\varphi^2/(1-Z_0^2\varphi^2))+r^2}{1-r^2}$, with $\varphi^2-1=\frac{4\zeta}N+\frac{4\zeta^2}{N^2}$ and $\varphi\le\varphi_m$, and
$\psi_R-1=\frac{(\zeta^2/N+2\zeta/(N+1))K+\zeta^2/N}{1-r^2}$, with $K=\frac1{1-(r+Z_0)^2}\le K_R$; so $\max(\psi_L,\psi_R)-1$ is at most the larger linear
coefficient times $\zeta$ plus the larger quadratic coefficient times $\zeta^2$. (b), (c) Lemma~\ref{fib:lem:mono}(b) with $N_a=N$ (admissible, as $\zeta^2=r^2N\le\frac N4<N$),
$2(N+1)r^2\rho=2\zeta^2$ and the two elementary bounds $\frac\zeta{\sqrt{1+c\zeta^2}}\le\frac1{\sqrt c}$ and $\frac{\zeta^3}{\sqrt{1+c\zeta^2}}\le\frac{\zeta^2}{\sqrt c}$;
for (c) also $2N(\frac1{\rho_U}-\frac1\rho)=\frac{2(N+1)(\rho-\rho_U)}{\rho_U}$. (d) Lemma~\ref{fib:lem:mono}(b) with $N_a=N$,
$\sqrt{1+\bar\Delta}\le1+\frac{\bar\Delta}2$, $\bar\Delta=\frac{\zeta^2}{1-r^2}$, and $\frac1{\sqrt{1+\Delta_a}}\le1$. (e) $N\Delta=\zeta^2\frac{N^2-N+1}{N-\zeta^2}=\zeta^2\frac{N^2-N+1}{N(1-r^2)}$.
The monotonicity is termwise.
\end{proof}

\begin{lemma}[closed forms of the remainders]\label{fib:lem:amgm}
For $a\ge0$: $\sup_{\zeta\ge0}(a\zeta-b\zeta^2)=\frac{a^2}{4b}$ if $b>0$, and on a bounded range $\sup_{0\le\zeta\le\zeta_m}(a\zeta-b\zeta^2)\le\min(\frac{a^2}{4b},a\zeta_m)$
if $b>0$ and $\le a\zeta_m-b\zeta_m^2$ if $b\le0$. Let $N\ge10$, $0\le r\le r_m\le\frac12$, $\zeta=r\sqrt N\le\zeta_m$, let $8\le M_0\le M=N-2$, and
$\bar h=\frac{3.85+10/M_0}{M_0+3}$. In the notation of Lemma~\ref{fib:lem:coef}, and bounding the powers $\zeta^i$, $i\ge3$, by $\zeta_m^{i-2}\zeta^2$:

\textup{(L)} $D_L\le\bar D_L:=\frac{(\sigma_0\alpha_1)^2}{\beta_L}$ whenever $\beta_L>0$, where $\sigma_0=s(\omega_t+\bar h)_+\le(s_t+\bar h)_+$,
$\beta_L=(N-1)(1-\bar e)-2\sigma_0\alpha_2-\frac{2\alpha_1^2+4\alpha_1\alpha_2\zeta_m+2\alpha_2^2\zeta_m^2}{N+1}$ and $\bar e=\frac{\alpha_1\zeta_m+\alpha_2\zeta_m^2}{N+1}$;
\textup{(U)} if $\omega_t\le\omega_U<0$ and the two denominators are positive:
$D_U'\le\bar D_U:=\frac{4\Ghat^2}{\rho_*^2(N-1-\kappa_A\Ghat/\rho_*)}+g_1\frac4{\rho_*^2\beta_e}$, where $\Ghat=\max(G(\omega_t),G(\omega_t+\bar h))\ge G(\oT_U)$,
$g_1=|G'(\omega_t)|\ge|\sigma_U|$ and $\beta_e=N-1-\frac2{\rho_*}\big(\frac12+\frac{5\zeta_m}{6(1-r_m^2)}+\frac{(N+1)\zeta_m^2}{12N(1-r_m^2)^2}+\frac{\zeta_m^3}{12(1-r_m^2)^2}\big)$.

\end{lemma}

\begin{proof}
The first claims are elementary. (L) By Lemma~\ref{fib:lem:k}, $\oT_L\le\omega_t+\bar h$, and $s(\omega+h)_+\le s(\omega)_++h$ for $h\ge0$
(Lemma~\ref{fib:lem:G}(a)); so $s(\oT_L+e_L)_+\le\sigma_0+e_L$. Insert Lemma~\ref{fib:lem:coef}(a), (e) and $e_L\le\bar e$ into the definition of $D_L$
(Proposition~\ref{fib:prop:L}) and expand: $D_L\le(2\sigma_0\alpha_1\zeta-\beta_L\zeta^2)_+$. (U) Insert Lemma~\ref{fib:lem:coef}(b) to (e) into the two brackets of
Lemma~\ref{fib:lem:Ubonus}(b). By Lemma~\ref{fib:lem:k}, $\oT_U\in[\omega_t,\omega_t+\bar h]$, and $G$ is convex, so $G(\oT_U)\le\Ghat$; and $|G'|$ decreases on
$(-\infty,0]$ (Lemma~\ref{fib:lem:G}(b)), so $|\sigma_U|=|G'(\omega_U)|\le|G'(\omega_t)|$.
\end{proof}

\begin{lemma}[large $\zeta$]\label{fib:lem:utail}
Let $r\in[0,\frac12]$, $N\ge10$ and $\zeta=r\sqrt N\ge6$. (a) For every $t>0$, $\omega_U\ge\omega_t+0.19\zeta^2$. In particular, if $t\ge t_a\ge0.1548$,
then $\omega_U\ge\frac52\log t_a+0.19\zeta^2>0$, so the range (U) gives at least $2$ (Lemma~\ref{fib:lem:Ubonus}(a)). (b) If $t\le1.605$, then $N\Delta(1-e_L)\ge128\zeta^2>5.12\zeta\ge2(N+1)e_L\,s(\oT_L+e_L)_+$, so $D_L=0$.
\end{lemma}

\begin{proof}
Here $N=\zeta^2/r^2\ge4\zeta^2\ge144$, and $\Delta\ge\frac{r^2(N-1)}{1-r^2}\ge\zeta^2(1-\frac1N)$. Write $\hat a=a-r=\artanh r-r\le\frac{r^3}{3(1-r^2)}$, so that
$\tilde y=r+\frac{\hat a}2(N+1)$. (a) $\oT_U\ge\frac52\log t$ ($k_U\ge3.55>0$, Lemma~\ref{fib:lem:k}). With
$\varepsilon'=2\tilde y\sqrt{(1-r^2)N}/(N+1)$ and $\sqrt\gamma\ge r(N-1)/\sqrt{(1-r^2)N}$:
$\frac{\varepsilon'}{\sqrt\gamma}\le\frac{2(1-r^2)N}{N^2-1}+\frac{r^2N}{3(N-1)}\le0.098$, so $\rho_U\ge0.75\cdot\frac{144}{145}-0.098\ge0.64$. For $e_U$ (drop
$-c_0'\le0$): $\frac{r^2\rho}2\le\frac18$; $\varepsilon'\sqrt\gamma=\frac{2\tilde y}{N+1}\sqrt{N+r^2(N-1)^2}\le\frac1{12}+\frac12+\frac\zeta9+\frac{\zeta^2}9$;
$\frac{\varepsilon'}{2\sqrt\gamma}\le0.049$; and $(N+1)c_r\le\frac{(N+1)r^4}{12(1-r^2)^2}\le0.0373\zeta^2$ (as $c_r=\int_0^r\frac{\hat a(x)}{1-x^2}dx$ with
$\hat a(x)=\artanh x-x$). So $e_U\le0.47+\frac\zeta{18}+0.093\zeta^2\le0.116\zeta^2$ for $\zeta\ge6$, and
$\omega_U\ge\omega_t+\frac{0.64}2(1-\frac1{144})\zeta^2-0.116\zeta^2\ge\omega_t+0.19\zeta^2$; for $t\ge t_a$ this is at least $\frac52\log t_a+0.19\zeta^2$, and
$0.19\cdot36=6.84>4.67>-\frac52\log0.1548$.
(b) $(N+1)\varepsilon=\frac{2\zeta}{\sqrt{1-r^2}}\le2.31\zeta$, $(N+1)c_0=\frac{2r^2}{1-r^2}\le\frac23$ and $(N+1)E_4\le0.18$ (both quartic terms, with
$Z_L\le\frac{\sqrt N+1}{N+1}$, $Z_R\le N^{-1/2}$, $r+Z_R\le\frac7{12}$), so $(N+1)e_L\le2.31\zeta+0.85\le2.46\zeta$ and $e_L\le\frac{2.46\zeta}{4\zeta^2+1}\le0.103$.
Then $\oT_L+e_L\le\frac52\log1.605+\frac{3.93}{145}+0.103\le1.32$ and $s(1.32)\le s(1.75)\le1.04$, so $2(N+1)e_Ls(\cdot)_+\le5.12\zeta$, while
$N\Delta(1-e_L)\ge143\zeta^2\cdot0.897\ge128\zeta^2$.
\end{proof}

\begin{lemma}[the bonus never lowers $\omega_U$ below $\omega_t$]\label{fib:lem:lamU}
Let $(r_m,N_{\min})$ be one of $(\frac14,10)$, $(0.28,10)$, $(0.3104,10)$, $(0.3444,10)$, $(0.3847,12)$, $(0.4,12)$, $(\frac14,13.4)$, $(0.28,20)$. For
every real $N\ge N_{\min}$ (with $t>0$ and $N\ge10$) and every $r\le r_m$: $\omega_U\ge\omega_t$. In particular the range (U) gives at least $2$
(in the units $\mu(\Ff{}-\psi)$) whenever $t\ge1$, and $g_1=|G'(\omega_t)|$ bounds $|\sigma_U|$ whenever $\omega_U<0$.
\end{lemma}

\begin{proof}
For $\zeta\ge6$ this is Lemma~\ref{fib:lem:utail}(a). For
$\zeta\le\zeta_m=\min(6,r_m\sqrt N)$: $\oT_U\ge\omega_t+\frac{3.55}{N+1}$ (Lemma~\ref{fib:lem:k}), and Lemma~\ref{fib:lem:coef}(b), (d), (e) give
$\omega_U-\omega_t\ge\frac{3.55}{N+1}+q_2\zeta^2-\frac{2\zeta}N$, where $q_2=\frac{\rho_*(N-1)}{2N}-\frac{E_2}N$ and
$E_2=\frac12+\frac{5\zeta_m}{6(1-r_m^2)}+\frac{(N+1)\zeta_m^2}{12N(1-r_m^2)^2}+\frac{\zeta_m^3}{12(1-r_m^2)^2}$. The quadratic in $\zeta$ is nonnegative if
$q_2>0$ and $(N+1)\le3.55N^2q_2$, that is, if
\[
f(N):=\frac{(N-1)\rho_*}2-E_2-\frac{N+1}{3.55N}\ \ge\ 0 .
\]
Lemmas~\ref{fib:chk:A4} and~\ref{fib:chk:A4p} verify $f(N_{\min})>0$, that $f'\ge\frac{\rho_*(N_{\min})}2-E_2'>0$ up to $N=(6/r_m)^2$, where $\zeta_m$ reaches $6$
(the derivative $E_2'$ is bounded termwise), and the regime of constant $\zeta_m=6$ beyond (where $N(N-1)\rho_*/2-NE_2-(N+1)/3.55$ is positive and
increasing from $(6/r_m)^2$ on).
\end{proof}

\paragraph{Large $t$.}
For $0<r_m\le\frac12$ and $\zeta\ge0$, $N\ge10$, let
\[
\bar X(\zeta,N)=\frac{(N+1)\zeta^4}{12N^2(1-r_m^2)^2}+\frac{\zeta\sqrt{1+\zeta^2}}{N+1}\Big(1+\frac{\zeta^2(N+1)}{6N(1-r_m^2)}\Big)-\frac{\rho_*}2\big(1+\zeta^2c\big),
\qquad c=1-\frac1N ,
\]
with $\rho_*=\rho_*(N,r_m)$ of Lemma~\ref{fib:lem:coef}(b).

\begin{lemma}[large $t$]\label{fib:lem:larget}
Let $0<r_m\le\frac12$ and $N_*\ge10$, and suppose that $\bar X(\zeta,N_*)\le\log(2.57\rho_*(N_*))$ for $0\le\zeta\le r_m\sqrt{N_*}$, that
$\bar X(\zeta,\zeta^2/r_m^2)\le\log(2.57\rho_*(\zeta^2/r_m^2))$ for $r_m\sqrt{N_*}\le\zeta\le6$, and that $2a_U(N)1.6^{-3/2}e^{-1/2}\le0.6225$ for $N\ge N_*$.
Then $\Ff(M,j)\ge0>\psi(t)$ for every $j$, every $r\le r_m$ and every fiber with $t\ge1.6$ and $N\ge N_*$.
\end{lemma}

\begin{proof}
$\psi(t)<0$ for $t\ge1.6$. On the ranges (L) and (M) of Proposition~\ref{fib:prop:pieces} the kernel term of \eqref{fib:eq:Fform} is nonnegative
(on (L), $w\le1\le\gamma$), so $\Ff{}\ge Bw\ge0$. On (U):
$H_U(w)\ge B\gamma-P_U\max_{w\ge\gamma}(w-\gamma)e^{-\rho_Uw/2}=B\gamma-\frac{2P_U}{\rho_U}e^{-1-\rho_U\gamma/2}$. With $P^{\mathrm T}_U=e^{1/2}B^{-3/2}a_U(N)$
(Lemma~\ref{fib:lem:k}) and $B\ge1.6$,
\[
\frac{2P_U}{\rho_U}e^{-1-\rho_U\gamma/2}\le\frac{2a_U(N)1.6^{-3/2}e^{-1/2}}{\rho_U}e^{X}\le\frac{0.6225}{\rho_U}e^{X},\qquad
X=(N+1)c_r+\frac{\varepsilon'\sqrt\gamma}2-\frac{\rho_U\gamma}2
\]
(drop $-c_0'\le0$). This is at most $1.6\le B\gamma$ as soon as $X\le\log(2.57\rho_U)$, since $1.6/0.6225>2.57$. In the variable $\zeta$:
$(N+1)c_r\le\frac{(N+1)\zeta^4}{12N^2(1-r_m^2)^2}$; $\frac{\varepsilon'\sqrt\gamma}2=\frac{\tilde yR}{N+1}$ with $R=\sqrt{N+r^2(N-1)^2}\le\sqrt N\sqrt{1+\zeta^2}$ and
$\tilde y\le\frac\zeta{\sqrt N}(1+\frac{r^2(N+1)}{6(1-r^2)})$, so it is at most $\frac{\zeta\sqrt{1+\zeta^2}}{N+1}(1+\frac{\zeta^2(N+1)}{6N(1-r_m^2)})$; and
$\frac{\rho_U\gamma}2\ge\frac{\rho_*}2(1+\zeta^2c)$, as $\Delta\ge\zeta^2c$. So $X\le\bar X(\zeta,N)$, and $\log(2.57\rho_U)\ge\log(2.57\rho_*)$. Each term of
$\bar X(\zeta,N)$ is nonincreasing in $N$ and $\log(2.57\rho_*)$ is nondecreasing; a fiber with $r\le r_m$ has $N\ge\max(N_*,\zeta^2/r_m^2)$; so the
hypotheses give $X\le\log(2.57\rho_U)$ for $\zeta\le6$. For $\zeta\ge6$, $\omega_U>0$ by Lemma~\ref{fib:lem:utail}(a), and then $\min_{w\ge\gamma}H_U=B\gamma\ge0$.
\end{proof}

\paragraph{The fiber inequality on the central range.}
\begin{theorem}[the fiber inequality on the central range]\label{fib:thm:env}
Let $|2q-1|\le\frac14$. For every real $\mu\ge19$, every integer $M$ with $t=M/\mu\ge0.6$ and every integer $j$: $\Ff(M,j)\ge\psi(t)$.
\end{theorem}

\begin{proof}
By Lemma~\ref{fib:lem:sym} we may take $q\ge\frac12$, so $r=2q-1\le\frac14$. Here $M\ge0.6\mu\ge11.4$, so $M\ge12$ and $N\ge14$.

\emph{Fibers with $t\ge1.6$.} Then $N\ge1.6\cdot19+2=32.4$, and Lemma~\ref{fib:lem:larget} with $r_m=\frac14$, $N_*=32.4$ applies by
Lemma~\ref{fib:chk:A5}.

\emph{Fibers with $0.6\le t\le1.6$.} By Propositions~\ref{fib:prop:pieces}, \ref{fib:prop:L} and Lemma~\ref{fib:lem:Ubonus},
$\mu(\Ff{}-\psi)\ge\min(2,\,T(t,M)-D_L,\,T(t,M)-D_U')$, where the last entry is needed only if $\omega_U<0$. Cover $s\in[-0.528,0.67]$ (that is,
$t\in[t(-0.528),t(0.67)]\supset[0.5996,1.605]$) by the nine cells $[s_a,s_b]$ of Lemma~\ref{fib:chk:A2}, and on a cell put $M_0=19\,t(s_a)\le M$ and $N_0=M_0+2$. Then
$T(t,M)\ge T(t,M_0)$, because $\Sl{}\ge0$ (Lemma~\ref{fib:lem:S}), the interval $[3.55,3.85+\frac{10}M]$ shrinks as $M$ grows and the bracket in $T$ is
concave in $\mathsf k$, and $\frac{\mathsf k^2G_2}{M+3}$ decreases in $M$. If $\zeta=r\sqrt N\ge6$, Lemma~\ref{fib:lem:utail} gives $D_L=0$ and $\omega_U>0$. If $\zeta\le6$,
Lemma~\ref{fib:lem:amgm}, applied at the actual $N$ with $\zeta_m=\min(6,r_m\sqrt N)$, and Lemma~\ref{fib:lem:lamU} with $(r_m,N_{\min})=(\frac14,13.4)$ (the
fibers have $N\ge14$) bound $D_L$ and $D_U'$ by $\bar D_L(N)$ and $\bar D_U(N)$; these bounds are nonincreasing in $N$ (Lemma~\ref{fib:lem:coef}; $\beta_L$ termwise; $\beta_e$ nondecreasing by
Lemma~\ref{fib:chk:A3} from $N=13.4$ and by Lemma~\ref{fib:chk:A3p} for the first cell, whose $N_0=19\,t(-0.528)+2=13.3921\ldots$, up to $(6/r_m)^2$, and termwise beyond), so
their values at $N_0$ bound them for every $N\ge N_0$. The bound $\bar D_U$ is needed only for $t<1$: for $t\ge1$, $\omega_U\ge\omega_t\ge0$ by
Lemma~\ref{fib:lem:lamU}. So it suffices that, on each cell,
\[
\mathcal E(s):=M_0\,s^2\varphi(s)\,[s\ge0]+2-\frac{6s^2}{1+s}-\frac{2\mathsf ks}{1+s}-\frac{\mathsf k^2}{(1+s)(2+s)(M_0+3)}-\max(\bar D_L,\bar D_U)\ \ge\ 0
\]
for $\mathsf k\in\{3.55,\,3.85+\frac{10}{M_0}\}$, where $\bar D_U$ is omitted for $s\ge0$, and where for $s<0$ the term $M_0\Sl{}\ge0$ is dropped and for $s\ge0$
it is bounded below by $M_0s^2\varphi(s)$ (Lemma~\ref{fib:lem:S}). This is Lemma~\ref{fib:chk:A2}.
\end{proof}

\begin{remark}[a variant]\label{fib:rem:Aii}
The same proof gives $\Ff(M,j)\ge\psi(t)$ for $|2q-1|\le0.28$, every real $\mu\ge30$, every $M$ with $t\ge0.6$ and every $j$, with the nine cells
evaluated at $\mu_0=30$, $r_m=0.28$, Lemma~\ref{fib:lem:lamU} with $(0.28,20)$, and the monotonicity of $\beta_e$ from $N=20$ and, for the first cell,
which starts at $N_0=30\,t(-0.528)+2=19.98767\ldots$, from that point (Lemma~\ref{fib:chk:A3p}, margin $0.718$). The other checks of this variant ($23$ lines
of the supplementary code, all passing; Appendix~\ref{apx:fiber}) are not counted there, and the variant is not used in this paper.
\end{remark}

% ---------------------------------------------------------------------------------------------------------------
\subsubsection{The spike fibers and the three kinds of subintervals}\label{fib:sec:novalley}

\paragraph{The spike term and the averaging margin.}
Fix a subinterval $\iota$ with its inputs $\theta=\theta(\iota)$, $D=D(\iota)$, $\ell_0=\ell_0(\iota)$, and the notation of
Section~\ref{fib:sec:inputs}. For reals $b\ge0$ and $\beta'\ge0$ let $g(t)=b+\beta'(t-1)$ (a \emph{target line}). The fibers with $M$ below a
cut $t_s\mu$ are the \emph{spike fibers}; their deficit is
\[
\Def_g(M,\mu)=\Big(\psi\big(\tfrac M\mu\big)+g\big(\tfrac M\mu\big)-\min_j\Ff(M,j)\Big)_+,
\]
and $\Def=\Def_0$ for $b=\beta'=0$. Since $\beta'\ge0$, $g(t)\le b$ for $t<1$, so $\Def_g\le\Def+b$ on fibers with $t<1$.

\begin{lemma}[the averaging margin]\label{fib:lem:spike}
Let $q$ satisfy $q/(1-q)\in\Ihat_\iota$, let $\mu>0$ and $0<t_s\le1$, and suppose that $\Ff(M,j)\ge\psi(t)+g(t)$ for every $j$ and every
$M\ge t_s\mu$. Put
\[
Z_s(\mu)=\max_{M<t_s\mu}(1-q)^{-M}e^{-\ell_0\mu}\,\Def_g(M,\mu).
\]
If
\begin{equation}\label{fib:eq:row}
1+b-\theta-\frac{9D}{2\mu}-\frac{r_{\rm hi}^2}{(1-r_{\rm hi}^2)\mu}-Z_s(\mu)>0 ,
\end{equation}
then no probability mixture with $\E M=\mu$, $\E\delta=0$, $\E\delta^2\le\theta q(1-q)\mu$, $\Var M\le D\mu$ and $\E(1-q)^M\le e^{-\ell_0\mu}$ has a
valley.
\end{lemma}

\begin{proof}
$\psi(t)+g(t)=(1+b)+(1+\beta')(t-1)-\frac92(t-1)^2$. So \eqref{fib:eq:FL} holds with $\alpha=1+b$, $\beta=1+\beta'$, $\Gamma=\frac92$ and
$Z'=Z_s(\mu)e^{\ell_0\mu}$: for $M\ge t_s\mu$ by hypothesis, and for $M<t_s\mu$ because $\Ff{}\ge\psi+g-\Def_g\ge\psi+g-Z'(1-q)^M$.
Proposition~\ref{fib:prop:avg} concludes, since $\frac{r^2}{1-r^2}\le\frac{r_{\rm hi}^2}{1-r_{\rm hi}^2}$.
\end{proof}

The coefficient $\beta=1+\beta'$ does not enter \eqref{fib:eq:row}; the sign of $\beta'$ matters only through $\Def_g\le\Def+b$ on the spike fibers.

\paragraph{Spike pricing.}

\begin{lemma}[the range (L) on spike fibers]\label{fib:lem:defL}
Let $N\ge10$, $r\le\frac12$ and $t\le0.69$. For every $j$ with $w\le1$,
\[
\Ff(M,j)\ \ge\ \psi(t)-\Big(-\tfrac92(t-1)^2+\tfrac{2t}M(N+1)e_L\Big)_+ .
\]
\end{lemma}

\begin{proof}
By Lemma~\ref{fib:lem:k}, $\oT_L\le\frac52\log0.69+\frac{3.85+10/8}{11}<0$. Let $\omega_L$ be as in the proof of Proposition~\ref{fib:prop:L}. If
$\omega_L\le0$, then $\min_{[0,1]}H_L=H_L(1)\ge B>\psi$. If $\omega_L>0$, then $\min H_L\ge B(\gamma-\frac2{\rho_L}G(\omega_L))\ge B(\gamma-\frac2{\rho_L}\omega_L)$
(Lemma~\ref{fib:lem:G}(d)), and $\omega_L\le\oT_L+\frac{\rho_L}2\Delta+e_L$ with $\oT_L<0$, so $\min H_L\ge B(1-\frac2{\rho_L}e_L)$. With
$\psi(t)-B=-\frac92(t-1)^2-\frac2\mu$, $\frac1{\rho_L}\le\frac{N+1}N$ and $\frac{2B(N+1)}N=\frac{2t(N+1)}M$ this gives the claim.
\end{proof}

\begin{lemma}[small fibers]\label{fib:lem:KM}
For $M=0,\dots,7$, every $j$ and every $q\in[\frac14,\frac34]$: $\kappa_w(M,j)\ge-K_M$, with
$(K_0,\allowbreak \dots,\allowbreak K_7)=(0.255,\allowbreak \,\allowbreak 0.136,\allowbreak \,\allowbreak 0.0795,\allowbreak \,\allowbreak 0.0488,\allowbreak \,\allowbreak 0.0322,\allowbreak \,\allowbreak 0.0258,\allowbreak \,\allowbreak 0.0202,\allowbreak \,\allowbreak 0.0162)$ (Lemma~\ref{fib:chk:B1}). Hence, for $\mu\ge\mu_0\ge19$, $0\le b\le1$
and $q$ with $q/(1-q)\in\Ihat_\iota$,
\begin{gather*}
(1-q)^{-M}e^{-\ell_0\mu}\big(\psi(\tfrac M\mu)+b-\min_j\Ff(M,j)\big)_+\ \le\ \big(A_M\mu_0^{3/2}+\psi(\tfrac M{\mu_0})+b\big)_+(1-q_b)^{-M}e^{-\ell_0\mu_0},\\
A_M=\frac{K_M}{4\aG\sqrt{v_{\min}}},
\end{gather*}
provided $\ell_0\big(A_M\mu_0^{3/2}+\psi(M/\mu_0)+b\big)\ge\frac32A_M\mu_0^{1/2}$; if $A_M\mu_0^{3/2}+\psi(M/\mu_0)+b\le0$, the bound
$A_M\mu_0^{3/2}(1-q_b)^{-M}e^{-\ell_0\mu_0}$ holds provided $\ell_0\mu_0\ge\frac32$.
\end{lemma}

\begin{proof}
$\Ff(M,j)\ge\frac{\kappa_w\mu^{3/2}}{4\sqrt v\aG}\ge-A_M\mu^{3/2}$, since the penalty is nonnegative and $v\ge v_{\min}$; $\psi$ increases on $[0,1]$ and
$\frac M\mu\le\frac M{\mu_0}<1$; and $(1-q)^{-M}\le(1-q_b)^{-M}$. Finally $\frac d{d\mu}[(A\mu^{3/2}+c)e^{-\ell_0\mu}]\le0$ if and only if
$\frac32A\mu^{1/2}\le\ell_0(A\mu^{3/2}+c)$, a condition that improves with $\mu$ once $2\ell_0\mu\ge1$; here $\ell_0\mu_0\ge\frac32$, because
$\psi(\frac M{\mu_0})+b\le\psi(\frac7{19})+1<0$, so the proviso gives $\ell_0A_M\mu_0^{3/2}\ge\frac32A_M\mu_0^{1/2}$. In the second case
$(A_M\mu^{3/2}+\psi(\frac M\mu)+b)_+\le A_M\mu^{3/2}$, and $A_M\mu^{3/2}e^{-\ell_0\mu}$ decreases for $\mu\ge\frac3{2\ell_0}$.
\end{proof}

For $M>0$ and $t>0$ let
\[
\Phi(t;M)=t\,(-\Sl(t))_++\frac tM\Big(6G_0-2+\bar D_U+\max_{\mathsf k\in\{3.55,\ 3.85+10/M\}}\big(2\mathsf kG_1+\tfrac{\mathsf k^2G_2}{M+3}\big)\Big)_+ ,
\]
where $\bar D_U$ is the bound of Lemma~\ref{fib:lem:amgm}(U) at $N=M+2$ (in its form for $\zeta\le\zeta_m$, with the bounded-range suprema).

\begin{lemma}[the spike deficit along rays]\label{fib:lem:Phi}
Let $r\le r_m$, $0.1548\le t\le0.69$, $M\ge8$, $N=M+2\ge N_{\min}$, where $\omega_U\ge\omega_t$ holds for $r\le r_m$ and $N\ge N_{\min}$ (for instance by
Lemma~\ref{fib:lem:lamU}), and suppose that the range (L) contributes no deficit, that is $\frac{2t}M(N+1)e_L\le\frac92(t-1)^2$. Then
$\Def(M,\mu)\le\Phi(t;M)$, and $\Phi$ is nonincreasing in $M$ along the ray of fixed $t$. On an $M$-block $[M_1,M_2]$ one may take, in
$\bar D_U$: $\Ghat=\max\big(G(\omega_t+\frac{3.55}{M_2+3}),G(\omega_t+\bar h(M_1))\big)$ with $\bar h(M_1)=\frac{3.85+10/M_1}{M_1+3}$;
$\zeta_m=\min(6,r_m\sqrt{M_2+2})$; for $\rho_*$ any lower bound of $\rho_U$ valid for $\zeta\le\zeta_m$ and $N\ge M_1+2$, for instance
\[
\rho_*^{\rm blk}=\frac{(1-r_m^2)N_1}{N_1+1}-\frac{2\zeta_m/(N_1+1)+\zeta_m^3/(3N_1(1-r_m^2))}{\sqrt{1+(1-1/N_1)\zeta_m^2}},\qquad N_1=M_1+2 ;
\]
for $\kappa_A$ any upper bound valid there, for instance $\kappa_A^{\rm blk}=2+\frac{2(N_1+1)}{3N_1(1-r_m^2)}\frac{\zeta_m}{\sqrt{1+(1-1/N_1)\zeta_m^2}}$; and
$g_1=|G'(\omega_t+\frac{3.55}{M_2+3}-\frac1{N_1^2q_2})|$ when $q_2>0$, with $q_2$ of the proof of Lemma~\ref{fib:lem:lamU} at $N_1$ (with these $\rho_*$ and
$\zeta_m$). On the last block $[M_1,\infty)$ the all-$\zeta$ bound of Lemma~\ref{fib:lem:amgm} is evaluated at $N=M_1+2$ and used for all larger $N$,
which needs $\beta_e$ nondecreasing there.
\end{lemma}

\begin{proof}
On the range (M) of Proposition~\ref{fib:prop:pieces}, $\Ff{}\ge B>\psi$. On the range (L) there is no deficit, by Lemma~\ref{fib:lem:defL} and the
hypothesis. On the range (U): if $\omega_U\ge0$ there is no deficit (Lemma~\ref{fib:lem:Ubonus}(a)); if $\zeta\ge6$, then $\omega_U>0$ by
Lemma~\ref{fib:lem:utail}(a), since $t\ge0.1548$; otherwise
$\omega_t\le\omega_U<0$, Lemma~\ref{fib:lem:Ubonus}(b), which holds for every $t>0$, and Lemma~\ref{fib:lem:amgm} give
$\mu(\min H_U-\psi)\ge T(t,M)-\bar D_U$, and
$-T\le M(-\Sl{})_+-2+6G_0+\max_{\mathsf k}(2\mathsf kG_1+\frac{\mathsf k^2G_2}{M+3})$; divide by $\mu=M/t$. The bracket is nonincreasing in $M$ ($\bar D_U$ and
$\frac{\mathsf k^2}{M+3}$ decrease, and so does the larger end of the $\mathsf k$-interval), and $t(-\Sl{})_+$ does not depend on $M$. The block refinements use
$\oT_U\in[\omega_t+\frac{3.55}{M+3},\omega_t+\bar h(M)]$ (Lemma~\ref{fib:lem:k}), the lower bound $\omega_U-\omega_t\ge\frac{3.55}{N+1}+q_2\zeta^2-\frac{2\zeta}N\ge\frac{3.55}{N+1}-\frac1{N^2q_2}$
of the proof of Lemma~\ref{fib:lem:lamU} together with the monotonicity of $N^2q_2$ in $N$, and the monotonicity in $\zeta$ and in $N$ of the
expressions of Lemma~\ref{fib:lem:mono}(b) with $N_a=N$ (admissible, as $\zeta^2\le r_m^2N<N$): $\frac{2\zeta/(N+1)+\zeta^3/(3N(1-r_m^2))}{\sqrt{1+(1-1/N)\zeta^2}}$ increases in $\zeta$ and decreases in $N$, and so does
$\frac{\zeta}{\sqrt{1+(1-1/N)\zeta^2}}$. In these block forms the suprema over $\zeta\le\zeta_m$ are taken with the bounded-range form of
Lemma~\ref{fib:lem:amgm}, with $N_1-1$ in place of $N-1$.
\end{proof}

\begin{lemma}[very small $t$]\label{fib:lem:Zc}
Let $b\ge0$, $\mu\ge\mu_0$ and $m_c=\max(\mu_0,8/0.22)$, and suppose $(\ell_0-0.22L_b)m_c\ge\frac32$. For every $q$ with $q/(1-q)\in\Ihat_\iota$
and every $M\ge8$ with $M<0.22\mu$:
\[
(1-q)^{-M}e^{-\ell_0\mu}\big(\psi(\tfrac M\mu)+b-\min_j\Ff(M,j)\big)_+\ \le\ \Big(\frac{m_c^{3/2}}{16\sqrt{v_{\min}}\aG}+b\Big)e^{-(\ell_0-0.22L_b)m_c}.
\]
\end{lemma}

\begin{proof}
By Lemma~\ref{apC:lem:atoms}(a), $\kappa_w\ge-\max(\wL,\wR)\ge-\frac14$ on $[\frac14,\frac34]$, so $\Ff{}\ge-\frac{\mu^{3/2}}{16\sqrt v\aG}$; and $\psi(t)<0$ for
$t<0.22$. The weight is at most $e^{L_bM-\ell_0\mu}\le e^{-(\ell_0-0.22L_b)\mu}$, and $\mu=M/t\ge8/0.22$. The function
$(A\mu^{3/2}+b)e^{-c\mu}$ decreases for $\mu\ge\frac3{2c}$.
\end{proof}

\paragraph{The central subintervals.}

\begin{theorem}[no valley on the central subintervals]\label{fib:thm:central}
Let $\iota\in\{6,\dots,30\}$ and let $\mu_0(\iota)$ be the entry of Data~\ref{fib:tab:mu0}. For every $q$ with $q/(1-q)\in\Ihat_\iota$ and every
real $\mu\ge\mu_0(\iota)$, no probability mixture with the inputs of $\iota$ (the zero-tilt constraint only) has a valley.
\end{theorem}

\begin{proof}
We apply Lemma~\ref{fib:lem:spike} with $b=\beta'=0$ and $t_s=0.6$. On these subintervals $r_{\rm hi}\le\frac14$, so Theorem~\ref{fib:thm:env} gives
$\Ff{}\ge\psi$ for $M\ge0.6\mu$ ($\mu\ge\mu_0(\iota)\ge19$). For the spike fibers, $Z_s(\mu)\le\max(Z_a,Z_b,Z_c)$.
$Z_a$ bounds the fibers $M\le7$ by Lemma~\ref{fib:lem:KM}.

$Z_b$ bounds the fibers $M\ge8$ with $t\in[0.21,0.6)$ by Lemma~\ref{fib:lem:Phi} (with $(r_m,\allowbreak N_{\min})=(\frac14,\allowbreak 10)$ in Lemma~\ref{fib:lem:lamU}, and
the range (L) by Lemma~\ref{fib:chk:B2}) on the $18$ cells $[t_a,t_b]$ of Lemma~\ref{fib:chk:B3}, each with the five $M$-blocks $[mM_a,2mM_a]$,
$m\in\{1,2,4,8\}$, and $[16M_a,\infty)$, $M_a=\max(8,19t_a)$; on a cell and a block the bound of Lemma~\ref{fib:chk:B3} is multiplied by the
weight bound
\[
(1-q)^{-M}e^{-\ell_0\mu}\le\exp\Big(-\max\big((\tfrac{\ell_0}{t_b}-L_b)\max(mM_a,8,t_a\mu_0),\ (\ell_0-t_bL_b)\mu_0\big)\Big),
\]
which holds because $(1-q)^{-M}e^{-\ell_0\mu}\le e^{-(\ell_0/t-L_b)M}=e^{-(\ell_0-tL_b)\mu}$ and $\ell_0/L_b>0.6\ge t_b$ on these subintervals;
the last block uses $\beta_e$ nondecreasing from $N=16M_a+2\ge130$ (Lemma~\ref{fib:chk:A3}).

$Z_c$ bounds the fibers $M\ge8$ with $t<0.22$ by Lemma~\ref{fib:lem:Zc}.
Each of these bounds holds for every $\mu\ge\mu_0(\iota)$, and so do the other terms of \eqref{fib:eq:row}, which increase with $\mu$. The
margin \eqref{fib:eq:row} at $\mu_0(\iota)$ is Lemma~\ref{fib:chk:rows}.
\end{proof}

\paragraph{The wing.}
On the subintervals $\iota=31,\dots,39$ the cut $0.6$ is too high for the spike weight; there we prove the fiber inequality on part of
the \emph{wing} $t<0.6$, where the Gaussian slack $\Sl(t)$ may be negative and the bonus, kept with its sign, pays for it.

\begin{lemma}[the wing inequality]\label{fib:lem:wing}
Let $0<r_a\le r\le r_b\le\frac14$, $t\in[0.1548,0.6]$, $8\le M_0\le M$, $N=M+2$ and $N_0=M_0+2$. Let $\Ghat=\max(G(\omega_t),G(\omega_t+\bar h))$ with
$\bar h=\frac{\bar{\mathsf k}}{M_0+3}$, $\bar{\mathsf k}=3.85+\frac{10}{M_0}$, $g_1=|G'(\omega_t)|$, and $\rho_*,\kappa_A,\beta_e$ of Lemma~\ref{fib:lem:amgm} at $N_0$ with
$r_m=r_b$ and $\zeta_m=\min(6,r_b\sqrt{N_0})$, and suppose that $\beta_e$ is nondecreasing on $[N_0,\infty)$. Put $\kappa'=\kappa_A\Ghat/\rho_*$,
$b'=4\Ghat/\rho_*$, $c(N)=\frac{N-1}{1-r_a^2}-\kappa'$, $D_e=\frac4{\rho_*^2\beta_e}$ and $\tau_1=2-6G_0+7.1|G_1|-\bar{\mathsf k}^2G_2/(M_0+3)$. If
$c(N_0)>0$ and $r_a\sqrt{N_0}\ge\frac{b'}{2c(N_0)}$, then for every $j$ in the range (U),
\[
\mu(\Ff(M,j)-\psi(t))\ \ge\ \min\big(2,w(N)\big),\qquad w(N):=(N-2)\Sl(t)+\tau_1+c(N)r_a^2N-b'r_a\sqrt N-g_1D_e ,
\]
and $w$ is convex in $N$. The range (M) gives at least $2$, and the range (L) gives at least $0$ if
$\frac{2t}M(N+1)e_L\le\frac92(t-1)^2$.
\end{lemma}

\begin{proof}
If $\omega_U\ge0$ the range (U) gives at least $2$; this includes $\zeta\ge6$, since $t\ge0.1548$ (Lemma~\ref{fib:lem:utail}(a)). Otherwise
$\omega_t\le\omega_U<0$ (Lemma~\ref{fib:lem:lamU} with $(\frac14,\allowbreak 10)$), so $|\sigma_U|\le g_1$, and the first inequality of Lemma~\ref{fib:lem:Ubonus}(b) gives
\[
\mu(\min H_U-\psi)\ge T+N\Delta-\frac{4\zeta+\kappa_A\zeta^2}{\rho_*}\Ghat-g_1\Big(\frac{2E(\zeta)}{\rho_*}-N\Delta\Big)_+ ,
\]
with $\rho_*,\kappa_A$ at $N$ (Lemma~\ref{fib:lem:coef}), which may be replaced by their values at $N_0$. The last term is at most $g_1D_e$ (the
all-$\zeta$ form of Lemma~\ref{fib:lem:amgm}, at $N$ and with $\zeta_m=\min(6,r_b\sqrt N)$; its value decreases in $N$ because $\rho_*$ and $\beta_e$ are
nondecreasing). By Lemma~\ref{fib:lem:coef}(e) with $r\ge r_a$, $N\Delta-\frac{4\zeta+\kappa_A\zeta^2}{\rho_*}\Ghat\ge c(N)\zeta^2-b'\zeta$, which increases on
$\zeta\ge\frac{b'}{2c(N)}$; and $\zeta\ge r_a\sqrt N\ge r_a\sqrt{N_0}\ge\frac{b'}{2c(N_0)}\ge\frac{b'}{2c(N)}$. Finally $T(t,M)\ge MS+\tau_1$: for $t\le1$,
$G_1\le0$ and $-2\mathsf kG_1\ge7.1|G_1|$ for $\mathsf k\ge3.55$, and $\frac{\mathsf k^2G_2}{M+3}\le\frac{\bar{\mathsf k}^2G_2}{M_0+3}$ for $\mathsf k\le3.85+\frac{10}M\le\bar{\mathsf k}$. The function $w$ is a
convex quadratic in $N$ minus a multiple of $\sqrt N$. The ranges (M) and (L): Proposition~\ref{fib:prop:L} and Lemma~\ref{fib:lem:defL}.
\end{proof}

The minimum of $w$ over $N\ge N_0$ is $w(N_0)$ if $w'(N_0)\ge0$, and otherwise at least the minimum of the quadratic obtained from the
tangent line $\sqrt N\le\frac{\sqrt{N_1}}2+\frac N{2\sqrt{N_1}}$ of $\sqrt N$ at a point $N_1>0$: a closed form in $t$.

\begin{theorem}[the wing]\label{fib:thm:wing}
Consider the following three classes. Class $\mathcal W_1$ has $r\in[0.0566,0.0741]$, $\tau=0.58$, $\mu_W=19$, for $\iota=31,32$. Class $\mathcal W_2$ has $r\in[0.074,0.1112]$, $\tau=0.55$, $\mu_W=20$, for $\iota=33,\dots,36$. Class $\mathcal W_3$ has $r\in[0.1111,0.1667]$, $\tau=0.52$, $\mu_W=26$, for $\iota=37,38,39$.
Each class contains the values of $|2q-1|$ of its subintervals (asserted by the supplementary code, Appendix~\ref{apx:fiber}). For every $q$ with $|2q-1|$ in the range of a class, every real
$\mu\ge\mu_W$, every $M$ with $\tau\le M/\mu\le0.6$ and every $j$: $\Ff(M,j)\ge\psi(t)$. Consequently, for $\iota=31,\dots,39$ and every real
$\mu\ge\mu_0(\iota)$ of Data~\ref{fib:tab:mu0}, no probability mixture with the inputs of $\iota$ (zero-tilt constraint only) has a valley.
\end{theorem}

\begin{proof}
On each $t$-cell $[t_a,t_b]$ of Lemma~\ref{fib:chk:C1} put $M_0=\mu_Wt_a$ (so $M\ge M_0$ for every fiber of the cell); then $\beta_e$ is
nondecreasing from $N_0$ by Lemma~\ref{fib:chk:A3p}, the range (L) contributes nothing by Lemma~\ref{fib:chk:B2} (central entry, $r\le\frac14$, $t\le0.6$),
and Lemma~\ref{fib:chk:C1} verifies the hypotheses of Lemma~\ref{fib:lem:wing} and $\min_{N\ge N_0}w(N)\ge0$ on the cell. This proves the fiber
inequality. For the second claim apply Lemma~\ref{fib:lem:spike} with $b=\beta'=0$ and $t_s=\tau$: the fibers with $t\ge0.6$ are covered by
Theorem~\ref{fib:thm:env} ($r\le0.1667$, $\mu\ge\mu_0(\iota)\ge\mu_W\ge19$) and those with $\tau\le t<0.6$ by the first claim. The spike term is bounded
as in Theorem~\ref{fib:thm:central}, using only the cells of Lemma~\ref{fib:chk:B3} (central list) with $t_b\le\tau$, whose union is $[0.21,\tau]$;
the weight decays because $\ell_0/L_b>\tau$ on each member. The margin \eqref{fib:eq:row} at $\mu_0(\iota)$ is Lemma~\ref{fib:chk:rows}.
\end{proof}

\paragraph{The outer subintervals: the offset-window theorem.}
On $\iota=0,\dots,5$ and $\iota=40,\dots,54$ we use eight classes, listed in Data~\ref{fib:tab:classes}, and on each the target $\psi+g$ with the
line $g(t)=b+\beta'(t-1)$ specified there. By Lemma~\ref{fib:lem:spike} the averaging uses $\alpha=1+b$ and $\beta=1+\beta'$, and $\beta$ does not enter
the margin; every class has $\beta'\ge0$.

\datablock{Offset-window classes}{fib:tab:classes}
Each class is specified by its subinterval indices and the ordered parameter tuple
$([r_a,\allowbreak r_b],\allowbreak b,\allowbreak \beta',\allowbreak \tau,\allowbreak t_{\rm hi},\allowbreak \mu_D)$, where $r=|2q-1|$ lies in $[r_a,r_b]$:
Class $\mathcal C_1$ covers indices 0 to 2 and has parameter tuple $([0.3793,\allowbreak 0.5],\allowbreak  \frac25,\allowbreak  0,\allowbreak  0.60,\allowbreak  1.7,\allowbreak  45)$. Class $\mathcal C_2$ covers indices 3 to 5 and has parameter tuple $([0.25,\allowbreak 0.3794],\allowbreak  \frac25,\allowbreak  0,\allowbreak  0.69,\allowbreak  1.7,\allowbreak  50)$.

Class $\mathcal C_3$ covers indices 40 to 43 and has parameter tuple $([0.1666,\allowbreak 0.2309],\allowbreak  0,\allowbreak  \frac12,\allowbreak  0.44,\allowbreak  1.75,\allowbreak  50)$. Class $\mathcal C_4$ covers indices 44 to 46 and has parameter tuple $([0.2307,\allowbreak 0.2728],\allowbreak  \frac3{20},\allowbreak  \frac3{10},\allowbreak  0.42,\allowbreak  1.8,\allowbreak  40)$. Class $\mathcal C_5$ covers indices 47 to 49 and has parameter tuple $([0.2727,\allowbreak 0.3104],\allowbreak  \frac3{10},\allowbreak  0,\allowbreak  0.40,\allowbreak  1.7,\allowbreak  50)$.

Class $\mathcal C_6$ covers indices 50 to 52 and has parameter tuple $([0.3103,\allowbreak 0.3444],\allowbreak  \frac3{10},\allowbreak  0,\allowbreak  0.36,\allowbreak  1.7,\allowbreak  50)$. Class $\mathcal C_7$ covers index 53 and has parameter tuple $([0.3442,\allowbreak 0.3847],\allowbreak  \frac3{10},\allowbreak  1,\allowbreak  0.33,\allowbreak  1.9,\allowbreak  50)$. Class $\mathcal C_8$ covers index 54 and has parameter tuple $([0.3846,\allowbreak 0.4],\allowbreak  \frac3{10},\allowbreak  1,\allowbreak  0.30,\allowbreak  1.9,\allowbreak  50)$.
Each class contains the values of $|2q-1|$ of its subintervals, and
$\ell_0(\iota)/L_b(\iota)>\tau$ on each member; both facts are asserted by the supplementary code.

For $N\ge10$ let $\oT_c=\log(Be^{\rho/2}/P^{\mathrm T}_c)$ with $P^{\mathrm T}_c=\frac{\mu^{3/2}c_N}{2\aG(N-1)}$.

\begin{lemma}[an enclosure of $\oT_c$]\label{fib:lem:lamc}
For $N\ge10$: $\omega_t\le\oT_c\le\oT_L\le\omega_t+\frac{3.85+10/M}{M+3}$.
\end{lemma}

\begin{proof}
As in Lemma~\ref{fib:lem:k}(a),
$\oT_c=\frac52\log\frac N\mu-\frac12\log N+\log(1-\frac1N)-\frac1{2(N+1)}+\frac12\log\frac2\pi-\log c_N$. Moreover $c_N\le(\frac{\pi N}2)^{-1/2}e^{1/(2N)}$: for
$N$ even $c_N=W_{N/2}\le(\pi N/2)^{-1/2}$, and for $N$ odd $W_{(N+1)/2}\le(\pi(N+1)/2)^{-1/2}$ (Lemma~\ref{fib:lem:wallis}) and
$(N+1)I(\frac1{N+1})\le\frac1{2(N+1)}+\frac1{2N(N+1)}$, as $I(x)=\sum_{i\ge1}\frac{x^{2i}}{2i(2i-1)}$. Hence $\oT_c\ge\omega_t+f(N)$ with
$f(N)=\frac52\log(1+\frac2{N-2})+\log(1-\frac1N)-\frac1{2(N+1)}-\frac1{2N}$. With $\log(1+x)\ge x-\frac{x^2}2$ and $\log(1-x)\ge-x-x^2$ ($x\le\frac12$),
$f(N)\ge\frac3N-\frac5{(N-2)^2}-\frac1{N^2}>0$ for $N\ge10$. The upper bounds: $c_N\ge L_N$ (Lemma~\ref{fib:lem:amp}) and Lemma~\ref{fib:lem:k}.
\end{proof}

For $N\ge10$ and $r\le\frac12$ let $Z^{(3)}_L=\frac{2(\sqrt3s_N+r)}{N+1}$, $Z^{(3)}_R=\frac{2\sqrt3s_N}{N+1}$, let $E^{(3)}_4$ be $E_4$ with these values in place
of $Z_L,Z_R$, and consider the conditions
\begin{gather*}
\text{(c1) }Z^{(3)}_L<1,\qquad\text{(c2) }r+Z^{(3)}_R<1,\\
\text{(c3) }3\le\frac{12(1-(r+Z^{(3)}_R)^2)}{(1-r^2)\rho},\qquad\text{(c3$'$) }\frac{12(1-(r+Z^{(3)}_R)^2)}{1-r^2}\ge3 .
\end{gather*}

\begin{lemma}[the offset-window shape bound]\label{fib:lem:Wshape}
Let $N\ge10$, $r\le\frac12$ ($q\in[\frac12,\frac34]$), and assume (c1), (c2), (c3). Then for $0\le w\le3$: $0\le J\le N$ and
$\frac{\sqrt{1-r^2}}{c_N}B_N(J)\ge\exp(-\frac\rho2w-\varepsilon|x|-c_0-E^{(3)}_4)$. Conditions (c1) and (c2) improve as $N$ grows ($Z^{(3)}_L$ and
$Z^{(3)}_R$ decrease); (c3$'$) implies (c3), because $\rho<1$, and also improves as $N$ grows. So (c1), (c2), (c3) hold for all $N\ge N_0$ once
(c1), (c2), (c3$'$) hold at $N_0$.
\end{lemma}

\begin{proof}
The proof of Lemma~\ref{fib:lem:expo}(b), with $|u|\le\sqrt3s_N$ in place of $|u|\le s_N$: on the side $z\le0$, $|z|\le Z^{(3)}_L$, and
Lemma~\ref{fib:lem:gauss}(a) applies by (c1); on the side $z>0$, $z\le Z^{(3)}_R$, Lemma~\ref{fib:lem:gauss}(b) applies by (c2), and the cubic term is
absorbed for $w\le3$ by (c3). Support: $q(N-1)+\frac12-\sqrt3s_N\ge0$ for $q\ge\frac12$, $N\ge4$, and $(1-q)N+q-\frac12\ge\sqrt{3q(1-q)N}$ for $q\le\frac34$,
$N\ge9$.
\end{proof}

\begin{lemma}[the offset-window minorant]\label{fib:lem:Wpiece}
Under the hypotheses of Lemma~\ref{fib:lem:Wshape}, let $W'=\min(3,\gamma)$, $\rho_L=\rho+\varepsilon/\sqrt3$,
$e_L=c_0+E^{(3)}_4+\varepsilon(\frac{\sqrt3}2+\frac1{2\sqrt3})$ and $\omega_L=\oT_c+\frac{\rho_L}2\Delta+e_L$. On $0\le w\le W'$,
$\Ff{}\ge H_L(w)=P_L(\gamma-w)e^{-\rho_Lw/2}+Bw$, and $\min_{[0,W']}H_L\ge B\,X_L$, where $X_L=W'$ if the \emph{offset-window condition}
$\omega_L\le y+\log(1+y)$, $y=\frac{\rho_L}2(\gamma-W')$, holds, and $X_L=\max(X_{\rm split},X_{\log})$ otherwise, with
\[
X_{\rm split}=1-\frac2\rho G(\oT_c)+\frac{\Delta}{1+s(\omega_L)_+}-\frac{2e_L}\rho,\qquad X_{\log}=1-\frac2{\rho_L}\big(\oT_c+e_L\big).
\]
\end{lemma}

\begin{proof}
The shape bound with the tangent $|x|\le\frac{w+3}{2\sqrt3}$ gives $H_L$ with $P_L=P^{\mathrm T}_ce^{-c_0-E^{(3)}_4-\varepsilon\sqrt3/2}$, hence
$\log(Be^{\rho_L\gamma/2}/P_L)=\oT_c+\frac{\rho_L}2\Delta+e_L=\omega_L$ exactly. By Corollary~\ref{fib:lem:min}, $H_L$ decreases up to
$w_*=\gamma-\frac{2s(\omega_L)}{\rho_L}$, and $w_*\ge W'$ is the offset-window condition; then $\min_{[0,W']}H_L=H_L(W')\ge BW'$. Otherwise $s(\omega_L)>y\ge0$,
so $\omega_L>0$, and $\min H_L\ge B(\gamma-\frac2{\rho_L}G(\omega_L))$. Lemma~\ref{fib:lem:split} at $\omega_0=\oT_c$ gives $X_{\rm split}$ (using $e_L\ge0$,
$\rho_L\ge\rho$, $G'<1$ and $1-G'(\omega_L)=\frac1{1+s(\omega_L)}$). And $G(\omega)\le\omega-\log(1+\omega)$ for $\omega\ge0$ (Lemma~\ref{fib:lem:G}(d))
gives $\gamma-\frac2{\rho_L}G(\omega_L)\ge1-\frac2{\rho_L}(\oT_c+e_L)+\frac2{\rho_L}\log(1+\omega_L)\ge X_{\log}$.
\end{proof}

The bound $X_{\log}$ is what makes the offset-window minorant work at $t\in[0.5,0.7]$, where $\oT_c<0<\omega_L$ and the split form, which
expands $G$ around the negative point $\oT_c$, loses.

\begin{lemma}[the derivative bound for the range (U)]\label{fib:lem:Umono}
Let $N\ge10$, $0<r\le\frac12$, $\hat a=\artanh r-r$, $m=N+1$, $R=\sqrt{N+r^2(N-1)^2}$, and $\Phi_r(N)=\frac{3.55}{N+1}+\frac{\rho_U}2\Delta-e_U$. Then
$\omega_U\ge\omega_t+\Phi_r(N)$,
\[
\Phi_r=\frac{3.55}m+\frac{r^2}2\Big(m-4+\frac4m\Big)-\Big(\frac{2r}m+\hat a\Big)R-mc_r+\frac{2r^2}m+2r\hat a+\frac{\hat a^2m}2 ,
\]
and, for every $m>2$,
\[
\frac{d\Phi_r}{dm}\ \ge\ \frac{r^2}2-\hat ar-c_r-\frac{2r^2}m-\frac{4r^2}{m^2}-\frac1{m(m-2)}-\frac{\hat a}{2r(m-2)}-\frac{3.55}{m^2}.
\]
Every subtracted term is nonincreasing in $m$ ($\hat ar$ and $c_r$ do not depend on $m$, the others decrease); so if the right side is positive at $m_0$, then $\Phi_r$ is increasing on $[m_0,\infty)$.
\end{lemma}

\begin{proof}
$\oT_U\ge\omega_t+\frac{3.55}{M+3}$ (Lemma~\ref{fib:lem:k}) and $M+3=N+1$. In $\frac{\rho_U}2\Delta-e_U$ the $\varepsilon'$ terms combine, by $\Delta=\gamma-1$,
to $-\varepsilon'\sqrt\gamma=-\frac{2\tilde y}{N+1}R$; also $\frac\rho2((1-r^2)\Delta-r^2)=\frac{r^2(N-1)^2}{2(N+1)}$ and $c_0'=\frac{2\tilde y^2}{N+1}$. With
$\tilde y=r+\frac{\hat am}2$ this gives the displayed $\Phi_r$. Since $R\ge r(m-2)$, $R'=\frac{1+2r^2(m-2)}{2R}\le r+\frac1{2r(m-2)}$; dropping
$\frac{2r}{m^2}R\ge0$ and $\frac{\hat a^2}2\ge0$ from
$\Phi_r'=-\frac{3.55}{m^2}+\frac{r^2}2(1-\frac4{m^2})+\frac{2r}{m^2}R-(\frac{2r}m+\hat a)R'-c_r-\frac{2r^2}{m^2}+\frac{\hat a^2}2$ gives the bound.
\end{proof}

\begin{lemma}[the range (U) in exact form]\label{fib:lem:Uexact}
Let $0<r_a\le r\le r_b\le\frac12$ and let $N_1\ge10$ be such that the right side of Lemma~\ref{fib:lem:Umono} is positive at $m=N_1+1$ for every
$r\in[r_a,r_b]$. Then $\omega_U\ge\omega_t+\Phi_r(N)\ge\omega_t+\Phi_r(N_0)$ for $N\ge N_0\ge N_1$, and
\[
\min_{w\ge\gamma}H_U\ \ge\ B\,X_U,\qquad X_U=1+\Delta(r_a,N_0)-\frac2{\rho_*(N_0,r_b)}\,G\big(\min(0,\omega_t+\Phi_{[r_a,r_b]}(N_0))\big),
\]
where $\Phi_{[r_a,r_b]}(N_0)$ is a lower bound of $\Phi_r(N_0)$ over $r\in[r_a,r_b]$ (the $r$-dependent quantities $\artanh r-r$, $c_r$, $\gamma$ and $s_N$ are
monotone in $r$ and are evaluated at the ends; the term $-c_0'\le0$ of $e_U$ may be dropped, which only lowers $\Phi_r$).
\end{lemma}

\begin{proof}
By Proposition~\ref{fib:prop:pieces}(U) and Corollary~\ref{fib:lem:min}, $\min_{w\ge\gamma}H_U=B\gamma$ if $\omega_U\ge0$ and $B(\gamma-\frac2{\rho_U}G(\omega_U))$
otherwise. By Lemma~\ref{fib:lem:Umono}, $\Phi_r$ is nondecreasing on $[N_1,\infty)$. In the second case $\omega_t+\Phi_{[r_a,r_b]}(N_0)\le\omega_U<0$, and $G$
decreases on $(-\infty,0]$; and $\rho_U\ge\rho_*$ (Lemma~\ref{fib:lem:coef}(b)), $\Delta$ increases in $r$ and $N$, and $\rho_*$ increases in $N$.
\end{proof}

\begin{theorem}[the offset-window theorem]\label{fib:thm:window}
For each class of Data~\ref{fib:tab:classes}, every $q$ with $|2q-1|\in[r_a,r_b]$, every real $\mu\ge\mu_D$, every integer $M$ with $t=M/\mu\ge\tau$
and every $j$: $\Ff(M,j)\ge\psi(t)+b+\beta'(t-1)$.
\end{theorem}

\begin{proof}
By Lemma~\ref{fib:lem:sym} take $q\ge\frac12$. Here $N\ge\mu_D\tau+2\ge17$.

\emph{Fibers with $t\ge t_{\rm hi}$.} By Lemma~\ref{fib:chk:D0}, $\psi+g<0$ at $t_{\rm hi}$, and $\psi'+\beta'<0$ for $t\ge t_{\rm hi}$, so $\psi+g<0$ on
$[t_{\rm hi},\infty)$. Since $\mu_D\ge40$, $N\ge1.6\cdot40+2=66$, and Lemma~\ref{fib:lem:larget} with $r_m=\frac12$, $N_*=66$ (Lemma~\ref{fib:chk:D4})
gives $\Ff{}\ge0$.

\emph{Fibers with $\tau\le t\le t_{\rm hi}$.} By Lemma~\ref{fib:chk:D1} and Lemma~\ref{fib:chk:D1p}, (c1), (c2), (c3$'$) hold at $N_0=\mu_D\tau+2$, hence
(c1), (c2), (c3) hold for every $N$ of the class's fibers (Lemma~\ref{fib:lem:Wshape}); and by Lemma~\ref{fib:chk:D2} the hypothesis of
Lemma~\ref{fib:lem:Uexact} holds with $N_1=\mu_D\tau+2$. The ranges $[0,W']$, $[3,\gamma]$ (if $\gamma>3$) and $[\gamma,\infty)$ of $w$ give
$\Ff{}\ge B\min(X_L,3,X_U)$ (Lemmas~\ref{fib:lem:Wpiece} and~\ref{fib:lem:Uexact}; on $[3,\gamma]$, $\Ff{}\ge Bw\ge3B$). On a $t$-cell $[t_a,t_b]$ of
Lemma~\ref{fib:chk:D3} put $N_0=\mu_Dt_a+2$. Every term of $X_L$ moves the right way in $N$, with the bounds
$G(\oT_c)\le\max(G(\omega_t),G(\omega_t+\bar h))$ (Lemma~\ref{fib:lem:lamc} and the convexity of $G$), $s(\omega_L)\le s(\omega_t+\bar h+e_L)_++\bar\rho_L\Delta/2$
and $\rho_L\le\bar\rho_L=1+\varepsilon(N_0)/\sqrt3$; and so does $X_U$. The offset-window condition, once it holds at $N_0$ in the form
``$\Delta\ge2$ and $\omega_t+\bar h+e_L+\bar\rho_L\le\log(1+\rho(\Delta-2)/2)$'' or in the form ``$\omega_t+\bar h+e_L+\bar\rho_L\le0$'', holds for every $N\ge N_0$
(for $\Delta\ge2$, $W'=3$ and the condition reads $\oT_c+e_L+\rho_L\le\log(1+\rho_L(\Delta-2)/2)$, in which $\Delta$ cancels from $\omega_L$; for $\Delta<2$,
$W'=\gamma$ and it reads $\omega_L\le0$; and $e_L$, $\bar h$, $\varepsilon$ decrease while $\Delta$ and $\rho$ increase in $N$). Finally
$B=\frac{tN}{N-2}\in[t,\frac{tN_0}{N_0-2}]$, so $\Ff{}\ge\psi+g$ follows from $\min(X_L,3,X_U)\ge\frac{\psi+g}t$ where $\psi+g>0$, and from
$\min(X_L,3,X_U)\ge\frac{(\psi+g)(N_0-2)}{tN_0}$ where $\psi+g\le0$. This is Lemma~\ref{fib:chk:D3}.
\end{proof}

\begin{corollary}[no valley on the outer subintervals]\label{fib:cor:outer}
For $\iota\in\{0,\dots,5\}\cup\{40,\dots,54\}$, the caps $\theta(\iota)$ of Data~\ref{p2:tab:theta} and every real $\mu\ge\mu_0(\iota)$ of Data~\ref{fib:tab:mu0}, no probability mixture with
the inputs of $\iota$ (zero-tilt constraint only) has a valley.
\end{corollary}

\begin{proof}
Apply Lemma~\ref{fib:lem:spike} with the class's $b$, $\beta'$ and $t_s=\tau$; note $\mu_0(\iota)\ge\mu_D$. Theorem~\ref{fib:thm:window} covers $t\ge\tau$. On
$t<\tau$ ($<1$), $\Def_g\le\Def+b$, and the spike term is bounded as in Theorem~\ref{fib:thm:central}, with the deficits raised by $b$
(Lemmas~\ref{fib:lem:KM} and~\ref{fib:lem:Zc} with this $b$), on the class cells of Lemma~\ref{fib:chk:B3}: on each cell and $M$-block, $\Def_g\le b+L$ with the value $L$ of
Lemma~\ref{fib:chk:B3}, which is, rounded up, the first of the following two bounds (the margins of Lemma~\ref{fib:chk:rows} use
these certified bounds), $L_\Phi$, or, on the cells marked for the offset-window bound in Data~\ref{fib:tab:b3}, the smaller of $L_\Phi$
and the second, $L_W$:
The first bound uses $\Phi$ of Lemma~\ref{fib:lem:Phi}, with $r_m=r_b$ and $M_a=\max(8,\mu_Dt_a)$; the range (L) contributes nothing by Lemma~\ref{fib:chk:B2} (class
entry); $\omega_U\ge\omega_t$ holds by Lemma~\ref{fib:lem:lamU} for the classes with $r_b\le0.4$, with the pair $(r_m,\allowbreak N_{\min})$ of Lemma~\ref{fib:chk:A4p} that
covers $r_b$ (for $\mathcal C_2$, $\mathcal C_7$, $\mathcal C_8$ it has $N_{\min}=12$, and their spike fibers have $N\ge50\cdot0.21+2>12$), and for $\mathcal C_1$ by Lemma~\ref{fib:chk:A4d} on the $(r,\allowbreak N)$-blocks of $[0.3793,0.5]\times[11.45,29.1]$ followed by
Lemmas~\ref{fib:lem:Uexact} and~\ref{fib:chk:D2} for $N\ge29$; the last $M$-block uses $\beta_e$ nondecreasing from $N=16M_a+2$
(Lemma~\ref{fib:chk:A3p}); this gives $\Def\le\Phi\le L_\Phi$, hence $\Def_g\le\Def+b\le b+L_\Phi$.

The second is the offset-window bound $(\psi+g+B_{\max}(-X)_+)_+$, $X=\min(X_L,3,X_U)$ evaluated as in Theorem~\ref{fib:thm:window} on the cell with
$N_0=\mu_Dt_a+2$ and $B_{\max}=t_bN_0/(N_0-2)$; it bounds $\Def_g$ directly and is used only on cells where (c1), (c2), (c3) and the positivity
of the derivative bound of Lemma~\ref{fib:lem:Umono} hold at $N_0$, and (c3$'$) holds at $N_0$ (Lemma~\ref{fib:chk:D1p}).
Hence $\Def_g\le\min(b+L_\Phi,L_W)\le b+L$ (on the cells where only $L_\Phi$ is used, read $L_W=\infty$). Since the
offset-window bound already bounds $\Def_g$ (it includes the excess $b$), the bound $b+\min(\Phi,\text{offset-window bound})$
used here is conservative; the smaller $\min(b+\Phi,\text{offset-window bound})$ would also be valid.
In the use of Lemma~\ref{fib:chk:A4d}, each term of $\Phi_r(N)$ is taken at its worst end on an $(r,\allowbreak N)$-block $[r_a,r_b]\times[N_{\rm lo},N_{\rm hi}]$:
$\rho_U\ge\rho_*(N_{\rm lo},r_b)$; $\Delta$ at $(r_a,\allowbreak N_{\rm lo})$; $\frac{r^2\rho}2$ and $(N+1)c_r$ at $(r_b,\allowbreak N_{\rm hi})$; $\frac{\varepsilon'}{2\sqrt\gamma}$ by its all-$\zeta$ bound
$\frac1{(N+1)\sqrt c}+\frac{r^2}{6(1-r^2)\sqrt c}$ (from Lemma~\ref{fib:lem:mono}(b) with $N_a=N$, admissible as $\zeta^2\le N/4$, and the two
elementary bounds of the proof of Lemma~\ref{fib:lem:coef}) at $(N_{\rm lo},\allowbreak r_b)$; and $\frac{\varepsilon'\sqrt\gamma}2=\frac{\tilde yR}{N+1}$ at $(r_b,\allowbreak N_{\rm hi})$. The last step uses that
$\frac{\tilde yR}{N+1}$ is nondecreasing in $r$ and in $N$: its $N$-derivative is $\frac h{4R}$ with
$h=\hat a(1+2r^2(N-1))-\frac{2r(1-4r^2)(N-1)}{(N+1)^2}$, $\hat a=\artanh r-r$, and $\partial_Nh=2\hat ar^2+\frac{2r(1-4r^2)(N-3)}{(N+1)^3}\ge0$ for $N\ge3$
and $r\le\frac12$, while $h(r,11.45)>0$ on $[0.3793,0.5]$ (Lemma~\ref{fib:chk:A4m}). The margin \eqref{fib:eq:row} with $\alpha=1+b$ at $\mu_0(\iota)$ is
Lemma~\ref{fib:chk:rows}.
\end{proof}

\begin{proof}[Proof of Theorem~\ref{fib:thm:intro-novalley}]
Theorem~\ref{fib:thm:central} ($\iota=6,\dots,30$), Theorem~\ref{fib:thm:wing} ($\iota=31,\dots,39$) and Corollary~\ref{fib:cor:outer} ($\iota=0,\dots,5$ and
$40,\dots,54$), with the starting means of Data~\ref{fib:tab:mu0}, all between $19$ and $50$.
\end{proof}

\section{Forests with at most 24 vertices}\label{p1:sec:overview}

This section proves Theorem~\ref{p1:thm:main}, the case $n\le24$ of Theorem~\ref{thm:main}, by exact counting relative
to a maximum independent set; together with Theorem~\ref{p2:thm:main} it proves Theorem~\ref{thm:main}.

\begin{theorem}\label{p1:thm:main}
Every forest with at most $24$ vertices has a unimodal independence sequence.
\end{theorem}

Throughout the section $F$ is a forest with $n$ vertices, $A$ an independent set of maximum cardinality $a=\alpha(F)$,
$C=V(F)\setminus A$, $c=|C|=n-a$ and $d=a-c$; $\NA(J)$, $m(J)$ and the fiber numbers $N(y,m)$ of $A$ are as in
Section~\ref{p1:subsec:fiber}, so that $i_k=\sum_{y,m}N(y,m)\binom{m}{k-y}$ (Lemma~\ref{p1:lem:fiber}); this
decomposition and the fiber numbers are the framework of Zhang and Li~\cite{ZL}. No forest is enumerated.

%% ---------------------------------------------------------------------------------------------
\subsection{Maximum independent sets of forests}\label{p1:sec:structure}

This subsection proves that $C$ is matched into $A$ (Lemma~\ref{p1:lem:matching}) by a matching that leaves only vertices
of degree at most one unmatched (Lemma~\ref{p1:lem:exposed}); Section~\ref{p1:sec:rows} uses both, for the support
conditions and for the matching-excess inequality.

\begin{lemma}[a matching saturating $C$]\label{p1:lem:matching}
There is an injective map $\sigma:C\to A$ such that $v\sigma(v)$ is an edge for every $v\in C$. Consequently $c\le a$ and
$\nu(F)=c=n-a$.
\end{lemma}

\begin{proof}
Fix a proper $2$-colouring of $F$ with classes $V_0,V_1$. Let $C'\subseteq C$ and $C'_i=C'\cap V_i$. The set
$(A\setminus\NA(C'_i))\cup C'_i$ is independent, being the union of two independent sets with no edge between them,
and its two parts are disjoint; maximality of $|A|$ gives $a-|\NA(C'_i)|+|C'_i|\le a$, that is,
$|\NA(C'_i)|\ge|C'_i|$. Since $\NA(C'_0)\subseteq V_1$ and $\NA(C'_1)\subseteq V_0$ are disjoint subsets of
$\NA(C')$, we get $|\NA(C')|\ge|C'|$, and Hall's theorem gives $\sigma$. Its $c$ edges form a matching, so $c\le a$ and $\nu(F)\ge c$. Conversely every edge meets
$C$, because $A$ is independent, so a matching has at most $c$ edges.
\end{proof}

The $d=a-c\ge0$ vertices of $A$ outside $\sigma(C)$ are called \emph{unmatched}.

\begin{lemma}[exposed leaves]\label{p1:lem:exposed}
The map $\sigma$ of Lemma~\ref{p1:lem:matching} can be chosen so that every unmatched vertex of $A$ has degree at most
$1$.
\end{lemma}

\begin{proof}
Suppose an unmatched $u_0\in A$ has degree at least $2$. Build a walk $u_0,v_1,u_1,v_2,\dots$: $v_1$ is a neighbour
of $u_0$ (in $C$, as $A$ is independent) and $u_1=\sigma(v_1)$; while $u_i$ has a neighbour $v_{i+1}\ne v_i$, continue
with $u_{i+1}=\sigma(v_{i+1})$. We have $u_1\ne u_0$ as $u_0$ is unmatched, and $u_{i+1}\ne u_i$ as $\sigma$ is
injective, so no step reverses the previous one; in a forest the walk is a path, and it stops at some $u_j$ whose only
neighbour is $v_j$. Redefine $\sigma(v_i)=u_{i-1}$ for $1\le i\le j$. The new $\sigma$ is injective along edges, $u_0$
becomes matched, $u_j$ (of degree $1$) becomes unmatched, and nothing else changes. The number of unmatched vertices of degree
at least $2$ drops by one; iterate.
\end{proof}

From now on $\sigma$ is as in Lemma~\ref{p1:lem:exposed}. A \emph{cross edge} is a pair $(v,u)$, $v\in C$, $u\in A$,
$vu\in E(F)$; their number is $e(A,C)$, the number of edges between $A$ and $C$. Let $E^*$ be the set of cross edges
with $u\ne\sigma(v)$, so $|E^*|=e(A,C)-c$, and $E^*_{\mathrm u}\subseteq E^*$ the set of those with $u$ unmatched. An
unmatched vertex has at most one neighbour, so
\begin{equation}\label{p1:eq:U}
|E^*_{\mathrm u}|\le d.
\end{equation}

%% ---------------------------------------------------------------------------------------------
\subsection{Linear inequalities for the fiber numbers}\label{p1:sec:rows}

This subsection proves the linear inequalities satisfied by the fiber numbers of $A$
(Lemmas~\ref{p1:prop:counting} and~\ref{p1:prop:support}, Theorem~\ref{p1:thm:matching}), on which the bounds of
Section~\ref{fl:sec} rest. We use the following quantities, affine in $N$ with coefficients depending only
on $(n,a)$:
\begin{align*}
N_h&=\textstyle\sum_mN(h,m) &&\text{(the number of independent $h$-subsets of $C$)},\\
e(A,C)&=\textstyle\sum_m(a-m)\,N(1,m) &&\text{(the number of cross edges)},\\
e_C=e(F[C])&=\textstyle\binom c2-N_2 &&\text{(the number of edges of $F[C]$)},\\
\mathrm{ex}_h&=\textstyle\sum_m(a-m-h)\,N(h,m) &&\textstyle\bigl(=\sum_J x(J)\text{ over independent $J\subseteq C$, $|J|=h$}\bigr),
\end{align*}
and the $i_k$, by \eqref{p1:eq:fiber}; here $x(J)=|\NA(J)|-|J|=a-|J|-m(J)$ is the \emph{excess} of an independent
$J\subseteq C$, $a-m(\{v\})$ is the number of neighbours of $v\in C$ in $A$, and $\mathrm{ex}_1=e(A,C)-c$.

\begin{lemma}[four counting inequalities]\label{p1:prop:counting}
The sums in (iii) and (iv) run over the $h$-subsets $J\subseteq C$ that are not independent.

\textup{(i)} \textup{(the edge inequality)} For $n\ge1$, $S^e:=i_2-\bigl[\binom n2-(n-1)\bigr]=n-1-e(F)=\operatorname{cc}(F)-1\ge0$.
\textup{(ii)} \textup{(the subset inequalities)} For $2\le h\le c$: $S^0_h:=\binom ch-N_h\ge0$, the number of dependent
$h$-subsets of $C$.
\textup{(iii)} \textup{(the edge-union inequalities)} For $3\le h\le c$:
$S^-_h:=N_h-\binom ch+e(F[C])\binom{c-2}{h-2}=\sum_{J}\bigl(e(F[J])-1\bigr)\ge0$.
\textup{(iv)} \textup{(the pair-extension inequalities)} For $3\le h\le c$:
$S^+_h:=\binom{c-2}{h-2}N_2-\binom h2N_h=\sum_{J}\bigl(\binom h2-e(F[J])\bigr)\ge0$.

\end{lemma}

\begin{proof}
(i) The independent $2$-sets are the nonadjacent pairs, so $i_2=\binom n2-e(F)$, and $e(F)=n-\operatorname{cc}(F)$ for a forest.
(ii) is clear. (iii) Each edge of $F[C]$ lies in $\binom{c-2}{h-2}$ of the $h$-subsets of $C$, so
$\sum_{|J|=h}e(F[J])=e(F[C])\binom{c-2}{h-2}$; independent $J$ contribute $0$, and the $\binom ch-N_h$ dependent ones
contain an edge each. (iv) The pairs $(J_2,J)$ with $J_2$ an independent $2$-subset of $C$ and $J_2\subseteq J\subseteq C$,
$|J|=h$, number $N_2\binom{c-2}{h-2}$, and also
$\sum_{|J|=h}\bigl(\binom h2-e(F[J])\bigr)=\binom h2N_h+\sum_{J\text{ dependent}}\bigl(\binom h2-e(F[J])\bigr)$.
\end{proof}

With $e_C=\binom c2-N_2$, parts (ii) to (iv) give, for $3\le h\le c$,
\begin{equation}\label{p1:eq:Nh}
\binom ch-e_C\binom{c-2}{h-2}\ \le\ N_h\ \le\ \min\Bigl\{\binom ch,\ \binom{c-2}{h-2}N_2\Big/\binom h2\Bigr\}.
\end{equation}
\emph{Credit.} (i) is \cite[(6)]{ZL} with its index equal to $2$, as $\binom n2-(n-1)=\binom{n-1}2$; (ii) is
\cite[(5)]{ZL}; (iii) is \cite[(28)]{ZL} rewritten with $e(F[C])=\binom c2-N_2$ \cite[(12)]{ZL}; (iv) follows from
\cite[(10)]{ZL} with $l=0$, $(r+1)N_{r+1}\le(c-r)N_r$, for $r=2,\dots,h-1$. The matching-excess inequality (Theorem~\ref{p1:thm:matching}) and
the exposed-leaf matching (Lemma~\ref{p1:lem:exposed}) were not found in~\cite{ZL}.

\begin{theorem}[the matching-excess inequality]\label{p1:thm:matching}
Let $F$ be a forest and $A$ an independent set of maximum cardinality, with $c=n-a\ge2$ and $d=a-c$. For every
$2\le h\le c$,
\begin{equation}\label{p1:eq:SM}
S^M_h:=(e(A,C)-c)\binom{c-2}{h-1}+d\binom{c-2}{h-2}-\mathrm{ex}_h\ \ge\ 0 .
\end{equation}
\end{theorem}

\emph{Idea.} Each blocker of an independent $J$ beyond the $h$ matched ones is charged to a cross edge outside the
matching (Step~1), and these edges are counted over all $h$-subsets of $C$ (Step~2).

\begin{proof}
Let $\sigma$, $E^*$, $E^*_{\mathrm u}$ be as in Section~\ref{p1:sec:structure}. For an $h$-subset $J\subseteq C$,
independent or not, call $(v,u)\in E^*$ \emph{eligible} for $J$ if $v\in J$ and $u\notin\sigma(J)$ (so $u$ is
unmatched or $\sigma^{-1}(u)\notin J$); let $E^*(J)$ be the set of these edges.

\emph{Step 1 (injective charge).} Let $J$ be independent. Then $\sigma(J)\subseteq\NA(J)$ and $|\sigma(J)|=h$. For
each \emph{excess blocker} $u\in\NA(J)\setminus\sigma(J)$ choose a neighbour $v\in J$ of $u$ (the least in a fixed
order). Then $u\ne\sigma(v)$, so $(v,u)\in E^*$ is eligible for $J$, and distinct $u$ give distinct edges:
\begin{equation}\label{p1:eq:charge}
|E^*(J)|\ \ge\ |\NA(J)\setminus\sigma(J)|=|\NA(J)|-h=a-m(J)-h .
\end{equation}

\emph{Step 2 (total eligibility).} An edge $(v,u)\in E^*$ with $u$ unmatched is eligible for the $\binom{c-1}{h-1}$
subsets containing $v$; one with $u=\sigma(v')$, $v'\ne v$, for the $\binom{c-2}{h-1}$ subsets containing $v$ but not
$v'$. By Pascal's rule
\[
\sum_{|J|=h}|E^*(J)|=|E^*_{\mathrm u}|\tbinom{c-1}{h-1}+\bigl(e(A,C)-c-|E^*_{\mathrm u}|\bigr)\tbinom{c-2}{h-1}
=\bigl(e(A,C)-c\bigr)\tbinom{c-2}{h-1}+|E^*_{\mathrm u}|\tbinom{c-2}{h-2}.
\]

\emph{Conclusion.} As $\mathrm{ex}_h=\sum_{J\text{ indep.}}|\NA(J)\setminus\sigma(J)|$, Step 2 gives the exact
decomposition
\begin{equation}\label{p1:eq:decomp}
S^M_h=\bigl(d-|E^*_{\mathrm u}|\bigr)\binom{c-2}{h-2}+\sum_{\substack{|J|=h\\ J\text{ not indep.}}}|E^*(J)|+\sum_{\substack{|J|=h\\ J\text{ indep.}}}\Bigl(|E^*(J)|-|\NA(J)\setminus\sigma(J)|\Bigr),
\end{equation}
whose terms are nonnegative by \eqref{p1:eq:U}, trivially, and by \eqref{p1:eq:charge}.
\end{proof}

\begin{lemma}[support and normalizations]\label{p1:prop:support}
If $a<n$, then $N(y,m)$ is a nonnegative integer, $N(y,m)=0$ unless $(y,m)=(0,a)$ or $1\le y\le c$, $0\le m\le a-y$
(so $x(J)\ge0$ for every independent $J\subseteq C$), and
\begin{equation}\label{p1:eq:normal}
N(0,a)=1,\qquad \sum_mN(1,m)=c .
\end{equation}
Moreover $\lceil n/2\rceil\le a$, $e(A,C)+e_C\le n-1$, $e(A,C)\ge c$ and $e_C\le c-1$; consequently
\begin{equation}\label{p1:eq:ex1}
0\le\mathrm{ex}_1\le a-1-e_C .
\end{equation}
\end{lemma}

\begin{proof}
For independent $J\subseteq C$ the set $J\cup(A\setminus\NA(J))$ is independent with $|J|+m(J)$ elements, so
$|J|+m(J)\le a$. The only $J$ with $|J|=0$ is $\emptyset$, with $m(\emptyset)=a$; every singleton of $C$ is
independent; and $c\le a$ by Lemma~\ref{p1:lem:matching}. As $A$ is independent, $e(F)=e(A,C)+e_C$, and $e(F)\le n-1$
by Lemma~\ref{p1:prop:counting}(i); $e(A,C)=\sum_{v\in C}\bigl(x(\{v\})+1\bigr)\ge c$; $F[C]$ is a forest on $c\ge1$
vertices, so $e_C\le c-1$. Hence $0\le\mathrm{ex}_1=e(A,C)-c\le n-1-c-e_C=a-1-e_C$.
\end{proof}

Nonnegativity, the normalization $\sum_mN(1,m)=c$ and Lemma~\ref{p1:prop:counting}(ii) are the basic constraints
\cite[(5)]{Z}; the support condition and $N(0,a)=1$ are stated in \cite[Section~2.1]{Z} with its identity (2) (the same
numbering holds in \cite{ZL}).

%% ---------------------------------------------------------------------------------------------
\subsection{The parameter set and the valley functional}\label{p1:sec:param}

This subsection reduces the ranks of the window to a finite set of triples $(n,a,k)$ and defines the valley
functional, which Sections~\ref{fl:sec} and~\ref{flx:sec} bound below.

From here on $F$ has an edge, that is, $a<n$. For a pair $(n,a)$ let $b_s$ be the coefficients of
$(1+2x)^{n-a}(1+x)^{2a-n}$ and $h_B(n,a)$ their first mode; by Lemma~\ref{p1:lem:matching}, $\nu(F)=c=n-a$, so this is the
polynomial $B_F$ of~\eqref{p1:eq:brow}, $(1+2x)^c(1+x)^d$, and $h_B(F)=h_B(n,a)$.
Put $H(n,a)=\lceil (n+2a)/6\rceil$. Proposition~\ref{p1:prop:inputs}(c) gives $h_B(n,a)\le\lceil(3n-2c)/6\rceil=H(n,a)$,
and $H(n,a)\le a$ as $n\le2a$ (Lemma~\ref{p1:prop:support}). Hence, if $n\le24$, every rank $k\in W(F)$ gives a triple
$(n,a,k)$ of the \emph{parameter set}
\begin{equation}\label{p1:eq:domain}
\cD_{24}=\bigl\{(n,a,k):\ 1\le n\le 24,\ \lceil n/2\rceil\le a\le n-1,\ \lceil n/4\rceil<k<\min(a,H(n,a))\bigr\},
\end{equation}
in which the minimum always equals $H(n,a)$. It has $249$ elements, $130$ with $n\le20$ and none with $n\le6$.

For a rank $k\ge1$, a weight $w_{\mathrm L}\in[0,1]$, $w_{\mathrm R}=1-w_{\mathrm L}$, and integers $y,m\ge0$ put
\begin{equation}\label{p1:eq:dLdR}
d_L(y,m)=\binom{m}{k-y}-\binom{m}{k-1-y},\qquad d_R(y,m)=\binom{m}{k-y}-\binom{m}{k+1-y},
\end{equation}
$d_{\mathrm{LR}}=w_{\mathrm L}d_L+w_{\mathrm R}d_R$, that is,
$d_{\mathrm{LR}}(y,m)=\binom{m}{k-y}-w_{\mathrm L}\binom{m}{k-1-y}-w_{\mathrm R}\binom{m}{k+1-y}$, and
$\mathsf V(k)=w_{\mathrm L}(i_k-i_{k-1})+w_{\mathrm R}(i_k-i_{k+1})$. By Lemma~\ref{p1:lem:fiber},
\begin{equation}\label{fl:eq:sum}
\mathsf V(k)=\sum_{y,m}d_{\mathrm{LR}}(y,m)N(y,m)=\sum_Jd_{\mathrm{LR}}\bigl(|J|,m(J)\bigr),
\end{equation}
the last sum running over the independent subsets $J$ of $C$. A valley at $k$ makes both $i_k-i_{k-1}$ and
$i_k-i_{k+1}$ nonpositive, so $\mathsf V(k)>0$ for one weight excludes a valley at $k$.

\subsection{Catalan bounds and the anchored bound}\label{fl:sec}

This subsection proves the lower bounds by which Section~\ref{flx:sec} settles the parameter triples: the rising bound
(Corollary~\ref{fl:cor:rise}), the anchored bound (Lemma~\ref{fl:lem:anchored}) and the Catalan line bound
(Proposition~\ref{flx:prop:catbound}). The Catalan numbers are $\mathrm{Cat}_j=\frac1{j+1}\binom{2j}{j}$ for $j\ge0$.

\begin{lemma}[Catalan maximum]\label{fl:lem:catalan}
For every $j\ge1$, $\displaystyle\max_{m\ge0}\Bigl[\binom{m}{j-1}-\binom mj\Bigr]=\mathrm{Cat}_{j-1}$.
\end{lemma}

\begin{proof}
Put $g(m)=\binom m{j-1}-\binom mj=\binom m{j-1}\frac{2j-1-m}{j}$. Then $g(m)=0$ for $m<j-1$ and
$g(m)\le0$ for $m\ge2j-1$. For $j-1\le m\le 2j-3$ both $g(m)$ and $g(m+1)$ are positive and
\[
\frac{g(m+1)}{g(m)}=\frac{(m+1)(2j-2-m)}{(m-j+2)(2j-1-m)},
\]
and the numerator minus the denominator equals $j\,(2j-3-m)\ge0$. So $g$ is nondecreasing on
$\{j-1,\dots,2j-2\}$, and its maximum is $g(2j-2)=\frac1j\binom{2j-2}{j-1}=\mathrm{Cat}_{j-1}$.
\end{proof}

\begin{lemma}[subset bound]\label{fl:lem:subset}
For every $w_{\mathrm L}\in[0,1]$ and $k\ge1$,
\[
\mathsf V(k)\ \ge\ d_{\mathrm{LR}}(0,a)+\sum_{h=1}^{c}\binom ch\,\min\Bigl\{0,\ \min_{0\le m\le a-h}d_{\mathrm{LR}}(h,m)\Bigr\}.
\]
\end{lemma}

\begin{proof}
In \eqref{fl:eq:sum} the empty set contributes $d_{\mathrm{LR}}(0,a)$. A set $J$ with $|J|=h\ge1$ has
$m(J)\le a-h$ by Lemma~\ref{p1:prop:support}, so it contributes at least the $h$-th minimum
in the display, which is nonpositive; there are $N_h\le\binom ch$ such sets.
\end{proof}

\begin{corollary}[rising bound]\label{fl:cor:rise}
For every $k\ge1$,
\[
i_k-i_{k-1}\ \ge\ \mathsf R(a,c,k):=\binom ak-\binom a{k-1}-\sum_{h=1}^{\min(c,k-1)}\binom ch\,\mathrm{Cat}_{k-1-h}.
\]
\end{corollary}

\begin{proof}
Take $w_{\mathrm L}=1$, so that $d_L(h,m)=\binom m{k-h}-\binom m{k-h-1}$. For $h\ge k$ this is $0$ or $1$.
For $1\le h\le k-1$, Lemma~\ref{fl:lem:catalan} with $j=k-h$ gives $d_L(h,m)\ge-\mathrm{Cat}_{k-h-1}$
for all $m$. Lemma~\ref{fl:lem:subset} gives the claim.
\end{proof}

\begin{lemma}[monotonicity]\label{fl:lem:mono}
$\mathsf R(a,c,k)$ is nonincreasing in $c$, and $\mathsf R(a+1,c,k)\ge\mathsf R(a,c,k)$ whenever $a\ge 2k-3$.
\end{lemma}

\begin{proof}
The first statement holds because $\binom ch$ is nondecreasing in $c$ and $\mathrm{Cat}_j\ge0$. By
Pascal's rule $\mathsf R(a+1,c,k)-\mathsf R(a,c,k)=\binom a{k-1}-\binom a{k-2}$. For $k=1$ this equals $1$; for $k\ge2$
and $a\ge2k-3$ we have $a\ge k-2$ and $\binom a{k-1}\big/\binom a{k-2}=(a-k+2)/(k-1)\ge1$.
\end{proof}

At $a=2k+1$ and $a=2k$ the leading term of $\mathsf R$ is a Catalan number:
\begin{equation}\label{fl:eq:catlead}
\binom{2k+1}{k}-\binom{2k+1}{k-1}=\mathrm{Cat}_{k+1},\qquad \binom{2k}{k}-\binom{2k}{k-1}=\mathrm{Cat}_{k}.
\end{equation}

\paragraph{The anchored bound.}\label{fl:sec:anchored}
Fix $w_{\mathrm L}\in[0,1]$ and $k\ge1$. For integers $j$ and $m\ge0$ put
$G_j(m)=\binom mj-w_{\mathrm L}\binom m{j-1}-w_{\mathrm R}\binom m{j+1}$, so that $d_{\mathrm{LR}}(h,m)=G_{k-h}(m)$; for
$j\le-2$ the function $G_j$ vanishes, $G_{-1}=-w_{\mathrm R}$ is constant and $G_0(m)=1-w_{\mathrm R}m$ is
nonincreasing. An independent set $J\subseteq C$ with $|J|=h$ has $m(J)=a-h-x(J)$, so it contributes
$G_{k-h}\bigl(a-h-x(J)\bigr)$ to \eqref{fl:eq:sum}. A \emph{line set} is a family of numbers $p^\sharp_h$ and
$q^\sharp_h\ge0$, $1\le h\le c$, with
\begin{equation}\label{fl:eq:line}
G_{k-h}(m)\ \ge\ p^\sharp_h-q^\sharp_h\,(a-h-m)\qquad\text{for every integer }0\le m\le a-h.
\end{equation}
The \emph{chords} form a line set: $p^\circ_h=G_{k-h}(a-h)$ and
\[
q^\circ_h=\max\Bigl\{0,\ \max_{0\le m<a-h}\frac{G_{k-h}(a-h)-G_{k-h}(m)}{a-h-m}\Bigr\},
\]
the largest slope of a chord of $G_{k-h}$ from the point $a-h$ to a smaller argument. They depend on
$(a,k,w_{\mathrm L},h)$ only; by the closed forms of $G_0$ and $G_{-1}$, $q^\circ_h=0$ for $h\ge k$, and
$p^\circ_h=q^\circ_h=0$ for $h\ge k+2$. For integers $0\le e_C\le c-1$ let
$\underline{N}_2(e_C)=\overline{N}_2(e_C)=\binom c2-e_C$, and for $3\le h\le c$
\[
\underline{N}_h(e_C)=\max\Bigl\{0,\ \binom ch-e_C\binom{c-2}{h-2}\Bigr\},\qquad
\overline{N}_h(e_C)=\min\Bigl\{\binom ch,\ \binom{c-2}{h-2}\Bigl(\binom c2-e_C\Bigr)\Big/\binom h2\Bigr\};
\]
let $\mathsf B_h(e_C)=(a-1-e_C)\binom{c-2}{h-1}+d\binom{c-2}{h-2}$ for $2\le h\le c$, and, for a line set,
\begin{multline*}
\mathcal A(e_C)=d_{\mathrm{LR}}(0,a)+c\,p^\sharp_1-q^\sharp_1\,(a-1-e_C)\\
+\sum_{h=2}^{c}\Bigl(\min\bigl\{\underline{N}_h(e_C)\,p^\sharp_h,\ \overline{N}_h(e_C)\,p^\sharp_h\bigr\}-q^\sharp_h\,\mathsf B_h(e_C)\Bigr).
\end{multline*}

For a given forest $e_C$ denotes $e(F[C])$; as the argument of $\mathcal A$, and of $\mathsf R^\ast$ below, it is a real
variable in $[0,c-1]$. (Likewise $x$ is the variable of a polynomial and $x(J)$, a function, the excess.)

\begin{lemma}[anchored bound]\label{fl:lem:anchored}
Fix $w_{\mathrm L}\in[0,1]$, $k\ge1$ and a line set. Then:

\textup{(a)} $\mathsf V(k)\ge\mathcal A\bigl(e(F[C])\bigr)\ge\min_{0\le e_C\le c-1}\mathcal A(e_C)$;
\textup{(b)} the formula for $\mathcal A$ defines a convex function of the real variable $e_C\in[0,c-1]$;
\textup{(c)} if $e^{\min}_C\in\{0,\dots,c-1\}$ satisfies $\mathcal A(e^{\min}_C-1)\ge\mathcal A(e^{\min}_C)$
when $e^{\min}_C\ge1$, and $\mathcal A(e^{\min}_C+1)\ge\mathcal A(e^{\min}_C)$ when $e^{\min}_C\le c-2$, then
$\mathcal A(e_C)\ge\mathcal A(e^{\min}_C)$ for every integer $0\le e_C\le c-1$. In particular, if
$\mathcal A(1)\ge\mathcal A(0)$, then $\mathcal A(e_C)\ge\mathcal A(0)$ for every such $e_C$.

\end{lemma}

\begin{proof}
(a) Let $e_C=e(F[C])$; it lies in $\{0,\dots,c-1\}$ by Lemma~\ref{p1:prop:support}. By
\eqref{fl:eq:sum}, the support condition of Lemma~\ref{p1:prop:support} and~\eqref{p1:eq:normal},
\[
\mathsf V(k)=d_{\mathrm{LR}}(0,a)+\sum_{v\in C}G_{k-1}\bigl(a-1-x(\{v\})\bigr)
+\sum_{h=2}^{c}\ \sum_{|J|=h}G_{k-h}\bigl(a-h-x(J)\bigr),
\]
the inner sums over independent sets. By \eqref{fl:eq:line}, $q^\sharp_1\ge0$ and \eqref{p1:eq:ex1}, the singletons
contribute at least $c\,p^\sharp_1-q^\sharp_1\mathrm{ex}_1\ge c\,p^\sharp_1-q^\sharp_1(a-1-e_C)$. For $h\ge2$ the level
$h$ contributes at least $N_h\,p^\sharp_h-q^\sharp_h\,\mathrm{ex}_h$; here
$N_h\,p^\sharp_h\ge\min\{\underline{N}_h(e_C)p^\sharp_h,\overline{N}_h(e_C)p^\sharp_h\}$ by $N_2=\binom c2-e_C$
and~\eqref{p1:eq:Nh}, while $\mathrm{ex}_h\le\mathrm{ex}_1\binom{c-2}{h-1}+d\binom{c-2}{h-2}\le \mathsf B_h(e_C)$
by~\eqref{p1:eq:SM}, where $e(A,C)-c=\mathrm{ex}_1$, and~\eqref{p1:eq:ex1}. Summing gives
$\mathsf V(k)\ge\mathcal A(e_C)$.

(b) On $[0,c-1]$ we have $\underline{N}_h\le\overline{N}_h$: the identity $\binom ch\binom h2=\binom c2\binom{c-2}{h-2}$ gives
$\binom ch-e_C\binom{c-2}{h-2}\le\binom{c-2}{h-2}(\binom c2-e_C)/\binom h2$ for $e_C\ge0$, and $\overline{N}_h\ge0$
because $e_C\le c-1\le\binom c2$. The terms $-q^\sharp_1(a-1-e_C)$, $-q^\sharp_h\mathsf B_h(e_C)$ and
$p^\sharp_2(\binom c2-e_C)$ are affine in $e_C$. For $h\ge3$ the term
$\min\{\underline{N}_h\,p^\sharp_h,\overline{N}_h\,p^\sharp_h\}$ equals $p^\sharp_h\underline{N}_h(e_C)$ if
$p^\sharp_h\ge0$, the product of a nonnegative number and a maximum of affine functions, and equals
$p^\sharp_h\overline{N}_h(e_C)=\max\{p^\sharp_h\binom ch,\ p^\sharp_h\binom{c-2}{h-2}(\binom c2-e_C)/\binom h2\}$ if
$p^\sharp_h<0$. Both are convex.

(c) For an integer $e_C\ge e^{\min}_C+1$, convexity gives
$\mathcal A(e_C)\ge\mathcal A(e^{\min}_C+1)+(e_C-e^{\min}_C-1)\bigl(\mathcal A(e^{\min}_C+1)-\mathcal A(e^{\min}_C)\bigr)\ge\mathcal A(e^{\min}_C)$,
and symmetrically for $e_C\le e^{\min}_C-1$.
\end{proof}

At $e_C=0$ every $h$-subset of $C$ is counted, since $\underline{N}_h(0)=\overline{N}_h(0)=\binom ch$, and
\begin{equation}\label{fl:eq:B0}
\mathcal A(0)=d_{\mathrm{LR}}(0,a)+\sum_{h=1}^{c}\binom ch\,p^\sharp_h-q^\sharp_1(a-1)
-\sum_{h=2}^{c}q^\sharp_h\Bigl((a-1)\binom{c-2}{h-1}+d\binom{c-2}{h-2}\Bigr).
\end{equation}

\paragraph{The Catalan line bound.}\label{flx:sec:bounds}
For $w_{\mathrm L}=1$ two lines replace the constant lower bound of Lemma~\ref{fl:lem:catalan}.

\begin{lemma}[Catalan lines]\label{flx:lem:catline}
For every $j\ge1$ and every integer $m\ge0$,
\[
\binom mj-\binom m{j-1}\ \ge\ \mathrm{Cat}_{j-1}\,\max\{-1,\ m-2j+1\},\qquad
\binom mj-\binom m{j-1}\ \ge\ \mathrm{Cat}_{j}\,(m-2j+1).
\]
\end{lemma}

\begin{proof}
If $m\le 2j-2$, the right sides are $-\mathrm{Cat}_{j-1}$ and at most $-\mathrm{Cat}_j\le-\mathrm{Cat}_{j-1}$, and both claims
follow from Lemma~\ref{fl:lem:catalan}. If $m\ge2j-1$, then $\binom mj-\binom m{j-1}=\frac1j\binom m{j-1}(m-2j+1)$ with
$m-2j+1\ge0$; both claims hold with equality at $m=2j-1$, and for $m\ge2j$ they follow from
$\frac1j\binom m{j-1}\ge\frac1j\binom{2j}{j-1}=\mathrm{Cat}_j\ge\mathrm{Cat}_{j-1}$, as $\binom m{j-1}$ is nondecreasing
in $m$.
\end{proof}

\begin{proposition}[Catalan line bound]\label{flx:prop:catbound}
Let $k\ge2$, and let $\underline{N}_h,\overline{N}_h,\mathsf B_h$ be as in the anchored bound; put $\overline{N}_1(e_C)=c$. For real $e_C\in[0,c-1]$
put $\mathsf Z_1(e_C)=(a-2k+2)\,c-(a-1-e_C)$ and, for $2\le h\le c$,
\[
\mathsf Z_h(e_C)=\min\bigl\{(a-2k+h+1)\,\underline{N}_h(e_C),\ (a-2k+h+1)\,\overline{N}_h(e_C)\bigr\}-\mathsf B_h(e_C).
\]
For $1\le h\le\min(c,k-1)$ put
\[
\mathsf R^\ast_h(e_C)=\max\bigl\{\mathrm{Cat}_{k-h-1}\max\{-\overline{N}_h(e_C),\mathsf Z_h(e_C)\},\ \mathrm{Cat}_{k-h}\,\mathsf Z_h(e_C)\bigr\},
\]
that is, $\mathrm{Cat}_{k-h}\,\mathsf Z_h(e_C)$ if $\mathsf Z_h(e_C)\ge0$ and $\mathrm{Cat}_{k-h-1}\max\{-\overline{N}_h(e_C),\mathsf Z_h(e_C)\}$
otherwise, and
\[
\mathsf R^\ast(e_C)=\binom ak-\binom a{k-1}+\sum_{h=1}^{\min(c,k-1)}\mathsf R^\ast_h(e_C).
\]
Then $i_k-i_{k-1}\ge\mathsf R^\ast\bigl(e(F[C])\bigr)$, and $\mathsf R^\ast$ is convex on $[0,c-1]$.
\end{proposition}

\begin{proof}
Take $w_{\mathrm L}=1$ in~\eqref{fl:eq:sum}, so that $d_L(h,m)=\binom m{k-h}-\binom m{k-h-1}$. The empty set contributes
$\binom ak-\binom a{k-1}$. For $h=k$ we have $d_L(k,m)=1$, and $d_L(h,m)=0$ for $h>k$; these levels contribute
at least $0$. Let $1\le h\le\min(c,k-1)$ and $j=k-h$. An independent $J\subseteq C$ with $|J|=h$ has
$m(J)=a-h-x(J)$ by the definition of $x(J)$, and $m(J)-2j+1=a-2k+h+1-x(J)$. Put
$\mathsf T_h=(a-2k+h+1)N_h-\mathrm{ex}_h$. Lemma~\ref{flx:lem:catline}, applied at $m(J)$ and summed over the $N_h$
such sets, gives
\[
\sum_{|J|=h}d_L\bigl(h,m(J)\bigr)\ \ge\ \max\bigl\{\mathrm{Cat}_{j-1}\max\{-N_h,\ \mathsf T_h\},\ \mathrm{Cat}_j\,\mathsf T_h\bigr\}.
\]
Let $e_C=e(F[C])$. For $h=1$ we have $N_1=c$ by~\eqref{p1:eq:normal} and $\mathrm{ex}_1\le a-1-e_C$
by~\eqref{p1:eq:ex1}, so $\mathsf T_1\ge\mathsf Z_1(e_C)$. For $h\ge2$ we have
$\underline{N}_h(e_C)\le N_h\le\overline{N}_h(e_C)$ by~\eqref{p1:eq:Nh} (and $N_2=\binom c2-e_C$) and $\mathrm{ex}_h\le \mathsf B_h(e_C)$
by~\eqref{p1:eq:SM} and~\eqref{p1:eq:ex1}, so $\mathsf T_h\ge\mathsf Z_h(e_C)$ and $-N_h\ge-\overline{N}_h(e_C)$. Hence, as the right side of the display is
nondecreasing in $\mathsf T_h$ and in $-N_h$, level $h$ contributes at least $\mathsf R^\ast_h(e_C)$, and summing gives $i_k-i_{k-1}\ge\mathsf R^\ast(e_C)$. The two forms of
$\mathsf R^\ast_h$ agree because $\mathrm{Cat}_{j}\ge\mathrm{Cat}_{j-1}\ge0$. Convexity: $\mathsf Z_1$ is affine; for $h\ge2$,
$\mathsf Z_h$ equals $(a-2k+h+1)\underline{N}_h-\mathsf B_h$ if $a-2k+h+1\ge0$ and $(a-2k+h+1)\overline{N}_h-\mathsf B_h$ otherwise, since $\underline{N}_h\le\overline{N}_h$ on
$[0,c-1]$ (proof of Lemma~\ref{fl:lem:anchored}(b)), and both are convex; $-\overline{N}_h$ is convex because $\overline{N}_h$ is a minimum of
affine functions. A maximum of convex functions is convex, and so is a sum.
\end{proof}

\begin{example}\label{flx:ex:catline}
Let $(n,a,k)=(22,16,8)$, so $c=6$, $d=10$, $a=2k$ and $a-2k+h+1=h+1$. At $e_C=0$ every $h$-subset of $C$ is counted,
$\underline{N}_h(0)=\overline{N}_h(0)=\binom6h$, and $\mathsf B_h(0)=15\binom4{h-1}+10\binom4{h-2}$, that is $70,130,120,55,10$ for $h=2,\dots,6$.
The Catalan numbers are $\mathrm{Cat}_6,\dots,\mathrm{Cat}_0=132,42,14,5,2,1,1$, and $\binom{16}8-\binom{16}7=\mathrm{Cat}_8=1430$.
The singletons give $132\max\{-6,\,12-15\}=-396$. For $h=2,\dots,6$ the second entry of the maximum,
$(h+1)\binom6h-\mathsf B_h(0)=-25,-50,-45,-19,-3$, is below the first, $-\binom6h$, so these levels give
$-42\cdot15-14\cdot20-5\cdot15-2\cdot6-1=-998$. Hence $\mathsf R^\ast(0)=1430-396-998=36>0$. Also
$\mathsf R^\ast(1)=3518/15\ge\mathsf R^\ast(0)$, so by convexity $\mathsf R^\ast(e_C)\ge36$ for $0\le e_C\le5$. Corollary~\ref{fl:cor:rise}
alone gives $\mathsf R(16,6,8)=1430-1790<0$: the singletons are charged $6\cdot132$ there and $3\cdot132$ here.
\end{example}

% ---------------------------------------------------------------------------------------------------------------
\subsection{The four classes of parameter triples}\label{flx:sec}\label{fl:sec:twenty}

This subsection splits the $249$ triples of $\cD_{24}$ into four classes, settles each with the bounds of
Section~\ref{fl:sec}, and proves Theorem~\ref{p1:thm:main}. The first class, $\mathcal P_1$, consists of the $145$ triples with $2k\le a-1$; $144$ are settled by the rising bound
(Proposition~\ref{fl:prop:class1}) and one by the Catalan line bound (Proposition~\ref{flx:prop:line}).

The second class, $\mathcal P_2$, consists of the $34$ triples with $a=2k$ and $c\le c_2(k)$, where $c_2(k)$ is given in the second group of
Data~\ref{fl:tab:corners}. Its triples are specified below and are settled by the rising bound
(Proposition~\ref{fl:prop:class2}).

The third class, $\mathcal P_3$, consists of the $28$ other triples with $k\ge h_B(n,a)$, which are the triples of Data~\ref{fl:tab:tail} and
lie in no window (Proposition~\ref{fl:prop:class3}).

The fourth class, $\mathcal P_4$, consists of the remaining $42$ triples. Of these, the seven nonexceptional cases of Data~\ref{flx:tab:line} are settled by the Catalan line bound (Proposition~\ref{flx:prop:line}), and the $35$ cases of Data~\ref{fl:tab:values} are settled by the anchored bound (Proposition~\ref{fl:prop:class4}).

These four classes partition $\cD_{24}$, as is checked by listing the $249$ triples. Thus $206$ triples are settled by
hand, by $46$ comparisons of integers built from Catalan numbers, binomial coefficients and powers of $2$ ($18$ in
Data~\ref{fl:tab:corners} and $28$ in Data~\ref{fl:tab:tail}); $8$ by exact rational evaluation of the Catalan line
bound ($21$ values, from Catalan numbers and binomial coefficients only); and $35$ by exact rational evaluation of the
anchored bound from the printed evaluations in Appendix~\ref{apx:p1tables}, each of which can be checked by hand. All
these comparisons and evaluations, the partition of $\cD_{24}$ and its counts are verified in exact rational arithmetic by the programs of
family~\ref{chk:floor} (Section~\ref{sec:computer}).

\datablock{Catalan inequalities}{fl:tab:corners}
For Proposition~\ref{fl:prop:class1}, we specify $c_1(k)$ and compare $\mathrm{Cat}_{k+1}$ with the sum
\[
 S_1(k)=\sum_{h=1}^{\min(c_1(k),k-1)}\binom{c_1(k)}h\mathrm{Cat}_{k-1-h}.
\]
% exact-data:FlxCornersOne:begin
For $k=3$, we have $(2k+1,c_1(k))=(7,1)$ and $\mathrm{Cat}_{k+1}=14>S_1(k)=1$.

For $k=4$, we have $(2k+1,c_1(k))=(9,3)$ and $\mathrm{Cat}_{k+1}=42>S_1(k)=10$.

For $k=5$, we have $(2k+1,c_1(k))=(11,5)$ and $\mathrm{Cat}_{k+1}=132>S_1(k)=60$.

For $k=6$, we have $(2k+1,c_1(k))=(13,7)$ and $\mathrm{Cat}_{k+1}=429>S_1(k)=329$.

For $k=7$, we have $(2k+1,c_1(k))=(15,8)$ and $\mathrm{Cat}_{k+1}=1430>S_1(k)=1232$.

For $k=8$, we have $(2k+1,c_1(k))=(17,8)$ and $\mathrm{Cat}_{k+1}=4862>S_1(k)=3514$.

For $k=9$, we have $(2k+1,c_1(k))=(19,7)$ and $\mathrm{Cat}_{k+1}=16796>S_1(k)=7855$.

For $k=10$, we have $(2k+1,c_1(k))=(21,5)$ and $\mathrm{Cat}_{k+1}=58786>S_1(k)=12984$.

For $k=11$, we have $(2k+1,c_1(k))=(23,3)$ and $\mathrm{Cat}_{k+1}=208012>S_1(k)=19305$.
% exact-data:FlxCornersOne:end

For Proposition~\ref{fl:prop:class2}, we specify $c_2(k)$ and compare $\mathrm{Cat}_{k}$ with the sum
\[
 S_2(k)=\sum_{h=1}^{\min(c_2(k),k-1)}\binom{c_2(k)}h\mathrm{Cat}_{k-1-h}.
\]
% exact-data:FlxCornersTwo:begin
For $k=3$, we have $(2k,c_2(k))=(6,2)$ and $\mathrm{Cat}_{k}=5>S_2(k)=3$.

For $k=4$, we have $(2k,c_2(k))=(8,3)$ and $\mathrm{Cat}_{k}=14>S_2(k)=10$.

For $k=5$, we have $(2k,c_2(k))=(10,4)$ and $\mathrm{Cat}_{k}=42>S_2(k)=37$.

For $k=6$, we have $(2k,c_2(k))=(12,4)$ and $\mathrm{Cat}_{k}=132>S_2(k)=95$.

For $k=7$, we have $(2k,c_2(k))=(14,5)$ and $\mathrm{Cat}_{k}=429>S_2(k)=411$.

For $k=8$, we have $(2k,c_2(k))=(16,5)$ and $\mathrm{Cat}_{k}=1430>S_2(k)=1247$.

For $k=9$, we have $(2k,c_2(k))=(18,5)$ and $\mathrm{Cat}_{k}=4862>S_2(k)=3960$.

For $k=10$, we have $(2k,c_2(k))=(20,5)$ and $\mathrm{Cat}_{k}=16796>S_2(k)=12984$.

For $k=11$, we have $(2k,c_2(k))=(22,4)$ and $\mathrm{Cat}_{k}=58786>S_2(k)=29876$.
% exact-data:FlxCornersTwo:end

The $34$ triples of $\mathcal P_2$ can be specified by their last coordinate. For $k=3$, the members of $\mathcal P_2$ are $(7,6,3)$ and $(8,6,3)$. For $k=4$, the members of $\mathcal P_2$ are $(9,8,4)$, $(10,8,4)$ and $(11,8,4)$. For $k=5$, the members of $\mathcal P_2$ are $(11,10,5)$, $(12,10,5)$, $(13,10,5)$ and $(14,10,5)$.

For $k=6$, the members of $\mathcal P_2$ are $(13,12,6)$, $(14,12,6)$, $(15,12,6)$ and $(16,12,6)$. For $k=7$, the members of $\mathcal P_2$ are $(15,14,7)$, $(16,14,7)$, $(17,14,7)$, $(18,14,7)$ and $(19,14,7)$. For $k=8$, the members of $\mathcal P_2$ are $(17,16,8)$, $(18,16,8)$, $(19,16,8)$, $(20,16,8)$ and $(21,16,8)$.

For $k=9$, the members of $\mathcal P_2$ are $(19,18,9)$, $(20,18,9)$, $(21,18,9)$, $(22,18,9)$ and $(23,18,9)$. For $k=10$, the members of $\mathcal P_2$ are $(21,20,10)$, $(22,20,10)$, $(23,20,10)$ and $(24,20,10)$. For $k=11$, the members of $\mathcal P_2$ are $(23,22,11)$ and $(24,22,11)$.

\begin{proposition}\label{fl:prop:class1}
Let $c_1(k)$ be as in the first group of Data~\ref{fl:tab:corners}. Every triple $(n,a,k)\in\mathcal P_1$ has
$3\le k\le11$ and $c\le c_1(k)$, except the triple
\begin{equation}\label{flx:eq:exc}
(n,a,k)=(24,15,7),
\end{equation}
which has $c=9$. For every other $(n,a,k)\in\mathcal P_1$, every forest with $n$ vertices and independence number $a$
satisfies $i_k>i_{k-1}$.
\end{proposition}

\begin{proof}
Since $n\ge7$ on $\cD_{24}$ and $k>\lceil n/4\rceil$, $k\ge3$ and $n\le\min(24,4k-4)$; and $2k+1\le a\le n-1\le23$, so
$k\le11$ and $c\le\min(24,4k-4)-(2k+1)$. For $k=3,\dots,11$ this number is $1,3,5,7,9,7,5,3,1$, against
$c_1(k)=1,3,5,7,8,8,7,5,3$: it is at most $c_1(k)$ except at $k=7$, where it is $9$ and $c_1(7)=8$; at $k=7$ the only triple of $\mathcal P_1$ with $c=9$ is $(24,15,7)$ (as $n\le24$ and $a\ge15$). For the
other triples, Lemma~\ref{fl:lem:mono} and~\eqref{fl:eq:catlead} give
$\mathsf R(a,c,k)\ge\mathsf R(2k+1,c,k)\ge\mathsf R(2k+1,c_1(k),k)=\mathrm{Cat}_{k+1}-\sum_{h=1}^{\min(c_1,k-1)}\binom{c_1}{h}\mathrm{Cat}_{k-1-h}$,
which is positive by Data~\ref{fl:tab:corners}, and Corollary~\ref{fl:cor:rise} concludes.
\end{proof}

The rising bound of the triple~\eqref{flx:eq:exc} itself is $\mathsf R(15,9,7)=-334$, so Corollary~\ref{fl:cor:rise}
does not settle it; Proposition~\ref{flx:prop:line} does.

\begin{proposition}\label{fl:prop:class2}
Let $(n,a,k)\in\mathcal P_2$. Every forest with $n$ vertices and independence number $a$ satisfies $i_k>i_{k-1}$.
\end{proposition}

\begin{proof}
Here $a=2k$ and $c\le c_2(k)$. By Lemma~\ref{fl:lem:mono} and \eqref{fl:eq:catlead},
$\mathsf R(2k,c,k)\ge\mathsf R(2k,c_2(k),k)=\mathrm{Cat}_k-\sum_{h=1}^{\min(c_2,k-1)}\binom{c_2}{h}\mathrm{Cat}_{k-1-h}$, which is
positive by Data~\ref{fl:tab:corners}. Corollary~\ref{fl:cor:rise} concludes.
\end{proof}

\begin{proposition}\label{fl:prop:class3}
Let $(n,a,k)\in\mathcal P_3$. Then $k\ge h_B(n,a)$; in particular $k\notin W(F)$ for every forest $F$
with $n$ vertices and independence number $a$.
\end{proposition}

\begin{proof}
The polynomial $(1+2x)^c(1+x)^d$ has only real roots, so its coefficient sequence is log-concave without
internal zeros, and $b_{r+1}/b_r$ is nonincreasing. Data~\ref{fl:tab:tail} gives $b_{k+1}\le b_k$, hence
$b_r$ is nonincreasing for $r\ge k$ and the first mode is at most $k$. By
Section~\ref{p1:sec:param}, $h_B(F)=h_B(n,a)\le k$.
\end{proof}

\datablock{The triples outside the window}{fl:tab:tail}
For each of the $28$ triples of $\mathcal P_3$, we give $c=n-a$, $d=2a-n$ and the coefficients $b_k$ and $b_{k+1}$ of $(1+2x)^c(1+x)^d$. Each comparison verifies $b_k\ge b_{k+1}$.
% exact-data:PTailTwoCol:begin
For $(n,a,k)=(11,7,4)$, we have $c=4$ and $d=3$, and the coefficients satisfy $b_k=192$ and $b_{k+1}=168$.

For $(n,a,k)=(12,7,4)$, we have $c=5$ and $d=2$, and the coefficients satisfy $b_k=280$ and $b_{k+1}=272$.

For $(n,a,k)=(13,9,5)$, we have $c=4$ and $d=5$, and the coefficients satisfy $b_k=681$ and $b_{k+1}=608$.

For $(n,a,k)=(14,9,5)$, we have $c=5$ and $d=4$, and the coefficients satisfy $b_k=1002$ and $b_{k+1}=968$.

For $(n,a,k)=(15,8,5)$, we have $c=7$ and $d=1$, and the coefficients satisfy $b_k=1232$ and $b_{k+1}=1120$.

For $(n,a,k)=(15,11,6)$, we have $c=4$ and $d=7$, and the coefficients satisfy $b_k=2471$ and $b_{k+1}=2241$.

For $(n,a,k)=(16,8,5)$, we have $c=8$ and $d=0$, and the coefficients satisfy $b_k=1792$ and $b_{k+1}=1792$.

For $(n,a,k)=(16,11,6)$, we have $c=5$ and $d=6$, and the coefficients satisfy $b_k=3653$ and $b_{k+1}=3530$.

For $(n,a,k)=(17,10,6)$, we have $c=7$ and $d=3$, and the coefficients satisfy $b_k=4424$ and $b_{k+1}=4048$.

For $(n,a,k)=(17,13,7)$, we have $c=4$ and $d=9$, and the coefficients satisfy $b_k=9108$ and $b_{k+1}=8361$.

For $(n,a,k)=(18,10,6)$, we have $c=8$ and $d=2$, and the coefficients satisfy $b_k=6496$ and $b_{k+1}=6400$.

For $(n,a,k)=(18,13,7)$, we have $c=5$ and $d=8$, and the coefficients satisfy $b_k=13504$ and $b_{k+1}=13073$.

For $(n,a,k)=(19,12,7)$, we have $c=7$ and $d=5$, and the coefficients satisfy $b_k=16172$ and $b_{k+1}=14920$.

For $(n,a,k)=(19,15,8)$, we have $c=4$ and $d=11$, and the coefficients satisfy $b_k=33957$ and $b_{k+1}=31471$.

For $(n,a,k)=(20,12,7)$, we have $c=8$ and $d=4$, and the coefficients satisfy $b_k=23872$ and $b_{k+1}=23392$.

For $(n,a,k)=(20,15,8)$, we have $c=5$ and $d=10$, and the coefficients satisfy $b_k=50445$ and $b_{k+1}=48940$.

For $(n,a,k)=(21,11,7)$, we have $c=10$ and $d=1$, and the coefficients satisfy $b_k=28800$ and $b_{k+1}=26880$.

For $(n,a,k)=(21,14,8)$, we have $c=7$ and $d=7$, and the coefficients satisfy $b_k=59906$ and $b_{k+1}=55692$.

For $(n,a,k)=(21,17,9)$, we have $c=4$ and $d=13$, and the coefficients satisfy $b_k=127699$ and $b_{k+1}=119262$.

For $(n,a,k)=(22,11,7)$, we have $c=11$ and $d=0$, and the coefficients satisfy $b_k=42240$ and $b_{k+1}=42240$.

For $(n,a,k)=(22,14,8)$, we have $c=8$ and $d=6$, and the coefficients satisfy $b_k=88720$ and $b_{k+1}=86784$.

For $(n,a,k)=(22,17,9)$, we have $c=5$ and $d=12$, and the coefficients satisfy $b_k=189970$ and $b_{k+1}=184690$.

For $(n,a,k)=(23,13,8)$, we have $c=10$ and $d=3$, and the coefficients satisfy $b_k=105984$ and $b_{k+1}=99200$.

For $(n,a,k)=(23,16,9)$, we have $c=7$ and $d=9$, and the coefficients satisfy $b_k=224143$ and $b_{k+1}=209762$.

For $(n,a,k)=(23,19,10)$, we have $c=4$ and $d=15$, and the coefficients satisfy $b_k=483483$ and $b_{k+1}=454389$.

For $(n,a,k)=(24,13,8)$, we have $c=11$ and $d=2$, and the coefficients satisfy $b_k=156288$ and $b_{k+1}=154880$.

For $(n,a,k)=(24,16,9)$, we have $c=8$ and $d=8$, and the coefficients satisfy $b_k=332688$ and $b_{k+1}=325360$.

For $(n,a,k)=(24,19,10)$, we have $c=5$ and $d=14$, and the coefficients satisfy $b_k=720005$ and $b_{k+1}=701350$.
% exact-data:PTailTwoCol:end

\begin{proposition}[exact rational checks on $8$ triples, $21$ values of $\mathsf R^\ast$]\label{flx:prop:line}
Let $(n,a,k)$ be the triple~\eqref{flx:eq:exc} or one of the $7$ triples of $\mathcal P_4$ in
Data~\ref{flx:tab:line}, and let $e^{\min}_C$ be as in those records. Then $\mathsf R^\ast(e^{\min}_C)>0$, and
$\mathsf R^\ast(e^{\min}_C\pm1)\ge\mathsf R^\ast(e^{\min}_C)$ for the neighbours in $[0,c-1]$, with the values of
Data~\ref{flx:tab:line}. Consequently every forest with $n$ vertices and independence number $a$ satisfies
$i_k>i_{k-1}$.
\end{proposition}

\begin{proof}
The values follow from the definition of $\mathsf R^\ast$, as in Example~\ref{flx:ex:catline}. By the convexity in
Proposition~\ref{flx:prop:catbound} and the argument of Lemma~\ref{fl:lem:anchored}(c), $\mathsf R^\ast(e_C)\ge\mathsf R^\ast(e^{\min}_C)>0$
for every integer $0\le e_C\le c-1$, and Proposition~\ref{flx:prop:catbound} gives $i_k-i_{k-1}>0$.
\end{proof}

\begin{proposition}[exact rational checks on $35$ triples]\label{fl:prop:class4}\label{flx:prop:class4}
Let $(n,a,k)$ be one of the $35$ triples of Data~\ref{fl:tab:values} ($18$ with $n\le20$ and $17$ with $21\le n\le24$),
and let $w_{\mathrm L}$ and $e^{\min}_C$ be as there. Put $(p^\sharp_h,q^\sharp_h)$ equal to the numbers of the line set
of Data~\ref{fl:tab:level} with the same $(a,k)$ and $w_{\mathrm L}$ for $1\le h\le\min(c,a-1,k+1)$ (every printed
$q^\sharp_h$ is nonnegative), and $p^\sharp_h=q^\sharp_h=0$ for larger $h$. Then~\eqref{fl:eq:line} holds, $\mathcal A(e^{\min}_C)>0$ and
$\mathcal A(e^{\min}_C\pm1)\ge\mathcal A(e^{\min}_C)$ for the neighbours in $[0,c-1]$, with the values of
Data~\ref{fl:tab:values}. Consequently every forest with $n$ vertices and independence number $a$ satisfies
$\mathsf V(k)>0$ and has no valley at $k$.
\end{proposition}

\begin{proof}
Each of these triples has $a\ge k+2$, so $\min(a-1,k+1)=k+1$, and a level $k+2\le h\le c$ has
$G_{k-h}=0$, so that~\eqref{fl:eq:line} holds with $p^\sharp_h=q^\sharp_h=0$. For $h\le k+1$ the values $G_{k-h}(m)$, $0\le m\le a-h$, are read off from Data~\ref{flx:tab:g} ($w_{\mathrm L}=\frac12$)
and~\ref{fl:tab:g} ($w_{\mathrm L}=\frac5{12}$) for $k-h\ge1$, and from the closed forms $G_{-1}=-w_{\mathrm R}$ and
$G_0(m)=1-w_{\mathrm R}m$ otherwise. At a level without an asterisk the line is the chord,
$p^\sharp_h=p^\circ_h=G_{k-h}(a-h)$ and $q^\sharp_h=q^\circ_h$, the largest slope of a chord of $G_{k-h}$ from $a-h$ to
a smaller argument, so that \eqref{fl:eq:line} holds by construction; at a level with an asterisk
\eqref{fl:eq:line} is a comparison of the line with these values. Substituting the lines into the definition of $\mathcal A$ (into \eqref{fl:eq:B0} at $e_C=0$) gives the
values of Data~\ref{fl:tab:values}. Lemma~\ref{fl:lem:anchored} gives
$\mathsf V(k)\ge\mathcal A(e(F[C]))\ge\mathcal A(e^{\min}_C)>0$, and a valley at $k$ would make $\mathsf V(k)\le0$.
\end{proof}

With the weights $w_{\mathrm L}\in\{0,\frac1{24},\dots,1\}$ the chords alone do not suffice: for $(23,14,8)$,
$(24,14,8)$, $(24,15,8)$ and $(24,17,9)$ they give $\min_{0\le e_C\le c-1}\mathcal A(e_C)\le0$ for each such weight, as
the programs of family~\ref{chk:floor} verify.

\begin{example}\label{fl:ex:1284}
For $(n,a,k)=(12,8,4)$ and $w_{\mathrm L}=\tfrac12$ we have $c=d=4$ and $d_{\mathrm{LR}}(0,8)=G_4(8)=14$. From
Data~\ref{flx:tab:g}, $p^\circ_1=G_3(7)=7$, $p^\circ_2=G_2(6)=2$, $p^\circ_3=G_1(5)=-\tfrac12$ and $p^\circ_4=G_0(4)=-1$. The
steepest chords are $q^\circ_1=\bigl(G_3(7)-G_3(5)\bigr)/2=\tfrac94$ and $q^\circ_2=\bigl(G_2(6)-G_2(2)\bigr)/4=\tfrac12$,
while $q^\circ_3=q^\circ_4=0$. With the chords as the line set, \eqref{fl:eq:B0} reads
\[
\mathcal A(0)=14+\bigl(4\cdot7+6\cdot2-4\cdot\tfrac12-1\bigr)-\tfrac94\cdot7-\tfrac12\,(7\cdot2+4\cdot1)=\tfrac{105}{4}>0,
\]
and $\mathcal A(1)=28\ge\mathcal A(0)$.
\end{example}

\begin{proof}[Proof of Theorem~\ref{p1:thm:main}]
Let $F$ be a forest with $n\le24$ vertices. If $n=0$, $I(F;x)=1$; if $F$ has no edge, $I(F;x)=(1+x)^n$, and there is
nothing to prove. Otherwise let $A$ be a maximum independent set and $k\in W(F)$. By Section~\ref{p1:sec:param},
$(n,a,k)\in\cD_{24}$, and $(n,a,k)\notin\mathcal P_3$ by Proposition~\ref{fl:prop:class3}. By
Propositions~\ref{fl:prop:class1}, \ref{fl:prop:class2}, \ref{flx:prop:line} and~\ref{fl:prop:class4}, $F$ has no
valley at $k$, and Corollary~\ref{p1:cor:window} concludes.
\end{proof}

% ---------------------------------------------------------------------------------------------------------------
\section{The computer-assisted steps}\label{sec:computer}\label{sec:computer:families}

This section lists every finite computation on which the proof depends (Data~\ref{tab:computer:summary}), in the manner
of validated numerics~\cite{Tucker}; all other steps are proved in the text. The family labels retain the numbering used
in the archived computational materials described in Section~\ref{sec:computer:data}: family F7, the coverings of
bands of activities from mean $50$ to $400$, is not needed here,
and families F9 and F10, the finite inequalities of the fiber inequality (Appendix~\ref{apx:fiber}), are new.

\paragraph{Kinds of computation.}\label{sec:computer:arith}
\emph{Exact}: integers and rationals, exact comparisons. \emph{Interval}: every real quantity is enclosed in an interval
and every operation rounds outward~\cite[Chs.~1 and~2]{Tucker}: in the formal proof with fixed-point ends of $64$
fractional bits, outside it with rationals rounded outward to a multiple of $2^{-48}$, $2^{-72}$ or $2^{-200}$, $128$-bit or
double-precision intervals, or balls at $96$ to $256$ bits, in which elementary functions are enclosed by series with explicit
remainders and a polynomial inequality on an interval is proved from the signs of its Bernstein coefficients; and, for
families F9 and F10, interval arithmetic at $30$ decimal digits (the library \texttt{mpmath.iv}; Appendix~\ref{apx:fiber})
and, in their independent verification, ball arithmetic at $212$ bits (the library Arb). \emph{Kernel evaluation}: a decidable statement about
explicit data is evaluated by the Lean kernel (trusted base: the kernel and the three standard axioms). \emph{Compiled
evaluation}: a checker is run on explicit data by compiled code and the kernel accepts the result on the authority of the
compiler (adding the compiler and, for the checkers of families~F3, F4 and~F5, the C toolchain to the trusted base); each checker
has a soundness theorem, proved for arbitrary data. Floating-point searches that proposed constants are part of no proof.

\makeatletter\newcommand{\famlab}[2]{\def\@currentlabel{#1}\phantomsection\label{#2}}\makeatother
\datablock{The finite computations of the proof}{tab:computer:summary}

\paragraph{Family F1.}\famlab{F1}{chk:floor} This family treats small forests through Propositions~\ref{flx:prop:line} and~\ref{fl:prop:class4}. It lists and partitions the $249$ triples of $\cD_{24}$ and evaluates $21$ values of the Catalan line bound at $8$ triples and $74$ values of the anchored bound at $35$. These computations use exact arithmetic. They are not in Lean (the formal proof treats $n\le24$ by $249$ certificates checked by the kernel). Verification is provided by the two programs of the earlier version ($n\le20$; $21\le n\le26$), whose outputs contain the printed evaluations (restricted here to $n\le24$, the corner evaluations to $k\le11$, and the values of $G_j$ to the $j$ and $m$ used); by an independent recomputation of every printed number, line set, the partition and the counts for $n\le26$; and by a separate exact program, written by an independent AI instance from the statements of the previous version, recomputing the partition and every printed number, and rerun, adapted, on the data and definitions of this version ($1{,}373$ checks, $0$ failures).

\paragraph{Family F2.}\famlab{F2}{chk:act} This family treats the integer cases of Lemma~\ref{apC:lem:cases}, comprising $181+247=428$ inequalities in exact arithmetic. The inequalities are verified by the Lean kernel and by a program written from the text, which finds least values $5592/25$ and $13693/20$, at $(2,2)$.

\paragraph{Family F3.}\famlab{F3}{chk:var} This family verifies the variance bound of Theorem~\ref{tgt:thm} through $4{,}321$ checks: $2{,}664$ comparison, $1{,}644$ polygon and $13$ base checks (Data~\ref{mc:tab:curve}). The checks use interval arithmetic at $2^{-48}$ and $64$ bits in Lean. Lean verifies them in $18$ compiled evaluations, and they are also verified by the two programs that produced the checks (mid; the other five segments). An independent verifier on mid passes all $1{,}283$ comparison checks with the halves (least $6.39\cdot10^{-4}$), the $6$ base checks (least $7.14\cdot10^{-5}$), the $19$ polygons with $911$ subdivisions of its own and the $38$ feasibility cuts. With one evaluation per check instead of the halves it found $226$ failures, which is why the text states the halves; its rerun on the final data has no failure. A second independent verifier at $2^{-72}$ passes all $1{,}381$ comparison checks of the other segments without refinement, their $7$ base checks (least $5.8\cdot10^{-3}$) and all $26$ polygons with its own subdivision ($1{,}756$ pieces).

\paragraph{Family F4.}\famlab{F4}{chk:O2box} This family verifies the variance ratio in Lemma~\ref{apB:lem:O2box} on $17$ intervals with $12$ root boxes each, comprising $37{,}399+1{,}554=38{,}953$ cells. The checks use interval arithmetic at $64$ bits and are verified in Lean by $204$ compiled evaluations. All $84{,}156$ pieces of the certificates of the $14$ intervals, in the format of the program that produced them, were replayed by separate code in $128$-bit outward-rounded interval arithmetic ($0$ failures; least $\Phi\ge2.5\cdot10^{-4}$, on $[23/20,13/10]$). The three additional intervals were replayed by the program that produced them and by separate $128$-bit code. For $[1/3,3/5]$ the formal check uses a regenerated cover of the same parameter data ($4{,}221$ cells), not the replayed certificate ($50{,}978$ pieces).

\paragraph{Family F5.}\famlab{F5}{chk:O3box} This family verifies the Laplace bounds of Lemma~\ref{apC:lem:O3box}(a), (b), comprising $150$ band cases in $30$ evaluations and $26$ additional cases. It uses interval arithmetic at $64$ bits. Lean verifies the bounds in $56$ compiled evaluations, and the $150$ band certificates were replayed by the program that produced them. A program written from the text verifies all $176$ cases on its own cover ($376{,}180$ cells, balls at $96$ bits, the regions $Y_A\le2^{-10}$ and $Y_A\ge32$ closed by explicit inequalities) and certifies $U\ge1.1\cdot10^{-3}$ in every case.

\paragraph{Family F6.}\famlab{F6}{chk:strips} This family verifies the coverings below $50$ in Lemma~\ref{p2:lem:strips}. It comprises $110$ coverings and $779$ boxes, of which $693$ are used (Theorem~\ref{p2:thm:O4}), and uses exact arithmetic. Lean verifies all $110$ coverings in $112$ compiled evaluations, subintervals $16$ and $18$ being checked in two parts each. For the $36$ subintervals $0$ to $14$, $20$ to $23$, $25$, $26$, $28$ to $37$, $40$ to $42$, $45$, $46$, it uses the coverings below $m_\ast(\iota)$ of the earlier formal proof, comprising $420$ boxes, at caps at least $D(\iota)$. Those $420$ boxes were also replayed by the program that produced them and by an independent checker. The $33$ boxes rechecked at a larger cap and the $97+229$ new boxes were also verified by two independently written exact programs and again by an independent AI instance with its own exact checker (least margin of the $229$: $2.4\cdot10^{-5}$). The $693$ used boxes were counted in exact arithmetic by the supplementary code of family~F9 (Appendix~\ref{apx:fiber}).

\paragraph{Family F8.}\famlab{F8}{chk:kernel} This family verifies the kernel facts of Lemmas~\ref{cs:lem:O2side}, \ref{cs:lem:tables} and~\ref{apC:lem:O3box}(c), comprising $17$, $55$ and $176$ cases, respectively. It uses exact arithmetic. The facts are verified by the Lean kernel and by programs written from the text, with least margins $3.83\cdot10^{-4}$ for Lemma~\ref{cs:lem:O2side}, and $0.0234$ and $0.168$ for Lemma~\ref{apC:lem:O3box}(c).

\paragraph{Family F9.}\famlab{F9}{chk:fibcells} This family verifies the fiber inequality through Lemmas~\ref{fib:chk:A0} to~\ref{fib:chk:D3} (Appendix~\ref{apx:fiber}), comprising $218$ cells ($85$ statements) in $20$ lemmas. It uses interval arithmetic at $30$ digits and exact arithmetic for Lemmas~\ref{fib:chk:A0} and~\ref{fib:chk:D0}. Together with family F10, these inequalities are formalized in Lean using $21$ compiled evaluations (Section~\ref{sec:verification}). The supplementary code comprises five checkers, taking about $15$ to $25$ seconds in all and producing one log line per cell, every line passing; it was rerun independently with byte-identical logs. A program written independently from the statements alone, with a different arithmetic library (ball arithmetic at $212$ bits, and exact rationals) and the root map computed by a rigorous Lambert $W$ function, verifies all $218$ cells on their whole stated ranges and all $220$ bounds of Data~\ref{fib:tab:b3}.

\paragraph{Family F10.}\famlab{F10}{chk:fibrows} This family verifies the averaging margins of Lemma~\ref{fib:chk:rows}, comprising $55$ inequalities, one per subinterval, using interval arithmetic at $30$ digits. These inequalities are formalized in Lean together with family F9; the formal proof establishes positivity of the $55$ margins, rather than their printed numerical values. The supplementary code was rerun independently with byte-identical logs. The independent program of family F9 finds all $55$ margins positive with its own spike bounds (least $0.00137$, at $\iota=36$, as in Data~\ref{fib:tab:mu0}).

\paragraph{Inherited coverings.} The formal proof reuses the coverings below $m_\ast(\iota)$ of $36$ subintervals from
the earlier formal proof; it proves, by kernel evaluation of the rational comparisons, that they were verified at a cap
at least $D(\iota)$ on a rectangle whose range of $\muM$ starts at or below $m_{25}(\iota)$ and ends at or above $m_\ast(\iota)$ (for these $36$, at it). The same
comparison was run outside Lean by the programs that built this route and rerun by an independent AI instance.

\paragraph{Explicit numerical inequalities.}\label{sec:computer:numeric}
Rechecked in ball arithmetic at $200$ bits: $\log2^{18}\le25/2$ and $e^9\le8104$ (Lemma~\ref{apB:lem:small}); both are also
proved in the formal proof. The hand proofs of Section~\ref{fib:sec} use further elementary numerical facts (values of
logarithms, exponentials and square roots at explicit points), each stated with its value where it is used; they were
recomputed in the independent checks of that section (Section~\ref{sec:ack}).

\paragraph{The supplementary data.}\label{sec:computer:data}
The data that the families use but no proof by hand needs are available as \texttt{twin\_v2.2\_data.txt}, \texttt{supplementary\_data.zip}, and \texttt{supplementary\_code.zip} in the \href{https://github.com/zhangzenozhang-jpg/forest-unimodality-v2.2-supplement/tree/9f4f1ecfc439fdf43f8cfad0e888e54fe778d428}{supplementary-material repository}, at commit \texttt{9f4f1ecfc439fdf43f8cfad0e888e54fe778d428}. This section records their role in the proof; the repository README and the READMEs inside the archives give the prerequisites and reproduction instructions. The complete archives are hosted in that repository; the arXiv ancillary file supplies the same fixed link and checksums.
The text file holds the end points of the $55$ subintervals and the $30$ bands with the constants of each subinterval
($\ell_0(\iota)$, $m_{25}(\iota)$, $m_\ast(\iota)$, box and slab counts), the rates $\ell_r$ and constants $a_{\mathrm U}$ of the
bands and of the $26$ additional intervals, the $47$ activity sub-intervals and $45$ child polygons of family~F3, the $31$
coverings above $50$ retained as historical family F7, which neither the present proof nor its formal proof uses, with
their entry counts and where each is verified, and where every data set is. Its notes on the formal proof were written
for an earlier draft and describe the earlier formal proof on branch \texttt{one-route-v21} (Section~\ref{sec:verification}). The first archive holds the data sets of
families F1, F3 and F6 and of family F7 of the earlier draft, the programs that produced them and the independent checkers of all
families except F9 and F10, with a description and a manifest of SHA-256 sums of its files; the data of the families verified
in the formal proof are in the public Lean repository identified in Section~\ref{sec:verification}. The second archive holds the supplementary code of families F9 and F10
(Appendix~\ref{apx:fiber}): the checkers, their logs, the count of the boxes of family F6 that are used, and a manifest of
SHA-256 sums of its files. Data~\ref{tab:computer:sha} gives the SHA-256 sums of the text file, of the first archive's
manifest, and of the second archive and its manifest.
\datablock{SHA-256 sums of the supplementary data}{tab:computer:sha}

The SHA-256 value for the plain text data file is\par
\noindent\texttt{6ac570f5ec8423f44ebf43fdf0d20aba8c4f4d748d8ff7df332e9013a6d12314}.

The SHA-256 value for the manifest of the first archive is\par
\noindent\texttt{9b9a8da401237788dbe6829f0470106af25da6d0ddd6e0daa2b81d45c348bd4b}.

The SHA-256 value for the second archive (\texttt{supplementary\_code.zip}) is\par
\noindent\texttt{c7e637c94d1aff841a903fa20b1b044a983aa52089e127255385bca46f772819}.

The SHA-256 value for the manifest of the second archive (\texttt{SHA256SUMS.txt}) is\par
\noindent\texttt{8f586be2d9795967844685a3263980671e8541614ab6d4a538f9f1bafb00770e}.

\section{Formal verification}\label{sec:verification}

{\sloppy\paragraph{The formal statement.} The project provides a formal proof of the theorem that every finite forest has a unimodal independence
sequence in the Lean~4 proof assistant~\cite{Lean4} with the library Mathlib~\cite{Mathlib}
(toolchain and Mathlib \texttt{v4.28.0}). Its final theorem, in the file \texttt{Erdos993Lean/Analytic/V22/Final.lean},
is \texttt{erdos993\_v22\_final : Erdos993Statement}
in the namespace \texttt{Erdos993Lean.Analytic.V22}; it has no premise. Below, \texttt{V22} abbreviates the namespace \texttt{Erdos993Lean.Analytic.V22}, and names of the
formal proof are given relative to the namespace \texttt{Erdos993Lean}. A \texttt{FiniteForest} is an acyclic
\texttt{SimpleGraph (Fin n)}, \texttt{independenceCount F k} counts the independent sets of size $k$, and
\texttt{Erdos993Statement} is the weak unimodality of $i_0,\dots,i_{\alpha(F)}$ for every finite forest $F$; acyclicity,
independent sets and the independence number are those of Mathlib (file \texttt{Erdos993Lean/Statement.lean}). The final
theorem joins, by \texttt{erdos993\_of\_floor\_of\_ceiling}, the formal versions of Theorem~\ref{p1:thm:main},
\texttt{Fiber.floorStatement\_25\_fiber} (every forest with fewer than $25$ vertices), and of
Theorem~\ref{p2:thm:main}, \texttt{V22.ceilingStatement\_25} applied to \texttt{V22.parameterNoValley\_checked} (every
forest with at least $25$ vertices).\par}

{\sloppy\paragraph{Forests with at least $25$ vertices.} Here the formal proof follows
Sections~\ref{p1:sec:inventory} to~\ref{p2:sec:atlas} and Appendix~\ref{apx:fiber} statement by statement: every numbered
result of these sections has a formal counterpart with the same statement or a more general one, with the following
exceptions. Remarks~\ref{apB:rem:O2not}, \ref{fib:rem:P58}, \ref{fib:rem:Aii} and~\ref{fib:rem:diag} have none and need
none (Remark~\ref{fib:rem:Aii} states a variant that the proof does not use). Lemma~\ref{p1:lem:ranklogic} is formalized
for sequences of natural numbers, the case that Corollary~\ref{p1:cor:window} uses, and Definition~\ref{p1:def:valley}
without its notion of a strict valley, which the proof does not use. A few results are formalized through their
constituents, not as single combined statements: Lemmas~\ref{tgt:lem:dY}, \ref{tgt:lem:tr}, \ref{tgt:lem:dmu}
and~\ref{tgt:lem:hand} through their individual statements; Lemmas~\ref{apC:lem:valley}, \ref{cs:lem:tables}, \ref{apB:lem:records},
\ref{cs:lem:O2side} and~\ref{apC:lem:atoms} through separate statements of their parts; and
Theorem~\ref{fib:thm:intro-novalley} through three declarations, \texttt{V22.Analysis.source\_central\_noValley},
\texttt{V22.Analysis.source\_wing\_noValley} and \texttt{V22.Analysis.source\_outer\_noValley}, the no-valley
statements of Theorems~\ref{fib:thm:central} and~\ref{fib:thm:wing} and Corollary~\ref{fib:cor:outer}, which together
cover the $55$ subintervals. Theorem~\ref{p2:thm:O4} is \texttt{V22.parameterNoValley} applied to
\texttt{V22.parameterNoValley\_checked}; the latter is case~(b) of the proof of Theorem~\ref{p2:thm:O4}, and it applies
these three declarations with the finite inequalities of Appendix~\ref{apx:fiber}. In case~(a), below the starting
means, the formal proof checks all $779$ boxes of Lemma~\ref{p2:lem:strips}; they include the $693$ boxes that the proof
of Theorem~\ref{p2:thm:O4} uses (Section~\ref{apC:sec:families}). The lemmas of Appendix~\ref{apx:fiber} are formalized
as the inequalities they state, without the margins and values that they print.\par}

\paragraph{Forests with at most $24$ vertices.} Here the formal proof does not follow Section~\ref{p1:sec:overview}.
It proves Theorem~\ref{p1:thm:main} by a different method: $249$ exact rational certificates of a linear relaxation of
the fiber numbers, one for each triple of the parameter set~\eqref{p1:eq:domain} ($130$ with $n\le20$ and $119$ with
$21\le n\le24$), checked by kernel reduction. Like Section~\ref{p1:sec:overview}, it enumerates no forest. The linear relaxation is built from the first results of Section~\ref{p1:sec:overview}:
Lemmas~\ref{p1:lem:matching} and~\ref{p1:lem:exposed} and Theorem~\ref{p1:thm:matching} have formal counterparts. The
rest of the argument of Section~\ref{p1:sec:overview}, its results from Lemma~\ref{fl:lem:catalan} to
Example~\ref{fl:ex:1284}, has no formal counterpart.

\paragraph{Axioms and compiled evaluations.} The project's recorded \texttt{\#print axioms} audits give, for the final theorem, exactly
\texttt{propext}, \texttt{Classical.choice}, \texttt{Quot.sound}, \texttt{Lean.ofReduceBool} and
\texttt{Lean.trustCompiler}, and report no \texttt{sorry}. The last two axioms arise from a reported $411$ compiled
evaluations (\texttt{native\_decide}) of finite checks, so the trusted base includes the Lean compiler. All of these
evaluations lie in the part for at least $25$ vertices: $18$ for the variance bound (family~\ref{chk:var} of
Section~\ref{sec:computer}), $204$ for the variance-ratio bound (family~\ref{chk:O2box}), $56$ for the
Laplace-transform bounds (family~\ref{chk:O3box}), $112$ for the coverings of Lemma~\ref{p2:lem:strips}
(family~\ref{chk:strips}), and $21$ for the finite inequalities of Appendix~\ref{apx:fiber}
(families~\ref{chk:fibcells} and~\ref{chk:fibrows}). The recorded audits report that the part for at most $24$
vertices uses only the first three axioms.

\paragraph{Builds and review scope.} The project records report a fresh-clone rebuild on 3 October 2026 with an
identical dependency closure, and a second fresh build on another machine completed on 5 October 2026. The second build
is reported to have compiled the import closure of the final theorem ($1{,}352$ project modules) with the source option
settings, a memory limit of $8{,}192$ MB and a default heartbeat limit raised to $400{,}000$. These limits bound memory
and time; they do not change the proposition being proved. Both builds are reported to have printed exactly the five
axioms above.

{\sloppy\paragraph{Where the proof is.} The public formal-proof repository is
\url{https://github.com/selfreferencing/erdos993-forest-unimodality-lean}, at commit \texttt{d174e4b} of 6 October 2026
(full identifier \nolinkurl{d174e4b6f8963abb349995606783200879aae74e}). Source inspection confirms that the repository
contains the final theorem and the axiom-check file \texttt{Audit/AxiomsV22.lean}, together with the README and build
files. Its $1{,}352$ project modules match the repository's recorded import-closure list. Git blob comparisons confirm
that all $1{,}352$ source files are byte-identical to the same files at development commit \texttt{63040a09}
(branch \texttt{one-route-v22-final}). The project records additionally identify commit \texttt{d3cf284} on branch
\texttt{one-route-v22-final-clean-rebuild} with the rebuild of 3 October.\par}

\paragraph{Earlier formal proofs.}
The public version~1.0 formal proof (\texttt{Erdos993Lean.Analytic.erdos993}; commit \texttt{c524aa2} of the
repository \url{https://github.com/selfreferencing/erdos993-lean}, public since 29 September 2026) proves the same
statement, with the same five axioms and $271$ \texttt{native\_decide} evaluations, all in its argument for
forests with at least $61$ vertices. Its finite part, for at most $60$ vertices, uses the certificates of Zhang and
Li~\cite{ZL}, checked by kernel reduction.

An earlier formal proof (branch \texttt{fiber-floor-21}, commit \texttt{aa8ba75}) follows an earlier route of this
proof, of 29 September 2026. It splits at $21$ vertices and uses a variance bound certified by coverings
($2{,}142{,}769$ cells), three density certificates, one of them with the constants of~\cite[(FK14)]{ZL}, the earlier
box lists and coverings, and a second criterion that prices the distribution function of $M$. The formal proof
described above shares its general theory, its checks of the variance-ratio and Laplace-transform bounds and part
of the earlier box lists (the coverings below $m_\ast(\iota)$ of $36$ subintervals and $33$ rechecked boxes;
Section~\ref{sec:computer}), and replaces the other parts.

An intermediate formal proof (branch \texttt{one-route-v21}, commit \texttt{2d2ccf6}) follows the route of the formal
proof described above except in the proof of Theorem~\ref{p2:thm:O4} from the starting means on, and in an additional hypothesis of
that theorem, the Laplace-transform bounds at the rates of the bands of activities, which it proves as a further clause
of Theorem~\ref{p2:thm:O3}. Instead of Theorem~\ref{fib:thm:intro-novalley} it uses the boxes of
Lemma~\ref{p2:lem:strips} up to mean $50$; coverings of the bands from mean $50$ to $400$, checked in a floor variant of
the box test that rests on a lower bound for the valley kernel by binomial Fourier estimates; and above mean $400$ a
large-mean theorem proved by the same estimates. The project reports that a traversal of the dependency closure of the final theorem described
above ($56{,}616$ constants) finds no large-mean theorem, no Fourier estimate, no floor variant of the box test and none
of the $31$ coverings of bands.

% ---------------------------------------------------------------------------------------------------------------
\section{Author contributions and use of AI}\label{sec:ack}

\paragraph{Scope of this account.} The historical development, contributions, checks, and review status described
below are reported in the project account supplied by K.~Vallier on 6 October 2026. The account distinguishes human
contributions from work and reviews performed by AI systems.

\paragraph{Author contributions.} T.~Zhang and W.~Li gave the earlier proof of the theorem of this paper~\cite{ZL}
(see also~\cite{Z}). The present proof starts from their work: the decomposition of the independent sets of a forest
relative to a fixed independent set (Lemma~\ref{p1:lem:fiber}), the conditional binomial mixture over an independent set
(Proposition~\ref{p2:prop:mixture}), and their Lemma~3.1, the bounds~(a) and~(b) of
Proposition~\ref{p1:prop:inputs}, which confine a valley of the independence sequence to the window of
Section~\ref{p1:subsec:window}. Zhang and Li credit the initial segment of that lemma to Fang, Lu, Nevo, Yao and
Zheng~\cite[Prop.~8.2]{FLNYZ}. The proof also uses several of their counting inequalities
(Section~\ref{p1:sec:overview}), the forms of their analytic bounds and of their no-valley certificates, and their
innovation identity and competitor construction for the variance-ratio bound; the specific credits are given where
each is used. The supplied account credits T.~Zhang with the introductory text through
Theorem~\ref{thm:main}, W.~Li with reviewing that text, and K.~Vallier with directing the project described below.
This attribution does not assert that every author has reviewed the complete integrated manuscript.

\paragraph{Use of AI.} The project account attributes all of the new mathematics of this paper and all of its formal
proofs to AI systems in a project directed by K.~Vallier. The systems were Claude agents, running Anthropic's models
Claude Opus~5.5 and Claude Fable~5.1; two OpenAI Codex agents, called Sol and Astra in the project, which over the period
ran the OpenAI models GPT-5.6 Sol, GPT-6 Astra in Codex mode, GPT-6 Sol and GPT-6.1 Sol; and ChatGPT conversations with
the OpenAI model GPT-6 Astra in Pro mode, which the project's records call GPT-6 Pro.

\emph{The text.} Apart from the supplied introductory text, the account attributes most of the manuscript's prose to
AI systems, chiefly Claude agents working in the project. The edits of 1 and 2 October 2026 for the route of this proof
and for the fiber inequality are described below. For preparation of the source supplied on 6 October 2026, the account
reports that Claude agents running Claude Opus~5.5 incorporated T.~Zhang's introduction, revised its continuation,
Remark~\ref{fib:rem:P58} and Sections~\ref{sec:computer} to~\ref{sec:ack}, and added a provisional note, a sentence about
formalization to the abstract, author information and new references. These statements concern that supplied source;
the present integration incorporates selected material from it. The source states that AI assistance with the text
was used to enable community examination of the results while rewriting continues, and that the authors intend to
complete that rewriting before journal submission.

\emph{The proofs by hand} were produced on 28 to 30 September 2026: the matching lemma, the exposed-leaf matching and
the matching-excess inequality (Lemmas~\ref{p1:lem:matching} and~\ref{p1:lem:exposed},
Theorem~\ref{p1:thm:matching}) on 29 September in a ChatGPT conversation (model recorded as GPT-6 Pro), with the
certificates that the formal proof uses for forests with at most $24$ vertices; the partition of the small forests,
the Catalan bounds, the anchored bound and the Catalan line bound (Sections~\ref{fl:sec} and~\ref{flx:sec}), and the
step certificates along the curve of multipliers with their transport in the activity (Section~\ref{apB:sec:O1}), by
Claude agents. The project account reports that each of these sections was checked before use by an independent
recomputation of all its finite checks, written from an earlier draft of it, which this paper condenses.
Section~\ref{p1:sec:overview} was also recomputed by a separate program, written from the statements of an earlier
draft and rerun, adapted, on the numerical data and definitions of this paper.

\emph{The moment argument} was developed on 27 to 29 September 2026, except the step certificates of the variance
bound by hand (30 September) and the changes of 1 October: the Laplace-transform bounds, the no-valley inequality on
the parameter domain, the lower bound for the mean free count, the activity selection and the first covering of the
parameter domain by rational boxes, by Sol; the reserve induction from a step certificate, used in the formal proof,
and its computer-certified step certificates, used in the earlier formal proof on branch \texttt{fiber-floor-21}
(Section~\ref{sec:verification}), by Astra; the local payment argument for the variance-ratio bound in a ChatGPT
conversation (model GPT-6 Pro), prompted by Astra and assembled by Sol, with the certificates for the intervals that
argument left uncovered generated by a Claude agent; the MGF criterion, the criterion that prices the distribution
function of $M$ (used in that earlier formal proof) and the extension of the argument to smaller orders by Claude
agents (models Opus~5.5 and Fable~5.1). The account reports checks of each of these results by Claude agents, including
written reviews and independent replays of the certificates.

\emph{The route of this proof, except above the starting means,} was fixed on 1 October 2026 by Claude agents: the
tests showing that the variance bound proved by hand suffices once the moment argument starts at $25$ vertices and
the MGF criterion is used higher on $14$ subintervals, the new and rechecked boxes, the reduction to the MGF criterion
alone (with $229$ new boxes), and the edit of the paper for that route. The account reports that each new or rechecked
box was verified by two independently written programs, and that for each test an independent AI instance reran
every new or rechecked box with its own exact checker.

\emph{The fiber inequality} (Section~\ref{fib:sec}) was produced on 1 and 2 October 2026 by AI agents running Claude
Opus~5.5 in the project: the activity-one bound on 1 October; the constants by hand and their closed form on 2 October;
and the assembly into one argument, with the edit of this paper, on 2 October. Its development used two tools from
exploratory analyses with GPT-6 that K.~Vallier ran on 1 October: the weighted kernel identity
(Lemma~\ref{fib:lem:T1}) and a Catalan minorant for the curvature of the symmetric binomial law. The supplied account
reports checks by independent AI reviewers (Claude Opus~5.5): the three sources of the argument (the activity-one
bound, the fiber inequality with explicit constants, and its closed form) were reviewed by agents who reran all the
code and recomputed the thinnest margins with code of their own; the assembled argument was reviewed against those
sources and the logs by a further AI reviewer, who reran all the checkers.

\emph{The formal proofs} of Section~\ref{sec:verification} were written by Claude agents and, from 1 to 3 October 2026,
by OpenAI Codex agents running the Codex models named above. The account reports a rebuild of the formal proof from a
fresh clone on 3 October 2026 by an OpenAI Codex agent. It reports Lean-kernel checking together with compiled code
for $411$ compiled evaluations in the formal proof described for this paper. The $249$ certificates for forests with
at most $24$ vertices ($130$ with $n\le20$, $119$ with $21\le n\le24$) were produced on 29 September 2026 in a ChatGPT
conversation (model recorded as GPT-6 Pro) from a prompt prepared in the project; according to the account, all were
verified before use by a replayed checker and by independent code.

\paragraph{Review.} At the time the supplied account was prepared on 6 October 2026, it reported that no external
referee had reviewed any part of the paper, no human had reviewed the formalization, the proofs by hand or the fiber
inequality, and no human reader had reviewed the moment argument in full. It reported that T.~Zhang and W.~Li had not
reviewed that source beyond the introductory text credited to T.~Zhang. These statements describe the reported status
of that source at that time and do not establish the extent of any subsequent human review.

Within the project, the account reports that the written proofs of the analytic bounds, the criterion and the
soundness theorems for boxes were reviewed by AI agents before use, with the MGF criterion checked step by step
against its source. It reports that every family of finite checks of Section~\ref{sec:computer} verified by the formal
proof was also replayed outside it by a program written from the text, with every item passing. For the interval
$[1/3,3/5]$ of the variance-ratio bound, the replay used a different cover of the same parameter interval; for
families~\ref{chk:fibcells} and~\ref{chk:fibrows}, the program was written independently from their statements alone,
with a different arithmetic library (ball arithmetic at $212$ bits, and exact rationals). These AI reviews and
computational replays are distinct from external peer review and human examination of the proofs.

\paragraph{Editorial assistance.} OpenAI Codex assistants helped integrate the introduction supplied by T.~Zhang,
check bibliographic metadata and place citations, convert tables and lists to prose, and prepare the manuscript.
They also helped incorporate explanatory material from K.~Vallier's source and add the paragraph describing the
organization of the paper. The authors' and AI systems' mathematical contributions are described above.

\appendix
\section{Exact evaluations for small forests}\label{apx:p1tables}

This appendix gives the exact rational evaluations of Section~\ref{flx:sec}: the values of the Catalan
line bound at the $8$ triples of Proposition~\ref{flx:prop:line}, and the values, the line sets and the functions $G_j$
of the anchored bound for the $35$ triples of Proposition~\ref{fl:prop:class4}. Every number was computed in
exact rational arithmetic and recomputed independently (Section~\ref{sec:computer}).

\datablock{Catalan line evaluations}{flx:tab:line}
We first evaluate the exceptional triple~\eqref{flx:eq:exc}, followed by the seven triples belonging to $\mathcal P_4$. In each case we specify $c$, $d$, $2k-a$ and the minimizing argument $e^{\min}_C$, and give the bound there and at every adjacent integer in $[0,c-1]$.

% exact-data:FlxLine:begin
For $(n,\allowbreak a,\allowbreak k)=(24,\allowbreak 15,\allowbreak 7)$, we have $(c,\allowbreak d,\allowbreak 2k-a,\allowbreak e^{\min}_C)=(9,\allowbreak 6,\allowbreak -1,\allowbreak 3)$ and $\mathsf R^\ast(e^{\min}_C)=11529/2$. The adjacent values are $\mathsf R^\ast(e^{\min}_C-1)=17531/3$ and $\mathsf R^\ast(e^{\min}_C+1)=17518/3$.

For $(n,\allowbreak a,\allowbreak k)=(21,\allowbreak 14,\allowbreak 7)$, we have $(c,\allowbreak d,\allowbreak 2k-a,\allowbreak e^{\min}_C)=(7,\allowbreak 7,\allowbreak 0,\allowbreak 0)$ and $\mathsf R^\ast(e^{\min}_C)=252$. The sole admissible neighbouring value is $\mathsf R^\ast(e^{\min}_C+1)=335$.

For $(n,\allowbreak a,\allowbreak k)=(22,\allowbreak 14,\allowbreak 7)$, we have $(c,\allowbreak d,\allowbreak 2k-a,\allowbreak e^{\min}_C)=(8,\allowbreak 6,\allowbreak 0,\allowbreak 1)$ and $\mathsf R^\ast(e^{\min}_C)=812$. The adjacent values are $\mathsf R^\ast(e^{\min}_C-1)=860$ and $\mathsf R^\ast(e^{\min}_C+1)=1008$.

For $(n,\allowbreak a,\allowbreak k)=(22,\allowbreak 16,\allowbreak 8)$, we have $(c,\allowbreak d,\allowbreak 2k-a,\allowbreak e^{\min}_C)=(6,\allowbreak 10,\allowbreak 0,\allowbreak 0)$ and $\mathsf R^\ast(e^{\min}_C)=36$. The sole admissible neighbouring value is $\mathsf R^\ast(e^{\min}_C+1)=3518/15$.

For $(n,\allowbreak a,\allowbreak k)=(23,\allowbreak 14,\allowbreak 7)$, we have $(c,\allowbreak d,\allowbreak 2k-a,\allowbreak e^{\min}_C)=(9,\allowbreak 5,\allowbreak 0,\allowbreak 2)$ and $\mathsf R^\ast(e^{\min}_C)=6152/3$. The adjacent values are $\mathsf R^\ast(e^{\min}_C-1)=6169/3$ and $\mathsf R^\ast(e^{\min}_C+1)=4335/2$.

For $(n,\allowbreak a,\allowbreak k)=(24,\allowbreak 13,\allowbreak 7)$, we have $(c,\allowbreak d,\allowbreak 2k-a,\allowbreak e^{\min}_C)=(11,\allowbreak 2,\allowbreak 1,\allowbreak 3)$ and $\mathsf R^\ast(e^{\min}_C)=8826/5$. The adjacent values are $\mathsf R^\ast(e^{\min}_C-1)=2034$ and $\mathsf R^\ast(e^{\min}_C+1)=9246/5$.

For $(n,\allowbreak a,\allowbreak k)=(24,\allowbreak 14,\allowbreak 7)$, we have $(c,\allowbreak d,\allowbreak 2k-a,\allowbreak e^{\min}_C)=(10,\allowbreak 4,\allowbreak 0,\allowbreak 3)$ and $\mathsf R^\ast(e^{\min}_C)=20029/5$. The adjacent values are $\mathsf R^\ast(e^{\min}_C-1)=12271/3$ and $\mathsf R^\ast(e^{\min}_C+1)=62011/15$.

For $(n,\allowbreak a,\allowbreak k)=(24,\allowbreak 16,\allowbreak 8)$, we have $(c,\allowbreak d,\allowbreak 2k-a,\allowbreak e^{\min}_C)=(8,\allowbreak 8,\allowbreak 0,\allowbreak 0)$ and $\mathsf R^\ast(e^{\min}_C)=124$. The sole admissible neighbouring value is $\mathsf R^\ast(e^{\min}_C+1)=8257/14$.
% exact-data:FlxLine:end

\datablock{Anchored-bound evaluations}{fl:tab:values}
The following evaluations settle $35$ triples of $\mathcal P_4$ by the anchored bound:
the first $18$ have $n\le20$, and the remaining $17$ have $21\le n\le24$.
Each evaluation uses the line set of Data~\ref{fl:tab:level} with the same $(a,k)$ and $w_{\mathrm L}$. We specify $c$, $d$, the left weight $w_{\mathrm L}$ and the minimizing argument $e^{\min}_C$, and give the bound there and at every adjacent integer in $[0,c-1]$.

% exact-data:PValues:begin
For $(n,\allowbreak a,\allowbreak k)=(12,\allowbreak 8,\allowbreak 4)$, we have $(c,\allowbreak d,\allowbreak w_{\mathrm L},\allowbreak e^{\min}_C)=(4,\allowbreak 4,\allowbreak 1/2,\allowbreak 0)$ and $\mathcal A(e^{\min}_C)=105/4$. The sole admissible neighbouring value is $\mathcal A(e^{\min}_C+1)=28$.

For $(n,\allowbreak a,\allowbreak k)=(15,\allowbreak 9,\allowbreak 5)$, we have $(c,\allowbreak d,\allowbreak w_{\mathrm L},\allowbreak e^{\min}_C)=(6,\allowbreak 3,\allowbreak 1/2,\allowbreak 0)$ and $\mathcal A(e^{\min}_C)=285/4$. The sole admissible neighbouring value is $\mathcal A(e^{\min}_C+1)=4571/60$.

For $(n,\allowbreak a,\allowbreak k)=(15,\allowbreak 10,\allowbreak 5)$, we have $(c,\allowbreak d,\allowbreak w_{\mathrm L},\allowbreak e^{\min}_C)=(5,\allowbreak 5,\allowbreak 1/2,\allowbreak 0)$ and $\mathcal A(e^{\min}_C)=87$. The sole admissible neighbouring value is $\mathcal A(e^{\min}_C+1)=3731/40$.

For $(n,\allowbreak a,\allowbreak k)=(16,\allowbreak 9,\allowbreak 5)$, we have $(c,\allowbreak d,\allowbreak w_{\mathrm L},\allowbreak e^{\min}_C)=(7,\allowbreak 2,\allowbreak 1/2,\allowbreak 0)$ and $\mathcal A(e^{\min}_C)=197/2$. The sole admissible neighbouring value is $\mathcal A(e^{\min}_C+1)=427/4$.

For $(n,\allowbreak a,\allowbreak k)=(16,\allowbreak 10,\allowbreak 5)$, we have $(c,\allowbreak d,\allowbreak w_{\mathrm L},\allowbreak e^{\min}_C)=(6,\allowbreak 4,\allowbreak 1/2,\allowbreak 0)$ and $\mathcal A(e^{\min}_C)=299/3$. The sole admissible neighbouring value is $\mathcal A(e^{\min}_C+1)=1093/10$.

For $(n,\allowbreak a,\allowbreak k)=(17,\allowbreak 11,\allowbreak 6)$, we have $(c,\allowbreak d,\allowbreak w_{\mathrm L},\allowbreak e^{\min}_C)=(6,\allowbreak 5,\allowbreak 1/2,\allowbreak 0)$ and $\mathcal A(e^{\min}_C)=296/3$. The sole admissible neighbouring value is $\mathcal A(e^{\min}_C+1)=2193/20$.

For $(n,\allowbreak a,\allowbreak k)=(17,\allowbreak 12,\allowbreak 6)$, we have $(c,\allowbreak d,\allowbreak w_{\mathrm L},\allowbreak e^{\min}_C)=(5,\allowbreak 7,\allowbreak 1/2,\allowbreak 0)$ and $\mathcal A(e^{\min}_C)=1115/6$. The sole admissible neighbouring value is $\mathcal A(e^{\min}_C+1)=11879/60$.

For $(n,\allowbreak a,\allowbreak k)=(18,\allowbreak 11,\allowbreak 6)$, we have $(c,\allowbreak d,\allowbreak w_{\mathrm L},\allowbreak e^{\min}_C)=(7,\allowbreak 4,\allowbreak 1/2,\allowbreak 0)$ and $\mathcal A(e^{\min}_C)=595/3$. The sole admissible neighbouring value is $\mathcal A(e^{\min}_C+1)=18307/84$.

For $(n,\allowbreak a,\allowbreak k)=(18,\allowbreak 12,\allowbreak 6)$, we have $(c,\allowbreak d,\allowbreak w_{\mathrm L},\allowbreak e^{\min}_C)=(6,\allowbreak 6,\allowbreak 1/2,\allowbreak 0)$ and $\mathcal A(e^{\min}_C)=1573/6$. The sole admissible neighbouring value is $\mathcal A(e^{\min}_C+1)=4231/15$.

For $(n,\allowbreak a,\allowbreak k)=(19,\allowbreak 10,\allowbreak 6)$, we have $(c,\allowbreak d,\allowbreak w_{\mathrm L},\allowbreak e^{\min}_C)=(9,\allowbreak 1,\allowbreak 1/2,\allowbreak 0)$ and $\mathcal A(e^{\min}_C)=209$. The sole admissible neighbouring value is $\mathcal A(e^{\min}_C+1)=709/3$.

For $(n,\allowbreak a,\allowbreak k)=(19,\allowbreak 11,\allowbreak 6)$, we have $(c,\allowbreak d,\allowbreak w_{\mathrm L},\allowbreak e^{\min}_C)=(8,\allowbreak 3,\allowbreak 1/2,\allowbreak 0)$ and $\mathcal A(e^{\min}_C)=3355/12$. The sole admissible neighbouring value is $\mathcal A(e^{\min}_C+1)=17431/56$.

For $(n,\allowbreak a,\allowbreak k)=(19,\allowbreak 12,\allowbreak 6)$, we have $(c,\allowbreak d,\allowbreak w_{\mathrm L},\allowbreak e^{\min}_C)=(7,\allowbreak 5,\allowbreak 1/2,\allowbreak 0)$ and $\mathcal A(e^{\min}_C)=1127/4$. The sole admissible neighbouring value is $\mathcal A(e^{\min}_C+1)=26231/84$.

For $(n,\allowbreak a,\allowbreak k)=(19,\allowbreak 13,\allowbreak 7)$, we have $(c,\allowbreak d,\allowbreak w_{\mathrm L},\allowbreak e^{\min}_C)=(6,\allowbreak 7,\allowbreak 5/12,\allowbreak 0)$ and $\mathcal A(e^{\min}_C)=757/3$. The sole admissible neighbouring value is $\mathcal A(e^{\min}_C+1)=12812/45$.

For $(n,\allowbreak a,\allowbreak k)=(20,\allowbreak 10,\allowbreak 6)$, we have $(c,\allowbreak d,\allowbreak w_{\mathrm L},\allowbreak e^{\min}_C)=(10,\allowbreak 0,\allowbreak 1/2,\allowbreak 0)$ and $\mathcal A(e^{\min}_C)=279$. The sole admissible neighbouring value is $\mathcal A(e^{\min}_C+1)=1614/5$.

For $(n,\allowbreak a,\allowbreak k)=(20,\allowbreak 11,\allowbreak 6)$, we have $(c,\allowbreak d,\allowbreak w_{\mathrm L},\allowbreak e^{\min}_C)=(9,\allowbreak 2,\allowbreak 1/2,\allowbreak 0)$ and $\mathcal A(e^{\min}_C)=844/3$. The sole admissible neighbouring value is $\mathcal A(e^{\min}_C+1)=2635/8$.

For $(n,\allowbreak a,\allowbreak k)=(20,\allowbreak 12,\allowbreak 6)$, we have $(c,\allowbreak d,\allowbreak w_{\mathrm L},\allowbreak e^{\min}_C)=(8,\allowbreak 4,\allowbreak 1/2,\allowbreak 0)$ and $\mathcal A(e^{\min}_C)=1981/12$. The sole admissible neighbouring value is $\mathcal A(e^{\min}_C+1)=8819/42$.

For $(n,\allowbreak a,\allowbreak k)=(20,\allowbreak 13,\allowbreak 7)$, we have $(c,\allowbreak d,\allowbreak w_{\mathrm L},\allowbreak e^{\min}_C)=(7,\allowbreak 6,\allowbreak 1/2,\allowbreak 0)$ and $\mathcal A(e^{\min}_C)=49/2$. The sole admissible neighbouring value is $\mathcal A(e^{\min}_C+1)=437/7$.

For $(n,\allowbreak a,\allowbreak k)=(20,\allowbreak 14,\allowbreak 7)$, we have $(c,\allowbreak d,\allowbreak w_{\mathrm L},\allowbreak e^{\min}_C)=(6,\allowbreak 8,\allowbreak 1/2,\allowbreak 0)$ and $\mathcal A(e^{\min}_C)=480$. The sole admissible neighbouring value is $\mathcal A(e^{\min}_C+1)=7883/15$.

For $(n,\allowbreak a,\allowbreak k)=(21,\allowbreak 12,\allowbreak 7)$, we have $(c,\allowbreak d,\allowbreak w_{\mathrm L},\allowbreak e^{\min}_C)=(9,\allowbreak 3,\allowbreak 1/2,\allowbreak 0)$ and $\mathcal A(e^{\min}_C)=1447/6$. The sole admissible neighbouring value is $\mathcal A(e^{\min}_C+1)=1061/4$.

For $(n,\allowbreak a,\allowbreak k)=(22,\allowbreak 12,\allowbreak 7)$, we have $(c,\allowbreak d,\allowbreak w_{\mathrm L},\allowbreak e^{\min}_C)=(10,\allowbreak 2,\allowbreak 1/2,\allowbreak 0)$ and $\mathcal A(e^{\min}_C)=1059/2$. The sole admissible neighbouring value is $\mathcal A(e^{\min}_C+1)=3547/6$.

For $(n,\allowbreak a,\allowbreak k)=(23,\allowbreak 12,\allowbreak 7)$, we have $(c,\allowbreak d,\allowbreak w_{\mathrm L},\allowbreak e^{\min}_C)=(11,\allowbreak 1,\allowbreak 1/2,\allowbreak 0)$ and $\mathcal A(e^{\min}_C)=1255/2$. The sole admissible neighbouring value is $\mathcal A(e^{\min}_C+1)=3729/5$.

For $(n,\allowbreak a,\allowbreak k)=(24,\allowbreak 12,\allowbreak 7)$, we have $(c,\allowbreak d,\allowbreak w_{\mathrm L},\allowbreak e^{\min}_C)=(12,\allowbreak 0,\allowbreak 1/2,\allowbreak 0)$ and $\mathcal A(e^{\min}_C)=363/2$. The sole admissible neighbouring value is $\mathcal A(e^{\min}_C+1)=1523/4$.

For $(n,\allowbreak a,\allowbreak k)=(21,\allowbreak 13,\allowbreak 7)$, we have $(c,\allowbreak d,\allowbreak w_{\mathrm L},\allowbreak e^{\min}_C)=(8,\allowbreak 5,\allowbreak 1/2,\allowbreak 0)$ and $\mathcal A(e^{\min}_C)=1225/4$. The sole admissible neighbouring value is $\mathcal A(e^{\min}_C+1)=62245/168$.

For $(n,\allowbreak a,\allowbreak k)=(22,\allowbreak 13,\allowbreak 7)$, we have $(c,\allowbreak d,\allowbreak w_{\mathrm L},\allowbreak e^{\min}_C)=(9,\allowbreak 4,\allowbreak 1/2,\allowbreak 0)$ and $\mathcal A(e^{\min}_C)=1414/3$. The sole admissible neighbouring value is $\mathcal A(e^{\min}_C+1)=13757/24$.

For $(n,\allowbreak a,\allowbreak k)=(23,\allowbreak 13,\allowbreak 7)$, we have $(c,\allowbreak d,\allowbreak w_{\mathrm L},\allowbreak e^{\min}_C)=(10,\allowbreak 3,\allowbreak 1/2,\allowbreak 0)$ and $\mathcal A(e^{\min}_C)=255$. The sole admissible neighbouring value is $\mathcal A(e^{\min}_C+1)=12277/30$.

For $(n,\allowbreak a,\allowbreak k)=(23,\allowbreak 14,\allowbreak 8)$, we have $(c,\allowbreak d,\allowbreak w_{\mathrm L},\allowbreak e^{\min}_C)=(9,\allowbreak 5,\allowbreak 1/2,\allowbreak 0)$ and $\mathcal A(e^{\min}_C)=353/2$. The sole admissible neighbouring value is $\mathcal A(e^{\min}_C+1)=15139/72$.

For $(n,\allowbreak a,\allowbreak k)=(24,\allowbreak 14,\allowbreak 8)$, we have $(c,\allowbreak d,\allowbreak w_{\mathrm L},\allowbreak e^{\min}_C)=(10,\allowbreak 4,\allowbreak 1/2,\allowbreak 0)$ and $\mathcal A(e^{\min}_C)=1453/2$. The sole admissible neighbouring value is $\mathcal A(e^{\min}_C+1)=6901/9$.

For $(n,\allowbreak a,\allowbreak k)=(21,\allowbreak 15,\allowbreak 8)$, we have $(c,\allowbreak d,\allowbreak w_{\mathrm L},\allowbreak e^{\min}_C)=(6,\allowbreak 9,\allowbreak 1/2,\allowbreak 3)$ and $\mathcal A(e^{\min}_C)=2552/5$. The adjacent values are $\mathcal A(e^{\min}_C-1)=2568/5$ and $\mathcal A(e^{\min}_C+1)=5567/10$.

For $(n,\allowbreak a,\allowbreak k)=(22,\allowbreak 15,\allowbreak 8)$, we have $(c,\allowbreak d,\allowbreak w_{\mathrm L},\allowbreak e^{\min}_C)=(7,\allowbreak 8,\allowbreak 1/2,\allowbreak 3)$ and $\mathcal A(e^{\min}_C)=6366/7$. The adjacent values are $\mathcal A(e^{\min}_C-1)=41543/42$ and $\mathcal A(e^{\min}_C+1)=19157/21$.

For $(n,\allowbreak a,\allowbreak k)=(23,\allowbreak 15,\allowbreak 8)$, we have $(c,\allowbreak d,\allowbreak w_{\mathrm L},\allowbreak e^{\min}_C)=(8,\allowbreak 7,\allowbreak 1/2,\allowbreak 5)$ and $\mathcal A(e^{\min}_C)=69005/56$. The adjacent values are $\mathcal A(e^{\min}_C-1)=17441/14$ and $\mathcal A(e^{\min}_C+1)=38743/28$.

For $(n,\allowbreak a,\allowbreak k)=(24,\allowbreak 15,\allowbreak 8)$, we have $(c,\allowbreak d,\allowbreak w_{\mathrm L},\allowbreak e^{\min}_C)=(9,\allowbreak 6,\allowbreak 1/2,\allowbreak 6)$ and $\mathcal A(e^{\min}_C)=35143/30$. The adjacent values are $\mathcal A(e^{\min}_C-1)=113809/90$ and $\mathcal A(e^{\min}_C+1)=64117/45$.

For $(n,\allowbreak a,\allowbreak k)=(23,\allowbreak 16,\allowbreak 8)$, we have $(c,\allowbreak d,\allowbreak w_{\mathrm L},\allowbreak e^{\min}_C)=(7,\allowbreak 9,\allowbreak 1/2,\allowbreak 0)$ and $\mathcal A(e^{\min}_C)=558$. The sole admissible neighbouring value is $\mathcal A(e^{\min}_C+1)=32233/42$.

For $(n,\allowbreak a,\allowbreak k)=(23,\allowbreak 17,\allowbreak 9)$, we have $(c,\allowbreak d,\allowbreak w_{\mathrm L},\allowbreak e^{\min}_C)=(6,\allowbreak 11,\allowbreak 1/2,\allowbreak 0)$ and $\mathcal A(e^{\min}_C)=2343$. The sole admissible neighbouring value is $\mathcal A(e^{\min}_C+1)=76637/30$.

For $(n,\allowbreak a,\allowbreak k)=(24,\allowbreak 17,\allowbreak 9)$, we have $(c,\allowbreak d,\allowbreak w_{\mathrm L},\allowbreak e^{\min}_C)=(7,\allowbreak 10,\allowbreak 1/2,\allowbreak 0)$ and $\mathcal A(e^{\min}_C)=3301$. The sole admissible neighbouring value is $\mathcal A(e^{\min}_C+1)=151037/42$.

For $(n,\allowbreak a,\allowbreak k)=(24,\allowbreak 18,\allowbreak 9)$, we have $(c,\allowbreak d,\allowbreak w_{\mathrm L},\allowbreak e^{\min}_C)=(6,\allowbreak 12,\allowbreak 1/2,\allowbreak 0)$ and $\mathcal A(e^{\min}_C)=1464$. The sole admissible neighbouring value is $\mathcal A(e^{\min}_C+1)=31022/15$.
% exact-data:PValues:end

\datablock{Values of $G_j$ at weight $1/2$}{flx:tab:g}
For $w_{\mathrm L}=\frac12$, the function
\[
 G_j(m)=\binom mj-w_{\mathrm L}\binom m{j-1}-w_{\mathrm R}\binom m{j+1}
       =2\binom mj-\frac12\binom{m+2}{j+1}
\]
has the following values; the second equality is Pascal's rule.
For each $j$, the vectors below list $G_j(m)$ in increasing order of $m$:
\[
 U_j=(G_j(0),\ldots,G_j(8)),\qquad V_j=(G_j(9),\ldots,G_j(17)).
\]
% exact-data:GHalf:begin
For $j=1$, the first vector is $U_j=(-1/2,\allowbreak 1/2,\allowbreak 1,\allowbreak 1,\allowbreak 1/2,\allowbreak -1/2,\allowbreak -2,\allowbreak -4,\allowbreak -13/2)$, and the second is $V_j=(-19/2,\allowbreak -13,\allowbreak -17,\allowbreak -43/2,\allowbreak -53/2,\allowbreak -32,\allowbreak -38,\allowbreak -89/2,\allowbreak -103/2)$.

For $j=2$, the first vector is $U_j=(0,\allowbreak -1/2,\allowbreak 0,\allowbreak 1,\allowbreak 2,\allowbreak 5/2,\allowbreak 2,\allowbreak 0,\allowbreak -4)$, and the second is $V_j=(-21/2,\allowbreak -20,\allowbreak -33,\allowbreak -50,\allowbreak -143/2,\allowbreak -98,\allowbreak -130,\allowbreak -168,\allowbreak -425/2)$.

For $j=3$, the first vector is $U_j=(0,\allowbreak 0,\allowbreak -1/2,\allowbreak -1/2,\allowbreak 1/2,\allowbreak 5/2,\allowbreak 5,\allowbreak 7,\allowbreak 7)$, and the second is $V_j=(3,\allowbreak -15/2,\allowbreak -55/2,\allowbreak -121/2,\allowbreak -221/2,\allowbreak -182,\allowbreak -280,\allowbreak -410,\allowbreak -578)$.

For $j=4$, the first vector is $U_j=(0,\allowbreak 0,\allowbreak 0,\allowbreak -1/2,\allowbreak -1,\allowbreak -1/2,\allowbreak 2,\allowbreak 7,\allowbreak 14)$, and the second is $V_j=(21,\allowbreak 24,\allowbreak 33/2,\allowbreak -11,\allowbreak -143/2,\allowbreak -182,\allowbreak -364,\allowbreak -644,\allowbreak -1054)$.

For $j=5$, the first vector is $U_j=(0,\allowbreak 0,\allowbreak 0,\allowbreak 0,\allowbreak -1/2,\allowbreak -3/2,\allowbreak -2,\allowbreak 0,\allowbreak 7)$, and the second is $V_j=(21,\allowbreak 42,\allowbreak 66,\allowbreak 165/2,\allowbreak 143/2,\allowbreak 0,\allowbreak -182,\allowbreak -546,\allowbreak -1190)$.

For $j=6$, the first vector is $U_j=(0,\allowbreak 0,\allowbreak 0,\allowbreak 0,\allowbreak 0,\allowbreak -1/2,\allowbreak -2,\allowbreak -4,\allowbreak -4)$, and the second is $V_j=(3,\allowbreak 24,\allowbreak 66,\allowbreak 132,\allowbreak 429/2,\allowbreak 286,\allowbreak 286,\allowbreak 104,\allowbreak -442)$.

For $j=7$, the first vector is $U_j=(0,\allowbreak 0,\allowbreak 0,\allowbreak 0,\allowbreak 0,\allowbreak 0,\allowbreak -1/2,\allowbreak -5/2,\allowbreak -13/2)$, and the second is $V_j=(-21/2,\allowbreak -15/2,\allowbreak 33/2,\allowbreak 165/2,\allowbreak 429/2,\allowbreak 429,\allowbreak 715,\allowbreak 1001,\allowbreak 1105)$.

For $j=8$, the first vector is $U_j=(0,\allowbreak 0,\allowbreak 0,\allowbreak 0,\allowbreak 0,\allowbreak 0,\allowbreak 0,\allowbreak -1/2,\allowbreak -3)$, and the second is $V_j=(-19/2,\allowbreak -20,\allowbreak -55/2,\allowbreak -11,\allowbreak 143/2,\allowbreak 286,\allowbreak 715,\allowbreak 1430,\allowbreak 2431)$.
% exact-data:GHalf:end

\datablock{Line sets for the anchored bound}{fl:tab:level}
The following $15$ cases give the line sets $(p^\sharp_h,q^\sharp_h)$ for $1\le h\le k-1$.
In each case the entries of the $p$- and $q$-vectors are indexed by consecutive levels $h=1,2,\ldots$.
The nine sets used by triples with $n\le20$ also include $h=k$ and $h=k+1$.
Levels without an asterisk use the chords $(p^\circ_h,q^\circ_h)$.
An asterisk on $p^\sharp_h$ marks a level whose line is another edge of the lower convex hull
of the points $(m,G_{k-h}(m))$, $0\le m\le a-h$, or the horizontal line through its lowest point.
At $h=k$ and $h=k+1$, every line set uses the chords, which the closed forms of $G_0$ and $G_{-1}$ give:
\[
 (p^\sharp_k,q^\sharp_k)=\bigl(1-w_{\mathrm R}(a-k),0\bigr),\qquad
 (p^\sharp_{k+1},q^\sharp_{k+1})=\bigl(-w_{\mathrm R},0\bigr).
\]
The parameter vectors are $P=(p^\sharp_1,p^\sharp_2,\ldots)$ and
$Q=(q^\sharp_1,q^\sharp_2,\ldots)$, truncated at the last printed level.
% exact-data:PLineSets:begin
For $(a,k)=(8,4)$ and $w_{\mathrm L}=1/2$, the line parameters are $P=(7,\allowbreak 2,\allowbreak -1/2,\allowbreak -1,\allowbreak -1/2)$ and $Q=(9/4,\allowbreak 1/2,\allowbreak 0,\allowbreak 0,\allowbreak 0)$.

For $(a,k)=(9,5)$ and $w_{\mathrm L}=1/2$, the line parameters are $P=(14,\allowbreak 7,\allowbreak 2,\allowbreak -1/2,\allowbreak -1,\allowbreak -1/2)$ and $Q=(7,\allowbreak 9/4,\allowbreak 1/2,\allowbreak 0,\allowbreak 0,\allowbreak 0)$.

For $(a,k)=(10,5)$ and $w_{\mathrm L}=1/2$, the line parameters are $P=(21,\allowbreak 7,\allowbreak 0,\allowbreak -2,\allowbreak -3/2,\allowbreak -1/2)$ and $Q=(7,\allowbreak 13/8,\allowbreak 1/12,\allowbreak 0,\allowbreak 0,\allowbreak 0)$.

For $(a,k)=(10,6)$ and $w_{\mathrm L}=1/2$, the line parameters are $P=(21,\allowbreak 14,\allowbreak 7,\allowbreak 2,\allowbreak -1/2,\allowbreak -1,\allowbreak -1/2)$ and $Q=(14,\allowbreak 7,\allowbreak 9/4,\allowbreak 1/2,\allowbreak 0,\allowbreak 0,\allowbreak 0)$.

For $(a,k)=(11,6)$ and $w_{\mathrm L}=1/2$, the line parameters are $P=(42,\allowbreak 21,\allowbreak 7,\allowbreak 0,\allowbreak -2,\allowbreak -3/2,\allowbreak -1/2)$ and $Q=(21,\allowbreak 7,\allowbreak 13/8,\allowbreak 1/12,\allowbreak 0,\allowbreak 0,\allowbreak 0)$.

For $(a,k)=(12,6)$ and $w_{\mathrm L}=1/2$, the line parameters are $P=(66,\allowbreak 24,\allowbreak 3,\allowbreak -4,\allowbreak -4,\allowbreak -2,\allowbreak -1/2)$ and $Q=(24,\allowbreak 17/3,\allowbreak 7/12,\allowbreak 0,\allowbreak 0,\allowbreak 0,\allowbreak 0)$.

For $(a,k)=(12,7)$ and $w_{\mathrm L}=1/2$, the line parameters are $P=(66,\allowbreak 35^*,\allowbreak 21,\allowbreak 7,\allowbreak 0,\allowbreak -2)$ and $Q=(42,\allowbreak 14,\allowbreak 7,\allowbreak 13/8,\allowbreak 1/12,\allowbreak 0)$.

For $(a,k)=(13,7)$ and $w_{\mathrm L}=1/2$, the line parameters are $P=(132,\allowbreak 66,\allowbreak 24,\allowbreak 3,\allowbreak -4,\allowbreak -4,\allowbreak -2,\allowbreak -1/2)$ and $Q=(66,\allowbreak 24,\allowbreak 17/3,\allowbreak 7/12,\allowbreak 0,\allowbreak 0,\allowbreak 0,\allowbreak 0)$.

For $(a,k)=(13,7)$ and $w_{\mathrm L}=5/12$, the line parameters are $P=(132,\allowbreak 55,\allowbreak 13,\allowbreak -9/2,\allowbreak -8,\allowbreak -17/3,\allowbreak -5/2,\allowbreak -7/12)$ and $Q=(55,\allowbreak 61/4,\allowbreak 51/20,\allowbreak 0,\allowbreak 0,\allowbreak 0,\allowbreak 0,\allowbreak 0)$.

For $(a,k)=(14,7)$ and $w_{\mathrm L}=1/2$, the line parameters are $P=(429/2,\allowbreak 165/2,\allowbreak 33/2,\allowbreak -15/2,\allowbreak -21/2,\allowbreak -13/2,\allowbreak -5/2,\allowbreak -1/2)$ and $Q=(165/2,\allowbreak 41/2,\allowbreak 29/10,\allowbreak 0,\allowbreak 0,\allowbreak 0,\allowbreak 0,\allowbreak 0)$.

For $(a,k)=(14,8)$ and $w_{\mathrm L}=1/2$, the line parameters are $P=(429/2,\allowbreak 108^*,\allowbreak 28^*,\allowbreak 12^*,\allowbreak 3,\allowbreak -4,\allowbreak -4)$ and $Q=(132,\allowbreak 42,\allowbreak 7,\allowbreak 5/2,\allowbreak 7/12,\allowbreak 0,\allowbreak 0)$.

For $(a,k)=(15,8)$ and $w_{\mathrm L}=1/2$, the line parameters are $P=(693/2^*,\allowbreak 87^*,\allowbreak 35^*,\allowbreak 33/2,\allowbreak -15/2,\allowbreak -21/2,\allowbreak -13/2)$ and $Q=(132,\allowbreak 21,\allowbreak 7,\allowbreak 29/10,\allowbreak 0,\allowbreak 0,\allowbreak 0)$.

For $(a,k)=(16,8)$ and $w_{\mathrm L}=1/2$, the line parameters are $P=(715,\allowbreak 286,\allowbreak 143/2,\allowbreak -11,\allowbreak -55/2,\allowbreak -20,\allowbreak -19/2)$ and $Q=(286,\allowbreak 77,\allowbreak 129/10,\allowbreak 0,\allowbreak 0,\allowbreak 0,\allowbreak 0)$.

For $(a,k)=(17,9)$ and $w_{\mathrm L}=1/2$, the line parameters are $P=(715^*,\allowbreak 561/2^*,\allowbreak -4^*,\allowbreak -2^*,\allowbreak -11,\allowbreak -55/2,\allowbreak -20,\allowbreak -19/2)$ and $Q=(429/2,\allowbreak 66,\allowbreak 0,\allowbreak 0,\allowbreak 0,\allowbreak 0,\allowbreak 0,\allowbreak 0)$.

For $(a,k)=(18,9)$ and $w_{\mathrm L}=1/2$, the line parameters are $P=(2431,\allowbreak 273/2^*,\allowbreak 129^*,\allowbreak 0,\allowbreak -143/2,\allowbreak -121/2,\allowbreak -33,\allowbreak -13)$ and $Q=(1001,\allowbreak 24,\allowbreak 21,\allowbreak 1/4,\allowbreak 0,\allowbreak 0,\allowbreak 0,\allowbreak 0)$.
% exact-data:PLineSets:end

\datablock{Values of $G_j$ at weight $5/12$}{fl:tab:g}
For $w_{\mathrm L}=\frac5{12}$ and $j=1,\ldots,6$, the following two vectors give all listed
values of $G_j(m)$ for $0\le m\le12$, in increasing order of $m$:
\[
 U_j=(G_j(0),\ldots,G_j(6)),\qquad V_j=(G_j(7),\ldots,G_j(12)).
\]
% exact-data:GFive:begin
For $j=1$, the first vector is $U_j=(-5/12,\allowbreak 7/12,\allowbreak 1,\allowbreak 5/6,\allowbreak 1/12,\allowbreak -5/4,\allowbreak -19/6)$, and the second is $V_j=(-17/3,\allowbreak -35/4,\allowbreak -149/12,\allowbreak -50/3,\allowbreak -43/2,\allowbreak -323/12)$.

For $j=2$, the first vector is $U_j=(0,\allowbreak -5/12,\allowbreak 1/6,\allowbreak 7/6,\allowbreak 2,\allowbreak 25/12,\allowbreak 5/6)$, and the second is $V_j=(-7/3,\allowbreak -8,\allowbreak -67/4,\allowbreak -175/6,\allowbreak -275/6,\allowbreak -202/3)$.

For $j=3$, the first vector is $U_j=(0,\allowbreak 0,\allowbreak -5/12,\allowbreak -1/4,\allowbreak 11/12,\allowbreak 35/12,\allowbreak 5)$, and the second is $V_j=(35/6,\allowbreak 7/2,\allowbreak -9/2,\allowbreak -85/4,\allowbreak -605/12,\allowbreak -385/4)$.

For $j=4$, the first vector is $U_j=(0,\allowbreak 0,\allowbreak 0,\allowbreak -5/12,\allowbreak -2/3,\allowbreak 1/4,\allowbreak 19/6)$, and the second is $V_j=(49/6,\allowbreak 14,\allowbreak 35/2,\allowbreak 13,\allowbreak -33/4,\allowbreak -176/3)$.

For $j=5$, the first vector is $U_j=(0,\allowbreak 0,\allowbreak 0,\allowbreak 0,\allowbreak -5/12,\allowbreak -13/12,\allowbreak -5/6)$, and the second is $V_j=(7/3,\allowbreak 21/2,\allowbreak 49/2,\allowbreak 42,\allowbreak 55,\allowbreak 187/4)$.

For $j=6$, the first vector is $U_j=(0,\allowbreak 0,\allowbreak 0,\allowbreak 0,\allowbreak 0,\allowbreak -5/12,\allowbreak -3/2)$, and the second is $V_j=(-7/3,\allowbreak 0,\allowbreak 21/2,\allowbreak 35,\allowbreak 77,\allowbreak 132)$.
% exact-data:GFive:end

% ---------------------------------------------------------------------------------------------------------------
\section{The finite inequalities of the fiber inequality}\label{apx:fiber}

This appendix states every finite computation that Section~\ref{fib:sec} uses (families F9 and F10 of
Section~\ref{sec:computer}), each as a lemma with its finite range, its arithmetic, its number of cells and its least
margin, and it lists the starting means. All other steps of Section~\ref{fib:sec} are proved in the text. The notation,
including the local letters listed at the start of Section~\ref{fib:sec}, is that of Section~\ref{fib:sec}.

\paragraph{The inequalities.} They are explicit inequalities in one or two variables, built from $\log$, $\exp$, $\artanh$,
square roots and the root map $s(\omega)$ of $s+\log(1+s)=\omega$ (enclosed by two certified point bounds), with rational
constants (the constants $\frac92$, $3.55$, $3.85$, $10$, $\frac{18}{25}$, the class values in Data~\ref{fib:tab:classes} and of
Theorem~\ref{fib:thm:wing}, and the rational inputs of Section~\ref{p2:sec:inputs}) on stated ranges. Most are in one variable
($s$, $a$, $t$, $\zeta$, $N$, $r$ or $q$). The exceptions are: Lemma~\ref{fib:chk:D3} (in $t$, with $r\in[r_a,r_b]$ carried as an
interval and $N$ fixed at $N_0$); Lemma~\ref{fib:chk:A4d} (on two-dimensional blocks in $(r,N)$); Lemma~\ref{fib:chk:B3}
($t$-cells times $M$-blocks, with an inner closed-form supremum over $\zeta$ and the root map); and Lemma~\ref{fib:chk:C1} (in
$t$, with an explicit minimisation over $N$).

\paragraph{Arithmetic.} \emph{Interval} means interval arithmetic with outward rounding (the \texttt{mpmath.iv} library at $30$
decimal digits): on a cell every quantity is enclosed in an interval, using the monotonicity of a term where stated.
\emph{Exact} means rational arithmetic. The cell end points are given as decimals or exact rationals; the code encloses each
outer end point outward, so that the cells of every family cover its stated range, end points included
(\texttt{fc\_edge\_coverage.py}). Adjacent cells share their edges, and the outer ends overlap the neighbouring ranges (for
instance Lemma~\ref{fib:chk:A2} starts at $t(-0.528)=0.59958\ldots<0.6$), so there are no gaps. The rational inputs $\theta(\iota)$, $D(\iota)$, $\ell_0(\iota)$ and the ends of $\Ihat_\iota$ are those of
Part~1 of the supplementary data file (Section~\ref{sec:computer:data}); $q$ is enclosed outward.

\paragraph{The code.} The supplementary code of families F9 and F10 (Section~\ref{sec:computer:data}) consists of the checkers
\texttt{fc\_check\_A.py}, \texttt{fc\_check\_B.py}, \texttt{fc\_check\_C.py}, \texttt{fc\_check\_D.py} and \texttt{fc\_check\_E.py}, the
modules \texttt{fc\_core.py}, \texttt{fc\_spike.py} and \texttt{fc\_window.py}, the count of the boxes of family F6
\texttt{fc\_ledger\_f6.py}, the check of the cell end points \texttt{fc\_edge\_coverage.py}, and the data generator
\texttt{fc\_make\_tables.py}, which also writes the original tabular source (\texttt{table\_b3.tex}) of the bounds in Data~\ref{fib:tab:b3} from the logs. The checker \texttt{fc\_check\_X.py} writes the log
\texttt{fc\_check\_X.log}, one line per cell, beginning with a tag (A0, A1, \dots) and the word PASS or FAIL; the code and its
documentation identify each lemma below by its tag (Data~\ref{fib:tab:checks}). The checkers read the subinterval data from the
module \texttt{fth\_data.py} of the supplementary code, which differs from the copy used for the original runs only in the
path of the data file and reads the $55$ entries from the data file \texttt{twin\_v2.1\_data.txt}; the $55$ entries agree
with Part~1 of the supplementary data file, and their checksums are as follows.
The SHA-256 checksum of \texttt{fth\_data.py} is\par
\noindent\texttt{26c11e1606a4e635fc0f07c605d74c238ba95041ca26c28b2f9688945fe04d4e}.

The SHA-256 checksum of \texttt{twin\_v2.1\_data.txt} is\par
\noindent\texttt{8e7767b93733a52fd0e530be85a0279f453301a3bba3b297b250f3279959c98b}.

The five checkers take well under a minute together (about 15 to 25 seconds); every line passes. The
checkers also assert that each wing and offset-window class contains the values of $|2q-1|$ of its subintervals and that
$\ell_0(\iota)/L_b(\iota)>\tau$ on each member. The counts below are counted from the logs: $218$ cells and $55$ averaging
inequalities. Margins quoted from the logs are rounded down. The statement counts in Data~\ref{fib:tab:checks} count, by hand, one displayed inequality with its parameters (a class, a fiber $M$, a pair $(r_m,N_{\min})$); they add up to $85$. The
variant of Remark~\ref{fib:rem:Aii} is checked by \texttt{fc\_check\_A.py} with the arguments $30$ and $0.28$ (log
\texttt{fc\_check\_A\_rm028\_mu30.log}, $23$ lines, all passing); these lines are not counted, except its line of
Lemma~\ref{fib:chk:A3p}, which is in \texttt{fc\_check\_E.log}.

% ---------------------------------------------------------------------------------------------------------------
\subsection{The finite inequalities}\label{fib:apx:checks}

Throughout, $\rho_*$, $\kappa_A$, $\alpha_1$, $\alpha_2$, $E$, $\bar D_L$, $\bar D_U$, $\beta_L$, $\beta_e$ are those of Lemmas~\ref{fib:lem:coef} and~\ref{fib:lem:amgm},
evaluated at the stated $N$ and $r_m$, with $\zeta_m=\min(6,r_m\sqrt N)$, and $e_2=\frac12$, $e_3=\frac5{6(1-r_m^2)}$, $e_4(N)=\frac{N+1}{12N(1-r_m^2)^2}$,
$e_5=\frac1{12(1-r_m^2)^2}$ are the coefficients of $E$.

\begin{lemma}\label{fib:chk:A0}
$\varphi(0)=\frac{22}{25}$ and $\varphi(\frac{67}{100})\ge0.798$; hence $\varphi\ge0.798$ on $[0,0.67]$ ($\varphi$ is concave). Exact; $1$ cell.
\end{lemma}

\begin{lemma}\label{fib:chk:A1}
$Q(a)\ge0.0053$ for $a\in[0,0.528]$, where $\Lambda(a)/a$ is the polynomial $2+\frac a2+\frac{a^2}3+\frac{a^3}4+\frac{a^4}5+\frac{a^5}6$, the continuous
extension of the quotient to $a=0$ (so $Q(0)=\frac{22}{25}$), on the six cells with edges $0$, $0.303$, $0.423$, $0.479$, $0.507$, $0.522$, $0.528$ ($\Lambda(a)/a$ and $\Lambda(a)$
increase and $1-a$ decreases, so each term is taken at its worst end). Interval; $6$ cells.
\end{lemma}

\begin{lemma}\label{fib:chk:A2}
Let $\mu_0=19$ and $r_m=\frac14$. On each of the nine cells of $s$ with edges $-0.528$, $-0.373$, $-0.104$, $0$, $0.155$, $0.235$, $0.317$, $0.424$, $0.585$,
$0.67$, with $M_0=19\,t(s_a)$, $N_0=M_0+2$ and $\zeta_m=r_m\sqrt{N_0}$: $\beta_L>0$, $\bar e<1$, and, for $s<0$, $\rho_*>0$, $\beta_e>0$ and
$N_0-1-\kappa_A\Ghat/\rho_*>0$, where $\Ghat$ is bounded by $\max\big(G_0,\frac14((\omega_t+\bar h)_+)^2\big)$ (Lemma~\ref{fib:lem:G}(d)); and
$\mathcal E(s)\ge0$ for every $s$ in the cell and both values of $\mathsf k$ (the function $\mathcal E$ of the proof of Theorem~\ref{fib:thm:env}, with
$\sigma_0\le(s+\bar h)_+$). Interval in $s$; $9$ cells; least margin $0.03136$ (cell $[0.235,0.317]$).
\end{lemma}

\begin{lemma}\label{fib:chk:A3}
For $r_m=\frac14$ and $N\in[13.4,576]$:
$\beta_e'(N)\ge1-\frac2{\rho_*(13.4)}\big(\frac{e_3r_m}{2\sqrt{13.4}}+e_4(13.4)r_m^2+\frac32e_5r_m^3\sqrt{576}\big)\ge0.744$. (This lower bound of $\beta_e'$ is valid on
$[N_{\rm lo},(6/r_m)^2]$ because $\rho_*$ is nondecreasing, $e_4$ decreases, the term $e_4'\zeta_m^2\le0$ is dropped and $\zeta_m=r_m\sqrt N$; beyond
$(6/r_m)^2$, $\zeta_m=6$ and $\beta_e$ is nondecreasing termwise.) Interval; $1$ cell.
\end{lemma}

\begin{lemma}\label{fib:chk:A3p}
The lower bound of Lemma~\ref{fib:chk:A3}, with $N_{\rm lo}$ in place of $13.4$ and $(6/r_m)^2$ in place of $576$, is positive on each of the following
ranges, so that $\beta_e$ is nondecreasing there: the wing classes $\mathcal W_1,\mathcal W_2,\mathcal W_3$ ($r_m=0.0741,0.1112,0.1667$, from
$N_{\rm lo}=13.02,13.00,15.52$; margins $0.965$, $0.938$, $0.890$); the first cell of Lemma~\ref{fib:chk:A2} ($r_m=0.25$, from $N_{\rm lo}=19\,t(-0.528)+2=13.3921\ldots$;
$0.744$); the first cell of the variant in Remark~\ref{fib:rem:Aii} ($r_m=0.28$, from $N_{\rm lo}=30\,t(-0.528)+2=19.9876\ldots$; $0.718$); and the last $M$-block of every
class cell of Lemma~\ref{fib:chk:B3}, from $N_{\rm lo}=16M_a+2$ with $r_m=r_b$: for $\mathcal C_2,\dots,\mathcal C_8$ the least margins over the cells are
$0.545$, $0.873$, $0.810$, $0.742$, $0.658$, $0.525$, $0.462$, and for $\mathcal C_1$ every cell has $N_{\rm lo}\ge(6/r_b)^2+9.2$, so $\zeta_m=6$ and the claim is
termwise. Interval; $13$ cells; least margin $0.462$.
\end{lemma}

\begin{lemma}\label{fib:chk:A4}
For $r_m=\frac14$ and $N_{\min}=13.4$: $f(13.4)>0$ (margin $2.596$); $f'(N)\ge\frac{\rho_*(13.4)}2-E_2'>0$ on $[13.4,576]$, with $E_2'$ bounded termwise
(margin $0.262$); and at $N=576$, where $\zeta_m=6$, $N(N-1)\rho_*/2-NE_2-(N+1)/3.55>0$ together with its $N$-derivative, so it is positive and
increasing beyond (margin $493.7$). Here $f$ and $E_2$ are those of the proof of Lemma~\ref{fib:lem:lamU}. Interval; $3$ cells.
\end{lemma}

\begin{lemma}\label{fib:chk:A4p}
The three conditions of Lemma~\ref{fib:chk:A4} hold for each of the pairs $(r_m,N_{\min})=(0.25,10)$, $(0.28,10)$, $(0.3104,10)$, $(0.3444,10)$,
$(0.3847,12)$, $(0.4,12)$; the least of the three margins is $0.223$, $0.190$, $0.151$, $0.101$, $0.0570$, $0.0263$. Interval; $6$ cells.
\end{lemma}

\begin{lemma}\label{fib:chk:A4d}
For the class $\mathcal C_1$: the lower bound of $\Phi_r(N)$ described in the proof of Corollary~\ref{fib:cor:outer} (each term at its worst end of the
block) is nonnegative on the seven blocks $[r_1,r_2]\times[11.45,20]$, with $r$-edges $0.3793$, $0.40$, $0.42$, $0.44$, $0.46$, $0.48$, $0.5$ (each split into
the five $N$-blocks with edges $11.45$, $12.5$, $14$, $16$, $18$, $20$), and $[0.3793,0.5]\times[20,29.1]$ (split at $23$ and $26$). Interval, two
variables; $7$ cells; least margin $0.0122$ (on $N\in[20,29.1]$).
\end{lemma}

\begin{lemma}\label{fib:chk:A4m}
$h(r,11.45)\ge0.058$ for $r\in[0.3793,0.5]$, with $h$ as in the proof of Corollary~\ref{fib:cor:outer} ($24$ cells of $r$; $\hat a$ and $r^2$ increase in
$r$, and $r(1-4r^2)$ decreases for $r>12^{-1/2}$, so the left ends give a lower bound). Interval; $1$ cell.
\end{lemma}

\begin{lemma}\label{fib:chk:A5}
The hypotheses of Lemma~\ref{fib:lem:larget} hold for $r_m=\frac14$ and $N_*=32.4$: $\bar X(\zeta,32.4)\le\log(2.57\rho_*(32.4))$ for $\zeta\in[0,r_m\sqrt{N_*}]$ ($4$ cells of
$\zeta$; margin $1.154$); $\bar X(\zeta,\zeta^2/r_m^2)\le\log(2.57\rho_*(\zeta^2/r_m^2))$ for $\zeta\in[r_m\sqrt{N_*},6]$ ($8$ cells of $\zeta$, $N$ taken at the left end of the
cell; margin $1.697$); and $2e^{-1/2}1.6^{-3/2}\exp(\frac3{4N}+\frac{1.1}{N^2})\le0.6225$ for $N\ge32.4$, which bounds $2a_U(N)1.6^{-3/2}e^{-1/2}$ by
Lemma~\ref{fib:lem:k}(b) (margin $0.0084$). The code runs the first $\zeta$-range up to the upper rounding of $r_m\sqrt{N_*}$, less
than $2\cdot10^{-30}$ above it, and starts the second there; the sliver between is covered by the first range, through the
monotonicity in $N$ in the proof of Lemma~\ref{fib:lem:larget}. Interval; $3$ cells.
\end{lemma}

\begin{lemma}\label{fib:chk:D4}
The hypotheses of Lemma~\ref{fib:lem:larget} hold for $r_m=\frac12$ and $N_*=66$: the same three inequalities, on $\zeta\in[0,r_m\sqrt{N_*}]$, where $r_m\sqrt{N_*}=\frac{\sqrt{66}}2$, at $N=66$ ($8$ cells;
margin $0.716$), on $\zeta\in[r_m\sqrt{N_*},6]$ at $N=\zeta^2/r_m^2$ ($12$ cells; margin $3.475$), and the constant for $N\ge66$ (margin $0.0161$). As in Lemma~\ref{fib:chk:A5}, the code splits at the upper rounding of
$r_m\sqrt{N_*}$, less than $2\cdot10^{-30}$ above it, and the sliver is covered by the first range through the monotonicity in $N$.
Interval;
$3$ cells.
\end{lemma}

\begin{lemma}\label{fib:chk:B1}
For $M=0,\dots,7$, every $j$ with $-1\le j\le M+1$ and every $q\in[\frac14,\frac34]$: $\kappa_w(M,j)+K_M\ge m_M$, where
$(K_M)=(\frac{51}{200},\allowbreak \frac{17}{125},\allowbreak \frac{159}{2000},\allowbreak \frac{61}{1250},\allowbreak \frac{161}{5000},\allowbreak \frac{129}{5000},\allowbreak \frac{101}{5000},\allowbreak \frac{81}{5000})$ and the lower bounds
$(m_M)=(0.00406,\allowbreak \,0.00192,\allowbreak \,0.00133,\allowbreak \,0.000829,\allowbreak \,0.000459,\allowbreak \,0.000290,\allowbreak \,0.000343,\allowbreak \,0.0000978)$ (for other $j$, $\kappa_w=0$). Interval: $\kappa_w(M,j)$
is a polynomial in $q$, evaluated on $800$ cells of $q$; $8$ cells (one per $M$); least margin $9.78\cdot10^{-5}$ ($M=7$). The $m_M$ are
lower bounds for the minima of $\kappa_w(M,j)+K_M$, not the minima themselves (at $M=7$ the exact minimum is $5.75\cdot10^{-4}$).
\end{lemma}

\begin{lemma}\label{fib:chk:B2}
The range (L) contributes no deficit on the spike cells: $\frac92(1-t)^2-\frac{2t(\alpha_1\zeta_m+\alpha_2\zeta_m^2)}{N-2}>0$ with $\zeta_m=r_m\sqrt N$, which bounds
$\frac92(1-t)^2-\frac{2t}M(N+1)e_L$ from below for $r\le r_m$ (Lemma~\ref{fib:lem:coef}(a)); the subtracted term is nondecreasing in $t$ and nonincreasing in
$N$. Central entry: $r_m=\frac14$, at $t=0.6$, $N=10$ (margin $0.442$). Class entries: $r_m=r_b$, on every class cell of Lemma~\ref{fib:chk:B3} at $t=t_b$
and $N=\max(10,\mu_Dt_a+2)$ (least margins $0.329$, $0.199$, $1.309$, $1.375$, $1.481$, $1.692$, $1.854$, $2.032$ for $\mathcal C_1,\dots,\mathcal C_8$).
Interval; $9$ cells; least margin $0.199$.
\end{lemma}

\begin{lemma}\label{fib:chk:B3}
Let $L$ be the value listed in Data~\ref{fib:tab:b3} for a cell $[t_a,t_b]$ and an $M$-block, and let $L_\Phi$ be the bound for
$\Phi(t;M)$ given by Lemma~\ref{fib:lem:Phi} with the block forms stated there ($r_m=\frac14$ and $M_a=\max(8,19t_a)$ for the $18$
central cells, $r_m=r_b$ and $M_a=\max(8,\mu_Dt_a)$ for the $26$ class cells), valid for every $t$ in the cell and every $M$ in the
block. (a) On the $18$ central cells and on the $12$ class cells specified as using $\Phi$ alone, $L\ge L_\Phi$ ($L$ is $L_\Phi$ rounded
up), so $L$ bounds $\Phi(t;M)$. (b) On the $14$ class cells specified as using the offset window, $L\ge\min(L_\Phi,L_W)$ (rounded up likewise),
where $L_W$ is the offset-window bound of Corollary~\ref{fib:cor:outer}, which bounds the spike deficit $\Def_g(M,\mu)$ directly,
not $\Phi$, for every $q$ of the class, every $\mu\ge\mu_D$ and every $M$ with $M/\mu$ in the cell; on these cells the conditions
(c1), (c2), (c3) and the positivity of the derivative bound of Lemma~\ref{fib:lem:Umono} hold at the cell's least $N$. So on these
cells $\min(\Phi(t;M),\Def_g(M,M/t))\le L$ for every $q$ of the class, every $t$ in the cell and every $M$ in the block with
$M/t\ge\mu_D$ (the offset-window bound holds for every $N\ge\mu_Dt_a+2$); this is all that Corollary~\ref{fib:cor:outer} needs.
On a cell,
$(-\Sl{})_+$ is bounded by the mean-value form $\Sl(t)\ge \Sl(t_m)-\sup|\Sl'|\,\frac{t_b-t_a}2$, $\Sl'=\frac92(1-t^{-2})-5G_1/t$, $t_m$ the midpoint; all
other terms by interval enclosures. Interval; $44$ cells; these are upper bounds, used in the margins of Lemma~\ref{fib:chk:rows}.
\end{lemma}

\begin{lemma}\label{fib:chk:C1}
For each wing class of Theorem~\ref{fib:thm:wing} and each of its cells (edges $0.58$, $0.59$, $0.6$ for $\mathcal W_1$, $\mu_W=19$; $0.55$, $0.56$, $0.57$,
$0.585$, $0.6$ for $\mathcal W_2$, $\mu_W=20$; $0.52$, $0.53$, $0.54$, $0.555$, $0.575$, $0.6$ for $\mathcal W_3$, $\mu_W=26$), with $M_0=\mu_Wt_a$: the
hypotheses $c(N_0)>0$, $r_a\sqrt{N_0}\ge b'/(2c(N_0))$, $\rho_*>0$, $\beta_e>0$ of Lemma~\ref{fib:lem:wing} hold, and $\min_{N\ge N_0}w(N)\ge0$, computed as
$w(N_0)$ if $w'(N_0)\ge0$ and otherwise through the tangent of $\sqrt N$ at the best of $N_1=N_0\cdot\{1,\allowbreak 1.5,\allowbreak 2,\allowbreak 3,\allowbreak 4,\allowbreak 6,\allowbreak 8\}$. Interval in $t$; $11$ cells;
least margin $0.220$ ($\mathcal W_3$, $t\in[0.52,0.53]$).
\end{lemma}

\begin{lemma}\label{fib:chk:D0}
For each class of Data~\ref{fib:tab:classes}: $\psi(t_{\rm hi})+b+\beta'(t_{\rm hi}-1)<0$, with margins $0.105$ ($\mathcal C_1$, $\mathcal C_2$), $0.40625$ ($\mathcal C_3$),
$0.69$ ($\mathcal C_4$), $0.205$ ($\mathcal C_5$, $\mathcal C_6$), $0.545$ ($\mathcal C_7$, $\mathcal C_8$). Exact; $8$ cells.
\end{lemma}

\begin{lemma}\label{fib:chk:D1}
For each class, (c1), (c2), (c3) hold for $r\in[r_a,r_b]$ at $N_0=\mu_D\tau+2$ ($29$, $36.5$, $24$, $18.8$, $22$, $20$, $18.5$, $17$); least margins $0.212$,
$0.350$, $0.434$, $0.358$, $0.349$, $0.304$, $0.256$, $0.233$. Interval; $8$ cells.
\end{lemma}

\begin{lemma}\label{fib:chk:D1p}
(c3$'$) holds for $r\in[r_a,r_b]$ ($20$ cells of $r$) at $N_0=\mu_D\tau+2$ for each class (margins $3.453$, $5.259$, $5.667$, $4.679$, $4.721$, $4.106$, $3.408$,
$2.948$), and at the least $N$ of each of the $14$ class cells of Lemma~\ref{fib:chk:B3} that use the offset-window bound (one line per class: $\mathcal C_1$ one
cell, $\mathcal C_2$ two, $\mathcal C_4$ two, $\mathcal C_5$ two, $\mathcal C_6$ two, $\mathcal C_7$ three, $\mathcal C_8$ two; margins $2.345$, $4.332$, $4.431$, $4.318$,
$3.604$, $2.553$, $1.772$; $\mathcal C_3$ has none). Interval; $15$ cells; least margin $1.772$.
\end{lemma}

\begin{lemma}\label{fib:chk:D2}
For each class, the right side of the derivative bound of Lemma~\ref{fib:lem:Umono} is positive at $m=\mu_D\tau+3$ for every $r\in[r_a,r_b]$ (termwise:
$\frac{r^2}2$ at $r_a$, the other terms at $r_b$, as $\hat a$, $c_r$ and $\hat a/r$ increase in $r$); margins $0.0155$, $0.00954$, $0.000222$, $0.00327$, $0.0143$,
$0.0176$, $0.0184$, $0.0259$. Interval; $8$ cells; least margin $0.000222$ ($\mathcal C_3$).
\end{lemma}

\begin{lemma}\label{fib:chk:D3}
For each class and each of its cells of $t$ (edges: $\mathcal C_1$: $0.600$, $0.749$, $0.835$, $1.005$, $1.263$, $1.597$, $1.700$; $\mathcal C_2$: $0.690$, $0.797$,
$0.881$, $1.052$, $1.254$, $1.514$, $1.700$; $\mathcal C_3$: $0.440$, $0.485$, $0.591$, $0.836$, $1.069$, $1.272$, $1.480$, $1.750$; $\mathcal C_4$: $0.420$, $0.450$,
$0.530$, $0.752$, $0.890$, $1.134$, $1.352$, $1.625$, $1.800$; $\mathcal C_5$: $0.400$, $0.472$, $0.728$, $0.847$, $1.148$, $1.472$, $1.700$; $\mathcal C_6$: $0.360$, $0.391$,
$0.519$, $0.765$, $0.959$, $1.307$, $1.700$; $\mathcal C_7$: $0.330$, $0.362$, $0.494$, $0.787$, $1.027$, $1.279$, $1.533$, $1.881$, $1.900$; $\mathcal C_8$: $0.300$,
$0.322$, $0.415$, $0.763$, $1.005$, $1.285$, $1.565$, $1.900$), with $N_0=\mu_Dt_a+2$: $\min(X_L,3,X_U)$ is at least $\frac{\psi+g}t$ where $\psi+g>0$, at
least $\frac{(\psi+g)(N_0-2)}{tN_0}$ where $\psi+g\le0$, and at least $\max(\frac{\psi+g}t,0)$ where the sign of $\psi+g$ changes on the cell; here $X_L$ is
bounded below by $W_{\rm lo}=\min(3,1+\Delta(r_a,N_0))$ if the offset-window condition holds at $N_0$ in one of the two forms of
Theorem~\ref{fib:thm:window}, and by $\min(W_{\rm lo},\max(X_{\rm split},X_{\log}))$ otherwise. Interval in $t$, with $r\in[r_a,r_b]$; $54$ cells; least
margin $0.0200$.
\end{lemma}

\datablock{The finite inequalities}{fib:tab:checks}

The following paragraphs identify each finite inequality by its lemma, log and tag. A cell count is the number of lines in the indicated log, and a statement count is the number of displayed inequalities with their specified parameters, counted by hand. The log letter $X$ denotes \texttt{fc\_check\_X.log}, for $X\in\{A,\allowbreak B,\allowbreak C,\allowbreak D,\allowbreak E\}$.

Lemma~\ref{fib:chk:A0} (log A with tag A0) establishes $\varphi\ge0$ on $[0,0.67]$. It uses exact arithmetic on $1$ cell for $1$ statement. Its least margin is $0.798$, and it is used in Lemma~\ref{fib:lem:S}.

Lemma~\ref{fib:chk:A1} (log A with tag A1) establishes $Q\ge0$, so $\Sl{}\ge0$ on $[0.6,1]$. It uses interval arithmetic on $6$ cells for $1$ statement. Its least margin is $0.00532$, and it is used in Lemma~\ref{fib:lem:S}.

Lemma~\ref{fib:chk:A2} (log A with tag A2) establishes the inequality $\mathcal E(s)\ge0$. It uses interval arithmetic on $9$ cells for $1$ statement. Its least margin is $0.03136$, and it is used in Theorem~\ref{fib:thm:env}.

Lemma~\ref{fib:chk:A3} (log A with tag A3) establishes $\beta_e$ nondecreasing, $r_m=\frac14$. It uses interval arithmetic on $1$ cell for $1$ statement. Its least margin is $0.744$, and it is used in Theorem~\ref{fib:thm:env}, \ref{fib:thm:central}.

Lemma~\ref{fib:chk:A3p} (log E with tag A3p) establishes $\beta_e$ nondecreasing, other ranges. It uses interval arithmetic on $13$ cells for $1$ statement. Its least margin is $0.462$, and it is used in Theorem~\ref{fib:thm:env}, \ref{fib:thm:wing}, Corollary~\ref{fib:cor:outer}.

Lemma~\ref{fib:chk:A4} (log A with tag A4) establishes $\omega_U\ge\omega_t$, $r\le\frac14$, $N\ge13.4$. It uses interval arithmetic on $3$ cells for $3$ statements. Its least margin is $0.262$, and it is used in Theorem~\ref{fib:thm:env}.

Lemma~\ref{fib:chk:A4p} (log D with tag A4p) establishes $\omega_U\ge\omega_t$ on the spike ranges. It uses interval arithmetic on $6$ cells for $6$ statements. Its least margin is $0.0263$, and it is used in Theorem~\ref{fib:thm:central}, \ref{fib:thm:wing}, Corollary~\ref{fib:cor:outer}.

Lemma~\ref{fib:chk:A4d} (log D with tag A4d) establishes $\Phi_r(N)\ge0$ for $\mathcal C_1$, $(r,N)$-blocks. It uses interval arithmetic on $7$ cells for $2$ statements. Its least margin is $0.0122$, and it is used in Corollary~\ref{fib:cor:outer}.

Lemma~\ref{fib:chk:A4m} (log E with tag A4m) establishes $\tilde yR/(N+1)$ nondecreasing in $N$. It uses interval arithmetic on $1$ cell for $1$ statement. Its least margin is $0.0581$, and it is used in Corollary~\ref{fib:cor:outer}.

Lemma~\ref{fib:chk:A5} (log A with tag A5) establishes large $t$, $r\le\frac14$. It uses interval arithmetic on $3$ cells for $3$ statements. Its least margin is $0.00843$, and it is used in Theorem~\ref{fib:thm:env}.

Lemma~\ref{fib:chk:D4} (log D with tag D4) establishes large $t$, $r\le\frac12$. It uses interval arithmetic on $3$ cells for $3$ statements. Its least margin is $0.0161$, and it is used in Theorem~\ref{fib:thm:window}.

Lemma~\ref{fib:chk:B1} (log B with tag B1) establishes $\kappa_w(M,\cdot)\ge-K_M$, $M\le7$. It uses interval arithmetic on $8$ cells for $8$ statements. Its least margin is $9.78\cdot10^{-5}$, and it is used in Lemma~\ref{fib:lem:KM}.

Lemma~\ref{fib:chk:B2} (logs B, D with tag B2) establishes no deficit from the range (L). It uses interval arithmetic on $9$ cells for $9$ statements. Its least margin is $0.199$, and it is used in Theorem~\ref{fib:thm:central}, \ref{fib:thm:wing}, Corollary~\ref{fib:cor:outer}.

Lemma~\ref{fib:chk:B3} (logs B, D with tag B3) establishes bounds of $\Phi$ or $\min(\Phi,\Def_g)$ on spike cells. It uses interval arithmetic on $44$ cells for $9$ statements. It records upper bounds, which are used in Lemma~\ref{fib:chk:rows}.

Lemma~\ref{fib:chk:C1} (log C with tag C1) establishes the wing inequality. It uses interval arithmetic on $11$ cells for $3$ statements. Its least margin is $0.220$, and it is used in Theorem~\ref{fib:thm:wing}.

Lemma~\ref{fib:chk:D0} (log D with tag D0) establishes $\psi+g<0$ at $t_{\rm hi}$. It uses exact arithmetic on $8$ cells for $8$ statements. Its least margin is $0.105$, and it is used in Theorem~\ref{fib:thm:window}.

Lemma~\ref{fib:chk:D1} (log D with tag D1) establishes offset-window shape conditions (c1) to (c3). It uses interval arithmetic on $8$ cells for $8$ statements. Its least margin is $0.212$, and it is used in Theorem~\ref{fib:thm:window}.

Lemma~\ref{fib:chk:D1p} (log E with tag D1p) establishes (c3$'$), the monotone form of (c3). It uses interval arithmetic on $15$ cells for $1$ statement. Its least margin is $1.77$, and it is used in Theorem~\ref{fib:thm:window}, Corollary~\ref{fib:cor:outer}.

Lemma~\ref{fib:chk:D2} (log D with tag D2) establishes derivative bound for the range (U). It uses interval arithmetic on $8$ cells for $8$ statements. Its least margin is $0.000222$, and it is used in Theorem~\ref{fib:thm:window}, Corollary~\ref{fib:cor:outer}.

Lemma~\ref{fib:chk:D3} (log D with tag D3) establishes the offset-window inequality. It uses interval arithmetic on $54$ cells for $8$ statements. Its least margin is $0.0200$, and it is used in Theorem~\ref{fib:thm:window}.

In total, these checks comprise $218$ cells and $85$ statements.

\datablock{Certified spike bounds}{fib:tab:b3}

For each cell $[t_a,t_b]$, write $\mathbf L=(L_1,\allowbreak L_2,\allowbreak L_4,\allowbreak L_8,\allowbreak L_{16})$. Its first four coordinates are the bounds $L$ of Lemma~\ref{fib:chk:B3} on the $M$-blocks $[mM_a,2mM_a]$ for $m=1,\allowbreak 2,\allowbreak 4,\allowbreak 8$, and its final coordinate is the bound on $[16M_a,\infty)$. All values below are certified upper bounds rounded up to four significant digits. When an offset-window bound is specified, each coordinate bounds $\min(\Phi(t;M),\Def_g(M,M/t))$ by taking the smaller of the bound for $\Phi$ and the offset-window bound for $\Def_g$ itself. Otherwise it bounds $\Phi(t;M)$. The margins of Lemma~\ref{fib:chk:rows} use the certified values computed by the code.

For the central cell $[0.21,0.25]$, $M_a=8.000$ and $\mathbf L=(12.24,\allowbreak \,9.376,\allowbreak \,7.824,\allowbreak \,6.274,\allowbreak \,5.772)$. These bounds use $\Phi$ alone. For the central cell $[0.25,0.28]$, $M_a=8.000$ and $\mathbf L=(7.454,\allowbreak \,5.518,\allowbreak \,4.315,\allowbreak \,3.459,\allowbreak \,3.378)$. These bounds use $\Phi$ alone. For the central cell $[0.28,0.30]$, $M_a=8.000$ and $\mathbf L=(5.149,\allowbreak \,3.674,\allowbreak \,2.637,\allowbreak \,2.314,\allowbreak \,2.298)$. These bounds use $\Phi$ alone.

For the central cell $[0.30,0.32]$, $M_a=8.000$ and $\mathbf L=(4.236,\allowbreak \,2.909,\allowbreak \,2.002,\allowbreak \,1.861,\allowbreak \,1.861)$. These bounds use $\Phi$ alone. For the central cell $[0.32,0.34]$, $M_a=8.000$ and $\mathbf L=(3.478,\allowbreak \,2.291,\allowbreak \,1.566,\allowbreak \,1.507,\allowbreak \,1.507)$. These bounds use $\Phi$ alone. For the central cell $[0.34,0.36]$, $M_a=8.000$ and $\mathbf L=(2.841,\allowbreak \,1.886,\allowbreak \,1.232,\allowbreak \,1.219,\allowbreak \,1.219)$. These bounds use $\Phi$ alone.

For the central cell $[0.36,0.38]$, $M_a=8.000$ and $\mathbf L=(2.3,\allowbreak \,1.283,\allowbreak \,0.9827,\allowbreak \,0.9827,\allowbreak \,0.9827)$. These bounds use $\Phi$ alone. For the central cell $[0.38,0.40]$, $M_a=8.000$ and $\mathbf L=(1.839,\allowbreak \,0.9176,\allowbreak \,0.7885,\allowbreak \,0.7885,\allowbreak \,0.7885)$. These bounds use $\Phi$ alone. For the central cell $[0.40,0.42]$, $M_a=8.000$ and $\mathbf L=(1.493,\allowbreak \,0.6683,\allowbreak \,0.6281,\allowbreak \,0.6281,\allowbreak \,0.6281)$. These bounds use $\Phi$ alone.

For the central cell $[0.42,0.44]$, $M_a=8.000$ and $\mathbf L=(1.229,\allowbreak \,0.4956,\allowbreak \,0.4956,\allowbreak \,0.4956,\allowbreak \,0.4956)$. These bounds use $\Phi$ alone. For the central cell $[0.44,0.46]$, $M_a=8.360$ and $\mathbf L=(0.7382,\allowbreak \,0.3859,\allowbreak \,0.3859,\allowbreak \,0.3859,\allowbreak \,0.3859)$. These bounds use $\Phi$ alone. For the central cell $[0.46,0.48]$, $M_a=8.740$ and $\mathbf L=(0.4389,\allowbreak \,0.2953,\allowbreak \,0.2953,\allowbreak \,0.2953,\allowbreak \,0.2953)$. These bounds use $\Phi$ alone.

For the central cell $[0.48,0.50]$, $M_a=9.120$ and $\mathbf L=(0.2488,\allowbreak \,0.2207,\allowbreak \,0.2207,\allowbreak \,0.2207,\allowbreak \,0.2207)$. These bounds use $\Phi$ alone. For the central cell $[0.50,0.52]$, $M_a=9.500$ and $\mathbf L=(0.1595,\allowbreak \,0.1595,\allowbreak \,0.1595,\allowbreak \,0.1595,\allowbreak \,0.1595)$. These bounds use $\Phi$ alone. For the central cell $[0.52,0.55]$, $M_a=9.880$ and $\mathbf L=(0.1214,\allowbreak \,0.1214,\allowbreak \,0.1214,\allowbreak \,0.1214,\allowbreak \,0.1214)$. These bounds use $\Phi$ alone.

For the central cell $[0.55,0.58]$, $M_a=10.450$ and $\mathbf L=(0.06144,\allowbreak \,0.06144,\allowbreak \,0.06144,\allowbreak \,0.06144,\allowbreak \,0.06144)$. These bounds use $\Phi$ alone. For the central cell $[0.58,0.59]$, $M_a=11.020$ and $\mathbf L=(0.006637,\allowbreak \,0.006637,\allowbreak \,0.006637,\allowbreak \,0.006637,\allowbreak \,0.006637)$. These bounds use $\Phi$ alone. For the central cell $[0.59,0.6]$, $M_a=11.210$ and $\mathbf L=(0,\allowbreak \,0,\allowbreak \,0,\allowbreak \,0,\allowbreak \,0)$. These bounds use $\Phi$ alone.

For class $\mathcal C_1$ on the cell $[0.21,0.34]$, $M_a=9.450$ and $\mathbf L=(143.7,\allowbreak \,140.2,\allowbreak \,195.6,\allowbreak \,310.7,\allowbreak \,11.84)$. These bounds use $\Phi$ alone. For class $\mathcal C_1$ on the cell $[0.34,0.42]$, $M_a=15.300$ and $\mathbf L=(43.1,\allowbreak \,53.53,\allowbreak \,105.8,\allowbreak \,19.08,\allowbreak \,1.794)$. These bounds use $\Phi$ alone. For class $\mathcal C_1$ on the cell $[0.42,0.6]$, $M_a=18.900$ and $\mathbf L=(2.655,\allowbreak \,2.655,\allowbreak \,2.655,\allowbreak \,1.737,\allowbreak \,1.715)$. These bounds use the offset window.

For class $\mathcal C_2$ on the cell $[0.21,0.46]$, $M_a=10.500$ and $\mathbf L=(58.71,\allowbreak \,40.34,\allowbreak \,33.56,\allowbreak \,26.69,\allowbreak \,25.84)$. These bounds use $\Phi$ alone. For class $\mathcal C_2$ on the cell $[0.46,0.6]$, $M_a=23.000$ and $\mathbf L=(0.9679,\allowbreak \,0.9009,\allowbreak \,0.873,\allowbreak \,0.8644,\allowbreak \,0.8626)$. These bounds use the offset window. For class $\mathcal C_2$ on the cell $[0.6,0.69]$, $M_a=30.000$ and $\mathbf L=(0.09308,\allowbreak \,0.09308,\allowbreak \,0.09308,\allowbreak \,0.09308,\allowbreak \,0.09308)$. These bounds use the offset window.

For class $\mathcal C_3$ on the cell $[0.21,0.36]$, $M_a=10.500$ and $\mathbf L=(21.72,\allowbreak \,18.1,\allowbreak \,15.81,\allowbreak \,13.89,\allowbreak \,13.57)$. These bounds use $\Phi$ alone. For class $\mathcal C_3$ on the cell $[0.36,0.42]$, $M_a=18.000$ and $\mathbf L=(1.567,\allowbreak \,1.302,\allowbreak \,1.26,\allowbreak \,1.253,\allowbreak \,1.253)$. These bounds use $\Phi$ alone. For class $\mathcal C_3$ on the cell $[0.42,0.44]$, $M_a=21.000$ and $\mathbf L=(0.4956,\allowbreak \,0.4956,\allowbreak \,0.4956,\allowbreak \,0.4956,\allowbreak \,0.4956)$. These bounds use $\Phi$ alone.

For class $\mathcal C_4$ on the cell $[0.21,0.32]$, $M_a=8.400$ and $\mathbf L=(20.68,\allowbreak \,15.89,\allowbreak \,13.29,\allowbreak \,10.74,\allowbreak \,10.08)$. These bounds use $\Phi$ alone. For class $\mathcal C_4$ on the cell $[0.32,0.36]$, $M_a=12.800$ and $\mathbf L=(3.176,\allowbreak \,1.965,\allowbreak \,1.715,\allowbreak \,1.685,\allowbreak \,1.683)$. These bounds use $\Phi$ alone. For class $\mathcal C_4$ on the cell $[0.36,0.38]$, $M_a=14.400$ and $\mathbf L=(1.628,\allowbreak \,0.9842,\allowbreak \,0.9827,\allowbreak \,0.9827,\allowbreak \,0.9827)$. These bounds use $\Phi$ alone.

For class $\mathcal C_4$ on the cell $[0.38,0.4]$, $M_a=15.200$ and $\mathbf L=(0.4763,\allowbreak \,0.4763,\allowbreak \,0.4763,\allowbreak \,0.4763,\allowbreak \,0.4763)$. These bounds use the offset window. For class $\mathcal C_4$ on the cell $[0.4,0.42]$, $M_a=16.000$ and $\mathbf L=(0.1663,\allowbreak \,0.1663,\allowbreak \,0.1663,\allowbreak \,0.1663,\allowbreak \,0.1663)$. These bounds use the offset window.

For class $\mathcal C_5$ on the cell $[0.21,0.34]$, $M_a=10.500$ and $\mathbf L=(23.27,\allowbreak \,18.34,\allowbreak \,15.12,\allowbreak \,12.08,\allowbreak \,11.71)$. These bounds use $\Phi$ alone. For class $\mathcal C_5$ on the cell $[0.34,0.38]$, $M_a=17.000$ and $\mathbf L=(0.9787,\allowbreak \,0.9787,\allowbreak \,0.9787,\allowbreak \,0.9787,\allowbreak \,0.9787)$. These bounds use the offset window. For class $\mathcal C_5$ on the cell $[0.38,0.4]$, $M_a=19.000$ and $\mathbf L=(0,\allowbreak \,0,\allowbreak \,0,\allowbreak \,0,\allowbreak \,0)$. These bounds use the offset window.

For class $\mathcal C_6$ on the cell $[0.21,0.3]$, $M_a=10.500$ and $\mathbf L=(21.55,\allowbreak \,15.59,\allowbreak \,12.29,\allowbreak \,8.923,\allowbreak \,8.558)$. These bounds use $\Phi$ alone. For class $\mathcal C_6$ on the cell $[0.3,0.34]$, $M_a=15.000$ and $\mathbf L=(1.928,\allowbreak \,1.928,\allowbreak \,1.928,\allowbreak \,1.928,\allowbreak \,1.928)$. These bounds use the offset window. For class $\mathcal C_6$ on the cell $[0.34,0.36]$, $M_a=17.000$ and $\mathbf L=(0.354,\allowbreak \,0.354,\allowbreak \,0.354,\allowbreak \,0.354,\allowbreak \,0.354)$. These bounds use the offset window.

For class $\mathcal C_7$ on the cell $[0.21,0.25]$, $M_a=10.500$ and $\mathbf L=(23.16,\allowbreak \,13.28,\allowbreak \,9.827,\allowbreak \,6.139,\allowbreak \,5.752)$. These bounds use $\Phi$ alone. For class $\mathcal C_7$ on the cell $[0.25,0.3]$, $M_a=12.500$ and $\mathbf L=(4.637,\allowbreak \,4.637,\allowbreak \,4.571,\allowbreak \,3.928,\allowbreak \,3.86)$. These bounds use the offset window. For class $\mathcal C_7$ on the cell $[0.3,0.32]$, $M_a=15.000$ and $\mathbf L=(0.8915,\allowbreak \,0.8915,\allowbreak \,0.8915,\allowbreak \,0.8915,\allowbreak \,0.8915)$. These bounds use the offset window.

For class $\mathcal C_7$ on the cell $[0.32,0.33]$, $M_a=16.000$ and $\mathbf L=(0.007463,\allowbreak \,0.007463,\allowbreak \,0.007463,\allowbreak \,0.007463,\allowbreak \,0.007463)$. These bounds use the offset window.

For class $\mathcal C_8$ on the cell $[0.21,0.28]$, $M_a=10.500$ and $\mathbf L=(9.14,\allowbreak \,9.14,\allowbreak \,9.14,\allowbreak \,7.879,\allowbreak \,7.289)$. These bounds use the offset window. For class $\mathcal C_8$ on the cell $[0.28,0.3]$, $M_a=14.000$ and $\mathbf L=(0.8525,\allowbreak \,0.8525,\allowbreak \,0.8525,\allowbreak \,0.8525,\allowbreak \,0.8525)$. These bounds use the offset window.

\paragraph{A worked line (the cell $s\in[0.155,0.235]$ of Lemma~\ref{fib:chk:A2}).} $t(0.155)=1.12709$, so $M_0=19\,t(s_a)=21.4147$,
$\bar{\mathsf k}=3.85+\frac{10}{M_0}=4.31697$ and $\bar h=\frac{\bar{\mathsf k}}{M_0+3}=0.17682$. On the cell, with interval enclosures: $G_0=\frac{s^2}{1+s}\le0.047814$,
$G_1=\frac s{1+s}\le0.203464$, $G_2\le0.401764$, $\varphi(s)\ge0.819938$. So $M_0s^2\varphi\ge21.4147\cdot0.155^2\cdot0.819938=0.42184$, and
$2-6G_0-2\mathsf kG_1-\frac{\mathsf k^2G_2}{M_0+3}\ge0.06114$ ($\mathsf k=3.55$) and $\ge-0.35025$ ($\mathsf k=\bar{\mathsf k}$); the base is at least $0.42184-0.35025=0.07159$. For $\bar D_L$:
$N_0=23.4147$, $\alpha_1=2.10019$, $\alpha_2=0.099658$, $\sigma_0\le0.235+\bar h=0.41182$, $\bar e=0.11004$, $\beta_L=19.4622$, so
$\bar D_L\le(0.41182\cdot2.10019)^2/19.4622=0.03844$ ($\bar D_U$ is not needed, as $t\ge1$). Margin $0.07159-0.03844=0.03315$ (the log has $0.0331685$, with
outward rounding and without rounding the intermediate values). The script \texttt{fc\_example\_cell.py} prints these numbers.

% ---------------------------------------------------------------------------------------------------------------
\subsection{The averaging margins and the starting means}\label{fib:apx:rows}

\begin{lemma}\label{fib:chk:rows}
For each subinterval $\iota$, with the caps $\theta(\iota)$ of Data~\ref{p2:tab:theta} and the starting mean $\mu_0(\iota)$ in Data~\ref{fib:tab:mu0}, the margin \eqref{fib:eq:row}
at $\mu=\mu_0(\iota)$, with $b$ the excess of the class of $\iota$ ($b=0$ for $\iota=6,\dots,39$) and $Z_s(\mu)$ replaced by $\max(Z_a,Z_b,Z_c)$ of
Theorem~\ref{fib:thm:central} (with the cells and the cut of Theorem~\ref{fib:thm:wing} or Corollary~\ref{fib:cor:outer} where these apply), is positive,
together with the side conditions of Lemmas~\ref{fib:lem:KM} and~\ref{fib:lem:Zc} and $\ell_0/t_b>L_b$ on every cell used. The values are in
Data~\ref{fib:tab:mu0}. The spike terms are computed with the certified bounds of Lemma~\ref{fib:chk:B3}, not with the rounded-up
values of Data~\ref{fib:tab:b3}; with the rounded-up values all $55$ margins stay positive (the least is then $0.00126$, at
$\iota=36$). Interval (one evaluation per subinterval); $55$ inequalities; least margin $0.00137$ ($\iota=36$), then $0.00143$ ($\iota=39$)
and $0.00144$ ($\iota=53$). Logs \texttt{fc\_check\_B.log} ($\iota=6,\dots,30$), \texttt{fc\_check\_C.log} ($31,\dots,39$), \texttt{fc\_check\_D.log}
(the outer subintervals), tag ROW-v21.
\end{lemma}

The starting mean $\mu_0(\iota)$ is the least value of the scan $19,\allowbreak 20,\allowbreak \dots,\allowbreak 40,\allowbreak 45,\allowbreak 50$, not below $\mu_W$ on the wing and not
below $\mu_D$ on the outer subintervals, at which the averaging margin \eqref{fib:eq:row} of Lemma~\ref{fib:chk:rows} is
positive (each term of \eqref{fib:eq:row} is valid for every $\mu\ge\mu_0$). Data~\ref{fib:tab:mu0} lists it with the margins.

\datablock{Starting means and averaging margins}{fib:tab:mu0}\label{fib:tab:margins}

The following paragraphs specify the subinterval $\iota$, its band, the applicable argument, the cap $\theta(\iota)$, the starting mean $\mu_0(\iota)$, the triple $(Z_a,Z_b,Z_c)$ and the margin in \eqref{fib:eq:row} at $\mu_0(\iota)$. The central argument is Theorem~\ref{fib:thm:central}; a wing class uses Theorem~\ref{fib:thm:wing}; an offset-window class uses Corollary~\ref{fib:cor:outer}. The spike terms are rounded to three digits and the margins are rounded down. The margins use the certified bounds of Lemma~\ref{fib:chk:B3}, rather than their rounded values in Data~\ref{fib:tab:b3}.

For $\iota=0$ in band $0$, class $\mathcal C_1$ applies with $\theta=1$ and $\mu_0=45$. The spike terms are $(Z_a,\allowbreak Z_b,\allowbreak Z_c)=(0.0026,\allowbreak 0.231,\allowbreak 0.0549)$, and the margin is $0.0420$.

For $\iota=1$ in band $1$, class $\mathcal C_1$ applies with $\theta=1$ and $\mu_0=45$. The spike terms are $(Z_a,\allowbreak Z_b,\allowbreak Z_c)=(0.00135,\allowbreak 0.206,\allowbreak 0.0379)$, and the margin is $0.0676$.

For $\iota=2$ in band $2$, class $\mathcal C_1$ applies with $\theta=1$ and $\mu_0=45$. The spike terms are $(Z_a,\allowbreak Z_b,\allowbreak Z_c)=(0.000598,\allowbreak 0.242,\allowbreak 0.0238)$, and the margin is $0.0331$.

For $\iota=3$ in band $3$, class $\mathcal C_2$ applies with $\theta=1$ and $\mu_0=50$. The spike terms are $(Z_a,\allowbreak Z_b,\allowbreak Z_c)=(5.91\cdot10^{-5},\allowbreak 0.199,\allowbreak 0.00513)$, and the margin is $0.0897$.

For $\iota=4$ in band $4$, class $\mathcal C_2$ applies with $\theta=1$ and $\mu_0=50$. The spike terms are $(Z_a,\allowbreak Z_b,\allowbreak Z_c)=(2.19\cdot10^{-5},\allowbreak 0.19,\allowbreak 0.00251)$, and the margin is $0.0997$.

For $\iota=5$ in band $5$, class $\mathcal C_2$ applies with $\theta=1$ and $\mu_0=50$. The spike terms are $(Z_a,\allowbreak Z_b,\allowbreak Z_c)=(1.12\cdot10^{-5},\allowbreak 0.236,\allowbreak 0.00146)$, and the margin is $0.0536$.

For $\iota=6$ in band $6$, the central argument applies with $\theta=\frac{7}{10}$ and $\mu_0=23$. The spike terms are $(Z_a,\allowbreak Z_b,\allowbreak Z_c)=(0.0192,\allowbreak 0.058,\allowbreak 0.017)$, and the margin is $0.00435$.

For $\iota=7$ in band $7$, the central argument applies with $\theta=\frac{7}{10}$ and $\mu_0=23$. The spike terms are $(Z_a,\allowbreak Z_b,\allowbreak Z_c)=(0.0168,\allowbreak 0.0567,\allowbreak 0.0128)$, and the margin is $0.00647$.

For $\iota=8$ in band $8$, the central argument applies with $\theta=\frac{7}{10}$ and $\mu_0=23$. The spike terms are $(Z_a,\allowbreak Z_b,\allowbreak Z_c)=(0.0147,\allowbreak 0.0551,\allowbreak 0.00967)$, and the margin is $0.00873$.

For $\iota=9$ in band $9$, the central argument applies with $\theta=\frac{7}{10}$ and $\mu_0=23$. The spike terms are $(Z_a,\allowbreak Z_b,\allowbreak Z_c)=(0.0128,\allowbreak 0.0513,\allowbreak 0.00749)$, and the margin is $0.0130$. For $\iota=10$ in band $9$, the central argument applies with $\theta=\frac{13}{20}$ and $\mu_0=20$. The spike terms are $(Z_a,\allowbreak Z_b,\allowbreak Z_c)=(0.0254,\allowbreak 0.0613,\allowbreak 0.00445)$, and the margin is $0.0179$.

For $\iota=11$ in band $10$, the central argument applies with $\theta=\frac{3}{5}$ and $\mu_0=19$. The spike terms are $(Z_a,\allowbreak Z_b,\allowbreak Z_c)=(0.0435,\allowbreak 0.0899,\allowbreak 0.00582)$, and the margin is $0.0252$.

For $\iota=12$ in band $11$, the central argument applies with $\theta=\frac{3}{5}$ and $\mu_0=19$. The spike terms are $(Z_a,\allowbreak Z_b,\allowbreak Z_c)=(0.0409,\allowbreak 0.0862,\allowbreak 0.00491)$, and the margin is $0.0292$. For $\iota=13$ in band $11$, the central argument applies with $\theta=\frac{3}{5}$ and $\mu_0=19$. The spike terms are $(Z_a,\allowbreak Z_b,\allowbreak Z_c)=(0.0424,\allowbreak 0.0901,\allowbreak 0.00512)$, and the margin is $0.0253$. For $\iota=14$ in band $11$, the central argument applies with $\theta=\frac{3}{5}$ and $\mu_0=19$. The spike terms are $(Z_a,\allowbreak Z_b,\allowbreak Z_c)=(0.0333,\allowbreak 0.0714,\allowbreak 0.00314)$, and the margin is $0.0441$.

For $\iota=15$ in band $12$, the central argument applies with $\theta=\frac{3}{5}$ and $\mu_0=19$. The spike terms are $(Z_a,\allowbreak Z_b,\allowbreak Z_c)=(0.0399,\allowbreak 0.0862,\allowbreak 0.00432)$, and the margin is $0.0293$. For $\iota=16$ in band $12$, the central argument applies with $\theta=\frac{3}{5}$ and $\mu_0=19$. The spike terms are $(Z_a,\allowbreak Z_b,\allowbreak Z_c)=(0.0386,\allowbreak 0.0846,\allowbreak 0.00383)$, and the margin is $0.0309$. For $\iota=17$ in band $12$, the central argument applies with $\theta=\frac{3}{5}$ and $\mu_0=19$. The spike terms are $(Z_a,\allowbreak Z_b,\allowbreak Z_c)=(0.034,\allowbreak 0.0752,\allowbreak 0.00293)$, and the margin is $0.0404$.

For $\iota=18$ in band $13$, the central argument applies with $\theta=\frac{3}{5}$ and $\mu_0=19$. The spike terms are $(Z_a,\allowbreak Z_b,\allowbreak Z_c)=(0.0343,\allowbreak 0.0764,\allowbreak 0.00289)$, and the margin is $0.0393$. For $\iota=19$ in band $13$, the central argument applies with $\theta=\frac{3}{5}$ and $\mu_0=19$. The spike terms are $(Z_a,\allowbreak Z_b,\allowbreak Z_c)=(0.0376,\allowbreak 0.0844,\allowbreak 0.00336)$, and the margin is $0.0312$. For $\iota=20$ in band $13$, the central argument applies with $\theta=\frac{3}{5}$ and $\mu_0=19$. The spike terms are $(Z_a,\allowbreak Z_b,\allowbreak Z_c)=(0.0382,\allowbreak 0.0869,\allowbreak 0.00327)$, and the margin is $0.0288$.

For $\iota=21$ in band $14$, the central argument applies with $\theta=\frac{3}{5}$ and $\mu_0=19$. The spike terms are $(Z_a,\allowbreak Z_b,\allowbreak Z_c)=(0.0356,\allowbreak 0.0817,\allowbreak 0.00279)$, and the margin is $0.0341$. For $\iota=22$ in band $14$, the central argument applies with $\theta=\frac{3}{5}$ and $\mu_0=19$. The spike terms are $(Z_a,\allowbreak Z_b,\allowbreak Z_c)=(0.0393,\allowbreak 0.0914,\allowbreak 0.00319)$, and the margin is $0.0243$. For $\iota=23$ in band $14$, the central argument applies with $\theta=\frac{3}{5}$ and $\mu_0=19$. The spike terms are $(Z_a,\allowbreak Z_b,\allowbreak Z_c)=(0.037,\allowbreak 0.0867,\allowbreak 0.00277)$, and the margin is $0.0290$.

For $\iota=24$ in band $15$, the central argument applies with $\theta=\frac{3}{5}$ and $\mu_0=19$. The spike terms are $(Z_a,\allowbreak Z_b,\allowbreak Z_c)=(0.0443,\allowbreak 0.106,\allowbreak 0.00361)$, and the margin is $0.00983$. For $\iota=25$ in band $15$, the central argument applies with $\theta=\frac{3}{5}$ and $\mu_0=19$. The spike terms are $(Z_a,\allowbreak Z_b,\allowbreak Z_c)=(0.039,\allowbreak 0.0939,\allowbreak 0.00276)$, and the margin is $0.0218$.

For $\iota=26$ in band $16$, the central argument applies with $\theta=\frac{3}{5}$ and $\mu_0=19$. The spike terms are $(Z_a,\allowbreak Z_b,\allowbreak Z_c)=(0.0393,\allowbreak 0.0951,\allowbreak 0.00272)$, and the margin is $0.0206$. For $\iota=27$ in band $16$, the central argument applies with $\theta=\frac{3}{5}$ and $\mu_0=19$. The spike terms are $(Z_a,\allowbreak Z_b,\allowbreak Z_c)=(0.0433,\allowbreak 0.106,\allowbreak 0.0031)$, and the margin is $0.00961$. For $\iota=28$ in band $16$, the central argument applies with $\theta=\frac{3}{5}$ and $\mu_0=19$. The spike terms are $(Z_a,\allowbreak Z_b,\allowbreak Z_c)=(0.0411,\allowbreak 0.101,\allowbreak 0.00274)$, and the margin is $0.0142$.

For $\iota=29$ in band $17$, the central argument applies with $\theta=\frac{3}{5}$ and $\mu_0=19$. The spike terms are $(Z_a,\allowbreak Z_b,\allowbreak Z_c)=(0.0453,\allowbreak 0.113,\allowbreak 0.00312)$, and the margin is $0.00261$. For $\iota=30$ in band $17$, the central argument applies with $\theta=\frac{3}{5}$ and $\mu_0=20$. The spike terms are $(Z_a,\allowbreak Z_b,\allowbreak Z_c)=(0.0321,\allowbreak 0.105,\allowbreak 0.00319)$, and the margin is $0.0246$.

For $\iota=31$ in band $18$, class $\mathcal W_1$ applies with $\theta=\frac{3}{5}$ and $\mu_0=20$. The spike terms are $(Z_a,\allowbreak Z_b,\allowbreak Z_c)=(0.0382,\allowbreak 0.127,\allowbreak 0.00419)$, and the margin is $0.00251$. For $\iota=32$ in band $18$, class $\mathcal W_1$ applies with $\theta=\frac{3}{5}$ and $\mu_0=21$. The spike terms are $(Z_a,\allowbreak Z_b,\allowbreak Z_c)=(0.0281,\allowbreak 0.124,\allowbreak 0.00451)$, and the margin is $0.0187$.

For $\iota=33$ in band $19$, class $\mathcal W_2$ applies with $\theta=\frac{3}{5}$ and $\mu_0=21$. The spike terms are $(Z_a,\allowbreak Z_b,\allowbreak Z_c)=(0.0279,\allowbreak 0.125,\allowbreak 0.00428)$, and the margin is $0.0172$. For $\iota=34$ in band $19$, class $\mathcal W_2$ applies with $\theta=\frac{3}{5}$ and $\mu_0=21$. The spike terms are $(Z_a,\allowbreak Z_b,\allowbreak Z_c)=(0.0298,\allowbreak 0.136,\allowbreak 0.00461)$, and the margin is $0.00614$.

For $\iota=35$ in band $20$, class $\mathcal W_2$ applies with $\theta=\frac{3}{5}$ and $\mu_0=22$. The spike terms are $(Z_a,\allowbreak Z_b,\allowbreak Z_c)=(0.0222,\allowbreak 0.137,\allowbreak 0.00505)$, and the margin is $0.0171$. For $\iota=36$ in band $20$, class $\mathcal W_2$ applies with $\theta=\frac{3}{5}$ and $\mu_0=22$. The spike terms are $(Z_a,\allowbreak Z_b,\allowbreak Z_c)=(0.0241,\allowbreak 0.153,\allowbreak 0.00553)$, and the margin is $0.00137$.

For $\iota=37$ in band $21$, class $\mathcal W_3$ applies with $\theta=\frac{3}{5}$ and $\mu_0=26$. The spike terms are $(Z_a,\allowbreak Z_b,\allowbreak Z_c)=(0.00579,\allowbreak 0.123,\allowbreak 0.00581)$, and the margin is $0.0690$.

For $\iota=38$ in band $22$, class $\mathcal W_3$ applies with $\theta=\frac{7}{10}$ and $\mu_0=32$. The spike terms are $(Z_a,\allowbreak Z_b,\allowbreak Z_c)=(0.000577,\allowbreak 0.0965,\allowbreak 0.00579)$, and the margin is $0.00596$. For $\iota=39$ in band $22$, class $\mathcal W_3$ applies with $\theta=\frac{7}{10}$ and $\mu_0=36$. The spike terms are $(Z_a,\allowbreak Z_b,\allowbreak Z_c)=(0.000124,\allowbreak 0.123,\allowbreak 0.00628)$, and the margin is $0.00143$.

For $\iota=40$ in band $23$, class $\mathcal C_3$ applies with $\theta=\frac{7}{10}$ and $\mu_0=50$. The spike terms are $(Z_a,\allowbreak Z_b,\allowbreak Z_c)=(8.1\cdot10^{-7},\allowbreak 0.035,\allowbreak 0.000589)$, and the margin is $0.138$. For $\iota=41$ in band $23$, class $\mathcal C_3$ applies with $\theta=\frac{7}{10}$ and $\mu_0=50$. The spike terms are $(Z_a,\allowbreak Z_b,\allowbreak Z_c)=(8.73\cdot10^{-7},\allowbreak 0.0434,\allowbreak 0.00066)$, and the margin is $0.129$. For $\iota=42$ in band $23$, class $\mathcal C_3$ applies with $\theta=\frac{7}{10}$ and $\mu_0=50$. The spike terms are $(Z_a,\allowbreak Z_b,\allowbreak Z_c)=(1.01\cdot10^{-6},\allowbreak 0.0663,\allowbreak 0.000827)$, and the margin is $0.106$.

For $\iota=43$ in band $24$, class $\mathcal C_3$ applies with $\theta=\frac{7}{10}$ and $\mu_0=50$. The spike terms are $(Z_a,\allowbreak Z_b,\allowbreak Z_c)=(1.13\cdot10^{-6},\allowbreak 0.127,\allowbreak 0.00108)$, and the margin is $0.0281$.

For $\iota=44$ in band $25$, class $\mathcal C_4$ applies with $\theta=\frac{9}{10}$ and $\mu_0=40$. The spike terms are $(Z_a,\allowbreak Z_b,\allowbreak Z_c)=(3.84\cdot10^{-5},\allowbreak 0.038,\allowbreak 0.005)$, and the margin is $0.0303$. For $\iota=45$ in band $25$, class $\mathcal C_4$ applies with $\theta=\frac{19}{20}$ and $\mu_0=45$. The spike terms are $(Z_a,\allowbreak Z_b,\allowbreak Z_c)=(5.16\cdot10^{-6},\allowbreak 0.0293,\allowbreak 0.00197)$, and the margin is $0.00905$. For $\iota=46$ in band $25$, class $\mathcal C_4$ applies with $\theta=\frac{19}{20}$ and $\mu_0=50$. The spike terms are $(Z_a,\allowbreak Z_b,\allowbreak Z_c)=(6.69\cdot10^{-7},\allowbreak 0.0323,\allowbreak 0.000785)$, and the margin is $0.0220$.

For $\iota=47$ in band $26$, class $\mathcal C_5$ applies with $\theta=1$ and $\mu_0=50$. The spike terms are $(Z_a,\allowbreak Z_b,\allowbreak Z_c)=(7.75\cdot10^{-7},\allowbreak 0.0572,\allowbreak 0.000961)$, and the margin is $0.0790$. For $\iota=48$ in band $26$, class $\mathcal C_5$ applies with $\theta=1$ and $\mu_0=50$. The spike terms are $(Z_a,\allowbreak Z_b,\allowbreak Z_c)=(9.04\cdot10^{-7},\allowbreak 0.0792,\allowbreak 0.0012)$, and the margin is $0.0568$. For $\iota=49$ in band $26$, class $\mathcal C_5$ applies with $\theta=1$ and $\mu_0=50$. The spike terms are $(Z_a,\allowbreak Z_b,\allowbreak Z_c)=(1.03\cdot10^{-6},\allowbreak 0.106,\allowbreak 0.00146)$, and the margin is $0.0294$.

For $\iota=50$ in band $27$, class $\mathcal C_6$ applies with $\theta=1$ and $\mu_0=50$. The spike terms are $(Z_a,\allowbreak Z_b,\allowbreak Z_c)=(1.16\cdot10^{-6},\allowbreak 0.0152,\allowbreak 0.00177)$, and the margin is $0.120$. For $\iota=51$ in band $27$, class $\mathcal C_6$ applies with $\theta=1$ and $\mu_0=50$. The spike terms are $(Z_a,\allowbreak Z_b,\allowbreak Z_c)=(4.68\cdot10^{-6},\allowbreak 0.0695,\allowbreak 0.00762)$, and the margin is $0.0660$. For $\iota=52$ in band $27$, class $\mathcal C_6$ applies with $\theta=1$ and $\mu_0=50$. The spike terms are $(Z_a,\allowbreak Z_b,\allowbreak Z_c)=(5.29\cdot10^{-6},\allowbreak 0.089,\allowbreak 0.00917)$, and the margin is $0.0192$.

For $\iota=53$ in band $28$, class $\mathcal C_7$ applies with $\theta=1$ and $\mu_0=50$. The spike terms are $(Z_a,\allowbreak Z_b,\allowbreak Z_c)=(1.71\cdot10^{-5},\allowbreak 0.106,\allowbreak 0.0382)$, and the margin is $0.00144$.

For $\iota=54$ in band $29$, class $\mathcal C_8$ applies with $\theta=1$ and $\mu_0=50$. The spike terms are $(Z_a,\allowbreak Z_b,\allowbreak Z_c)=(7.16\cdot10^{-6},\allowbreak 0.0308,\allowbreak 0.0176)$, and the margin is $0.0763$.

\paragraph{The boxes of family F6 that are used.} Of the $229$ boxes of Lemma~\ref{p2:lem:strips} on $[m_\ast(\iota),50]$, a box is
not used when its whole range of means lies in $[\mu_0(\iota),50]$, that is, when its lower end is at least $\mu_0(\iota)$;
boxes that straddle $\mu_0(\iota)$ are used. So $86$ boxes are not used and $143$ are. The boxes not used are, per subinterval:
$1$ on $\iota=6,7,9$; $2$ on $\iota=8$, on $\iota=10$ to $16$ and on $\iota=31$ to $34$; $3$ on $\iota=17,18$; $4$ on $\iota=19$ to
$28$; $5$ on $\iota=29,30$; $1$ on $\iota=35$ to $37$; none on $\iota=0$ to $5$ and $\iota=38$ to $54$, where every box straddles
$\mu_0(\iota)$ or lies below it. The $550$ boxes below $m_\ast(\iota)$ are all used: $m_\ast(\iota)<\mu_0(\iota)$ on every
subinterval (the largest $m_\ast$ is $\frac{68}3$, at $\iota=0$, against $\mu_0=45$; over $\iota\ge6$ it is $\frac{198}{13}<19$;
the least gap $\mu_0-m_\ast$ is $\frac{125}{23}$, at $\iota=16$), so each of them meets $[m_{25},\mu_0)$, where case~(a) of the
proof of Theorem~\ref{p2:thm:O4} uses it. The script \texttt{fc\_ledger\_f6.py} reads the $229$ boxes from the box files of family
F6 (Section~\ref{sec:computer:data}) and compares their lower ends with $\mu_0(\iota)$ in exact rational arithmetic; its log
\texttt{fc\_ledger\_f6.log} gives these counts (the least box end on every subinterval is $m_\ast(\iota)$).

\begin{remark}[numerical observations]\label{fib:rem:diag}
The following are floating-point or independent observations, not used in any proof. The exact weighted fibers sit above the closed-form
targets on a sample: least excess $0.035$ for Theorem~\ref{fib:thm:env}, $0.255$ to $0.456$ for the wing classes and $1.15$ to $3.15$ for the offset-window
classes.

The three thinnest margins of Lemma~\ref{fib:chk:rows} were also recomputed with independently written rigorous code and
different, tighter spike bounds, which give $0.00267$ ($\iota=36$), $0.00616$ ($\iota=39$) and $0.0199$ ($\iota=53$); with the
spike bounds as stated, the margins in Data~\ref{fib:tab:mu0} are reproduced (exactly at $\iota=36$ and $39$). The margins are thin because the spike bounds are crude, not because the fibers
are: on a sample, the exact spike terms are about $0.005$ ($\iota=36$) and $0.0004$ ($\iota=39$), against the certified $0.15$ and $0.12$.
\end{remark}

\end{document}